\documentclass[reqno,centertags,12pt]{amsart}

\usepackage{amsmath}
\usepackage{amscd}
\usepackage{stackrel}
\usepackage{amssymb}
\usepackage{amsthm}
\usepackage{bbm}
\usepackage{latexsym}
\usepackage{mathrsfs}
\usepackage{verbatim}
\usepackage{tikz-cd}
\usepackage{mathtools}
\usepackage[hidelinks]{hyperref}

\usepackage[shortlabels]{enumitem}

\usepackage{xcolor, cite}

\usepackage{tikz}
\pgfdeclarelayer{nodelayer}
\pgfdeclarelayer{edgelayer}
\pgfsetlayers{edgelayer,nodelayer,main}
\tikzstyle{arrow}=[draw=black,arrows=-latex]

-\marginparwidth \oddsidemargin -4mm \evensidemargin \oddsidemargin

\newtheorem{theorem}{Theorem}[section]

\newtheorem{proposition}[theorem]{Proposition}
\newtheorem{lemma}[theorem]{Lemma}
\newtheorem{corollary}[theorem]{Corollary}
\theoremstyle{definition}
\newtheorem{definition}[theorem]{Definition}

\newcounter{smalllist}

\DeclareMathOperator*{\sgn}{sgn}

\allowdisplaybreaks
\numberwithin{equation}{section}

\newcommand{\abs}[1]{\left\lvert#1\right\rvert}

\newcommand{\norm}[1]{\left\|#1\right\|}

\newcommand{\set}[1]{\left\{ #1 \right\}}
\newcommand{\seq}[1]{\left\{ #1 \right\}}

\newcommand{\lb}{\label}

\newcommand{\supp}{\text{\rm{supp}}}

\newcommand{\beq}{\begin{equation}}
\newcommand{\eeq}{\end{equation}}

\newcommand{\bal}{\begin{align}}
\newcommand{\eal}{\end{align}}
\newcommand{\bals}{\begin{align*}}
\newcommand{\eals}{\end{align*}}

\newcommand{\bbN}{{\mathbb{N}}}
\newcommand{\bbR}{{\mathbb{R}}}

\newcommand{\bbZ}{{\mathbb{Z}}}

\newcommand{\bbT}{{\mathbb{T}}}

\newcommand{\eps}{\varepsilon}

\newcommand{\tht}{\theta}

\begin{document}
\title[Touching g-SQG Patches on the Plane]
{Well-Posedness and Finite Time Singularity \\ for Touching g-SQG Patches on the Plane}

\author{Junekey Jeon and Andrej Zlato\v{s}}

\address{\noindent Department of Mathematics \\ University of
California San Diego \\ La Jolla, CA 92093 \newline Email: \tt
zlatos@ucsd.edu,
j6jeon@ucsd.edu}

\begin{abstract}
We prove local well-posedness as well as  singularity formation for the g-SQG patch model on the plane (so on a domain without a boundary), with  $\alpha\in(0,\frac 16]$ and  patches being allowed to  touch each other.  We do this by bypassing any auxiliary contour equations and tracking patch boundary curves directly instead of their parametrizations.  In our results, which are sharp in terms of $\alpha$, the patch boundaries have $L^2$ curvatures   and a singularity occurs when at least one of these $L^2$-norms blows up in finite time.
\end{abstract}

\maketitle

%%%%%%%%%%%%%%%%%%%%%%%%%%%%%%%%%%%%%%%%%%%%%%
\section{Introduction and main results} \label{S1}

\bigskip
    \textbf{Background and motivation.} 
The best-known family of active scalar PDE in two dimensions is 
\beq\lb{111.25}
\partial_t \theta+u\cdot\nabla \theta=0
\eeq
with 
\beq\lb{111.26}
u\coloneqq-\nabla^\perp(-\Delta)^{-1+\alpha}\theta
\eeq
and $\alpha\in[0,\frac 12]$, where  $(x_1,x_2)^{\perp}\coloneqq (-x_{2},x_{1})$ and $\nabla^{\perp}\coloneqq(-\partial_{x_{2}},\partial_{x_{1}})$.  This becomes the vorticity form of the 2D Euler equation when $\alpha=0$ (with $u$ the fluid velocity and $\tht=\nabla^\perp\cdot u$ its vorticity), the inviscid surface quasi-geostrophic equation (SQG) when $\alpha=\frac 12$, and the generalized SQG equation (g-SQG) when $\alpha\in (0,\frac 12)$.  The latter two equations appear in atmospheric science and geophysical models (see, e.g., \cite{Blu, Ped, PHS, Smith}), and their first rigorous studies were undertaken by Constantin, Majda, and Tabak \cite{CMT} and by Constantin, Iyer, and Wu \cite{CIW}, respectively.
It has been well known for almost a century, since the works of H\" older \cite{Holder} and Wolibner \cite{Wolibner}, that solutions to the 2D Euler equation remain globally regular, while only local regularity in time has been obtained for SQG and g-SQG.  This is not surprising because the latter two PDE are more singular than 2D Euler, which was proved by Kiselev and \v Sver\' ak over a decade ago to possess smooth solutions with double-exponentially growing gradients on a disc \cite{KS} (see \cite{Xu} for some other bounded domains), while the second author recently extended this to the half-plane and some other unbounded domains
%, and also identified the maximal possible rate of this growth 
\cite{ZlaEulermax}.

Existence of finite time singularities for SQG and g-SQG on domains without a boundary (i.e., $\bbR^2$ and $\bbT^2$) was originally suggested by the numerical study \cite{CorFonManRod}, but today it still remains one of the most important open questions in fluid dynamics.  \hbox{Nevertheless,} it has been partially answered when a boundary is present.  First, Kiselev, Ryzhik, Yao, and the second author showed a decade ago that finite time singularity does happen for the  {\it g-SQG patch problem} on the half-plane when $\alpha\in(0,\frac 1{24})$ \cite{KRYZ}.  The patch problem is a special case of \eqref{111.25}--\eqref{111.26} that concerns weak solutions that are linear combinations of characteristic functions of sufficiently regular sets (having $H^3$ boundaries in \cite{KRYZ}) with initially disjoint boundaries.  It was introduced by Rodrigo for SQG \cite{Rodrigo} and Gancedo for g-SQG \cite{Gancedo} on the plane (see also the earlier works \cite{Chemin, bc, Butt, Majda} for the 2D Euler case), and singularity develops either when a patch boundary loses the prescribed regularity or a touch happens in finite time (either a self-touch or a touch of two distinct patch boundaries).  Later, Gancedo and Patel \cite{GanPat} showed that the same result holds for $H^2$ patches on the half-plane and $\alpha\in(0,\frac 16)$, and the authors proved in  \cite{JeoZla}  that in fact finite time singularity must always be accompanied by a loss of $C^{1,\frac{2\alpha}{1-2\alpha}}$ regularity of some patch boundary when $\alpha\in(0,\frac 14]$ (this then also applies in the settings of \cite{KRYZ, GanPat}; see also \cite{GanStr, KisLuo}).
Finally, the second author recently showed that \eqref{111.25}--\eqref{111.26} is well-posed on the half-plane  in appropriate spaces of differentiable (i.e., non-patch) solutions when $\alpha\in(0,\frac 14]$, and finite time singularity can develop  from even smooth initial data and for any such $\alpha$ in this setting \cite{ZlaSQGsing}.

Presence of a domain boundary is crucial in all these singularity results and so, in particular, they have not been extended to the whole plane $\bbR^2$.  One of the motivations for this paper is  extension of the finite time patch singularity results above to $\bbR^2$, which we do in Theorem~\ref{T2.7} below.  
The results in \cite{KRYZ,GanPat} both feature a pair of symmetrically positioned (with respect to the vertical axis) patches on $\bbR\times\bbR^+$ with opposite strengths that are touching the half-plane boundary, and approach each other along it quickly enough for them to  lose the relevant boundary regularity in finite time.
% (as was proved in \cite{JeoZla}, no patch collisions can happen without that).  
We show that the same dynamic  plays out on the whole plane when the initial data is made up of such a pair of patches plus their reflections across the horizontal axis (again with opposite strengths, yielding a ``doubly odd'' scenario).  This means that we have two pairs of patches that touch initially, which seems to pose a problem.  Indeed, it appears that no previous results for patch-like solutions involve initial data with patch touches, except for the 2D Euler case, where  Yudovich well-posedness theory for general bounded vorticities on smooth domains \cite{Yud} naturally allows one to consider any number of patches of arbitrary shapes (e.g., patches with cusps were studied in \cite{Danchin, Danchin2, JeoElg, JeoElg2}).
% such as the Muskat problem \cite{}) 
This may suggest that one should not expect a general local well-posedness result in this setting, except when one a priori restricts the space of solutions to those that are odd across the horizontal axis and touches can only happen between a patch and its reflection.  Of course, the latter  just coincides with the no-touch half-plane  setting and is not a proper well-posedness result on $\bbR^2$, but it also a priori excludes certain high frequency ``unstable modes'' as we discuss at the end of this introduction.  

Nevertheless, our other main result --- and second motivation ---   is precisely  that: local well-posedness for the g-SQG patch problem with $H^2$ patch boundaries and $\alpha\in(0,\frac 16]$ that allows patch boundaries to touch (see Theorem \ref{T2.5}).  However, we have to exclude initial data with exterior touches of patches whose strengths have the same sign (which includes any self-touches) as well as interior touches of patches with opposite sign strengths.  This is not only a technical requirement, as we explain at the end of this introduction.  But unlike in the reflected half-plane scenario described above, we make no requirement on the strengths of the touching patches, only on their signs.  We also note that one can think of pairs of disjoint patches that touch along a curve and have  opposite sign strengths as sharp interface versions of certain special continuous  solutions to these types of PDE ---  such as the Lamb-Chaplygin dipole for the 2D Euler equation \cite{Lamb, Chaplygin}, which is a traveling wave solution  that changes sign across a line parallel to its direction of travel.  Moreover, our approach can be used to study low regularity continuous solutions to \eqref{111.25}--\eqref{111.26} with more regular level sets as well \cite{JeoZla3}.

%  For instance, if two patches with the same strengths (i.e., same coefficients in front of their characteristic functions in the expression for the solution $\theta$) have $H^2$ boundaries and touch, one can use the integral form \eqref{111.30} of the Biot-Savart law \eqref{111.26} to show that

The third motivation for our work is introduction of a new framework for the study of patch problems associated with active scalar PDE.  The usual approach to these involves derivation of a {\it contour PDE} for some specific parametrizations of the patch boundaries, followed by a proof of well-posedness for it and a proof that the unique solution also represents a weak solution to the original PDE \eqref{111.25}--\eqref{111.26} (see, e.g., \cite{GanNguPat, KisYaoZla, GanPat, Gancedo,CorCorGan}).  Unfortunately, the contour PDE can often introduce non-trivial complications: it is far from unique because so are patch boundary parametrizations, it may involve quite singular integral kernels,
% that may need to be adjusted via reparametrizations, 
and its solutions may develop singularities even if the boundary curves they represent remain regular.  

We therefore wish to dispense with the contour equation ``middleman'' and work directly with patch boundaries as (non-parametrized and hence uniquely represented) curves in $\bbR^2$.  (We will still work with curve parametrizations when finding solutions to certain smooth approximations of \eqref{111.25}--\eqref{111.26}, but this will be much easier.)  An additional benefit of this will be a simplification of another related concept, the {\it arc-chord constant} ${\rm ac}(\tilde \gamma):=\sup_{\xi\neq \zeta} \frac{|\xi-\zeta|}{|\tilde\gamma(\xi)-\tilde\gamma(\zeta)|}$ for  curve parametrizations $\tilde\gamma:\bbT\to\bbR^2$.   Finiteness of ${\rm ac}(\tilde \gamma)$ means that the (closed) curve is simple and the parametrization $\tilde\gamma$ is non-degenerate.  It can also be used to estimate how close the curve comes to touching (or crossing) itself, but only from one side because it may be large for close-to-degenerate parametrizations of curves that do not have distinct segments very close to each other.  However,  once we work directly with curves, there is no need to ensure non-degeneracy of some specific parametrization, and we can replace ${\rm ac}(\tilde \gamma)$ by a simpler quantity that only controls how close distinct curve segments are to each other.  If we define it properly --- see below for details, but the idea is to use $\sup_{\xi,\zeta} \frac 1{|\tilde\gamma(\xi)-\tilde\gamma(\zeta)|}$, with the supremum taken over all pairs $\xi,\zeta\in\bbT$ such that the arclength distance of $\tilde\gamma(\xi)$ and $\tilde\gamma(\zeta)$ is no less than the reciprocal of the square of the $\dot H^2$-norm of the curve, with the latter signifying that the points lie on ``distinct'' curve segments --- this will also allow us to control the rate of change of this quantity in time much more easily than that of the arc-chord constant.

Finally, we note that our results are sharp in terms of $\alpha$.  Indeed, when a patch whose boundary contains $\{(x_1,|x_1|^\beta)\,:\, |x_1|< 1\}$ touches its reflection across the horizontal axis and their strengths are opposite, it is not difficult to compute that the horizontal component of the corresponding velocity $u$ behaves like $C-cx_2^{1-2\alpha}$ near the origin (see also \cite{ZlaSQGsing}), which is $C-c|x_1|^{\beta(1-2\alpha)}$ on the above boundary segment.  If we take $\alpha>\frac 16$ and $\beta\in(\frac 32,\frac 1{1-2\alpha})$  (so that the patch boundary can be $H^2$),  the patch would have to experience instantaneous loss of $H^2$ regularity because ${\beta(1-2\alpha)}<1$.  Therefore  the range $\alpha\in(0,\frac 16]$  is the maximal possible in which our main results can hold.  We also mention that $H^k$-level of boundary regularity is the right one for the g-SQG patch problem because Kiselev and Luo showed that the problem is ill-posed in H\" older spaces as well as in Sobolev spaces with $p\neq 2$ \cite{KisLuo2}.

With the above %these motivations 
in mind let us now present the basic setup for our analysis, followed by statements of our main results, an overview of the proof of the main well-posedness result, and an explanation of why certain patch touches can be allowed and others must be excluded.

\bigskip
    \textbf{Spaces of closed curves in $\bbR^2$.} 
%In this paper, we distinguish the notion of \emph{paths} (or \emph{parametrized curves})
%and \emph{curves}. 
% (the reason for writing $\tilde{\gamma}$ instead of
%just $\gamma$ here is due to our notational convention
%that is elaborated after Definition~\ref{D2.3}).
A \emph{closed path} is any $\tilde{\gamma}\in C(\bbT;\bbR^{2})$, with $\bbT\coloneqq\bbR/\bbZ$.  In order to identify paths that parametrize the same curves, for closed paths $ \tilde{\gamma}_{1},\tilde{\gamma}_{2}$ we let
%define the pseudometric
\beq\lb{111.24}
    d_{\mathrm{F}} (\tilde{\gamma}_{1},\tilde{\gamma}_{2})\coloneqq
    \inf_{\phi}    \norm{\tilde{\gamma}_{1} - \tilde{\gamma}_{2}\circ\phi}_{L^{\infty}(\bbT)},
\eeq
 where the infimum is taken over all orientation-preserving homeomorphisms
$\phi \colon\bbT\to\bbT$.   Lemma~\ref{LA.1} below shows that  $d_\mathrm{F}$ would not change if we used orientation-preserving diffeomorphisms instead.  Also, there need not always be
$\phi\colon\bbT\to\bbT$ minimizing the $\inf$ in \eqref{111.24}, so in particular,
$d_{\mathrm{F}}(\tilde{\gamma}_{1},\tilde{\gamma}_{2}) = 0$
does not imply existence of $\phi$ such that
$\tilde{\gamma}_{1} = \tilde{\gamma}_{2}\circ\phi$.
%The corresponding metric is called the \emph{Fr\'{e}chet metric} (references...).
%
%To define the notion of \emph{closed curves}, consider the space $C(\bbT;\bbR^{2})$ of all
%continuous closed paths in $\bbR^{2}$ endowed with the uniform norm
%$\norm{\,\cdot\,}_{L^{\infty}(\bbT)}$. On this Banach space, we define a pseudometric
%\[
%    d_{\mathrm{F}}\colon (\tilde{\gamma}_{1},\tilde{\gamma}_{2})\mapsto
%    \inf_{\phi}
%    \norm{\tilde{\gamma}_{1} - \tilde{\gamma}_{2}\circ\phi}_{L^{\infty}(\bbT)},
%\]
%; i.e., any homeomorphism obtained by projecting a
%strictly increasing homeomorphism $\tilde{\phi}\colon \bbR\to \bbR$ such that
%$\tilde{\phi}(x+n) = \tilde{\phi}(x)+n$ holds for all $x\in\bbR$ and $n\in\bbZ$.
%By identifying paths with the zero distance to each other, we obtain the notion of
%\emph{Fr\'{e}chet metric} (references needed).

\begin{definition}\label{D2.1}
    Let $\operatorname{CC}(\bbR^{2})$ be the quotient space of
    $C(\bbT;\bbR^{2})$ with respect to the equivalence relation 
    $\tilde{\gamma}_{1}\sim\tilde{\gamma}_{2}$ iff
    $d_{\mathrm{F}}(\tilde{\gamma}_{1},\tilde{\gamma}_{2}) = 0$.  
    The corresponding metric on $\operatorname{CC}(\bbR^{2})$, also denoted $d_{\mathrm{F}}$, is
    the \emph{Fr\'{e}chet metric}.
    A \emph{closed curve} in $\bbR^{2}$ is any $\gamma\in \operatorname{CC}(\bbR^{2})$, and a \emph{parametrization} of $\gamma$ is any
      $\tilde{\gamma}\in C(\ell\bbT;\bbR^{2})$ with $\ell\geq 0$ such that
    $\tilde{\gamma}(\ell \,\cdot)\in C(\bbT;\bbR^{2})$ is a representative of $\gamma$.
    The \emph{image} of $\gamma$, denoted $\operatorname{im}(\gamma)$, is the image of any parametrization of $\gamma$.
    
    The \emph{length} $\ell(\gamma)$ of $\gamma$ is the total variation of
        any parametrization of $\gamma$. If $\ell(\gamma)<\infty$, then $\gamma$ is \emph{rectifiable}, and the set of all such closed curves in $\bbR^{2}$ is denoted
        $\operatorname{RC}(\bbR^{2})$.  An {\it arclength parametrization} of a rectifiable $\gamma$ is any parametrization $\tilde{\gamma}\in C(\ell (\gamma) \bbT;\bbR^{2})$ of $\gamma$ satisfying $|\partial_\xi \tilde\gamma(\xi)| =1$ for almost all $\xi\in \ell (\gamma) \bbT$, while a {\it constant-speed parametrization} of  $\gamma$ is any parametrization $\tilde{\gamma}\in C(\bbT;\bbR^{2})$ of $\gamma$ satisfying $|\partial_\xi \tilde\gamma(\xi)| =\ell (\gamma)$ for almost all $\xi\in  \bbT$.
\end{definition}

%\textit{Remark}. In the proof of the main results (Theorem~\ref{T2.5} and Theorem~\ref{T2.6}),
%all curves we consider will be $C^{1,\beta}$ or $H^{2}$
%(see Definition~\ref{D2.3} below).\\

{\it Remark.}
In the rest of the paper, we will denote arclength parametrizations of a closed curve $\gamma$ by $\gamma(\cdot)$ (or just $\gamma$), while general parametrizations will have $\tilde{}$ (including constant-speed ones, which we employ to study convergence in $\operatorname{CC}(\bbR^{2})$).  We will also use the argument $s$ in arclength parametrizations and $\xi$ in parametrizations over $\bbT$.
However,  we will continue using $s$ in a couple of instances when we introduce not-necessarily-arclength approximations of arclength parametrizations of curves.
\smallskip

We note that $d_{\rm F}$ will work better for our purposes than the Hausdorff distance of the {\it images} of curves, defined in Remark 2 after Theorem \ref{T2.5}, because curves with the same image may traverse it very differently.  Nevertheless, this will not be an issue within the class of positively oriented simple planar curves, which we will mainly consider here, and we shall see that our main results will also hold when we replace $d_\mathrm{F}$ by the Hausdorff distance.

\begin{definition}
    A closed curve $\gamma\in\operatorname{CC}(\bbR^{2})$
    is \emph{simple} if  it has a representative that is an injective non-trivial closed path (i.e., it has no self-intersections).  We denote by $\Omega(\gamma)$ the bounded component of $\bbR^{2}\setminus\operatorname{im}(\gamma)$.  We say that $\gamma$ is
    \emph{positive} (or \emph{oriented counterclockwise}) if for any $x\in\Omega(\gamma)$,
    any parametrization of $\gamma$ is homotopic in $\bbR^{2}\setminus\set{x}$
    to the  counterclockwise  oriented unit circle centered at $x$.
    We denote the set of all simple closed curves in $\bbR^{2}$
    by $\operatorname{SC}(\bbR^{2})$, and the set of all positive simple closed curves
    in $\bbR^{2}$ by $\operatorname{PSC}(\bbR^{2})$.
\end{definition}

{\it Remark.}
Note that if $\gamma_{1},\gamma_{2}$ are positive simple closed curves,
then $\gamma_{1} = \gamma_{2}$ if and only if
$\operatorname{im}(\gamma_{1}) = \operatorname{im}(\gamma_{2})$ because
for any injective parametrizations $\tilde{\gamma}_{1},\tilde{\gamma}_{2}$
of $\gamma_{1},\gamma_{2}$, respectively, $\tilde{\gamma}_{1}^{-1}\circ\tilde{\gamma}_{2}$
is an orientation-preserving homeomorphism that identifies $\gamma_{1}$ and $\gamma_{2}$.
\smallskip

We next define various norms of closed curves.

\begin{definition}\label{D2.3}
    Consider any $(p,\beta,k)\in[1,\infty]\times[0,1]\times \bbN$ and any $\gamma\in\operatorname{CC}(\bbR^{2})$.  If $\gamma$ is rectifiable, let ${\gamma(\cdot)}$ be any arclength parametrization of $\gamma$ and define
 %$\norm{\gamma}_{C^{0,\beta}}\coloneqq\norm{\gamma(\cdot)}_{C^{0,\beta}}$, 
 $\norm{\gamma}_{L^{p}}\coloneqq\norm{\gamma(\cdot)}_{L^{p}}$,
%for any     parametrization $\tilde\gamma$ of $\gamma$.
%
%        \item The \emph{length} $\ell(\gamma)$ of $\gamma$ is the total variation of
%        any parametrization of $\gamma$. If $\ell(\gamma)<\infty$, then $\gamma$ is \emph{rectifiable}, and the set of all such closed curves in $\bbR^{2}$ is denoted
%        $\operatorname{RC}(\bbR^{2})$.
        \[
            \norm{\gamma}_{\dot{C}^{k,\beta}} \coloneqq \begin{cases}
                     \norm{{\gamma(\cdot)}}_{\dot{C}^{k,\beta}(\ell(\gamma)\bbT)} &
                \textrm{if $\gamma$ is non-trivial,} \\
                1 & \textrm{if $\gamma$ is trivial and $(k,\beta)=(1,0)$,} \\
                \infty & \textrm{otherwise,} \\
            \end{cases}
        \]
        and
         \[
            \norm{\gamma}_{\dot{H}^{k}} \coloneqq \begin{cases}
                     \norm{{\gamma(\cdot)}}_{\dot{H}^{k}(\ell(\gamma)\bbT)} &
                \textrm{if $\gamma$ is non-trivial,} \\
                0 & \textrm{if $\gamma$ is trivial and $k=1$,} \\
                \infty & \textrm{otherwise.} \\
            \end{cases}
        \]       
%        where in both first cases ${\gamma(\cdot)}$ is any arclength parametrization of $\gamma$.
        If $\gamma$ is not rectifiable, let all these values be $\infty$.  We then 
        %define  $\norm{\gamma}_{C^{k,\beta}}$ and $\norm{\gamma}_{H^{k}}$ via the above
        %in the standard way,
      %  $\norm{\,\cdot\,}_{\dot{C}^{k,\beta}}$, $\norm{\,\cdot\,}_{\dot{H}^{k}}$.
          say that $\gamma$ is a $C^{k,\beta}$ resp.~$H^{k}$ curve when $\norm{\gamma}_{\dot{C}^{k,\beta}}<\infty$ resp.~$ \norm{\gamma}_{\dot{H}^{k}}  <\infty$.
%        if $\norm{\gamma}_{C^{k,\beta}}<\inft$ ($\norm{\gamma}_{H^{k}}$, respectively)
%        is finite.

%        \item For any $h\in[0,\frac{\ell(\gamma)}2]$ let
%        \[
%            \Delta_{h}(\gamma) \coloneqq
%            \begin{cases}
%                \min_{h\leq \abs{s-s'} \leq \ell(\gamma)/2}
%            \abs{{\gamma}(s) - {\gamma}(s')}
%            & \textrm{if $\gamma$ is rectifiable} \\
%            0 & \textrm{if $\gamma$ is not rectifiable.}
%            \end{cases}
%        \]
 %       where $\tilde{\gamma}$ is any arclength parametrization of $\gamma$.
%        is the minimum distance between points on the images of $\gamma$ and $\eta$.
\end{definition}

{\it Remarks.}
1.  Integration by parts and the Cauchy-Schwarz inequality show that for any $\gamma\in\operatorname{CC}(\bbR^{2})$ we have
\begin{equation}\label{2.1}
    \ell(\gamma) \leq \norm{\gamma}_{L^{\infty}}^{2}\norm{\gamma}_{\dot{H}^{2}}^{2}    .
\end{equation}
%\smallskip

%For a given closed curve $\gamma\in\operatorname{CC}(\bbR^{2})$,
%any fixed parametrization of $\gamma$ is generally denoted as $\tilde{\gamma}$.
%To avoid confusion, we generally use $\tilde{\cdot}$
%to denote \emph{parametrized curves} and notations without $\tilde{\cdot}$ generally
%mean curves without preferred choice of parametrizations, if not specified otherwise.
%However, as an exception to this rule, if $\gamma$ is rectifiable,
%then we denote any arclength parametrization of $\gamma$ as just $\gamma$.

2.  When we want to emphasize that we are considering the norm of some parametrization, rather than the norm of the curve it represents, we will often include the domain of the parametrization in the notation (e.g.,  $\norm{\tilde{\gamma}}_{\dot{H}^{2}(\ell\bbT)}$).  Of course, no such distinction is necessary when the parametrization is arclength or the norm is $L^\infty$.
% means the $\dot{H}^{2}$-norm of
%the given parametrization $\tilde{\gamma}\colon\ell\bbT\to\bbR^{2}$
%for some $\ell>0$ (either arclength or not), rather than
%the $\dot{H}^{2}$-norm of the curve $\tilde{\gamma}$ defines.
%Note that such a distinction is not necessary for the $L^{\infty}$-norm,
%or when the given parametrization is of arclength.\\
\smallskip

%with $\bar\gamma$ any shift of $\gamma$ whose image contains the origin in $\bbR^2$.\\

We will consider patches that may touch each other but their boundaries cannot cross.  This relates to the following concepts for the geometry of pairs of curves.

\begin{definition}
    For any $\gamma_1,\gamma_2\in\operatorname{CC}(\bbR^{2})$ let
    $\Delta(\gamma_1,\gamma_2) \coloneqq
    \min_{\xi_1,\xi_2\in\bbT}\abs{\tilde{\gamma}_1(\xi_1) - \tilde{\gamma}_2(\xi_2)}$,
    where $\tilde{\gamma}_1,\tilde{\gamma}_2\colon\bbT\to\bbR^{2}$ are
    any parametrizations of $\gamma_1,\gamma_2$, respectively.
    If $\gamma_1,\gamma_2$ are also $C^1$, we say that they
    % curves $\gamma_1,\gamma_2\in\operatorname{CC}(\bbR^{2})$  
    \emph{do not cross transversally} if for any $C^{1}$ parametrizations
    $\tilde{\gamma}_1,\tilde{\gamma}_2\colon\bbT\to\bbR^{2}$
    of $\gamma_1, \gamma_2$, and for any $\xi_1, \xi_2\in\bbT$ such that
    $\tilde{\gamma}_1(\xi_1) = \tilde{\gamma}_2(\xi_2)$, the vectors
    $\partial_{\xi}\tilde{\gamma}_1(\xi_1)$ and  $\partial_{\xi}\tilde{\gamma}_2(\xi_2)$ are parallel.
\end{definition}

\smallskip
    \textbf{Patch solutions to \eqref{111.25}--\eqref{111.26} and the main results.} 
We now fix some finite set $\mathcal{L}$ that we will use to index individual
patches within collections of patches.  Their boundary curves will be $z=\{z^\lambda\}_{\lambda\in {\mathcal{L}}}\in \operatorname{PSC}(\bbR^2)^\mathcal L$ and their strengths will be $\theta=\{\theta^\lambda\}_{\lambda\in {\mathcal{L}}}\in (\bbR\setminus\set{0})^{\mathcal{L}}$ (with $A^B$ the set of all $f:B\to A$). 
%We will use superscripts to denote this index
%from $\mathcal{L}$, e.g., $\theta^{\lambda}$ means the value of the function $\theta$
%at $\lambda$ (i.e., the strength of the $\lambda$-th patch).
Let also 
\begin{align*}
    d_{\mathrm{F}}(z_{1},z_{2})
    \coloneqq \max_{\lambda\in\mathcal{L}}
    d_{\mathrm{F}}(z_{1}^{\lambda},z_{2}^{\lambda})
\end{align*}
be the Fr\' echet distance on $\operatorname{CC}(\bbR^{2})^{\mathcal{L}}$, define
\[
\norm{z}_{\dot H^k} \coloneqq \max_{\lambda\in \mathcal{L}} \norm{z^{\lambda}}_{\dot H^k}
%\left( \sum_{\lambda\in \mathcal{L}}\norm{z^{\lambda}}_{\dot H^k}^2 \right)^{1/2}
\qquad\text{and}\qquad  \norm{z}_{ L^p} \coloneqq \max_{\lambda\in \mathcal{L}} \norm{z^{\lambda}}_{L^p},
\]
 and for each $\lambda\in\mathcal{L}$ denote
\begin{align*}
\Sigma^{\lambda}(z)  \coloneqq & \left\{ \lambda'\in\mathcal L\,:\, \theta^{\lambda}\theta^{\lambda'} > 0 \,\,\,\&\, \left(\Omega(z^{\lambda}) \subseteq \Omega(z^{\lambda'}) \text{ or } \Omega(z^{\lambda}) \supseteq \Omega(z^{\lambda'}) \right) \right\}
\\ & \qquad\qquad \cup  \left\{ \lambda'\in\mathcal L\,:\, \theta^{\lambda}\theta^{\lambda'} < 0 \,\,\,\&\,\,\, \Omega(z^{\lambda}) \cap \Omega(z^{\lambda'}) = \emptyset \right\}.
\end{align*}
%$\Sigma^{\lambda}(z)$
%is defined as the set of every $\lambda'\in\mathcal{L}$ such that:
%\begin{enumerate}
%    \item $\theta^{\lambda}\theta^{\lambda'} > 0$ and
%    $\Omega(z^{\lambda}) \subseteq \Omega(z^{\lambda'})$, or
%    \item $\theta^{\lambda}\theta^{\lambda'} > 0$ and
%    $\Omega(z^{\lambda}) \supseteq \Omega(z^{\lambda'})$, or
%    \item $\theta^{\lambda}\theta^{\lambda'} < 0$ and
%    $\Omega(z^{\lambda}) \cap \Omega(z^{\lambda'}) = \emptyset$.
%\end{enumerate}
This set represents those patches for which we will be able to control patch solutions to \eqref{111.25}--\eqref{111.26} independently of how close their boundaries are to ${\rm im} (z^\lambda)$ (see the end of this introduction for an explanation of this).  We  cannot allow self-touches, which are the same as exterior touches of patches with $\tht^\lambda\tht^{\lambda'}>0$, or any patch boundary crossings.  For the latter we note that non-transversal crossings look like an exterior touch on one side of the touching segment and an interior touch on the other, so they have to be excluded. On the other hand, a quick computation shows that transversal crossings would cause instantaneous loss of  regularity of both boundaries because the tangential velocity component generated by each patch is not even Lipschitz in the normal direction (it is only $(1-2\alpha)$-H\" older).

For any $\alpha\in\left(0,\frac{1}{2}\right)$
%({\color{red}Mention the separate note for $\alpha=0$?})
 the (continuous)  {\it g-SQG velocity field} generated by $z\in\operatorname{PSC}(\bbR^{2})^{\mathcal{L}}$
is 
\beq \lb{111.30}
    u(z;x)=
    c_{\alpha}\sum_{\lambda\in\mathcal{L}}\theta^{\lambda}
    \int_{\Omega(z^{\lambda})}\frac{(x - y)^{\perp}}{\abs{x - y}^{2 + 2\alpha}}\,dy
\eeq
with some $c_{\alpha} > 0$.  This is of course what we get when we let $\tht$ in \eqref{111.26} be $\sum_{\lambda\in\mathcal{L}}\theta^{\lambda}   \chi_{\Omega(z^{\lambda})}$. When each $z^{\lambda}$ is rectifiable and $z^\lambda(\cdot)$ is any of its arclength parametrizations, we have
\begin{equation}\label{2.3}
    u(z;x) = -\sum_{\lambda\in\mathcal{L}}\theta^{\lambda}\int_{\ell(z^{\lambda})\bbT}
    K(x - z^{\lambda}(s))\partial_{s}z^{\lambda}(s) \,ds
\end{equation}
for all $x\in\bbR^{2}\setminus\bigcup_{\lambda\in\mathcal{L}} \operatorname{im}(z^{\lambda})$ by Green's theorem,
where
% $K\colon\bbR^{2}\to(0,\infty]$ is given as
\begin{align*}
    K(x)\coloneqq \frac{c_{\alpha}}{2\alpha\abs{x}^{2\alpha}}.
\end{align*}
When each $z^{\lambda}$ is $C^{1}$, this extends to all $x\in\bbR^2$ and the integral in \eqref{2.3} always converges absolutely.  Since we will mainly consider here $C^1$ (and particularly $H^2$) curves, we will take \eqref{2.3} as the \emph{definition of} $u(z)$ whenever the integral converges absolutely for all $x\in\bbR^2$, even when
some $z^{\lambda}$ may not  be simple.

%When $x\in\operatorname{im}(z^{\lambda})$ for some $\lambda$ and $z^{\lambda}$
%is only assumed to be rectifiable, then
%the integral $\int_{\ell(z^{\lambda})\bbT}K(x - z^{\lambda}(s))\partial_{s}z^{\lambda}(s)\,ds$
%does not need to converge absolutely. However, if we further assume
%$z^{\lambda}$ is $C^{1,\beta}$ for any $\beta\in(0,1]$
%(which we eventually do because we assume it is $H^{2}$),
%then the integral does converge absolutely and thus is equal to
%$\int_{\Omega(z^{\lambda})}\frac{(x-y)^{\perp}}{\abs{x-y}^{2+2\alpha}}\,dy$
%for any $x\in\bbR^{2}$; see Lemma~\ref{L4.2}.
%For simplicity of presentation, let us take \eqref{2.3} as the \emph{definition}
%of $u(z)$ whenever the integrals are absolutely convergent, even when
%$z^{\lambda}$'s are not assumed to be simple.

For $\gamma\in\operatorname{CC}(\bbR^{2})$,
a vector field $v\in C(\bbR^{2};\bbR^{2})$,
and $h\in\bbR$, we let
$X_{v}^{h}[\gamma]\in\operatorname{CC}(\bbR^{2})$ be the curve whose parametrization is
\[
    \tilde{\gamma}(\xi) + hv(\tilde{\gamma}(\xi)),
\]
where $\tilde{\gamma}$ is any parametrization of $\gamma$ (so ${\rm im}(X_{v}^{h}[\gamma])$ is obtained by transporting each $x\in{\rm im}(\gamma)$ by velocity $v(x)$ for time $h$).
Continuity of $v$ shows that $X_{v}^{h}[\gamma]$ is well-defined and
independent of the choice  of $\tilde \gamma$.
Then for any $z\in\operatorname{CC}(\bbR^{2})^{\mathcal{L}}$ we let
$X_{v}^{h}[z]\coloneqq \{X_{v}^{h}[z^{\lambda}]\}_{\lambda\in\mathcal L}\in \operatorname{CC}(\bbR^2)^{\mathcal L}$.

Now we are ready to define the notion of patch solutions to \eqref{111.25}--\eqref{111.26}.  The patch boundaries will be collections of time-dependent positive simple closed curves, and in order to avoid too many brackets, we will include both the time variable and $\lambda$ in the superscript, with $z^{t,\lambda}\in \operatorname{PSC}(\bbR^2)$, $z^t\in \operatorname{PSC}(\bbR^2)^{\mathcal L}$, $z^\lambda\in C(I; \operatorname{PSC}(\bbR^2))$, and $z\in C(I; \operatorname{PSC}(\bbR^2)^{\mathcal L})$ for some time interval $I$ (using $\lambda$ instead of, e.g., $n$ will eliminate any potential confusion; we reserve the subscript for a future mollification parameter $\eps$). 
Since \eqref{111.25} is a transport PDE  with a non-Lipschitz velocity, and we only have a metric on $\operatorname{PSC}(\bbR^{2})^{\mathcal{L}}$, the following is the proper form of the patch version of \eqref{111.25}--\eqref{111.26} in terms of collections of curves from $\operatorname{PSC}(\bbR^{2})$.

\begin{definition} \lb{D2.5}
Given a finite set $\mathcal L$ and $\tht\in (\bbR\setminus\set{0})^{\mathcal{L}}$, a {\it patch solution} to \eqref{111.25}--\eqref{111.26} on a time interval $I$ is any $z\in C(I;\operatorname{PSC}(\bbR^{2})^{\mathcal{L}})$ such that
%    Let $z^{0}\in\operatorname{PSC}(\bbR^{2})^{\mathcal{L}}$ be such that
%    $L(z^{0}) < \infty$. Then for $T\in(0,\infty]$, we call a function
%    $z\colon [0,T)\to\operatorname{PSC}(\bbR^{2})^{\mathcal{L}}$
%    an \emph{$H^{2}$ patch solution to the g-SQG equation with the initial data $z^{0}$},
%    if $z(0) = z^{0}$,
    \begin{equation}\label{2.4}
        \lim_{h\to 0}\frac{d_{\mathrm{F}}\left(
            z^{t+h}, X_{u(z^{t})}^{h}[z^{t}]
        \right)}{h} = 0
    \end{equation}
    holds for all $t\in I$ with $u$ defined in \eqref{111.30}, where the limit is one-sided at any end-point of $I$. If   for any compact $J\subseteq I$ we have $\sup_{t\in J} \norm{z^t}_{\dot H^2} < \infty$, then $z$ is an {\it $H^2$ patch solution.}
\end{definition}

\emph{Remarks}.
1. Note that this definition places no a priori requirements on relative positions of pairs of patches (i.e.,  concerning touches or crossings of their boundaries).
\smallskip

2.  Remark 3 after \cite[Definition 1.2]{KisYaoZla} shows that $H^2$ patch solutions to \eqref{111.25}--\eqref{111.26} are also weak solutions to \eqref{111.25}--\eqref{111.26} in the usual sense.
\smallskip

Our first main result is  the following well-posedness theorem and blow-up criterion.

\begin{theorem}\label{T2.5}
    Let $\alpha\in\left(0,\frac{1}{6}\right]$, and let $\mathcal L$ be any finite set and $\tht\in (\bbR\setminus\set{0})^{\mathcal{L}}$. Then for each
    $z^{0}\in\operatorname{PSC}(\bbR^{2})^{\mathcal{L}}$ with $\norm{z^{0}}_{\dot H^2}<\infty$ and  $\Delta(z^{0,\lambda},z^{0,\lambda'})>0$ whenever $\lambda'\notin\Sigma^\lambda(z^0)$, %$L(z^{0})<\infty$,
    there is   a unique $H^{2}$ patch solution
    $z\in C_{\rm loc}([0,T_{z^0});\operatorname{PSC}(\bbR^{2})^{\mathcal{L}})$ to \eqref{111.25}--\eqref{111.26}
  with the given initial data $z^{0}$, where $T_{z^0}\in(0,\infty]$ and 
  \beq\lb{111.29}
  \sup_{t\in[0, T_{z^0})} \norm{z^{t}}_{\dot H^2} = \infty
  \eeq 
  holds if $T_{z^0}<\infty$.  Moreover, %$\sup_{t\in[0,T']}L(z^{t})<\infty$ for any $T'\in[0,T_{z^0})$, and 
  for any $(t,\lambda)\in[0,T_{z^0})\times \mathcal{L}$ we have 
    \beq\lb{111.27}
        \abs{\Omega(z^{t,\lambda})} = \abs{\Omega(z^{0,\lambda})},\qquad
        \Sigma^{\lambda}(z^{t}) = \Sigma^{\lambda}(z^{0}),  \qquad\text{and}\qquad
        \min_{\lambda'\notin\Sigma^\lambda(z^0)} \Delta(z^{t,\lambda},z^{t,\lambda'})>0,
    \eeq  
    and for any $\lambda,\lambda'\in\mathcal L$, the set ${\rm im}(z^{t,\lambda})\cap{\rm im}(z^{t,\lambda'})$ is either empty for all $t\in[0,T_{z^0})$ or non-empty for all $t\in[0,T_{z^0})$.
%    as well as $\min_{\lambda'\notin\Sigma^\lambda(z^0)} \Delta(z^{t,\lambda},z^{t,\lambda'})>0$.
\end{theorem}

\emph{Remarks}.
1.  The last claim shows that the set of touching pairs of patch boundaries does not change over time, and \eqref{111.27} shows that no patch boundary crossings (even non-transversal ones) can develop.  Self-touches also do not happen due to $z^{t}\in\operatorname{PSC}(\bbR^{2})^{\mathcal{L}}$.
\smallskip

2. We will in fact show that the solutions in Theorem \ref{T2.5} are unique even when they satisfy \eqref{2.4} with the {\it Hausdorff metric} $d_{\mathrm{H}}$ (which was used in \cite{KisYaoZla}) in place of $d_{\mathrm{F}}$. Recall that
\[
    d_{\mathrm{H}}(A,B) \coloneqq\max\set{
        \sup_{x\in A}d(x,B), \sup_{x\in B}d(x,A)
    }
\]
for closed sets $A,B\subseteq\bbR^{2}$, and we let
\[
    d_{\mathrm{H}}(\gamma_{1},\gamma_{2}) \coloneqq
    d_{\mathrm{H}}(\operatorname{im}(\gamma_{1}),\operatorname{im}(\gamma_{2}))
    \qquad\text{and}\qquad     d_{\mathrm{H}}(z_{1},z_{2})
    \coloneqq \max_{\lambda\in\mathcal{L}}
    d_{\mathrm{H}}(z_{1}^{\lambda},z_{2}^{\lambda})
\]
for $\gamma_{1},\gamma_{2}\in\operatorname{CC}(\bbR^{2})$ and  
 $z_1,z_2\in \operatorname{CC}(\bbR^2)^\mathcal L$.
Since clearly $d_{\mathrm{H}}\leq d_{\mathrm{F}}$, any solution
in terms of $d_{\rm F}$ is also a solution in terms of $d_{\rm H}$.
Our uniqueness proof therefore shows that the two notions of solutions are
equivalent in the setting of Theorem~\ref{T2.5}.
\smallskip

%3.  This result naturally extends to $\bbR\times\bbR^+$, as the proof of Theorem \ref{T2.7} shows.
%\smallskip

Our second main result shows that $\dot H^2$-norms of  solutions from Theorem \ref{T2.5} can  blow up in finite time.
%, that is, we have $T_{z^0}<\infty$ and hence also \eqref{111.29}.
While such results were proved in \cite{KRYZ,GanPat, JeoZla} for non-touching patches on the half-plane, this is the first one for this type of models on a domain without a boundary. 

\begin{theorem}\label{T2.7}
    For each $\alpha\in\left(0,\frac{1}{6}\right]$, there is $z^0$ as in Theorem~\ref{T2.5} such that $T_{z^0}<\infty$, and hence  \eqref{111.29} also holds.
\end{theorem}

{\it Remark.}  The proof shows that all curves constituting $z^0$ can be chosen to be smooth.
%\smallskip

\bigskip
    \textbf{Overview of the proof of Theorem \ref{T2.5}.} 
For any $\gamma\in\operatorname{CC}(\bbR^{2})$ and $h\in[0,\frac{\ell(\gamma)}2]$ let 
        \[
            \Delta_{h}(\gamma) \coloneqq
            \begin{cases}
                \min_{h\leq \abs{s-s'} \leq \ell(\gamma)/2}
            \abs{{\gamma}(s) - {\gamma}(s')}
            & \textrm{if $\gamma$ is rectifiable} \\
            0 & \textrm{if $\gamma$ is not rectifiable.}
            \end{cases}
        \]
Note that since $\norm{\gamma}_{\dot{C}^{1,\beta}}^{-1/\beta}\le  \frac{\ell(\gamma)}2$ by Lemma~\ref{L3.1}, $\Delta_{h}(\gamma)$ will always be well-defined when
$h\leq\norm{\gamma}_{\dot{C}^{1,\beta}}^{-1/\beta}$, and we will not comment on this again below.
Let us next define
\[
\abs{\theta}\coloneqq \sum_{\lambda\in\mathcal{L}}\abs{\theta^{\lambda}} \qquad
\text{and} \qquad m(\theta)\coloneqq\min_{\lambda\in\mathcal{L}}\abs{\theta^{\lambda}}>0,
\]
and for any $z\in\operatorname{CC}(\bbR^{2})^{\mathcal{L}}$ let
\beq\lb{111.50}
    Q(z) \coloneqq
    \frac{1}{m(\theta)}\sum_{\lambda\in\mathcal{L}}
    \abs{\theta^{\lambda}}
    \norm{z^{\lambda}}_{\dot{H}^{2}}^{2}
\eeq
Also, for any $z\in\operatorname{PSC}(\bbR^{2})^{\mathcal{L}}$ let
\[
    W(z) \coloneqq \max_{\lambda\in\mathcal{L}}\abs{\Omega(z^{\lambda})}
\]
and
\beq\lb{111.23}
    L(z) \coloneqq \max\set{
        2Q(z),
        \max_{\lambda\in\mathcal{L}} \, \max
        \set{
            \frac{1}{\Delta_{1/Q(z)}(z^{\lambda})},
            \max_{\lambda'\notin\Sigma^{\lambda}(z)}\frac{1}{\Delta(z^{\lambda},z^{\lambda'})}
        }
    } \in[0,\infty],
\eeq
where $0^{-1} \coloneqq \infty$, $\infty^{-1} \coloneqq 0$, and
$\max\emptyset \coloneqq 0$.  

Our first goal is to construct an $H^2$ patch solution $z$ to \eqref{2.4} with the given initial data $z^0$, and so having locally-in-time bounded $\norm{z}_{\dot{H}^{2}}$.  This can be done via   an appropriate bound on $\partial_t^+ \norm{z^t}_{\dot{H}^{2}}$, where   the
\emph{upper-right} and  \emph{lower-right Dini derivatives} of $f:(a,b)\to\bbR$ are
\[
    \partial_{t}^{+}f(t)\coloneqq \limsup_{h\to 0^{+}}\frac{f(t+h) - f(t)}{h} \qquad\text{and}\qquad
    \partial_{t+}f(t)\coloneqq \liminf_{h\to 0^{+}}\frac{f(t+h) - f(t)}{h}.
\]
However, as was already shown in \cite{GanPat}, such a bound in general needs $\min_{\lambda,\lambda'\in\mathcal L} \Delta (z^{t,\lambda},z^{t,\lambda'})>0$,  and depends on this quantity and on $\norm{z^t}_{\dot{H}^{2}}$,  as well as on how close the curves $z^{t,\lambda}$ are to having self-intersections. (Note that the latter is quantified by $\Delta_{1/Q(z)}(z^{t,\lambda})^{-1}$, which replaces the abovementioned arc-chord constant for curve parametrizations in our approach.)  If we want to allow patch touches, obtaining this level of control may not be possible.

Nevertheless, it turns out that if we instead want to bound $\partial_t^+ Q(z^t)$, that estimate needs to only control $\Delta (z^{t,\lambda},z^{t,\lambda'})$ for $\lambda'\notin \Sigma^\lambda(z^t)$ and $\Delta_{1/Q(z)}(z^{t,\lambda})$ for all $\lambda\in\mathcal L$.  This is thanks to the specific form of $Q$ that we chose above, and we also note that since 
\[
Q(z)\geq\norm{z^{\lambda}}_{\dot{H}^{2}}^{2} \geq \norm{z^{\lambda}}_{\dot{C}^{1,1/2}}^{2}
\]
 for each $z\in\operatorname{PSC}(\bbR^{2})^{\mathcal{L}}$ and  $\lambda\in\mathcal{L}$ thanks to the factor $\frac 1{m(\tht)}$,  Lemma \ref{L3.1} shows that $\Delta_{1/Q(z)}(z^{\lambda})$ is always well-defined.   We will however still need to also control $\partial_t^+ \Delta (z^{t,\lambda},z^{t,\lambda'})^{-1}$ for $\lambda'\notin \Sigma^\lambda(z^t)$ and $\partial_t^+\Delta_{1/Q(z)}(z^{t,\lambda})^{-1}$ for all $\lambda\in\mathcal L$, which is why we will work with the functional $L$ instead of just $Q$.
We note that $L$
%\colon \operatorname{PSC}(\bbR^{2})^{\mathcal{L}}\to [0,\infty]$
%is the functional we associate to the solution and track over the time evolution.
%It 
is lower semicontinuous on $\operatorname{PSC}(\bbR^{2})^{\mathcal{L}}$ by Lemma~\ref{LA.7}, and  that $L(z)<\infty$ allows any pair of curves $z^{\lambda},z^{\lambda'}$ to have either interior or exterior touches (but no crossings), depending on $\sgn(\theta^{\lambda}\theta^{\lambda'})$.
%implies that 
%\begin{enumerate}
%    \item $Q(z) < \infty$: all $z^{\lambda}$'s stay to have finite $\dot{H}^{2}$-norms,
%
%    \item $\min_{\lambda\in\mathcal{L}}\Delta_{1/Q(z)}(z^{\lambda}) > 0$:
%    all $z^{\lambda}$'s stay to be simple (see Lemma~\ref{L3.8}), and
%
%    \item $\min_{\lambda\in\mathcal{L}}
%    \min_{\lambda'\notin\Sigma^{\lambda}(z)}\Delta(z^{\lambda},z^{\lambda'}) > 0$:
%    for each pair $(\lambda,\lambda')$ of distinct elements in $\mathcal{L}$,
%    either $z^{\lambda}$ and $z^{\lambda'}$ are disjoint
%    ($\Delta(z^{\lambda},z^{\lambda'})>0$), or they may ``touch'' but never ``cross''
%    each other, and ``touch'' can happen only from inside if
%    $\theta^{\lambda}$ and $\theta^{\lambda'}$ are of the same sign, or only
%    from outside if they are of different signs.
%\end{enumerate}
Since 
\begin{equation}\label{2.2}
    \Delta_{h}(\gamma) \leq h
\end{equation}
clearly holds for any $\gamma\in\operatorname{CC}(\bbR^{2})$ and $h\in[0,\frac{\ell(\gamma)}2]$,
%always holds, since whenever $\gamma$ is rectifiable,
%\[
%    \Delta_{h}(\gamma) \leq \abs{\gamma(h) - \gamma(0)} \leq h
%\]
%holds for any arclength parametrization of $\gamma$. Therefore,
the $2Q(z)$ in the definition of $L(z)$ seems redundant when one is concerned with whether $L(z)<\infty$ or not.  However, including it (with the coefficient 2) will make estimating $\partial_t^+ L(z^t)$
much easier, which is why we do it.

We start our proofs by obtaining several estimates on the g-SQG velocity field \eqref{2.3} in Section \ref{S4}.   Then, to show existence of the solution $z$ in Theorem \ref{T2.5}, we define {\it mollified \hbox{g-SQG}  equations} and prove the result for them, after which we recover $z$ as their limit.
Consider some smooth even $\chi\colon\bbR\to\bbR$
such that $0\leq\chi\leq 1$, $\chi\equiv 1$ on $\bbR\setminus (-1,1)$, and $0\notin \supp\, \chi$.
%$\lim_{r\to 0}\frac{\chi(r)}{r^{n}} = 0$ for all $n\in\bbN$.
For each $\eps>0$ let $K_{\eps}(x)\coloneqq \chi(\frac{\abs{x}}{\eps})K(x)$.
Note that for any $n\in\bbN$ there is $C_{\alpha,n}$ that only depends
on $\alpha,n$ such that the norms of the $n$-linear forms
$D^{n}K_{\eps}(x)$ and $D^{n}K(x)$ are both bounded by
$\frac{C}{\abs{x}^{n+2\alpha}}$.
For any $z\in \operatorname{PSC}(\bbR^{2})^{\mathcal{L}}$ we  now define the mollified velocity field 
\[
    u_{\eps}(z;x) \coloneqq
    \sum_{\lambda\in\mathcal{L}}\theta^{\lambda}\int_{\Omega(z^{\lambda})}
    \nabla^{\perp}K_{\eps}(x-y)\,dy
\]
for $x\in\bbR^2$.  When each $z^{\lambda}$ is rectifiable and $z^\lambda(\cdot)$ its arclength parametrization, we also have
\beq \lb{111.2}
    u_{\eps}(z;x)=
    -\sum_{\lambda\in\mathcal{L}}\theta^{\lambda}\int_{\ell(z^{\lambda})\bbT}
    K_{\eps}(x - z^{\lambda}(s)) \,
    \partial_{s}z^{\lambda}(s) \,ds .
\eeq
 We again take \eqref{111.2}  as the \emph{definition of
$u_{\eps}(z)$} for any $z\in\operatorname{RC}(\bbR^{2})^{\mathcal{L}}$,
%(recall that $\operatorname{RC}(\bbR^{2})$ is the set of rectifiable closed curves in $\bbR^{2}$), 
even when some $z^{\lambda}$ may not be simple.

Since $u_{\eps}(z^t)$ is defined using a smooth kernel, it is much easier to solve \eqref{2.4} with this velocity in place of $u(z^t)$.  We do this in Section \ref{S5}  by introducing a relevant ODE \eqref{5.1} in $H^{2}(\bbT;\bbR^{2})^{\mathcal L}$ for some parametrizations of the curves $z^{t,\lambda}$ and showing that it has global solutions that also solve \eqref{2.4} with $u_\eps(z^t)$ (see Corollary \ref{C5.6} and Proposition \ref{P7.1}).  Then in Section \ref{S6} we obtain a uniform-in-$\eps$ bound on $\partial_t^+ L(z^t)$ for these solutions in Proposition~\ref{P6.7}, with a crucial cancellation that allows for patch touches when $\lambda'\in\Sigma^\lambda(z^t)$ (and which motivates the special choice of $Q(z)$, as mentioned above) happening in the proof of Lemma~\ref{L6.3} when we estimate  the term $G_6$.  Finally, after taking $\eps\to 0^+$ we obtain the desired solution satisfying Theorem~\ref{T2.5}, although not yet the blow-up criterion \eqref{111.29},  in Section \ref{S7}.

Section \ref{S8} then contains the proof of uniqueness of this solution (even with $d_{\rm H}$ in place of $d_{\rm F}$ in \eqref{2.4}), while Section \ref{S10}  establishes  \eqref{111.29} and that the set of pairs of patches that touch at any given time $t$ is in fact $t$-independent, as well as proves Theorem~\ref{T2.7}.  Various results  about closed curves in $\bbR^2$ 
%with different levels of regularity 
are proved in Appendices \ref{SAa} (several lemmas on the geometry of curves), \ref{SA} (results about the  spaces $\operatorname{CC}(\bbR^2)$ and $\operatorname{PSC}(\bbR^2)$, and functionals on them), and \ref{S9} (proof of a crucial technical lemma used in the uniqueness argument).  These are in principle unrelated to \eqref{111.25}--\eqref{111.26}, and may be useful for other patch-like models in the future.

%As usual, it is enough to show two things:
%\begin{itemize}
%    \item \emph{Existence}: for each $W_{0}\in[0,\infty)$, there exists a non-increasing
%    function $T_{0}\colon[0,\infty)\to(0,\infty]$ such that for any
%    $z^{0}\in\operatorname{PSC}(\bbR^{2})^{\mathcal{L}}$ with
%    $W(z^{0}) \leq W_{0}$ and $L(z^{0})<\infty$,
%    a solution $z$ having $z^{0}$ as its initial data always exists
%    on the time interval $[0,T_{0}(L(z^{0}))]$. Furthermore, $W(z)$ is invariant over time.
%    We prove this in Proposition~\ref{P7.3}.
%
%    \item \emph{Uniqueness}: any other solution
%    $w\colon [0,T]\to\operatorname{PSC}(\bbR^{2})^{\mathcal{L}}$
%    with the same initial data must coincide with $z$ on their common domain of existence.
%    We prove this at the end of Section~\ref{S8}.
%\end{itemize}

%We also prove in Section~\ref{S10} that a finite-time blow-up can happen only when
%the $\dot{H}^{2}$-norm of some $z^{\lambda}$ blows up:
%
%\begin{theorem}\label{T2.6}
%    In the setting of Theorem~\ref{T2.5}, if $T<\infty$, then
%    $\sup_{t\in[0,T)}Q(z^{t}) = \infty$.
%\end{theorem}

\bigskip
    \textbf{Explanation of which types of patch touches can be allowed.} 
We end this introduction by explaining why we are able to control our solutions independently of how close ${\rm im} (z^{\lambda'})$ and ${\rm im} (z^\lambda)$ are for $\lambda'\in\Sigma^{\lambda}(z)$, but not for $\lambda'\notin\Sigma^{\lambda}(z)$.
Pick some large $n\in\bbN$ and assume that two patch boundaries $z^{\lambda'}$ and $z^{\lambda}$
contain disjoint oscillating segments $I'\coloneqq\set{\left(x_{1}, \frac{\cos(nx_{1})}{n^{2}}\right)\colon
x_{1}\in\left[-1,1\right]}$ and
$I\coloneqq\set{\left(x_{1}, \frac{3 + \sin(nx_{1})}{n^{2}}\right)\colon
x_{1}\in\left[-1,1\right]}$.  Assume also that they have no high frequency oscillations elsewhere, their curvatures are bounded by $2$ and $H^2$-norms by 4, the patches are disjoint, and the rest of their boundaries are of distance at least $\frac 12$ from the middle halves of $I',I$ and do not approach each other too closely.

Let us now consider the contribution from the second patch to the instantaneous growth of
the curvature of the boundary of the first patch at $(0,\frac 1{n^2})$.  Since the tangent line at this point is horizontal, the growth of the curvature only depends on the vertical velocity (this is essentially also true at other points on $I'$ because
slopes of the tangents to $I'$ are $O(\frac 1n)$).  The vertical component of the kernel in \eqref{111.30} at this point is $\frac{-y_1}{\abs{(y_1,  y_2-{n^{-2}})}^{2 + 2\alpha}}$, which results in a cancellation in the integral corresponding to $\lambda$ whenever points $(y_1,y_2)$ and $(-y_1,y_2)$ both lie in the second patch.  Therefore the contribution of the part of that patch lying in $B_{1/2}(0)$  to the vertical component of the integral at $(0,\frac 1{n^2})$ is the same as the contribution from the intersection of $B_{1/2}(0)$ with the union of the sets 
$A_{k}\coloneqq \set{ y\in\bbR^2\colon y_{1}\in\left(\frac{(2k-1)\pi}{n},\frac{2k\pi}{n}\right) \,\,\&\,\, |y_2-\frac{3}{n^{2}}| <  \frac{ |\sin(ny_{1})|}{n^{2}}}$ ($k\in\bbZ$).  All these have the same shape and area $\frac 4{n^3}$, and the contributions from $A_k$ and $A_{1-k}$ have opposite signs for each $k\ge 1$.  The latter is roughly $\frac \pi n$ closer to  $(0,\frac 1{n^2})$ than the former, so their contributions sum to a number that is between two positive constants multiplied by $(\frac kn)^{-2-2\alpha}\frac 1n\frac 1{n^3} = n^{-2+2\alpha}k^{-2-2\alpha}$ for $k\ge 2$.  A simple computation shows that the same is true for $k=1$ because $\alpha\le\frac 12$, and adding this over $k\ge 2$ yields a positive contribution of the order $n^{-2+2\alpha}$.

A similar computation shows that  the contribution of the part of the second patch lying inside $B_{1/2}(0)$  to the vertical component of the integral at $(\frac \pi n,-\frac 1{n^2})$ is between two negative constants multiplied by $n^{-2+2\alpha}$, and this repeats periodically at the other vertical extrema of $I'$ that are inside $B_{1/4}(0)$.  Since the contribution from $\bbR^2\setminus B_{1/2}(0)$ to the integral is a smooth function of $x\in B_{1/4}(0)$ with an $n$-independent bound on any derivative, we see that the instantaneous growth rate of the amplitude of the oscillations of $I'$ inside $B_{1/4}(0)$ that is due to the second patch is proportional to $\tht^\lambda n^{-2+2\alpha}$.  Hence the instantaneous growth rate of the squared $L^2$-norm of the curvature of the patch boundary (i.e., of $\|z^{\lambda'}\|_{\dot H^2}^2$)  is proportional to $\tht^\lambda n^{2\alpha}$.
 The proof of Theorem \ref{T2.5} in the one patch case (one can also use the proofs in \cite{GanPat}) shows that the instantaneous growth rate of $\|z^{\lambda'}\|_{\dot H^2}^2$ that is due to the first patch itself is bounded by some constant times $\|z^{\lambda'}\|_{\dot H^2}^{6+4\alpha}$ ($\le 4^7$), so the overall growth rate is again proportional to $\tht^\lambda n^{2\alpha}$ when $n$ is large.
 
 The same computation shows that the instantaneous growth rate of $\|z^{\lambda}\|_{\dot H^2}^2$ is proportional to $\tht^{\lambda'} n^{2\alpha}$ for large $n$ (with the ratio of proportionality constants converging to 1 as $n\to\infty$).  This means that the only linear combinations of squares of these  norms that might have an $n$-independent bound on their growth rates are multiples of $\tht^{\lambda'}\|z^{\lambda'}\|_{\dot H^2}^2- \tht^\lambda \|z^{\lambda}\|_{\dot H^2}^2$ (as well as that $\tht^\lambda\tht^{\lambda'}>0$ should cause ill-posedness if the patches have an external touch).  Such a quantity only yields a bound on the individual $\dot H^2$-norms when $\tht^\lambda\tht^{\lambda'}<0$ (for exterior patch touches), and a similar computation yields the requirement $\tht^\lambda\tht^{\lambda'}>0$ for interior touches, with the relevant quantity being $\tht^{\lambda'}\|z^{\lambda'}\|_{\dot H^2}^2+ \tht^\lambda \|z^{\lambda}\|_{\dot H^2}^2$. Corollary~\ref{C6.4} below --- specifically the proofs of Lemmas \ref{L6.2} and \ref{L6.3}, with the crucial cancellations appearing in our estimates on the term $G_6$ in the latter --- shows that these are also sufficient for general constellations of finitely many patches, and it turns out that it then suffices to control  the single  quantity $Q(z)$ from \eqref{111.50}.  So even though we will not be able to control {\it growth rates of $\|z^{\lambda}\|_{\dot H^2}$} for individual patch boundaries, our approach will  still yield good control of all these norms. 
 
 Finally, note that if the oscillations of $I'$ and $I$ are offset by half a period (instead of quarter as above), then the  unbounded-in-$n$ amplitude growth rate disappears in the above analysis.  This always happens for solutions that are odd in $x_2$ and only touches  between patches and their reflections are allowed (i.e., in the no-touch half-plane case),  which also demonstrates why the general setting considered here is much more challenging to study.

\vskip 3mm
\noindent
{\bf Acknowledgement.}  Both authors were supported in part by NSF grant DMS-2407615.

%%%%%%%%%%%%%%%%%%%%%%%%%%%%%%%%%%%%%%%%%%%%%%
\section{Estimates on g-SQG velocity fields}\label{S4}
%%%%%%%%%%%%%%%%%%%%%%%%%%%%%%%%%%%%%%%%%%%%%%

In this section we prove various estimates on  the g-SQG velocity field that we use throughout the rest of the paper.
We note that the proofs show that all these estimates continue to hold when we replace the kernel $K$ in \eqref{2.3} by $K_{\eps}$ in \eqref{111.2} (possibly with different constants that are independent of $\eps$), and hence the velocity $u$ by $u_{\eps}$.  We also note that all constants $C$ in this section depend only on the indicated variables (e.g., $C_\alpha$ only depends on $\alpha$) and can change from one inequality to another.  We always assume that $\alpha\in(0,\frac 12)$ and $\beta\in(0,1]$.

We start with a uniform bound on the velocity field
 in terms of the sizes and strengths of the patches, 
independently of the regularity of their boundaries.

\begin{lemma}\label{L4.1}
    There is $C_\alpha$  such that
    for any simple closed curve $\gamma$ in $\bbR^{2}$ and $x\in\bbR^2$ we have
    \[
        \int_{\Omega(\gamma)}\frac{dy}{\abs{x - y}^{1 + 2\alpha}}
        \leq C_\alpha\abs{\Omega(\gamma)}^{\frac{1}{2} - \alpha}.
    \]
In particular, for any $z\in\operatorname{PSC}(\bbR^{2})^{\mathcal{L}}$ we obtain
    \[
        \norm{u(z)}_{L^{\infty}}
        \leq c_{\alpha}C_\alpha\sum_{\lambda\in\mathcal{L}}\abs{\theta^{\lambda}}
        \abs{\Omega(z^{\lambda})}^{\frac{1}{2} - \alpha}
        \leq c_{\alpha}C_\alpha\abs{\theta}W(z)^{\frac{1}{2}-\alpha}.
    \]    
%    holds with the same constant $C$.
\end{lemma}

\begin{proof}
    The first claim immediately follows from Hardy-Littlewood rearrangement inequality
    \begin{align*}
        \int_{\Omega(\gamma)}\frac{dy}{\abs{x - y}^{1+2\alpha}}
        %= \int_{\bbR^{2}}\frac{\mathbbm{1}_{\Omega(\gamma)}(y)}{\abs{x - y}^{1+2\alpha}}\,dy
        \leq \int_{0}^{2\pi}\int_{0}^{(\abs{\Omega(\gamma)}/\pi)^{1/2}}
        \frac{1}{r^{2\alpha}}\,dr\,d\theta 
%        &= \frac{2\pi}{1 - 2\alpha}
%        \left(\frac{\abs{\Omega(\gamma)}}{\pi}\right)^{\frac{1}{2} - \alpha}
        = \frac{2\pi^{\frac{1}{2}+\alpha}}{1 - 2\alpha}
        \abs{\Omega(\gamma)}^{\frac{1}{2} - \alpha}.
    \end{align*}
    The second claim is then immediate from the definition of $u(z)$.
\end{proof}

Next is an estimate on the integral of the kernel $K$ along the curve $\gamma$ without
the tangent vector $\partial_{s}\gamma(s)$. As a result, there is no cancellation of contributions to the integral from nearby curve segments (with roughly opposite tangent vectors), which means that we get a worse estimate than
in Lemma~\ref{L4.1}. 
%In fact, as noted in Section~\ref{S2}, the integral may not converge if $\gamma$ is only assumed to be rectifiable.

\begin{lemma}\label{L4.2}
    There is $C_\alpha$ such that for any $C^{1,\beta}$ closed curve
    $\gamma\colon\ell\bbT\to\bbR^{2}$ parametrized by arclength and any $x\in\bbR^2$ we have
    \[
        \int_{\ell\bbT}K(x-\gamma(s))\,ds \leq
        C_\alpha\ell\norm{\gamma}_{\dot{C}^{1,\beta}}^{2\alpha/\beta}.
    \]
\end{lemma}

\begin{proof}
    Let $d: = \frac{1}{4}\norm{\gamma}_{\dot{C}^{1,\beta}}^{-1/\beta}$.
    Then Lemma~\ref{L3.2} shows that
    \begin{align*}
        \int_{\ell\bbT}\frac{ds}{\abs{x - \gamma(s)}^{2\alpha}}
        &\leq \frac{\ell}{4d}\int_{-2d}^{2d}\frac{ds}{\abs{s/2}^{2\alpha}}
        + \int_{\ell\bbT}\frac{ds}{d^{2\alpha}}
        = \frac{2 - 2\alpha}{1 - 2\alpha}\frac{\ell}{d^{2\alpha}}
        = C_\alpha\ell\norm{\gamma}_{\dot{C}^{1,\beta}}^{2\alpha/\beta}.
    \end{align*}
\end{proof}

Next we bound the difference between the velocity fields generated
by $K$ and  $K_{\eps}$.

\begin{lemma}\label{L4.3}
    There is $C_\alpha$ such that
    for any $C^{1,\beta}$ closed curve
    $\gamma\colon\ell\bbT\to\bbR^{2}$ parametrized by arclength, any  $\eps\in\left(0,\frac{1}{4}\norm{\gamma}_{\dot{C}^{1,\beta}}^{-1/\beta}\right]$, and any
    $x\in\bbR^{2}$  we have
    \[
        \int_{\abs{x - \gamma(s)}\leq\eps}
        K(x - \gamma(s))\,ds
        \leq C_\alpha\ell\norm{\gamma}_{\dot{C}^{1,\beta}}^{1/\beta}
        \eps^{1-2\alpha}.
    \]
\end{lemma}

\begin{proof}
    Lemma~\ref{L3.2} shows that
    \begin{align*}
        \int_{\abs{x - \gamma(s)}\leq\eps}\frac{ds}{\abs{x - \gamma(s)}^{2\alpha}}
        &\leq \ell\norm{\gamma}_{\dot{C}^{1,\beta}}^{1/\beta}\int_{-2\eps}^{2\eps}
        \frac{ds}{\abs{s/2}^{2\alpha}}
        = \frac{4}{1-2\alpha}\ell\norm{\gamma}_{\dot{C}^{1,\beta}}^{1/\beta}\eps^{1-2\alpha}.
    \end{align*}
\end{proof}

We now estimate the difference of  velocity fields generated by two nearby curves at two nearby points.  This result is related to \cite[Lemma 4.9]{KisYaoZla}, but is more quantitative and it is stated for more general curves.

%({\color{red}Below corresponds to KYZ, Lemma 4.9. This is technically a bit more general
%than what is stated in KYZ as we do not assume curves to be simple, but
%we do not actually need this for non-simple curves anyway.
%KYZ also states it only for $H^{3}$ curves, but I believe the proof given there
%should just work fine for $C^{1,\beta}$ curves as well.})

\begin{lemma}\label{L4.4}
    There is $C_\alpha$ such that for any $C^{1,\beta}$ closed curves
    $\gamma_{i}\colon\ell_{i}\bbT\to\bbR^{2}$ ($i=1,2$) parametrized by arclength and any  $x_{1},x_{2}\in\bbR^{2}$ we have
    \begin{align*}
        &\abs{\int_{\ell_{1}\bbT}K(x_{1} - \gamma_{1}(s))
        \partial_{s}\gamma_{1}(s)\,ds
        - \int_{\ell_{2}\bbT}K(x_{2} - \gamma_{2}(s))
        \partial_{s}\gamma_{2}(s)\,ds} \\
        &\quad\quad\quad\quad\leq
        C_\alpha(\ell_{1}+\ell_{2})\max\left\{
            \norm{\gamma_{1}}_{\dot{C}^{1,\beta}},
            \norm{\gamma_{2}}_{\dot{C}^{1,\beta}}
        \right\}^{1/\beta}
        \left(
            \abs{x_{1}-x_{2}} + d_{\mathrm{F}}(\gamma_{1},\gamma_{2})
        \right)^{1-2\alpha}.
    \end{align*}
\end{lemma}

\begin{proof}
    Fix any $\eps > 0$. Lemma~\ref{LA.1} then
    yields an orientation-preserving diffeomorphism
    $\phi\colon\ell_{1}\bbT\to\ell_{2}\bbT$ such that
    $\norm{\gamma_{1} - \gamma_{2}\circ\phi}_{L^{\infty}}
    \leq d_{\mathrm{F}}(\gamma_{1},\gamma_{2}) + \eps$.
    Let $d\coloneqq \abs{x_{1}-x_{2}} + d_{\mathrm{F}}(\gamma_{1},\gamma_{2}) + \eps$ and
    $d_{0}\coloneqq \min\left\{\frac{1}{4}\norm{\gamma_{1}}_{\dot{C}^{1,\beta}}^{-1/\beta},
    \frac{1}{4}\norm{\gamma_{2}}_{\dot{C}^{1,\beta}}^{-1/\beta}\right\}$.

    When $d > d_{0}$, Lemma~\ref{L4.2} shows that
    \begin{align*}
        &\abs{\int_{\ell_{1}\bbT}K(x_{1} - \gamma_{1}(s))
        \partial_{s}\gamma_{1}(s)\,ds
        - \int_{\ell_{2}\bbT}K(x_{2} - \gamma_{2}(s))
        \partial_{s}\gamma_{2}(s)\,ds} 
 %  \\     &\quad\quad\quad\quad
 \leq       C_\alpha\left(
            \ell_{1}\norm{\gamma_{1}}_{\dot{C}^{1,\beta}}^{2\alpha/\beta}
            + \ell_{2}\norm{\gamma_{2}}_{\dot{C}^{1,\beta}}^{2\alpha/\beta}
        \right) \\ & \qquad\qquad\qquad
        \leq \frac{C_\alpha(\ell_{1}+\ell_{2})}{4^{2\alpha}d_{0}}d_{0}^{1-2\alpha} 
        %\quad\quad\quad\quad
        \leq
        C_\alpha(\ell_{1}+\ell_{2})\max\left\{
            \norm{\gamma_{1}}_{\dot{C}^{1,\beta}},
            \norm{\gamma_{2}}_{\dot{C}^{1,\beta}}
        \right\}^{1/\beta}d^{1-2\alpha}.
    \end{align*}
    Since $\eps > 0$ was arbitrary, we get the desired estimate.
    
    Let us now assume $d \leq d_{0}$.
    For each $i=1,2$, take $\set{s_{i,j}}_{j=1}^{N_{i}}\subseteq\ell_{i}\bbT$ with
    $N_{i}\leq \frac{\ell_{i}}{4d_{0}}$ obtained by applying Lemma~\ref{L3.2} with
    $\gamma\coloneqq\gamma_{i}$ and $x\coloneqq x_{i}$, and then define
    \[
        S_{i}\coloneqq \set{j\in\set{1,\dots,N_{i}}\colon
        \abs{x_{i} - \gamma_{i}(s_{i,j})}\leq d}
        \qquad\textrm{and}\qquad
        A_{i}\coloneqq\bigcup_{j\in S_{i}}[s_{i,j}-2d,s_{i,j}+2d].
    \]
    Then Lemma~\ref{L3.2} shows that $\set{s\in\ell_{i}\bbT \colon
    \abs{x_{i} - \gamma_{i}(s)} \leq d} \subseteq A_{i}$, and we let
    \[
        B\coloneqq \set{s\in\ell_{1}\bbT\colon
        \abs{x_{1}-\gamma_{1}(s)} < \abs{x_{2}-(\gamma_{2}\circ\phi)(s)}}.
    \]
      Let $        \tilde{K}(r)\coloneqq \frac{c_{\alpha}}{2\alpha\, r^{2\alpha}}$
    so that $K(x) = \tilde{K}(\abs{x})$.    
    We now write the integral
    \[
        \int_{\ell_{1}\bbT}
        \left[ K(x_{1} - \gamma_{1}(s))\partial_{s}\gamma_{1}(s)
        - K(x_{2} - (\gamma_{2}\circ\phi)(s))\partial_{s}(\gamma_{2}\circ\phi)(s)
        \right] ds
    \]
    as the sum of the following four terms that we will estimate separately by
    $C_\alpha \frac{\ell_{1}+\ell_{2}}{d_{0}}d^{1-2\alpha}$:
    \begin{align*}
        V_{1} &\coloneqq \int_{B}
        \left[
            K(x_{1} - \gamma_{1}(s))
            - K(x_{2} - (\gamma_{2}\circ\phi)(s))
        \right]
        \partial_{s}\gamma_{1}(s)\,ds, \\
        V_{2} &\coloneqq \int_{\ell_{2}\bbT\setminus\phi(B)}
        \left[
            K(x_{1} - (\gamma_{1}\circ\phi^{-1})(s))
            - K(x_{2} - \gamma_{2}(s))
        \right]
        \partial_{s}\gamma_{2}(s)\,ds, \\
        V_{3} &\coloneqq \int_{A_{1}\cap\phi^{-1}(A_{2})}
        \tilde{K}\left(
            \max\left\{\abs{x_{1} - \gamma_{1}(s)},
        \abs{x_{2} - (\gamma_{2}\circ\phi)(s)}\right\}
        \right)\partial_{s}(\gamma_{1} - \gamma_{2}\circ\phi)(s)\,ds,\\
        V_{4} &\coloneqq \int_{\ell_{1}\bbT\setminus(A_{1}\cap\phi^{-1}(A_{2}))}
        \tilde{K}\left(
            \max\left\{\abs{x_{1} - \gamma_{1}(s)},
        \abs{x_{2} - (\gamma_{2}\circ\phi)(s)}\right\}
        \right)\partial_{s}(\gamma_{1} - \gamma_{2}\circ\phi)(s)\,ds.
    \end{align*}

    \textbf{Estimate for $V_{1}$.} By definition of $B$, we have
    \begin{equation*}
        \abs{\frac{1}{\abs{x_{1} - \gamma_{1}(s)}^{2\alpha}}
        - \frac{1}{\abs{x_{2} - (\gamma_{2}\circ\phi)(s)}^{2\alpha}}}
        \leq \frac{1}{\abs{x_{1} - \gamma_{1}(s)}^{2\alpha}}
    \end{equation*}
    for any $s\in B$. Hence, for any fixed $j$, Lemma~\ref{L3.2} shows that
    \begin{align*}
        \int_{B\cap[s_{1,j}-2d,s_{1,j}+2d]}
        \abs{\frac{1}{\abs{x_{1} - \gamma_{1}(s)}^{2\alpha}}
        - \frac{1}{\abs{x_{2} - (\gamma_{2}\circ\phi)(s)}^{2\alpha}}}\,ds
        &\leq \int_{-2d}^{2d}\frac{ds}{\abs{s/2}^{2\alpha}}
        = \frac{4}{1-2\alpha}d^{1-2\alpha}.
    \end{align*}
    On the other hand, the mean value theorem applied to the function
    $\tilde{K}$ shows that
    \begin{equation}\label{4.1}
        \abs{K(x_{1} - \gamma_{1}(s))
        - K(x_{2} - (\gamma_{2}\circ\phi)(s))}
        \leq \frac{c_{\alpha}d}{\abs{x_{1} - \gamma_{1}(s)}^{1+2\alpha}}
    \end{equation}
    holds for any $s\in B$, so again Lemma~\ref{L3.2} shows that
    \begin{align*}
        &\int_{B\cap[s_{1,j}-2d_{0},s_{1,j}+2d_{0}]
        \setminus [s_{1,j}-2d,s_{1,j}+2d]}
        \abs{K(x_{1} - \gamma_{1}(s))
        - K(x_{2} - (\gamma_{2}\circ\phi)(s))}\,ds \\
        &\quad\quad\quad\leq
        \int_{2d}^{2d_{0}}\frac{2c_{\alpha}d}{(s/2)^{1+2\alpha}}\,ds
        \leq \frac{2c_{\alpha}}{\alpha}d^{1-2\alpha}.
    \end{align*}
    Hence, summing up these estimates over $j=1, \dots ,N_{1}$ gives
    \begin{equation}\label{4.2}
        \int_{B\cap\bigcup_{j=1}^{N_{1}}[s_{1,j}-2d_{0},s_{1,j}+2d_{0}]}
        \abs{K(x_{1} - \gamma_{1}(s))
        - K(x_{2} - (\gamma_{2}\circ\phi)(s))}\,ds
        \leq C_\alpha\frac{\ell_{1}}{d_{0}}d^{1-2\alpha}.
    \end{equation}
    Also, \eqref{4.1}, $d \leq d_{0}$, and  $\abs{x_{1} - \gamma_{1}(s)} \geq d_{0}$ for 
    $s\notin \bigcup_{j=1}^{N_{1}}[s_{1,j}-2d_{0},s_{1,j}+2d_{0}]$ show that
    \begin{equation*}
        \abs{K(x_{1} - \gamma_{1}(s))
        - K(x_{2} - (\gamma_{2}\circ\phi)(s))}
        \leq \frac{c_{\alpha}d}{d_{0}^{1+2\alpha}}
        \leq \frac{c_{\alpha}}{d_{0}}d^{1-2\alpha}
    \end{equation*}
 for all $s\in B\setminus\bigcup_{j=1}^{N_{1}}[s_{1,j}-2d_{0},s_{1,j}+2d_{0}]$.
    Together with \eqref{4.2}, this yields
    $\abs{V_{1}} \leq C_\alpha\frac{\ell_{1}}{d_{0}}d^{1-2\alpha}$.

    \textbf{Estimate for $V_{2}$.} This is identical to the estimate for $V_1$.

    \textbf{Estimate for $V_{3}$.} Clearly
    \begin{align*}
        \abs{V_{3}} &\leq
        \int_{A_{1}\cap\phi^{-1}(A_{2})}
        \tilde{K}\left(
            \max\set{\abs{x_{1} - \gamma_{1}(s)},
            \abs{x_{2} - (\gamma_{2}\circ\phi)(s)}}
        \right)\,ds
        \\&\quad\quad\quad\quad\quad\quad
        + \int_{\phi(A_{1})\cap A_{2}}
        \tilde{K}\left(
            \max\set{\abs{x_{1}-(\gamma_{1}\circ\phi^{-1})(s)},
            \abs{x_{2}-\gamma_{2}(s)}}
        \right)\,ds.
    \end{align*}
    By symmetry, it suffices to estimate the first term, which we do using Lemma~\ref{L3.2}:
    \begin{align*}
        &\int_{A_{1}\cap\phi^{-1}(A_{2})}
        \tilde{K}\left(
            \max\set{\abs{x_{1} - \gamma_{1}(s)},
            \abs{x_{2} - (\gamma_{2}\circ\phi)(s)}}
        \right)\,ds
        \leq \sum_{j\in S_{1}}\int_{s_{1,j} - 2d}^{s_{1,j} + 2d}
        \frac{c_{\alpha}}{2\alpha\abs{x_{1} - \gamma_{1}(s)}^{2\alpha}}\,ds
        \\&\quad\quad
        \leq N_{1}\int_{-2d}^{2d}
        \frac{c_{\alpha}}{2\alpha\abs{s/2}^{2\alpha}}\,ds
        \leq \frac{\ell_{1}}{4d_{0}} \frac{2c_{\alpha}}{\alpha(1-2\alpha)}d^{1-2\alpha}
        = C_\alpha\frac{\ell_{1}}{d_{0}}d^{1-2\alpha}.
    \end{align*}

    \textbf{Estimate for $V_{4}$.} Note that $\ell_{1}\bbT\setminus(A_{1}\cap\phi^{-1}(A_{2}))$
    is either a union of a disjoint family $\set{(a_{k},b_{k})}_{k=1}^{N}$ of open subintervals
    of $\ell_{1}\bbT$ with $N\leq N_{1} + N_{2}$, or all of $\ell_{1}\bbT$.
    In the latter case let $N: = 1$ and $(a_{1},b_{1}) \coloneqq (0,\ell_{1})$.  Then in either case we have
    \begin{align*}
        V_{4} = \sum_{k=1}^{N}\int_{a_{k}}^{b_{k}}
        \tilde{K}\left(
            \max\left\{\abs{x_{1} - \gamma_{1}(s)},
            \abs{x_{2} - (\gamma_{2}\circ\phi)(s)}\right\}
        \right)
        \partial_{s}(\gamma_{1} - \gamma_{2}\circ\phi)(s)\,ds,
    \end{align*}
    and then integration by parts yields
    \begin{align*}
        &\int_{a_{k}}^{b_{k}}
        \tilde{K}\left(
            \max\left\{\abs{x_{1} - \gamma_{1}(s)},
            \abs{x_{2} - (\gamma_{2}\circ\phi)(s)}\right\}
        \right)
        \partial_{s}(\gamma_{1} - \gamma_{2}\circ\phi)(s)\,ds
        \\&\quad\quad=
        \tilde{K}\left(
            \max\left\{\abs{x_{1} - \gamma_{1}(b_{k})},
            \abs{x_{2} - (\gamma_{2}\circ\phi)(b_{k})}\right\}
        \right)
        (\gamma_{1}(b_{k}) - (\gamma_{2}\circ\phi)(b_{k}))
        \\&\quad\quad\quad\quad\quad
        - \tilde{K}\left(
            \max\left\{\abs{x_{1} - \gamma_{1}(a_{k})},
            \abs{x_{2} - (\gamma_{2}\circ\phi)(a_{k})}\right\}
        \right)
        (\gamma_{1}(a_{k}) - (\gamma_{2}\circ\phi)(a_{k}))
        \\&\quad\quad\quad\quad\quad
        + c_{\alpha}\int_{(a_{k},b_{k})\cap B}
        \frac{(\gamma_{1}(s) - (\gamma_{2}\circ\phi)(s))
        \big((x_{2} - (\gamma_{2}\circ\phi)(s))
        \cdot \partial_{s}(\gamma_{2}\circ\phi)(s)\big)}
        {\abs{x_{2} - (\gamma_{2}\circ\phi)(s)}^{2+2\alpha}}\,ds
        \\&\quad\quad\quad\quad\quad
        + c_{\alpha}\int_{(a_{k},b_{k})\setminus B}
        \frac{(\gamma_{1}(s) - (\gamma_{2}\circ\phi)(s))
        \big((x_{1} - \gamma_{1}(s))
        \cdot \partial_{s}\gamma_{1}(s)\big)}
        {\abs{x_{1} - \gamma_{1}(s)}^{2+2\alpha}}\,ds.
    \end{align*}
    (Note that the definition of $A_i$ ensures that the argument of $\tilde K$ in the above integrals stays away from 0, and one can then use $\max\set{f,g} = \frac{f + g + \abs{f - g}}{2}$ and for instance \cite[Theorem 2]{SerrinVarberg} to perform this integration by parts.)
%    ({\color{red}Reference for the chain rule for Sobolev functions is
%    \href{https://www.jstor.org/stable/2316959?seq=1}{here (link)}.
%    I am not sure how much detail should be presented.
%    My argument is basically that, (1) the singularity of $\tilde{K}$ is not a problem
%    because the definition of $(a_{k},b_{k})$ ensures it is never touched,
%    and (2) after applying the identity $\max\set{f,g} = \frac{f + g + \abs{f - g}}{2}$,
%    we use Theorem 2 from the reference to deal with the singularity of $\abs{f - g}$.})
    Let $V_{5,k}$ be the sum of the first two terms from the right-hand side above, $V_{6,k}$ the third term,
    and $V_{7,k}$ the last term.
    Since $(a_{k},b_{k})$ is disjoint from $A_{1}\cap\phi^{-1}(A_{2})$, at least one of
    $\abs{x_{1} - \gamma_{1}(s)}$ and $\abs{x_{2} - (\gamma_{2}\circ\phi)(s)}$
    is bounded below by $d$ for all $s\in[a_{k},b_{k}]$. Hence $\abs{V_{5,k}}$ is bounded by
$2\frac{c_{\alpha}}{2\alpha\, d^{2\alpha}}d
        = C_\alpha d^{1-2\alpha}$,
    so we have
    \[
        \abs{\sum_{k=1}^{N}V_{5,k}}
        \leq C_\alpha  Nd^{1-2\alpha}
        \leq C_\alpha\frac{\ell_{1} + \ell_{2}}{d_{0}}d^{1-2\alpha}.
    \]

    On the other hand, since $B$ is open, applying the change of variables formula
    to each connected component of $(a_{k},b_{k})\cap B$ yields
    \begin{align*}
        \abs{\sum_{k=1}^{N}V_{6,k}}
        \leq c_{\alpha}d\int_{\phi(B)\setminus(\phi(A_{1})\cap A_{2})}
        \frac{ds}{\abs{x_{2} - \gamma_{2}(s)}^{1+2\alpha}}.
    \end{align*}
    Since $s\in \phi(B)\setminus(\phi(A_{1})\cap A_{2})$ implies
    $\abs{x_{2} - \gamma_{2}(s)} \geq d$, Lemma~\ref{L3.2} shows that
    \begin{align*}
        &\int_{\big(\phi(B)\setminus (\phi(A_{1})\cap A_{2})\big)
        \cap[s_{2,j}-2d_{0},s_{2,j}+2d_{0}]}
        \frac{ds}{\abs{x_{2} - \gamma_{2}(s)}^{1+2\alpha}}
        \leq \int_{-2d_{0}}^{2d_{0}}\frac{ds}{\max\{d,\abs{s/2}\}^{1+2\alpha}}
        \leq \frac{2(1+2\alpha)}{\alpha}d^{-2\alpha}
    \end{align*}
    for $j=1 \dots ,N_{2}$, and
    \begin{align*}
        &\int_{\big(\phi(B)\setminus (\phi(A_{1})\cap A_{2})\big)
        \setminus\bigcup_{j=1}^{N_{2}}[s_{2,j}-2d_{0},s_{2,j}+2d_{0}]}
        \frac{ds}{\abs{x_{2} - \gamma_{2}(s)}^{1+2\alpha}}
        \leq \frac{\ell_{2}}{d_{0}^{1+2\alpha}}
        \leq \frac{\ell_{2}}{d_{0}}d^{-2\alpha}.
    \end{align*}
    We therefore obtain
    \begin{align*}
        \abs{\sum_{k=1}^{N}V_{6,k}}
        &\leq C_\alpha N_{2}d^{1-2\alpha} + \frac{c_{\alpha}\ell_{2}}{d_{0}}d^{1-2\alpha}
        \leq C_\alpha\frac{\ell_{2}}{d_{0}}d^{1-2\alpha},
    \end{align*}
    and $\abs{\sum_{k=1}^{N}V_{7,k}}
    \leq C_\alpha\frac{\ell_{1}}{d_{0}}d^{1-2\alpha}$ follows in the same way
    (but without the change of variables).  Thus
    $\abs{V_{4}} \leq C_\alpha\frac{\ell_{1} + \ell_{2}}{d_{0}}d^{1-2\alpha}$, which together with the same estimate for $V_{1}$, $V_{2}$, and $V_{3}$ (and $\eps>0$ being arbitrary) finishes the proof.
\end{proof}

We next obtain some crucial estimates that will allow us to leverage the bounds from Lemma~\ref{L3.3} and Lemma~\ref{L3.5} for curves that do not cross each other to obtain better estimates on the g-SQG velocity and related integrals in the following sections.

\begin{lemma}\label{L4.5}
    If $\frac{\beta}{1+\beta}\geq 2\alpha$, then there is
    $C_{\alpha,\beta}$    such that for any $C^{1,\beta}$ closed curves
    $\gamma_{i}\colon\ell_{i}\bbT\to\bbR^{2}$ ($i=1,2,3$) parametrized by arclength, with
    $\mathbf{T}_{i}\coloneqq\partial_{s}\gamma_{i}$ and
    $\mathbf{N}_{i}\coloneqq\mathbf{T}_{i}^{\perp}$, %for $i=1,2$,
     and for any
    $x\in\bbR^{2}$, $s'\in\ell_{1}\bbT$, and a Lipschitz continuous orientation-preserving
    homeomorphism $\phi\colon\ell_{3}\bbT\to\ell_{2}\bbT$, we have the following.
    \begin{enumerate}
        \item If $\gamma_{1}$ and $\gamma_{2}$ do not cross transversally, then
        \begin{align*}
            &\int_{\ell_{2}\bbT}
            \frac{\abs{\mathbf{N}_{1}(s')\cdot\mathbf{T}_{2}(s)}}
            {\abs{\gamma_{1}(s') - \gamma_{2}(s)}
            \abs{x - \gamma_{3}(\phi^{-1}(s))}^{2\alpha}}\,ds
            \\&\qquad\qquad\qquad\leq C_{\alpha,\beta}\max\set{\ell_{2},\ell_{3}}\max\left\{
                \norm{\gamma_{1}}_{\dot{C}^{1,\beta}},
                \norm{\gamma_{2}}_{\dot{C}^{1,\beta}},
                \norm{\gamma_{3}}_{\dot{C}^{1,\beta}}
            \right\}^{\frac{1+2\alpha}{\beta}}
            \left(\norm{\phi}_{\dot{C}^{0,1}}+1\right)^{2}.
        \end{align*}

        \item If $\gamma_{1},\gamma_{2}\in\operatorname{PSC}(\bbR^{2})$ and
        $\Omega(\gamma_{1})\subseteq\Omega(\gamma_{2})$ or
        $\Omega(\gamma_{2})\subseteq\Omega(\gamma_{1})$, then
        \begin{align*}
            &\int_{\ell_{2}\bbT}
            \frac{\abs{\mathbf{T}_{1}(s') - \mathbf{T}_{2}(s)}}
            {\abs{\gamma_{1}(s') - \gamma_{2}(s)}
            \abs{x - \gamma_{3}(\phi^{-1}(s))}^{2\alpha}}\,ds
            \\&\qquad\qquad\qquad\leq C_{\alpha,\beta}
            \frac{\max\set{\ell_{2},\ell_{3}}\left(\norm{\phi}_{\dot{C}^{0,1}}+1\right)^{2}}
            {\min\set{
                \Delta_{\norm{\gamma_{1}}_{\dot{C}^{1,\beta}}^{-1/\beta}}(\gamma_{1}),
                \Delta_{\norm{\gamma_{2}}_{\dot{C}^{1,\beta}}^{-1/\beta}}(\gamma_{2}),
                \norm{\gamma_{3}}_{\dot{C}^{1,\beta}}^{-1/\beta}
            }^{1+2\alpha}}.
        \end{align*}

        \item If $\gamma_{1},\gamma_{2}\in\operatorname{PSC}(\bbR^{2})$ and
        $\Omega(\gamma_{1})\cap\Omega(\gamma_{2})=\emptyset$, then
        \begin{align*}
            &\int_{\ell_{2}\bbT}
            \frac{\abs{\mathbf{T}_{1}(s') + \mathbf{T}_{2}(s)}}
            {\abs{\gamma_{1}(s') - \gamma_{2}(s)}
            \abs{x - \gamma_{3}(\phi^{-1}(s))}^{2\alpha}}\,ds
            \\&\qquad\qquad\qquad\leq C_{\alpha,\beta}
            \frac{\max\set{\ell_{2},\ell_{3}}\left(\norm{\phi}_{\dot{C}^{0,1}}+1\right)^{2}}
            {\min\set{
                \Delta_{\norm{\gamma_{1}}_{\dot{C}^{1,\beta}}^{-1/\beta}}(\gamma_{1}),
                \Delta_{\norm{\gamma_{2}}_{\dot{C}^{1,\beta}}^{-1/\beta}}(\gamma_{2}),
                \norm{\gamma_{3}}_{\dot{C}^{1,\beta}}^{-1/\beta}
            }^{1+2\alpha}}.
        \end{align*}
    \end{enumerate}
\end{lemma}

\begin{proof}
    We only prove (1) because (2) and (3) follow in the same way but
     using Lemma~\ref{L3.5} instead of Lemma~\ref{L3.3}. 
     This replaces $\norm{\gamma_1}_{\dot{C}^{1,\beta}}^{-1/\beta}$ and $\norm{\gamma_2}_{\dot{C}^{1,\beta}}^{-1/\beta}$ by $\Delta_{\norm{\gamma_1}_{\dot{C}^{1,\beta}}^{-1/\beta}}(\gamma_1)$ and $\Delta_{\norm{\gamma_2}_{\dot{C}^{1,\beta}}^{-1/\beta}}(\gamma_2)$, respectively, in the bounds, while $\norm{\gamma_3}_{\dot{C}^{1,\beta}}$ still enters the proof  via Lemma~\ref{L3.2}.  Let
    \begin{align*}
        B\coloneqq\set{s\in\ell_{2}\bbT \colon
        \abs{\gamma_{1}(s') - \gamma_{2}(s)}<\abs{x - \gamma_{3}(\phi^{-1}(s))}}
    \end{align*}
    and
    \begin{align*}
        F_{1} &\coloneqq \int_{B}
        \frac{\abs{\mathbf{N}_{1}(s')\cdot\mathbf{T}_{2}(s)}}
            {\abs{\gamma_{1}(s') - \gamma_{2}(s)}
            \abs{x - \gamma_{3}(\phi^{-1}(s))}^{2\alpha}}\,ds, \\
        F_{2} &\coloneqq \int_{\ell_{2}\bbT\setminus B}
        \frac{\abs{\mathbf{N}_{1}(s')\cdot\mathbf{T}_{2}(s)}}
            {\abs{\gamma_{1}(s') - \gamma_{2}(s)}
            \abs{x - \gamma_{3}(\phi^{-1}(s))}^{2\alpha}}\,ds.
    \end{align*}
    Note that it is enough to bound $F_{2}$. Indeed,  then we can apply  to the right-hand side of
    \begin{align*}
        F_{1} &\leq \int_{\ell_{2}\bbT}
        \frac{\abs{\mathbf{N}_{1}(s')\cdot\mathbf{T}_{2}(s)}}
        {\abs{\gamma_{1}(s') - \gamma_{2}(s)}^{1+2\alpha}}\,ds
    \end{align*}
     the obtained bound for $F_2$ with 
    $x\coloneqq\gamma_{1}(s')$, $\gamma_{3}\coloneqq\gamma_{2}$, and $\phi\coloneqq{\rm Id}$ to similarly estimate $F_1$.

    To bound $F_{2}$, we apply Lemma~\ref{L3.2} with
    $\gamma\coloneqq\gamma_{2}$ and $x\coloneqq\gamma_{1}(s')$ to obtain
    $\set{s_{2,i}}_{i=1}^{N_{2}}\subseteq\ell_{2}\bbT$ with
    $N_{2}\leq \ell_{2}\norm{\gamma_{2}}_{\dot{C}^{1,\beta}}^{1/\beta}$,
    and also with $\gamma\coloneqq\gamma_{3}$ and the given $x$ to obtain
    $\set{s_{3,j}}_{j=1}^{N_{3}}\subseteq\ell_{3}\bbT$ with
    $N_{3}\leq \ell_{3}\norm{\gamma_{3}}_{\dot{C}^{1,\beta}}^{1/\beta}$.
    Let $M\coloneqq\max\left\{\norm{\gamma_{1}}_{\dot{C}^{1,\beta}},
    \norm{\gamma_{2}}_{\dot{C}^{1,\beta}},
    \norm{\gamma_{3}}_{\dot{C}^{1,\beta}}\right\}$ and
    $d_{0}\coloneqq\frac{1}{4}M^{-1/\beta}$.

    Fix any $j\in\{1,\dots ,N_{3}\}$.  Since the intervals $[s_{2,i}-2d_{0},s_{2,i}+2d_{0}]$
    are pairwise disjoint and
    \begin{align*}
        \abs{\phi^{-1}\left(
            [s_{2,i}-2d_{0},s_{2,i}+2d_{0}]
        \right)} \geq \frac{4d_{0}}{\norm{\phi}_{\dot{C}^{0,1}}},
    \end{align*}
    there are at most $\left(\left\lfloor\norm{\phi}_{\dot{C}^{0,1}}
    \right\rfloor + 2\right)$-many $i$'s such that
    $\phi^{-1}\left([s_{2,i}-2d_{0},s_{2,i}+2d_{0}]\right)$ intersects
    $[s_{3,j}-2d_{0},s_{3,j}+2d_{0}]$. Fix any such $i$, let
    $r\coloneqq\abs{\gamma_{1}(s') - \gamma_{2}(s_{2,i})}$ and
    \begin{align*}
        I\coloneqq [s_{3,j}-2d_{0},s_{3,j}+2d_{0}]
        \cap \phi^{-1}\left(
            [s_{2,i}-2d_{0},s_{2,i}+2d_{0}]
        \right).
    \end{align*}
  Then for any $s\in I$, Lemmas~\ref{L3.3} and \ref{L3.2}(4) show that
    \begin{align*}
        \abs{\mathbf{N}_{1}(s')\cdot\mathbf{T}_{2}(\phi(s))}
        &\leq 12M^{\frac{1}{1+\beta}}
        r^{\frac{\beta}{1+\beta}}
        + \norm{\gamma_{2}}_{\dot{C}^{1,\beta}}
        \abs{\phi(s) - s_{2,i}}^{\beta} \\
        &\leq 12M^{\frac{1}{1+\beta}}
        r^{\frac{\beta}{1+\beta}}
        + 2\norm{\gamma_{2}}_{\dot{C}^{1,\beta}}
        \abs{\gamma_{1}(s') - \gamma_{2}(\phi(s))}^{\beta}.
    \end{align*}
    Combining this again with Lemma~\ref{L3.3} and Lemma~\ref{L3.2} yields
    \begin{align*}
        &\int_{\phi(I)\setminus B}
        \frac{\abs{\mathbf{N}_{1}(s')\cdot\mathbf{T}_{2}(s)}}
        {\abs{\gamma_{1}(s') - \gamma_{2}(s)}
        \abs{x - \gamma_{3}(\phi^{-1}(s))}^{2\alpha}}\,ds
        \\&\quad\quad\leq \int_{I\setminus\phi^{-1}(B)}
        \frac{\abs{\mathbf{N}_{1}(s')\cdot\mathbf{T}_{2}(\phi(s))}
        \norm{\phi}_{\dot{C}^{0,1}}}
        {\abs{\gamma_{1}(s') - \gamma_{2}(\phi(s))}
        \abs{x - \gamma_{3}(s)}^{2\alpha}}\,ds \\
        &\quad\quad\leq \int_{\abs{s - s_{3,j}}<\min\{r,2d_{0}\}}
        \frac{12M^{\frac{1}{1+\beta}}\norm{\phi}_{\dot{C}^{0,1}}}
        {r^{\frac{1}{1+\beta}}\abs{x - \gamma_{3}(s)}^{2\alpha}}\,ds
        + \int_{\min\{r,2d_{0}\}\leq \abs{s - s_{3,j}}\leq 2d_{0}}
        \frac{12M^{\frac{1}{1+\beta}}r^{\frac{\beta}{1+\beta}}
        \norm{\phi}_{\dot{C}^{0,1}}}
        {\abs{x - \gamma_{3}(s)}^{1+2\alpha}}\,ds
        \\&\quad\quad\quad\quad
        + \int_{\min\{r,2d_{0}\}\leq \abs{s - s_{3,j}}\leq 2d_{0}}
        \frac{2\norm{\gamma_{2}}_{\dot{C}^{1,\beta}}\norm{\phi}_{\dot{C}^{0,1}}}
        {\abs{x - \gamma_{3}(s)}^{1+2\alpha-\beta}}\,ds.
    \end{align*}
    Denote the three terms on the right-hand side above by
    $F_{3}(i,j)$, $F_{4}(i,j)$, and $F_{5}(i,j)$.
    Since Lemma~\ref{L3.2} shows that
    $\abs{x - \gamma_{3}(s)} \geq \frac{\abs{s - s_{3,j}}}{2}$ for $s\in I$, it  follows that
    \begin{align*}
       \max\{ F_{3}(i,j), F_{4}(i,j) \} \leq C_{\alpha,\beta}M^{\frac{1}{1+\beta}}
       \norm{\phi}_{\dot{C}^{0,1}} d_{0}^{\frac{\beta}{1+\beta}-2\alpha}
    \end{align*}
    and
    \begin{align*}
        F_{5}(i,j) \leq C_{\alpha,\beta}\norm{\gamma_{2}}_{\dot{C}^{1,\beta}}
        \norm{\phi}_{\dot{C}^{0,1}}
        d_{0}^{\beta-2\alpha}.
    \end{align*}
    From the definition of $d_{0}$ we now see that
    all  $F_{3}(i,j)$, $F_{4}(i,j)$, and $F_{5}(i,j)$ are all bounded by
    $C_{\alpha,\beta}M^{2\alpha/\beta}\norm{\phi}_{\dot{C}^{0,1}}$.
    
    On the other hand, since Lemma~\ref{L3.2} shows that any
    $s\in[s_{3,j}-2d_{0},s_{3,j}+2d_{0}]$ that does not belong to any 
    $\phi^{-1}\left([s_{2,i}-2d_{0},s_{2,i}+2d_{0}]\right)$
    satisfies $\abs{\gamma_{1}(s') - \gamma_{2}(\phi(s))}\geq d_{0}$,
     Lemma~\ref{L3.2} yields
    \begin{align*}
        &\int_{\phi\left([s_{3,j}-2d_{0},s_{3,j}+2d_{0}]\right)\setminus B}
        \frac{\abs{\mathbf{N}_{1}(s')\cdot\mathbf{T}_{2}(s)}}
        {\abs{\gamma_{1}(s') - \gamma_{2}(s)}
        \abs{x - \gamma_{3}(\phi^{-1}(s))}^{2\alpha}}\,ds \\
        &\quad\quad\leq
        \left(\left\lfloor\norm{\phi}_{\dot{C}^{0,1}}
        \right\rfloor + 2\right)
        C_{\alpha,\beta}M^{2\alpha/\beta}\norm{\phi}_{\dot{C}^{0,1}}
        + \int_{s_{3,j}-2d_{0}}^{s_{3,j}+2d_{0}}
        \frac{\norm{\phi}_{\dot{C}^{0,1}}}
        {d_{0}\abs{x - \gamma_{3}(s)}^{2\alpha}}\,ds \\
        &\quad\quad\leq
        C_{\alpha,\beta}M^{2\alpha/\beta}
        \left(\norm{\phi}_{\dot{C}^{0,1}}+1\right)^{2}
    \end{align*}
    for any $j=1, \dots ,N_{2}$.

    Finally, Lemma~\ref{L3.2} again shows that any $s\in\ell_{2}\bbT\setminus B$
    that does not belong to any $\phi\left([s_{3,j}-2d_{0},s_{3,j}+2d_{0}]\right)$
    satisfies $\abs{x - \gamma_{3}(\phi^{-1}(s))}\geq d_{0}$, thus Lemma~\ref{L3.2} yields
    \begin{align*}
        F_{2} &\leq N_{3} C_{\alpha,\beta}M^{2\alpha/\beta}
        \left(\norm{\phi}_{\dot{C}^{0,1}}+1\right)^{2}
        +\int_{\ell_{2}\bbT}\frac{ds}{d_{0}^{1+2\alpha}}
        \leq C_{\alpha,\beta}\max\set{\ell_{2},\ell_{3}}M^{\frac{1+2\alpha}{\beta}}
        \left(\norm{\phi}_{\dot{C}^{0,1}}+1\right)^{2}.
    \end{align*}
    This finishes the proof of (1).
\end{proof}

The first application of Lemma~\ref{L4.5} is a
uniform bound on the \emph{tangential derivative} of the velocity field on patch boundaries.
Note that we cannot expect such a bound to hold for the \emph{normal derivative} because
the velocity is only H\"{o}lder continuous at the patch boundaries.

\begin{lemma}\label{L4.6}
    If $\frac{\beta}{1+\beta}\geq 2\alpha$, then there is
     $C_{\alpha,\beta}$  such that for any $z\in\operatorname{CC}(\bbR^{2})^{\mathcal{L}}$ with
    each $z^{\lambda}$ being $C^{1,\beta}$ and not crossing
    itself and all the other $z^{\lambda'}$ transversally, we have
    \[
        \norm{\partial_{s}(u(z)\circ z^{\lambda})}_{L^{\infty}}
        \leq C_{\alpha,\beta}\sum_{\lambda'\in\mathcal{L}}\abs{\theta^{\lambda'}}
        \ell(z^{\lambda'})\max\left\{\norm{z^{\lambda}}_{\dot{C}^{1,\beta}},
        \norm{z^{\lambda'}}_{\dot{C}^{1,\beta}}\right\}^{\frac{1+2\alpha}{\beta}}
    \]
     for each $\lambda\in\mathcal{L}$, where $z^\lambda$ on the left-hand side is
    any arclength parametrization of this curve.
\end{lemma}

\begin{proof}
    For each $\lambda\in\mathcal{L}$, let
    $\mathbf{T}^{\lambda}\coloneqq\partial_{s}z^{\lambda}$
    and $\mathbf{N}^{\lambda}\coloneqq(\mathbf{T}^{\lambda})^{\perp}$.
    For any fixed $\lambda$, it is enough to show that the function
    \[
        \partial_{s}(u_{\eps}(z)\circ z^{\lambda}) (s)
= -\sum_{\lambda'\in\mathcal{L}}\theta^{\lambda'}
        \int_{\ell(z^{\lambda'})\bbT}
        DK_{\eps}(z^{\lambda}(s) - z^{\lambda'}(s'))(\mathbf{T}^{\lambda}(s))
        \mathbf{T}^{\lambda'}(s') \,ds'
    \]
    (where $DK_\eps\colon \bbR^{2} \to \operatorname{L}(\bbR^{2};\bbR)$)
    satisfies the desired estimate for all small enough $\eps>0$.
    That is, we need to show that
    \begin{equation} \lb{111.3}
        \abs{\int_{\ell(z^{\lambda'})\bbT}
        DK_{\eps}(z^{\lambda}(s) - z^{\lambda'}(s'))(\mathbf{T}^{\lambda}(s))
        \mathbf{T}^{\lambda'}(s')\,ds'}
       % \\&\quad\quad\quad 
        \leq
        C_{\alpha,\beta}\ell(z^{\lambda'})\max\left\{\norm{z^{\lambda}}_{\dot{C}^{1,\beta}},
        \norm{z^{\lambda'}}_{\dot{C}^{1,\beta}}\right\}^{\frac{1+2\alpha}{\beta}}
    \end{equation}
    holds for any $s\in\ell(z^{\lambda})\mathbb{T}$ and $\lambda'\in\mathcal{L}$.
    We separately estimate the $\mathbf{T}^{\lambda}(s)$-component
    and the $\mathbf{N}^{\lambda}(s)$-component of the integral above.

    \textbf{The $\mathbf{T}^{\lambda}(s)$-component}. Since
    \[
        \mathbf{T}^{\lambda'}(s')
        = (\mathbf{T}^{\lambda'}(s')\cdot\mathbf{T}^{\lambda}(s))\mathbf{T}^{\lambda}(s)
        + (\mathbf{T}^{\lambda'}(s')\cdot\mathbf{N}^{\lambda}(s))\mathbf{N}^{\lambda}(s)
    \]
    and
    \[
        \partial_{s'}\left(K_{\eps}(z^{\lambda}(s) - z^{\lambda'}(s'))\right)
        = -DK_{\eps}(z^{\lambda}(s) - z^{\lambda'}(s'))(\mathbf{T}^{\lambda'}(s'))
    \]
    integrates to zero over $\ell(z^{\lambda'})\bbT$, we obtain
    \begin{align*}
        &\int_{\ell(z^{\lambda'})\bbT}
        DK_{\eps}(z^{\lambda}(s) - z^{\lambda'}(s'))(\mathbf{T}^{\lambda}(s))
        (\mathbf{T}^{\lambda'}(s')\cdot\mathbf{T}^{\lambda}(s))\,ds'
        \\&\qquad\qquad\qquad =
        -\int_{\ell(z^{\lambda'})\bbT}
        DK_{\eps}(z^{\lambda}(s) - z^{\lambda'}(s'))(\mathbf{N}^{\lambda}(s))
        (\mathbf{T}^{\lambda'}(s')\cdot\mathbf{N}^{\lambda}(s))\,ds'.
    \end{align*}
    Now pick $\eps$-independent $C\geq 0$ such that
    $\abs{DK_{\eps}(x)(h)} \leq \frac{C\abs{h}}{\abs{x}^{1+2\alpha}}$ for all small $\eps>0$ and all $x,h\in\bbR^2$.  Then
    Lemma~\ref{L4.5} with $\gamma_{1} \coloneqq z^{\lambda}$,
    $\gamma_{2}=\gamma_{3} \coloneqq z^{\lambda'}$, $x \coloneqq z^{\lambda}(s)$, and
    $\phi\coloneqq{\rm Id}$ shows
    \begin{equation}\label{4.3}\begin{split}
        &\abs{\int_{\ell(z^{\lambda'})\bbT}
        DK_{\eps}(z^{\lambda}(s) - z^{\lambda'}(s'))(\mathbf{N}^{\lambda}(s))
        (\mathbf{T}^{\lambda'}(s')\cdot\mathbf{N}^{\lambda}(s)) \,ds'}
        \\&\quad\quad\leq
        \int_{\ell(z^{\lambda'})\bbT}
        \frac{C\abs{\mathbf{T}^{\lambda'}(s')\cdot\mathbf{N}^{\lambda}(s)}}
        {\abs{z^{\lambda}(s) - z^{\lambda'}(s')}^{1+2\alpha}}ds'
        \leq C_{\alpha,\beta}\ell(z^{\lambda'})\max\left\{\norm{z^{\lambda}}_{\dot{C}^{1,\beta}},
        \norm{z^{\lambda'}}_{\dot{C}^{1,\beta}}\right\}^{\frac{1+2\alpha}{\beta}}
    \end{split}\end{equation}
    as desired.

    \textbf{The $\mathbf{N}^{\lambda}(s)$-component}.
    The argument yielding \eqref{4.3} shows the same bound for 
    \[
    \abs{\int_{\ell(z^{\lambda'})\bbT}
        DK_{\eps}(z^{\lambda}(s) - z^{\lambda'}(s'))(\mathbf{T}^{\lambda}(s))
        (\mathbf{T}^{\lambda'}(s')\cdot\mathbf{N}^{\lambda}(s)) \,ds'}  .
        \]
      This and \eqref{4.3} now yield \eqref{111.3} and the proof is finished.
\end{proof}

Even though the normal derivative of the velocity field is unbounded at the patch boundaries,
we can still obtain a Lipschitz bound for
the \emph{normal component} of the velocity field % (cf.~\cite{JeoZla}).
when the patch boundaries have no transversal crossings.  However, in Section~\ref{S8} we will need a more general version of such a bound that allows for transversal crossings when we show that any $H^2$ patch solution with no  crossings initially must remain such.  This will add a H\" older term in the relevant bounds in the next two lemmas, with a constant depending on the maximal  crossing angle.

\begin{lemma}\label{L4.8}
    If $\frac{\beta}{1+\beta}\geq 2\alpha$, then there is
    $C_{\alpha,\beta}$ such that for any $C^{1,\beta}$ closed curves
    $\gamma_{i}\colon\ell_{i}\bbT\to\bbR^{2}$ ($i=1,2$) parametrized by arclength, with
    $\mathbf{T}_{i}\coloneqq\partial_{s}\gamma_{i}$ and
    $\mathbf{N}_{i}\coloneqq\mathbf{T}_{i}^{\perp}$, % for $i=1,2$, 
    and for any
    $x\in\bbR^{2}$ and $s'\in\ell_{1}\bbT$ we have
    \begin{align*}
        &\int_{\ell_{2}\bbT}
        \abs{K(x - \gamma_{2}(s)) - K(\gamma_{1}(s') - \gamma_{2}(s)) }
        \abs{\mathbf{N}_{1}(s')\cdot\mathbf{T}_{2}(s)}\,ds
        \\&\qquad\qquad\qquad\leq C_{\alpha,\beta}\, \ell_{2} M^{1/\beta}
        \left(
            M^{2\alpha/\beta}\abs{x - \gamma_{1}(s') }
            + \Xi\abs{x - \gamma_{1}(s') }^{1-2\alpha}
        \right),
    \end{align*}
    where $M\coloneqq\max\left\{
            \norm{\gamma_{1}}_{\dot{C}^{1,\beta}},
            \norm{\gamma_{2}}_{\dot{C}^{1,\beta}}
        \right\}$ and (with $\max\emptyset\coloneqq 0$)
    \[
        \Xi\coloneqq \max\left\{\abs{\mathbf{N}_{1}(s_{1})\cdot\mathbf{T}_{2}(s_{2})}
        \colon (s_{1},s_{2})\in\ell_{1}\bbT\times\ell_{2}\bbT
        \,\, \&\,\, \gamma_{1}(s_{1}) = \gamma_{2}(s_{2}) \right\}.
    \]
\end{lemma}

\begin{proof}
    Let      $d\coloneqq\abs{x-\gamma_{1}(s')}$,
   % $M\coloneqq\max\set{\norm{\gamma_{1}}_{\dot{C}^{1,\beta}}   \norm{\gamma_{2}}_{\dot{C}^{1,\beta}}}$,
    $d_{0}\coloneqq\frac{1}{4}M^{-1/\beta}$, and 
    \[
        B\coloneqq\set{s\in\ell_{2}\bbT \colon
        \abs{\gamma_{1}(s') - \gamma_{2}(s)}<\abs{x - \gamma_{2}(s)}}.
    \]
    Apply Lemma~\ref{L3.2} twice --- first with $\gamma\coloneqq\gamma_{2}$ and
    $x\coloneqq\gamma_{1}(s')$, and then with the same $\gamma$ and the given $x$ ---
    to respectively obtain $\set{s_{1,i}}_{i=1}^{N_{1}}, \set{s_{2,j}}_{j=1}^{N_{2}}\subseteq\ell_{2}\bbT$ with
    $N_{1},N_{2}\leq \ell_{2}\norm{\gamma_{2}}_{\dot{C}^{1,\beta}}^{1/\beta}$.

    Fix any $j\in\{1,\dots ,N_{2}\}$.  Since the intervals $[s_{1,i}-2d_{0},s_{1,i}+2d_{0}]$ ($i=1,\dots,N_1$)
    are pairwise disjoint, there are at most two $i$'s such that
    $[s_{1,i}-2d_{0},s_{1,i}+2d_{0}]$ intersects
    $[s_{2,j}-2d_{0},s_{2,j}+2d_{0}]$. Fix any such $i$, let
    $r\coloneqq\abs{\gamma_{1}(s') - \gamma_{2}(s_{1,i})}=\min_{|s-s_{1,i}|\le 2d_0}\abs{\gamma_{1}(s') - \gamma_{2}(s)}$,
    $h\coloneqq (r/M)^{\frac{1}{1+\beta}}$, and
    \begin{align*}
        I\coloneqq [s_{2,j}-2d_{0},s_{2,j}+2d_{0}]
        \cap [s_{1,i}-2d_{0},s_{1,i}+2d_{0}].
    \end{align*}
    If $\gamma_{1}(s_{1}) = \gamma_{2}(s_{2})$  for some
    $(s_{1},s_{2})\in [s'-h,s'+h]\times [s_{1,i}-h,s_{1,i}+h]$, then
    \begin{align*}
        \abs{\mathbf{N}_{1}(s')\cdot\mathbf{T}_{2}(s_{1,i})}
        &\leq \norm{\gamma_{1}}_{\dot{C}^{1,\beta}}h^{\beta}
        + \norm{\gamma_{2}}_{\dot{C}^{1,\beta}}h^{\beta}
        + \abs{\mathbf{N}_{1}(s_{1})\cdot\mathbf{T}_{2}(s_{2})} \\
        &\leq 2M^{\frac{1}{1+\beta}}r^{\frac{\beta}{1+\beta}}
        + \abs{\mathbf{N}_{1}(s_{1})\cdot\mathbf{T}_{2}(s_{2})}.
    \end{align*}
    If there is no such pair $(s_{1},s_{2})$, then Lemma~\ref{L3.3} shows that
    $\abs{\mathbf{N}_{1}(s')\cdot\mathbf{T}_{2}(s_{1,i})} \le 12M^{\frac{1}{1+\beta}}r^{\frac{\beta}{1+\beta}}$.
    In either case, Lemma~\ref{L3.2}(4) now shows that for all $s\in I$ we have
    \begin{align*}
        \abs{\mathbf{N}_{1}(s')\cdot\mathbf{T}_{2}(s)}
        &\leq 12M^{\frac{1}{1+\beta}}
        r^{\frac{\beta}{1+\beta}}  + \Xi
        + \norm{\gamma_{2}}_{\dot{C}^{1,\beta}}
        \abs{s - s_{1,i}}^{\beta}\\
        &\leq 12M^{\frac{1}{1+\beta}}
        r^{\frac{\beta}{1+\beta}}  + \Xi
        + 2\norm{\gamma_{2}}_{\dot{C}^{1,\beta}}
        \abs{\gamma_{1}(s') - \gamma_{2}(s)}^{\beta} .
    \end{align*} 
    Combining this again with Lemma~\ref{L3.2} (recall  that  $r\le \inf_{s\in I}\abs{\gamma_{1}(s') - \gamma_{2}(s)}$), and with 
    \begin{equation}\label{222.2}
        \abs{\frac{1}{a^{2\alpha}} - \frac{1}{b^{2\alpha}}}
        \leq \min\set{\frac{1}{\min\set{a,b}^{2\alpha}},
        \frac{\abs{a - b}}{\abs{a}\abs{b}^{2\alpha}}}
        \leq \min\set{\frac{1}{\min\set{a,b}^{2\alpha}},
        \frac{\abs{a - b}}{\min\set{a,b}^{1+2\alpha}}}
    \end{equation}
    for $a,b>0$ (which uses $|c^{2\alpha}-1|\le |c-1|$ for $c\ge 0$), yields
    \begin{align*}
        &\int_{I}
        \abs{ K(x - \gamma_{2}(s)) - K(\gamma_{1}(s') - \gamma_{2}(s))}
        \abs{\mathbf{N}_{1}(s')\cdot\mathbf{T}_{2}(s)}\,ds \\
        &\ \leq \int_{I\cap B}
       \abs{ K(x - \gamma_{2}(s)) - K(\gamma_{1}(s') - \gamma_{2}(s))}
        \left( 12M^{\frac{1}{1+\beta}}        r^{\frac{\beta}{1+\beta}} +\Xi \right) \,ds
        \\&\quad\quad
        + \int_{I\setminus B}
       \abs{ K(x - \gamma_{2}(s)) - K(\gamma_{1}(s') - \gamma_{2}(s))}
        \left(12M^{\frac{1}{1+\beta}}        r^{\frac{\beta}{1+\beta}} + \Xi \right) \,ds 
          \\&\quad\quad
         + \int_{I}
        \abs{ K(x - \gamma_{2}(s)) - K(\gamma_{1}(s') - \gamma_{2}(s))}
      \,  2\norm{\gamma_{2}}_{\dot{C}^{1,\beta}}
        \abs{\gamma_{1}(s') - \gamma_{2}(s)}^{\beta}  \,ds \\
        &\ \leq \int_{\abs{s - s_{1,i}}<\min\{r,2d_{0}\}}
        \frac{C_{\alpha}M^{\frac{1}{1+\beta}}d}
        {r^{\frac{1}{1+\beta}+2\alpha}}\,ds
        + \int_{\min\{r,2d_{0}\}\leq \abs{s - s_{1,i}}\leq 2d_{0}}
        \frac{C_{\alpha}M^{\frac{1}{1+\beta}}r^{\frac{\beta}{1+\beta}}d}
        {\abs{\gamma_{1}(s') - \gamma_{2}(s)}^{1+2\alpha}}\,ds
        \\&\quad\quad
        + \int_{\abs{s - s_{2,j}}<\min\{r,2d_{0}\}}
        \frac{C_{\alpha}M^{\frac{1}{1+\beta}}d}
        {r^{\frac{1}{1+\beta}}\abs{x - \gamma_{2}(s)}^{2\alpha}}\,ds
        + \int_{\min\{r,2d_{0}\}\leq \abs{s - s_{2,j}}\leq 2d_{0}}
        \frac{C_{\alpha}M^{\frac{1}{1+\beta}}r^{\frac{\beta}{1+\beta}}
        d}
        {\abs{x - \gamma_{2}(s)}^{1+2\alpha}}\,ds
                \\&\quad\quad
        + \Xi\int_{I\cap B}
        \abs{K(\gamma_{1}(s') - \gamma_{2}(s))
        - K(x - \gamma_{2}(s))}\,ds
        \\&\quad\quad
        + \Xi\int_{I\setminus B}
        \abs{K(\gamma_{1}(s') - \gamma_{2}(s))
        - K(x - \gamma_{2}(s))}\,ds
        \\&\quad\quad
        + \int_{s_{1,i}-2d_{0}}^{s_{1,i}+2d_{0}}
        \frac{C_{\alpha}\norm{\gamma_{2}}_{\dot{C}^{1,\beta}}d}
        {\abs{\gamma_{1}(s') - \gamma_{2}(s)}^{1+2\alpha-\beta}}\,ds
        + \int_{s_{2,j}-2d_{0}}^{s_{2,j}+2d_{0}}
        \frac{C_{\alpha}\norm{\gamma_{2}}_{\dot{C}^{1,\beta}}
        d}
        {\abs{x - \gamma_{2}(s)}^{1+2\alpha-\beta}}\,ds.
    \end{align*}
    
    Denote the eight terms on the right-hand side above by $F_{k}(i,j)$ ($k=1,\dots,8$) in the order of appearance.
    Since Lemma~\ref{L3.2} shows that
    $\abs{\gamma_{1}(s') - \gamma_{2}(s)} \geq \frac{\abs{s - s_{1,i}}}{2}$ and
    $\abs{x - \gamma_{2}(s)} \geq \frac{\abs{s - s_{2,j}}}{2}$ for all $s\in I$, we see from this and the definition of $d_{0}$ that
    \begin{align*}
       \max_{k=1,2,3,4}F_{k}(i,j) \leq C_{\alpha}M^{\frac{1}{1+\beta}}
        d_{0}^{\frac{\beta}{1+\beta}-2\alpha}d \le C_{\alpha}M^{2\alpha/\beta}d
    \end{align*}
    and
    \begin{align*}
        \max\set{F_{7}(i,j), F_{8}(i,j)} \leq C_{\alpha,\beta}M d_{0}^{\beta-2\alpha}d \le C_{\alpha,\beta}M^{2\alpha/\beta}d.
    \end{align*}
%    The definition of $d_{0}$ now shows that
%    all $F_{n}(i,j)$ with $n=1,2,3,4,5,6$ are bounded by
%    $C_{\alpha,\beta}M^{2\alpha/\beta}d$.
%    To estimate $F_{7}(i,j)$ and $F_{8}(i,j)$,
    Note also that \eqref{222.2} and Lemma~\ref{L3.2} yield
    \begin{align*}
        F_{5}(i,j) &\leq \int_{\abs{s - s_{1,i}} < \min\set{d,2d_{0}}}
        \frac{C_{\alpha}\Xi}{\abs{\gamma_{1}(s') - \gamma_{2}(s)}^{2\alpha}}\,ds
        + \int_{\min\set{d,2d_{0}} \leq \abs{s - s_{1,i}} \leq 2d_{0}}
        \frac{C_{\alpha}\Xi d}{\abs{\gamma_{1}(s') - \gamma_{2}(s)}^{1+2\alpha}}\,ds \\
        &\leq C_{\alpha}\Xi\min\set{d,2d_{0}}^{1-2\alpha}
        + \frac{C_{\alpha}\Xi d}{\min\set{d,2d_{0}}^{2\alpha}}
        \leq C_{\alpha}\Xi d^{1-2\alpha} + C_{\alpha}M^{2\alpha/\beta}d,
    \end{align*}
%    where $\frac{0}{0}\coloneqq 0$, 
and similarly
    \[
        F_{6}(i,j) \leq C_{\alpha}\Xi d^{1-2\alpha} + C_{\alpha}M^{2\alpha/\beta}d.
    \]
These estimates, \eqref{222.2}, and Lemma~\ref{L3.2} (which also shows that $\abs{\gamma_{1}(s') - \gamma_{2}(s)}\geq d_{0}$ holds for any   $s\in\ell_{2}\bbT\setminus\bigcup_{i=1}^{N_{1}}[s_{1,i}-2d_{0},s_{1,i}+2d_{0}]$) now yield
%    , so
%    the estimates for $F_{n}(i,j)$ ($n=1,2,3,4,5,6,7,8$),
%    Lemma~\ref{L3.2}, and \eqref{222.2} yield
    \begin{align*}
        &\int_{s_{2,j}-2d_{0}}^{s_{2,j}+2d_{0}}
        \abs{K(\gamma_{1}(s') - \gamma_{2}(s))
        - K(x - \gamma_{2}(s))}
        \abs{\mathbf{N}_{1}(s')\cdot\mathbf{T}_{2}(s)}\,ds \\
        &\quad\quad\leq C_{\alpha,\beta}
        \left(
            M^{2\alpha/\beta}d + \Xi d^{1-2\alpha}
        \right)
        + \int_{s_{2,j}-2d_{0}}^{s_{2,j}+2d_{0}}
        \frac{d}
        {d_{0}\abs{x - \gamma_{2}(s)}^{2\alpha}}\,ds \\
        &\quad\quad\leq
        C_{\alpha,\beta}
        \left(
            M^{2\alpha/\beta}d + \Xi d^{1-2\alpha}
        \right),
    \end{align*}
   which therefore holds for each $j\in\{1, \dots ,N_{2}\}$.

    This and Lemma~\ref{L3.2}, which again shows that any
    $s\in\ell_{2}\bbT\setminus\bigcup_{j=1}^{N_{2}}[s_{2,j}-2d_{0},s_{2,j}+2d_{0}]$
    satisfies $\abs{x - \gamma_{2}(s)}\geq d_{0}$,  yield
    \begin{align*}
        &\int_{\ell_{2}\bbT}
        \abs{K(\gamma_{1}(s') - \gamma_{2}(s))
        - K(x - \gamma_{2}(s))}
        \abs{\mathbf{N}_{1}(s')\cdot\mathbf{T}_{2}(s)}\,ds \\
        &\quad\quad
        \leq  C_{\alpha,\beta} N_{2}
        \left(
            M^{2\alpha/\beta}d + \Xi d^{1-2\alpha}
        \right)
        + \int_{\ell_{2}\bbT}\frac{d}{d_{0}^{1+2\alpha}}\,ds \\
        &\quad\quad \leq C_{\alpha,\beta}\ell_{2}M^{1/\beta}
        \left(
            M^{2\alpha/\beta}d + \Xi d^{1-2\alpha}
        \right),
    \end{align*}
    finishing the proof.
\end{proof}

\begin{lemma}\label{L4.7}
    If $\frac{\beta}{1+\beta}\geq 2\alpha$, then there is
    $C_{\alpha,\beta}$ such that the following holds.  If $z\in\operatorname{CC}(\bbR^{2})^{\mathcal{L}}$
    is such that each $z^{\lambda}$ is $C^{1,\beta}$, 
    then for any $\lambda\in\mathcal{L}$, $s\in\ell(z^{\lambda})\bbT$, and $x\in\bbR^{2}$ we have
    \begin{align*}
        &\abs{\left[
            u(z;x) - u(z;z^{\lambda}(s))
        \right]
        \cdot \partial_{s}z^{\lambda}(s)^{\perp}}
        \\&\quad\quad\quad
        \leq C_{\alpha,\beta}
        \sum_{\lambda'\in\mathcal{L}}
        \abs{\theta^{\lambda'}}\ell(z^{\lambda'})
        M_{\lambda,\lambda'}^{1/\beta}
%        \\&\quad\quad\quad\quad\quad\quad\quad
        \left(
           M_{\lambda,\lambda'}^{2\alpha/\beta}\abs{x - z^{\lambda}(s)}
            + \Xi_\lambda \abs{x - z^{\lambda}(s)}^{1-2\alpha}
        \right)
    \end{align*}
    where $M_{\lambda,\lambda'} \coloneqq\max\left\{
            \norm{z^{\lambda}}_{\dot{C}^{1,\beta}},
            \norm{z^{\lambda'}}_{\dot{C}^{1,\beta}}
        \right\}$ and
    \[
        \Xi_\lambda\coloneqq \max_{\lambda'\in\mathcal{L}}\,
        \max\set{\abs{\partial_{s}z^{\lambda}(s)^{\perp}
        \cdot\partial_{s}z^{\lambda'}(s')}
        \colon (s,s')\in\ell(z^{\lambda})\bbT\times\ell(z^{\lambda'})\bbT
        \,\, \&\,\, z^{\lambda}(s) = z^{\lambda'}(s')}.
    \]
\end{lemma}

\begin{proof}
    For each $\lambda\in\mathcal{L}$, let
    $\mathbf{T}^{\lambda}\coloneqq\partial_{s}z^{\lambda}$
    and $\mathbf{N}^{\lambda}\coloneqq(\mathbf{T}^{\lambda})^{\perp}$. Then since
    \begin{align*}
        &\left[
            u(z;x) - u(z;z^{\lambda}(s))
        \right]\cdot\mathbf{N}^{\lambda}(s)
        \\&\qquad 
        =-\sum_{\lambda'\in\mathcal{L}}\theta^{\lambda'}
        \int_{\ell(z^{\lambda'})\bbT}
        \left[
            K(x - z^{\lambda'}(s'))
            - K(z^{\lambda}(s) - z^{\lambda'}(s'))
        \right]
        (\mathbf{T}^{\lambda'}(s')\cdot\mathbf{N}^{\lambda}(s))
        \,ds',
    \end{align*}
    Lemma~\ref{L4.8} with $\gamma_{1}\coloneqq z^{\lambda}$ and
    $\gamma_{2}\coloneqq z^{\lambda'}$ gives us the desired estimate.
\end{proof}

%%%%%%%%%%%%%%%%%%%%%%%%%%%%%%%%%%%%%%%%%%%%%%
\section{Mollified g-SQG equations}\label{S5}
%%%%%%%%%%%%%%%%%%%%%%%%%%%%%%%%%%%%%%%%%%%%%%

From now on we will assume that $\alpha\in \left(0,\frac{1}{6}\right]$, which implies $\frac{\beta}{1+\beta} \geq 2\alpha$ when $\beta = \frac 12$.
Let $\mathcal{H}\coloneqq H^{2}(\bbT;\bbR^{2})^{\mathcal{L}}$ and let
\[
    \pi\colon \mathcal{H}\to \operatorname{RC}(\bbR^{2})^{\mathcal{L}}
\]
be the natural projection. Fix any $\eps>0$, and for any $z\in \operatorname{RC}(\bbR^{2})^{\mathcal{L}}$
consider the mollified velocity field $u_{\eps}(z)$ from \eqref{111.2}.
%\[
%    u_{\eps}(z;x) \coloneqq
%    - \sum_{\lambda\in\mathcal{L}}\theta^{\lambda}\int_{\ell(z^{\lambda})\bbT}
%    K_{\eps}(x - z^{\lambda}(s))
%   \, \partial_{s}z^{\lambda}(s)
%    \,ds
%\]
%that is generated by $z$. 
Clearly it  is infinitely differentiable and all
its derivatives are bounded. 
%If $v\colon\bbR^{2}\to\bbR^{2}$ is infinitely differentiable and all of its derivatives are bounded, then it can be easily seen that the map
Then $\tilde{\gamma}\mapsto u_{\eps}(\pi(\tilde z))\circ\tilde{\gamma}$ is a locally Lipschitz map from
$H^{2}(\bbT;\bbR^{2})$ to itself for each $\tilde{z}\in\mathcal{H}$, and   we now define the locally Lipschitz map $F_{\eps}\colon \mathcal{H}\to \mathcal{H}$ via
\[   
    F_{\eps}(\tilde{z})^\lambda \coloneqq
    u_{\eps}(\pi(\tilde z)) \circ \tilde{z}^{\lambda}.
\]

Therefore, for any $\tilde{z}^{0}\in\mathcal{H}$, the differential equation
\begin{equation}\label{5.1}
    \partial_{t}\tilde{z}_{\eps}^{t} = F_{\eps}(\tilde{z}_{\eps}^{t})
\end{equation}
with initial data $\tilde{z}^{0}$ has a unique maximal solution in $\mathcal{H}$.
We will show in Corollary~\ref{C5.6} that when
$\pi(\tilde{z}^{0})\in\operatorname{PSC}(\bbR^{2})^{\mathcal{L}}$,
this solution must be global (i.e., defined on all of $\bbR$).
We will also show in Corollary \ref{C5.3} that for any 
$z^{0}\in \operatorname{RC}(\bbR^{2})^{\mathcal{L}}$ with  $\|z^0\|_{\dot{H}^{2}}<\infty$,
all solutions to \eqref{5.1} whose initial data  consist of $H^2$ parametrizations
of $z^{0}$ must have the same projection down to
$\operatorname{RC}(\bbR^{2})^{\mathcal{L}}$ (via $\pi$).
This shows that the dynamic of \eqref{5.1} is independent of the parametrization, and it follows from the next lemma, which is a Lipschitz version of the H\" older-type estimate in Lemma~\ref{L4.4} (but for the kernel $K_\eps$ and with the Lipschitz constant 
depending on $\eps$).
%
%, because it is determined by the norm of
%the \emph{full} Fr\'{e}chet derivative $DK_{\eps}$,
%including not only the tangential derivative but also the normal derivative.
%As noted in Lemma~\ref{L4.6}, the tangential derivative alone can be controlled
%independent of $\eps$ as long as the lengths and the $\dot{C}^{1,\beta}$-norms
%are uniformly bounded.

\begin{lemma}\label{L5.1}
    For any $x_{1},x_{2}\in\bbR^{2}$ and
    $C^{1}$ parametrized curves $\tilde{\gamma}_{1},\tilde{\gamma}_{2}\colon\bbT\to\bbR^{2}$ we have
    \begin{align*}
        &\abs{\int_{\bbT}K_{\eps}(x_{1} - \tilde{\gamma}_{1}(\xi))
        \partial_{\xi}\tilde{\gamma}_{1}(\xi)\,d\xi
        - \int_{\bbT}K_{\eps}(x_{2} - \tilde{\gamma}_{2}(\xi))
        \partial_{\xi}\tilde{\gamma}_{2}(\xi)\,d\xi}
        \\&\quad\quad\quad
        \leq \norm{DK_{\eps}}_{L^{\infty}}
        \left(\ell(\gamma_{1}) + \ell(\gamma_{2})\right)
        \left(
            \abs{x_{1} - x_{2}}
            + d_{\mathrm{F}}(\gamma_{1},\gamma_{2})
        \right)
    \end{align*}
    where $\gamma_{i}$ is the curve defined by $\tilde{\gamma}_{i}$.
\end{lemma}

\begin{proof}
    For any orientation-preserving diffeomorphism $\phi\colon\bbT\to\bbT$,
    integration by parts gives
    \begin{align*}
        &\abs{\int_{\bbT}K_{\eps}(x_{1} - \tilde{\gamma}_{1}(\xi))
        \partial_{\xi}\tilde{\gamma}_{1}(\xi)\,d\xi
        - \int_{\bbT}K_{\eps}(x_{2} - \tilde{\gamma}_{2}(\xi))
        \partial_{\xi}\tilde{\gamma}_{2}(\xi)\,d\xi}
        \\&\quad\quad\quad\leq
        \abs{\int_{\bbT}
            DK_{\eps}(x_{1} - \tilde{\gamma}_{1}\circ\phi(\xi))
            \left(\partial_{\xi}(\tilde{\gamma}_{1}\circ\phi)(\xi)\right)
        (\tilde{\gamma}_{1}\circ\phi(\xi) - \tilde{\gamma}_{2}(\xi))\,d\xi}
        \\&\quad\quad\quad\quad\quad\quad+
        \abs{\int_{\bbT}\left(
            K_{\eps}(x_{1} - \tilde{\gamma}_{1}\circ\phi(\xi))
            - K_{\eps}(x_{2} - \tilde{\gamma}_{2}(\xi))
        \right)\partial_{\xi}\tilde{\gamma}_{2}(\xi)\,d\xi}
        \\&\quad\quad\quad\leq
        \norm{DK_{\eps}}_{L^{\infty}}
        \norm{\tilde{\gamma}_{1}\circ\phi - \tilde{\gamma}_{2}}_{L^{\infty}}
        \ell(\gamma_{1})
        + \norm{DK_{\eps}}_{L^{\infty}}\left(
            \abs{x_{1} - x_{2}}
            + \norm{\tilde{\gamma}_{1}\circ\phi - \tilde{\gamma}_{2}}_{L^{\infty}}
        \right)\ell(\gamma_{2}).
    \end{align*}
    Taking the infimum over $\phi$ and using
    Lemma~\ref{LA.1} now finishes the proof.
\end{proof}

\begin{lemma}\label{L5.2}
    Let $\tilde{z}_{\eps,i}\in C^1(I_{i}; \mathcal{H})$ ($i=1,2$)
%    $\tilde{z}_{\eps,2}\colon I_{2}\to\mathcal{H}$ 
be the unique solution to
    \eqref{5.1} with initial data
    $\tilde{z}_{i}^{0}\in \mathcal{H}$. 
    %and $\tilde{z}_{2}^{0}\in \mathcal{H}$, respectively.
    Then there is $C$ that only depends on ${\alpha,|\theta|,\eps,\tilde{z}_{1}^{0},\tilde{z}_{2}^{0}}$ such that
    \[
        d_{\mathrm{F}}(
            \pi(\tilde{z}_{\eps,1}^{t}),
            \pi(\tilde{z}_{\eps,2}^{t})
        )
        \leq e^{C\abs{t}}
        d_{\mathrm{F}}(\pi(\tilde{z}_{1}^{0}), \pi(\tilde{z}_{2}^{0}))
    \]
    holds for all $t\in[-\frac 1C,\frac 1C]\cap I_1\cap I_2$.
\end{lemma}

\begin{proof}
    Let us denote $z_{\eps,i}^{t}\coloneqq\pi(\tilde{z}_{\eps,i}^{t})$
    for $t\in I_{i}$. Since $\tilde{z}_{\eps,i}\colon I_{i}\to \mathcal{H}$ is
    continuous, there is
    $M_{1}$ depending on ${\alpha,|\theta|,\eps,\tilde{z}_{1}^{0},\tilde{z}_{2}^{0}}$
    such that
    \[
        \ell(\tilde{z}_{\eps,1}^{t,\lambda})
        + \ell(\tilde{z}_{\eps,2}^{t,\lambda})
        \leq \norm{\tilde{z}_{\eps,1}^{t}}_{\mathcal{H}}
        + \norm{\tilde{z}_{\eps,2}^{t}}_{\mathcal{H}}
        < M_{1}
    \]
    for all $(t,\lambda)\in  (-M_{1}^{-1},M_{1}^{-1}) \times \mathcal{L}$.
    Then \eqref{5.1} gives for any such $(\lambda,t)$
    \[
        \tilde{z}_{\eps,i}^{t+h,\lambda}(\xi)
        = \tilde{z}_{\eps,i}^{t,\lambda}(\xi)
        + \int_{t}^{t+h}u_{\eps}(z_{\eps,i}^{\tau};
        \tilde{z}_{\eps,i}^{\tau,\lambda}(\xi))\,d\tau
    \]
    for  $i=1,2$, any $\xi\in\bbT$, and all small enough $h$.
    Then Lemma~\ref{L5.1} shows that
    there is $M_{2}$ depending on ${\alpha,|\theta|,\eps}$ such that
    for any orientation-preserving homeomorphism $\phi\colon\bbT\to\bbT$ and
    all small enough $h>0$,
    \begin{align*}
        &d_{\mathrm{F}}\left(z_{\eps,1}^{t+h,\lambda},
        z_{\eps,2}^{t+h,\lambda}\right)
        \\&\quad\ \leq
        \norm{\tilde{z}_{\eps,1}^{t,\lambda}
        - \tilde{z}_{\eps,2}^{t,\lambda}\circ\phi}_{L^{\infty}}
        + M_{1}M_{2}\int_{t}^{t+h}
        \left(
            \norm{\tilde{z}_{\eps,1}^{\tau,\lambda}
            - \tilde{z}_{\eps,2}^{\tau,\lambda}\circ\phi}_{L^{\infty}}
            + d_{\mathrm{F}}\left(z_{\eps,1}^{\tau},
            z_{\eps,2}^{\tau}\right)
        \right)
        d\tau
        \\&\quad\ \leq
        (1 + hM_{1}M_{2})
        \norm{\tilde{z}_{\eps,1}^{t,\lambda}
        - \tilde{z}_{\eps,2}^{t,\lambda}\circ\phi}_{L^{\infty}}
        + M_{1}M_{2}\int_{t}^{t+h}
        \left( d_{\mathrm{F}}(z_{\eps,1}^{\tau},z_{\eps,2}^{\tau})
        + \sum_{i=1}^{2}\norm{\tilde{z}_{\eps,i}^{\tau,\lambda}
        - \tilde{z}_{\eps,i}^{t,\lambda}}_{L^{\infty}} \right)d\tau.
    \end{align*}
    Taking the infimum over $\phi$, then maximum over $\lambda$, and using $\norm{f}_{L^{\infty}(\bbT)} \leq \norm{f}_{L^{2}(\bbT)} + \norm{f'}_{L^{2}(\bbT)}$, we find that
%     Since, in general for any function $f\colon\bbT\to\bbR^{2}$
%    in $H^{1}(\bbT;\bbR^{2})$, Cauchy-Schwarz inequality shows
%    \[
%        \norm{f}_{L^{\infty}(\bbT)} \leq
%        \abs{\int_{\bbT}f(\xi')\,d\xi'} +
%        \max_{\xi\in\bbT}\abs{f(\xi) - \int_{\bbT}f(\xi')\,d\xi'}
%        \leq \norm{f}_{L^{2}(\bbT)} + \norm{\partial_{\xi}f}_{L^{2}(\bbT)},
%    \]
%    taking the infimum over $\phi$ and then maximum over $\lambda$
%    on the inequality for $d_{\mathrm{F}}(z_{\eps,1}^{t+h,\lambda},
%    z_{\eps,2}^{t+h,\lambda})$ we derived above shows
    \begin{align*}
        &d_{\mathrm{F}}\left(z_{\eps,1}^{t+h},z_{\eps,2}^{t+h}\right)
        \\&\quad\quad \leq (1+hM_{1}M_{2})
        d_{\mathrm{F}}\left(z_{\eps,1}^{t},z_{\eps,2}^{t}\right)
        + M_{1}M_{2}\int_{t}^{t+h}
        \left( d_{\mathrm{F}}(z_{\eps,1}^{\tau},z_{\eps,2}^{\tau})
        + \sqrt{2}\sum_{i=1}^{2}\norm{\tilde{z}_{\eps,i}^{\tau}
        - \tilde{z}_{\eps,i}^{t}}_{\mathcal H} \right) d\tau.
    \end{align*}
    Since the integrand  is continuous in $\tau$, we obtain
    \[
        \partial_{t}^{+}d_{\mathrm{F}}\left(
            z_{\eps,1}^{t},z_{\eps,2}^{t}
        \right)
        \leq 2M_{1}M_{2}
        d_{\mathrm{F}}\left(
            z_{\eps,1}^{t},z_{\eps,2}^{t}
        \right).
    \]
    A similar argument for $h<0$ now yields the desired conclusion.
\end{proof}

From this
%    and continuity of $d_{\mathrm{F}}(\pi(\tilde{z}_{\eps,1}^{t}),   \pi(\tilde{z}_{\eps,2}^{t}))$ 
    we immediately obtain the following.

\begin{corollary}\label{C5.3}
    Let $\tilde{z}_{\eps,i}\in C^1(I_{i}; \mathcal{H})$ ($i=1,2$)
%    $\tilde{z}_{\eps,2}\colon I_{2}\to\mathcal{H}$ 
be the unique solution to
    \eqref{5.1} with initial data
    $\tilde{z}_{i}^{0}\in \mathcal{H}$. 
    If $\pi(\tilde{z}_{1}^{0}) = \pi(\tilde{z}_{2}^{0})$,
    then $\pi(\tilde{z}_{\eps,1}^{t}) = \pi(\tilde{z}_{\eps,2}^{t})$
     for all $t\in I_{1}\cap I_{2}$.
\end{corollary}

%\begin{proof}
%    Follows directly from Lemma~\ref{L5.2}
%    and continuity of $d_{\mathrm{F}}(\pi(\tilde{z}_{\eps,1}^{t}),
%    \pi(\tilde{z}_{\eps,2}^{t}))$.
%\end{proof}

Next we show that no singularities can arise in the dynamic of \eqref{5.1} on $\mathcal H$, which will then imply that its solutions are global.

\begin{lemma}\label{L5.4}
    Let $\tilde{z}_{\eps}\in C^1(I; \mathcal{H})$ be the unique maximal solution
    to \eqref{5.1} with initial data $\tilde{z}_{\eps}^{0}\in \mathcal{H}$, and
    let $z_{\eps}^{t}\coloneqq \pi(\tilde{z}_{\eps}^{t})$ for each $t\in I$.
    \begin{enumerate}
        \item If $\Delta(z_{\eps}^{0,\lambda},z_{\eps}^{0,\lambda'})>0$
        for some $\lambda,\lambda'\in\mathcal{L}$, then
        $\Delta(z_{\eps}^{t,\lambda},z_{\eps}^{t,\lambda'})>0$
        for all $t\in I$.

        \item If $z_{\eps}^{0,\lambda}\in\operatorname{PSC}(\bbR^{2})$
        for some $\lambda\in\mathcal{L}$, then
        $z_{\eps}^{t,\lambda}\in\operatorname{PSC}(\bbR^{2})$ and
        $\abs{\Omega(z_{\eps}^{t,\lambda})} = \abs{\Omega(z_{\eps}^{0,\lambda})}$
        for all $t\in I$. In particular, if
        $z_{\eps}^{0}\in\operatorname{PSC}(\bbR^{2})^\mathcal{L}$, then $W(z_{\eps}^{t}) = W(z_{\eps}^{0})$
        for all $t\in I$.
    \end{enumerate}
\end{lemma}

\begin{proof}
    Consider the ODE
    \[
        \partial_{t}x^{t} = u_{\eps}(z_{\eps}^{t};x^{t}).
    \]
    By Lemma~\ref{L5.1}, $\norm{u_{\eps}(z_{\eps}^{t})}_{\dot{C}^{0,1}}$
    is bounded on any compact interval $J \subseteq I$.
    Therefore there is a  unique 
    $Z_{\eps}\colon I\times\bbR^{2}\to\bbR^{2}$ such that
    $Z_{\eps}(\,\cdot\,,x^{0})$ solves the above ODE
    with initial data  $x^{0}$, and
    $\norm{Z_{\eps}(t,\,\cdot\,)}_{\dot{C}^{0,1}}$ is  bounded
    on any compact interval $J \subseteq I$.
    By uniqueness, $Z_{\eps}(t,\,\cdot\,)$ is 1-1 for each $t\in I$, and 
    $Z_{\eps}(t,\tilde{z}_{\eps}^{0,\lambda}(\xi))
    = \tilde{z}_{\eps}^{t,\lambda}(\xi)$ holds for any $(t,\lambda,\xi)\in I \times \mathcal{L} \times \bbT$.

This now shows (1), while the first claim in (2) follows by also using Lemma~\ref{LA.4} because $\{\tilde{z}_{\eps}^t\,:\, t\in I\}$   is a connected subset of $\mathcal{H}$.  The second claim in (2) follows  from
    Green's theorem since $u_{\eps}(z_{\eps}^{\tau})$
    is divergence-free for all $\tau\in I$.
%    
%    Therefore, for each $t\in I$ and $\lambda,\lambda'\in\mathcal{L}$,
%    if the images of $\tilde{z}_{\eps}^{0,\lambda}$ and
%    $\tilde{z}_{\eps}^{0,\lambda'}$ are disjoint from each other, then so is
%    $\tilde{z}_{\eps}^{t,\lambda}$ and $\tilde{z}_{\eps}^{t,\lambda'}$,
%    and if $\tilde{z}_{\eps}^{0,\lambda}$ is equivalent to an injective path,
%    then so is $\tilde{z}_{\eps}^{t,\lambda} = Z_{\eps}(t,\,\cdot\,)
%    \circ \tilde{z}_{\eps}^{0,\lambda}$.
%    Also, since the trajectory of $\tilde{z}_{\eps}$ over $I$
%    is a connected subset of $\mathcal{H}$, Lemma~\ref{LA.4} shows that
%    if $z_{\eps}^{0,\lambda}\in\operatorname{PSC}(\bbR^{2})$ then
%    so is $z_{\eps}^{t,\lambda}\in\operatorname{PSC}(\bbR^{2})$.
%    Finally, in this case, $\abs{\Omega(z_{\eps}^{0,\lambda})} =
%    \abs{\Omega(z_{\eps}^{t,\lambda})}$ easily follows from
%    Green's theorem since $u_{\eps}(z_{\eps}^{\tau})$
%    is divergence-free for all $\tau\in I$.
\end{proof}

\begin{lemma}\label{L5.5}
    For any $W_{0}\in[0,\infty)$, there is
    $C$ that only depends on ${\alpha,|\theta|,W_{0},\eps}$ such that for any solution
    $\tilde{z}_{\eps}\in C^1(I; \mathcal{H})$  to \eqref{5.1}    with 
    %initial data $\tilde{z}_{\eps}^{0}\in \mathcal{H}$ satisfying
    $z_{\eps}^{0}\coloneqq \pi(\tilde{z}_{\eps}^{0})
    \in\operatorname{PSC}(\bbR^{2})^{\mathcal{L}}$ and $W(z_{\eps}^{0})\leq W_{0}$ we have
    \beq\lb{111.5}
        e^{-C\abs{t}}\abs{\partial_{\xi}\tilde{z}_{\eps}^{0,\lambda}(\xi)}\leq
        \abs{\partial_{\xi}\tilde{z}_{\eps}^{t,\lambda}(\xi)} \leq 
        e^{C\abs{t}}\abs{\partial_{\xi}\tilde{z}_{\eps}^{0,\lambda}(\xi)}
    \eeq
    for all $(t,\lambda,\xi)\in I \times \mathcal{L} \times \bbT$.
\end{lemma}

\begin{proof}
    For each $t\in I$, let $z_{\eps}^{t}\coloneqq\pi(\tilde{z}_{\eps}^{t})$.
    We have
    \[
        \partial_{\xi}\tilde{z}_{\eps}^{t+h,\lambda}(\xi)
        - \partial_{\xi}\tilde{z}_{\eps}^{t,\lambda}(\xi)
        = \int_{t}^{t+h}D(u_{\eps}(z_{\eps}^{\tau}))
        (\tilde{z}_{\eps}^{\tau,\lambda}(\xi))
        \left(\partial_{\xi}\tilde{z}_{\eps}^{\tau,\lambda}(\xi)\right)\,d\tau
    \]
    when $t,t+h\in I$ and $(\lambda,\xi)\in\mathcal{L}\times\bbT$. The integrand is a continuous function of $\tau$, so
    taking $\abs{\,\cdot\,}$, dividing by $h$ and then letting $h\to 0^{+}$ shows that
    \[
        \max \left\{ \partial_{t}^{+}\abs{\partial_{\xi}\tilde{z}_{\eps}^{t,\lambda}(\xi)},
        -\partial_{t+}\abs{\partial_{\xi}\tilde{z}_{\eps}^{t,\lambda}(\xi)} \right\}
        \leq \norm{D(u_{\eps}(z_{\eps}^{t}))}_{L^{\infty}}
        \abs{\partial_{\xi}\tilde{z}_{\eps}^{t,\lambda}(\xi)}
%        ,\quad
%        \partial_{t+}\abs{\partial_{\xi}\tilde{z}_{\eps}^{t,\lambda}(\xi)}
%        \geq -\norm{D(u_{\eps}(z_{\eps}^{t}))}_{L^{\infty}}
%        \abs{\partial_{\xi}\tilde{z}_{\eps}^{t,\lambda}(\xi)}
    \]
    for all $(t,\lambda,\xi)\in I\times \mathcal{L} \times \bbT$.
    By Lemma~\ref{L5.4} we have
    \[
        \norm{D(u_{\eps}(z_{\eps}^{t}))}_{L^{\infty}}
        \leq \sum_{\lambda\in\mathcal{L}}\abs{\theta^{\lambda}}
        \int_{\Omega(z_{\eps}^{t,\lambda})}
        \norm{D\nabla^{\perp}K_{\eps}}_{L^{\infty}}dy
        \leq \abs{\theta}\norm{D^{2}K_{\eps}}_{L^{\infty}}W_{0},
    \]
    so a Gr\"{o}nwall-type argument shows that \eqref{111.5}
%    \[
%        e^{-Ct}\abs{\partial_{\xi}\tilde{z}_{\eps}^{0,\lambda}(\xi)}\leq
%        \abs{\partial_{\xi}\tilde{z}_{\eps}^{t,\lambda}(\xi)} \leq 
%        e^{Ct}\abs{\partial_{\xi}\tilde{z}_{\eps}^{0,\lambda}(\xi)}
%    \]
    holds for all $(t,\lambda,\xi)\in (I\cap[0,\infty))\times \mathcal{L}\times \bbT$,
    with $C\coloneqq \abs{\theta}\norm{D^{2}K_{\eps}}_{L^{\infty}}W_{0}$.
    A similar argument with the same $C$ applies to $t\in I\cap(-\infty,0]$.
\end{proof}

\begin{corollary}\label{C5.6}
    For any  $\tilde{z}_{\eps}^{0}\in \mathcal{H}$ with
    $\pi(\tilde{z}_{\eps}^{0})\in\operatorname{PSC}(\bbR^{2})^{\mathcal{L}}$,
    there is a unique global solution $\tilde{z}_{\eps}\in C^1(\bbR; \mathcal{H})$ to \eqref{5.1} with initial data $\tilde{z}_{\eps}^{0}$.
\end{corollary}

\begin{proof}
    Let $\tilde{z}_{\eps}\colon I\to \mathcal{H}$ be the unique maximal solution to \eqref{5.1}
    with initial data $\tilde{z}_{\eps}^{0}$, and let
    $z_{\eps}^{t}\coloneqq\pi(\tilde{z}_{\eps}^{t})$ for each $t\in I$.
    Then for each $t,t+h\in I$ and $\lambda\in\mathcal{L}$ we have
    \begin{align*}
        \partial_{\xi}^{2}&\tilde{z}_{\eps}^{t+h,\lambda}(\xi)
         - \partial_{\xi}^{2}\tilde{z}_{\eps}^{t,\lambda}(\xi)
        = \partial_{\xi}\left(
            \int_{t}^{t+h}D(u_{\eps}(z_{\eps}^{\tau}))
            (\tilde{z}_{\eps}^{\tau,\lambda}(\xi))
            \left(\partial_{\xi}\tilde{z}_{\eps}^{\tau,\lambda}(\xi)\right)\,d\tau
        \right) \\
        &= \int_{t}^{t+h} \left[ D^{2}(u_{\eps}(z_{\eps}^{\tau}))
        (\tilde{z}_{\eps}^{\tau,\lambda}(\xi))
        \left(\partial_{\xi}\tilde{z}_{\eps}^{\tau,\lambda}(\xi),
        \partial_{\xi}\tilde{z}_{\eps}^{\tau,\lambda}(\xi)\right)
 %       \\&\quad\quad\quad\quad
        + D(u_{\eps}(z_{\eps}^{\tau}))
        (\tilde{z}_{\eps}^{\tau,\lambda}(\xi))
        \left(\partial_{\xi}^{2}\tilde{z}_{\eps}^{\tau,\lambda}(\xi)\right) \right] d\tau
    \end{align*}
    for almost every $\xi\in\bbT$.  The integrand, as an $L^{2}(\bbT)$-valued function of $\xi$, is  continuous in $\tau$.
    So taking the ${L^{2}}$-norm (all norms are in $\xi$),
    dividing by $h$, and letting $h\to 0^{+}$ shows that
    \[
        \partial_{t}^{+}\norm{\partial_{\xi}^{2}\tilde{z}_{\eps}^{t,\lambda}}_{L^{2}}
        \leq \norm{D^{2}\left(u_{\eps}(z_{\eps}^{t})\right)}_{L^{\infty}}
        \norm{\partial_{\xi}\tilde{z}_{\eps}^{t,\lambda}}_{L^{\infty}}^{2}
        + \norm{D\left(u_{\eps}(z_{\eps}^{t})\right)}_{L^{\infty}}
        \norm{\partial_{\xi}^{2}\tilde{z}_{\eps}^{t,\lambda}}_{L^{2}}
    \]
    for all $(t,\lambda)\in I\times \mathcal{L}$.
    As in the proof of Lemma~\ref{L5.5}, Lemma~\ref{L5.4} shows that
    \[
        \norm{D^{n}\left(u_{\eps}(z_{\eps}^{t})\right)}_{L^{\infty}}
        \leq \sum_{\lambda\in\mathcal{L}}\abs{\theta^{\lambda}}
        \int_{\Omega(z_{\eps}^{t,\lambda})}
        \norm{D^{n}\nabla^{\perp}K_{\eps}}_{L^{\infty}}dy
        = \abs{\theta}\norm{D^{n+1}K_{\eps}}_{L^{\infty}} W(z_{\eps}^{0})
    \]
   for all $t\in I$ and $n=1,2$,
   % and $W_{0}\coloneqq W(z_{\eps}^{0})$, 
    so Lemma~\ref{L5.5} and a Gr\"{o}nwall-type argument show that
    $\norm{\partial_{\xi}^{2}\tilde{z}_{\eps}^{t,\lambda}}_{L^{2}}$
    is finite for all $(t,\lambda)\in(I\cap [0,\infty)) \times \mathcal{L}$.  Maximality of $I$
    now yields $[0,\infty) \subseteq I$, and  a similar argument applied to $t\leq 0$ also shows that 
    $(-\infty,0]\subseteq I$.
\end{proof}

\begin{proposition}\label{P7.1}
    For any $W_0,M\in[0,\infty)$, there is $C$ that only depends on ${\alpha,|\theta|,W_{0},M}$
    such that  whenever $\tilde{z}_{\eps}\in C^1(I; \mathcal{H})$ solves \eqref{5.1} on some interval $I$ and with $z_{\eps}^{t}\coloneqq\pi(\tilde{z}_{\eps}^{t})$ we have
    $z_{\eps}^{t_0}    \in\operatorname{PSC}(\bbR^{2})^{\mathcal{L}}$ and $W(z_{\eps}^{t_0})\leq W_{0}$ for some $t_0\in I$, as well as 
    \begin{align*}
        \sup_{t\in I}\max_{\lambda\in\mathcal{L}}\max\set{
            \ell(z_{\eps}^{t,\lambda}),
            \norm{z_{\eps}^{t,\lambda}}_{\dot{H}^{2}}^{2}
        } \leq M,
    \end{align*}
    then the following holds.
    For each $\lambda\in\mathcal{L}$, any $t,t+h\in I$, and any  arclength parametrizations of
    $z_{\eps}^{t,\lambda}$ and $z_{\eps}^{t+h,\lambda}$, there is a $C^{1}$
    %orientation-preserving 
    homeomorphism
    $\phi\colon\ell(z_{\eps}^{t,\lambda})\bbT\to\ell(z_{\eps}^{t+h,\lambda})\bbT$
    such that  $e^{-C\abs{h}} \leq \phi'(s) \leq e^{C\abs{h}}$ for all $s\in \ell(z_{\eps}^{t,\lambda})\bbT$ and 
    \begin{equation} \lb{111.11}
 %       \begin{gathered}
            \norm{z_{\eps}^{t+h,\lambda}\circ\phi
            - z_{\eps}^{t,\lambda}
            - hu_{\eps}(z_{\eps}^{t})\circ z_{\eps}^{t,\lambda}}_{L^{\infty}(\ell(z_\eps^{t,\lambda})\bbT)}
            \leq C\abs{h}^{2-2\alpha}.
%            \quad\textrm{and} \\
 %           e^{-C\abs{h}} \leq \phi' \leq e^{C\abs{h}}.
 %       \end{gathered}
    \end{equation}
    In particular, for any $t,t+h\in I$ we have
    \[
        d_{\mathrm{F}}(z_{\eps}^{t+h},
        X_{u_{\eps}(z_{\eps}^{t})}^{h}[z_{\eps}^{t}])
        \leq C\abs{h}^{2-2\alpha}.
    \]
%    holds for any such $I$, $t$ and $h$.
\end{proposition}

\textit{Remark}. Since Lemmas~\ref{L5.4} and \ref{L5.5} show that $W(z_{\eps}^{t})$ is constant in time and for any bounded $I'\subseteq I$ we have
\[
    \sup_{t\in I'}\max_{\lambda\in\mathcal{L}}\norm{z_{\eps}^{t,\lambda}}_{\dot{H}^{2}}^{2}
    \leq \sup_{t\in I}\frac{\norm{\tilde{z}_{\eps}^{t,\lambda}}_{\dot{H}^{2}(\bbT)}^{2}}
    {\min_{\xi\in\bbT}\abs{\partial_{\xi}\tilde{z}_{\eps}^{t,\lambda}(\xi)}^{3}} <\infty,
\]
 Proposition~\ref{P7.1} proves that \eqref{2.4} holds
with  $(u_{\eps},z_{\eps})$ in place of $(u,z)$ when $z_{\eps}^{t}\coloneqq\pi(\tilde{z}_{\eps}^{t})$ and $\tilde z_\eps$ is as in Corollary \ref{C5.6}.  Hence $z_\eps$ is  an $H^2$ patch solution to the \emph{mollified g-SQG equation}.

\begin{proof}
    Assume without loss that $t=0$ (we will then use $t$ below for the time argument), and consider only $h>0$ because the other case is analogous.
    Let $\tilde{w}_{\eps}\colon\bbR\to \mathcal H$
    be a solution to \eqref{5.1} such that
    $\tilde{w}_{\eps}^{0,\lambda}$ is a constant-speed parametrization of
    $z_{\eps}^{0,\lambda}$ for each $\lambda\in\mathcal L$. From Corollary~\ref{C5.3}
    we know that $\pi(\tilde{w}_{\eps}^{t}) = z_{\eps}^{t}$
    for all $t\in I$, so
    for any $t_{1},t_{2}\in I$ with $t_{1}\leq t_{2}$ and for any $\xi\in\bbT$ we have
    \begin{equation}\label{7.1}
        \tilde{w}_{\eps}^{t_{2},\lambda}(\xi) -
        \tilde{w}_{\eps}^{t_{1},\lambda}(\xi)
        = \int_{t_{1}}^{t_{2}}u_{\eps}(z_{\eps}^{\tau};
        \tilde{w}_{\eps}^{\tau,\lambda}(\xi))\,d\tau
    \end{equation}
    and
    \begin{equation}\label{7.2}
        \partial_{\xi}\tilde{w}_{\eps}^{t_{2},\lambda}(\xi) -
        \partial_{\xi}\tilde{w}_{\eps}^{t_{1},\lambda}(\xi)
        = \int_{t_{1}}^{t_{2}}\partial_{\xi}[u_{\eps}(z_{\eps}^{\tau};
        \tilde{w}_{\eps}^{\tau,\lambda}(\xi))]\,d\tau
    \end{equation}
   for any $\lambda\in\mathcal L$ (which we now fix). From \eqref{7.2} and Lemma~\ref{L4.6} (with $\beta=\frac 12$) we see that for all $(t,\xi)\in I\times\bbT$ we have
    \begin{align*}
        \max\set{\partial_{t}^{+}\abs{\partial_{\xi}\tilde{w}_{\eps}^{t,\lambda}(\xi)},
        -\partial_{t+}\abs{\partial_{\xi}\tilde{w}_{\eps}^{t,\lambda}(\xi)}}
        &\leq \abs{\partial_{\xi}[u_{\eps}
        (z_{\eps}^{t};\tilde{w}_{\eps}^{t,\lambda}(\xi))]} \\
        &\leq \norm{\partial_{s}(u_{\eps}(z_{\eps}^{t})\circ z_{\eps}^{t,\lambda})}_{L^{\infty}}
        \abs{\partial_{\xi}\tilde{w}_{\eps}^{t,\lambda}(\xi)} \\
                &\leq C_{\alpha}\abs{\theta}M^{2+2\alpha}
        \abs{\partial_{\xi}\tilde{w}_{\eps}^{t,\lambda}(\xi)}.
    \end{align*}
%     so Lemma~\ref{L4.6} (with $\beta=\frac 12$) yields
%    \begin{align*}
%        \max\set{\partial_{t}^{+}\abs{\partial_{\xi}\tilde{w}_{\eps}^{t,\lambda}(\xi)},
%        -\partial_{t+}\abs{\partial_{\xi}\tilde{w}_{\eps}^{t,\lambda}(\xi)}}
%        &\leq C_{\alpha}\abs{\theta}M^{2+2\alpha}
%        \abs{\partial_{\xi}\tilde{w}_{\eps}^{t,\lambda}(\xi)}.
%    \end{align*}
    Therefore, a Gr\"{o}nwall-type argument shows that
    \begin{equation}\label{7.3}
        e^{-Ch}\ell(z_{\eps}^{0,\lambda}) =e^{-Ch}\abs{\partial_{\xi}\tilde{w}_{\eps}^{0,\lambda}(\xi)}
        \leq \abs{\partial_{\xi}\tilde{w}_{\eps}^{h,\lambda}(\xi)}
        \leq e^{Ch}\abs{\partial_{\xi}\tilde{w}_{\eps}^{0,\lambda}(\xi)} = e^{Ch} \ell(z_{\eps}^{0,\lambda}) 
    \end{equation}
    for all $\xi\in\bbT$, with $C\coloneqq C_{\alpha}\abs{\theta}M^{2+2\alpha}$.
    In particular, this also yields
    % integrating the above inequality over $\bbT$ shows
    \[
        e^{-Ch}\ell(z_{\eps}^{0,\lambda})
        \leq \ell(z_{\eps}^{h,\lambda})
        \leq e^{Ch}\ell(z_{\eps}^{0,\lambda}).
    \]

    Let $\eta\colon\ell(z_{\eps}^{0,\lambda})\bbT\to\bbT$ be an orientation-preserving
    homeomorphism such that $z_{\eps}^{0,\lambda}
    = \tilde{w}_{\eps}^{0,\lambda}\circ\eta$. Since $\tilde{w}_{\eps}^{0,\lambda}$
    is a constant-speed parametrization, we see that $\eta'\equiv \ell(z_{\eps}^{0,\lambda})^{-1}$.
    Next take  an orientation-preserving homeomorphism $\psi\colon\bbT\to\ell(z_{\eps}^{h,\lambda})\bbT$ such that
    $z_{\eps}^{h,\lambda} \circ \psi = \tilde{w}_{\eps}^{h,\lambda}$ (it exists  by \eqref{7.3}),
%    Such $\psi$ can be found by first reparametrizing $\tilde{w}_{\eps}^{t+h,\lambda}$
%    by a constant-speed parametrization, and then apply an appropriate   scaled translation. 
and let $\phi\coloneqq \psi\circ\eta$.  Then for any
    $s\in\ell(z_{\eps}^{t,\lambda})\bbT$ we have
    \begin{align*}
        \phi'(s) = \psi'(\eta(s))\eta'(s)
        = \frac{\abs{\partial_{\xi}\tilde{w}_{\eps}^{h,\lambda}(\eta(s))}}
        {\ell(z_{\eps}^{0,\lambda})},
    \end{align*}
    so \eqref{7.3} shows that also $e^{-Ch}\leq \phi'(s) \leq e^{Ch}$.

    On the other hand, \eqref{7.1} and Lemma~\ref{L4.4} (with $\beta=\frac 12$) show that
    \begin{align*}
        \norm{z_{\eps}^{h,\lambda}\circ\phi
        - z_{\eps}^{0,\lambda}
        - hu_{\eps}(z_{\eps}^{0})\circ z_{\eps}^{0,\lambda}}_{L^{\infty}}
        &= \norm{\tilde{w}_{\eps}^{h,\lambda}
        - \tilde{w}_{\eps}^{0,\lambda}
        - hu_{\eps}(z_{\eps}^{0})\circ \tilde{w}_{\eps}^{0,\lambda}}_{L^{\infty}} \\
        &\leq \int_{0}^{h}
        \norm{u_{\eps}(z_{\eps}^{\tau})\circ \tilde{w}_{\eps}^{\tau,\lambda}
        - u_{\eps}(z_{\eps}^{0})\circ \tilde{w}_{\eps}^{0,\lambda}}_{L^{\infty}}
        d\tau \\
        &\leq C_{\alpha}\abs{\theta}M^{2}\int_{0}^{h}
        \max_{\lambda'\in\mathcal L} \norm{\tilde{w}_{\eps}^{\tau,\lambda'}
        - \tilde{w}_{\eps}^{0,\lambda'}}_{L^{\infty}}^{1-2\alpha}
        d\tau.
    \end{align*}
     Since   \eqref{7.1} and Lemmas~\ref{L4.1} and \ref{L5.4} show that for any $t_{1},t_{2}\in I$ with $t_{1}\leq t_{2}$ and any $\lambda'\in\mathcal L$ we have
    \begin{align*}
        \norm{\tilde{w}_{\eps}^{t_{2},\lambda'}
        - \tilde{w}_{\eps}^{t_{1},\lambda'}}_{L^{\infty}}
        \leq \int_{t_{1}}^{t_{2}}\norm{u_{\eps}(z_{\eps}^{\tau})}_{L^{\infty}}\,d\tau
        \leq C(t_{2}-t_{1})
    \end{align*}
   with $C$ depending only on ${\alpha,|\theta|,W_{0}}$, 
     we now obtain \eqref{111.11} with $t=0$.
%    \[
%        \norm{z_{\eps}^{h,\lambda}\circ\phi
%        - z_{\eps}^{0,\lambda}
%        - hu_{\eps}(z_{\eps}^{0})\circ z_{\eps}^{0,\lambda}}_{L^{\infty}}
%        \leq C_{\alpha,|\theta|,W_{0},M}h^{2-2\alpha}
%    \]
%    as desired.
\end{proof}

%%%%%%%%%%%%%%%%%%%%%%%%%%%%%%%%%%%%%%%%%%%%%%
\section{Uniform-in-$\eps$ estimates for  the mollified g-SQG equations}\label{S6}
%%%%%%%%%%%%%%%%%%%%%%%%%%%%%%%%%%%%%%%%%%%%%%

In this section, any constant $C_\alpha$ will only depend on $\alpha$.
Consider any initial data $z^{0}\in \operatorname{PSC}(\bbR^{2})^{\mathcal{L}}$
with $L(z^{0}) < \infty$, and let $W_{0}\coloneqq W(z^{0})$.
For any small  $\eps>0$, we can use
Lemma~\ref{L3.7} to construct perturbed initial data
$z_{\eps}^{0}\in\operatorname{PSC}(\bbR^{2})^{\mathcal{L}}$ such that
\beq\lb{111.12}
W(z_{\eps}^{0}) \leq W_{0}, \qquad
d_{\mathrm{F}}(z^{0},z_{\eps}^{0}) \leq \eps, \qquad
L(z_{\eps}^{0})\leq 56L(z^{0})
\eeq
% for some universal constant $C_{0}$, 
and
$\Delta(z_{\eps}^{0,\lambda},z_{\eps}^{0,\lambda'}) > 0$ for all
$\lambda\neq \lambda'$ from $\mathcal{L}$.
Indeed, if we rename the $\lambda$'s so that $\mathcal{L}=\{1,\dots,N\}$  and $\Omega(z^{0,i})\subseteq \Omega(z^{0,j})$
implies $i\leq j$, then we can apply Lemma~\ref{L3.7} with $c\coloneqq\frac 1{16}$ to
$z^{0,1}, \dots ,z^{0,N}$ (in that order and with $\eps,\frac \eps 5,\dots, \frac \eps {5^{N-1}}$ in place of $\eps$)
to obtain  $z_{\eps}^{0,1},\dots,
z_{\eps}^{0,N}\in\operatorname{PSC}(\bbR^{2})$ such that
for each $n=1, \dots ,N$ we have
\begin{itemize}
    \item $z_{\eps}^{0,n}$ lies inside
    $\Omega(z^{0,n})$,

    \item $d_{\mathrm{F}}(z^{0,n},z_{\eps}^{0,n}) \leq \eps$,

    \item $\norm{z_{\eps}^{0,n}}_{\dot{H}^{2}} \leq
    4\norm{z^{0,n}}_{\dot{H}^{2}}$,

    \item $\Delta_{1/16Q(z^{0})}(z_{\eps}^{0,n}) \geq
    \frac{1}{56}\Delta_{1/Q(z^{0})}(z^{0,n})$,

    \item $\Delta(z_{\eps}^{0,n},z_{\eps}^{0,i}) > \Delta(z^{0,n},z^{0,i})$
    holds for $i=1, \dots ,n-1$.
\end{itemize}
Since then clearly $Q(z_{\eps}^{0}) \leq 16Q(z^{0})$, we also have $L(z_{\eps}^{0}) \leq 56L(z^{0})$ for all $\eps>0$.

%To have the last condition, it is enough to ensure
%$d_{\mathrm{F}}(z^{0,\lambda_{k}},z_{\eps}^{0,\lambda_{k}})$ to be small enough.
%Indeed, note that for each $i=1,\dots,k-1$,
%the construction of the enumeration of $\mathcal{L}$ guarantees that
%$\operatorname{im}(z^{0,\lambda_{k}})$ is disjoint from $\Omega(z^{0,\lambda_{i}})$,
%which shows
%\[
%    \Delta(z^{0,\lambda_{k}},z_{\eps}^{0,\lambda_{i}})
%    \geq \Delta(z^{0,\lambda_{k}},z^{0,\lambda_{i}})
%    + \Delta(z^{0,\lambda_{i}},z_{\eps}^{0,\lambda_{i}}).
%\]
%Therefore, if $x\in\operatorname{im}(z_{\eps}^{0,\lambda_{k}})$ and
%$y\in\operatorname{im}(z_{\eps}^{0,\lambda_{i}})$ are such that
%$\Delta(z_{\eps}^{0,\lambda_{k}},z_{\eps}^{0,\lambda_{i}}) = \abs{x - y}$, then
%there exists $x'\in\operatorname{im}(z^{0,\lambda_{k}})$ such that
%\begin{align*}
%    \Delta(z_{\eps}^{0,\lambda_{k}},z_{\eps}^{0,\lambda_{i}})
%    &\geq \abs{x' - y} - \abs{x - x'}
%    \geq \Delta(z^{0,\lambda_{k}},z_{\eps}^{0,\lambda_{i}})
%    - d_{\mathrm{F}}(z^{0,\lambda_{k}},z_{\eps}^{0,\lambda_{k}}) \\
%    &\geq \Delta(z^{0,\lambda_{k}},z^{0,\lambda_{i}})
%    + (\Delta(z^{0,\lambda_{i}},z_{\eps}^{0,\lambda_{i}})
%    - d_{\mathrm{F}}(z^{0,\lambda_{k}},z_{\eps}^{0,\lambda_{k}})).
%\end{align*}
%Therefore, it is sufficient to inductively ensure that
%$d_{\mathrm{F}}(z^{0,\lambda_{k}},z_{\eps}^{0,\lambda_{k}})$ is strictly smaller than
%$\Delta(z^{0,\lambda_{i}},z_{\eps}^{0,\lambda_{i}})$ for each $i=1,\dots,k-1$,
%which are all strictly positive by the construction of $z_{\eps}^{0,\lambda_{i}}$'s.

Fix $\eps>0$ and pick $\tilde{z}_{\eps}^{0}\in H^{2}(\bbT;\bbR^{2})^{\mathcal{L}}$ so that $\tilde{z}_{\eps}^{0,\lambda}$ is a constant-speed parametrization of $z_{\eps}^{0,\lambda}$ for each $\lambda\in\mathcal L$.
Let $\tilde{z}_{\eps}\colon \bbR\to H^{2}(\bbT;\bbR^{2})^{\mathcal{L}}$ be from Corollary \ref{C5.6} and
$z_{\eps}^{t}\coloneqq \pi(\tilde{z}_{\eps}^{t})$ for each $t\in\bbR$.
Then Lemma~\ref{L5.4} shows that $z_{\eps}^{t}\in\operatorname{PSC}(\bbR^{2})^{\mathcal{L}}$
and $\min_{\lambda,\lambda'\in\mathcal{L}}\Delta(z_{\eps}^{t,\lambda},z_{\eps}^{t,\lambda'})>0$
 for each $t\in\bbR$.
Since $\set{z_{\eps}^{t}}_{t\in \bbR}$ is also a connected subset of
$\operatorname{PSC}(\bbR^{2})^{\mathcal{L}}$, Lemma~\ref{LA.5} shows that
\[
    \set{(\lambda,\lambda')\in\mathcal{L}^2\colon
    \lambda'\in\Sigma^{\lambda}(z_{\eps}^{t})}
\]
is independent of $t\in \bbR$.

We now derive a bound on $\partial_{t}^{+}L(z_{\eps}^{t})$ that is independent of $\eps>0$.
Let us first estimate the time derivative of
$Q(z_{\eps}^{t}) = \frac{1}{m(\theta)}\sum_{\lambda\in\mathcal{L}}\abs{\theta^{\lambda}}
\norm{z_{\eps}^{t,\lambda}}_{\dot{H}^{2}}^{2}$ (with $m(\theta)=
\min_{\lambda\in\mathcal{L}}\abs{\theta^{\lambda}}$).  To simplify notation, for any $\lambda,\lambda'\in\mathcal{L}$ and $s\in\ell(z_{\eps}^{0,\lambda})\bbT$ we denote
\begin{itemize}
    \item $\ell^{\lambda}
    \coloneqq \ell(z_{\eps}^{0,\lambda})$,

    \item $\mathbf{T}^{\lambda}(s)
    \coloneqq \partial_{s}z_{\eps}^{0,\lambda}(s)$,

    \item $\mathbf{N}^{\lambda}(s)
    \coloneqq \mathbf{T}^{\lambda}(s)^{\perp}$,

    \item $\kappa^{\lambda}(s)
    \coloneqq \partial_{s}^{2}z_{\eps}^{0,\lambda}(s)
    \cdot \mathbf{N}^{\lambda}(s)$,

    \item $\Delta^\lambda \coloneqq
    \Delta_{1/Q(z_{\eps}^{0})}(z_{\eps}^{0,\lambda})$,

    \item $\Delta^{\lambda,\lambda'}
    \coloneqq \Delta(z_{\eps}^{0,\lambda}, z_{\eps}^{0,\lambda'})$,
    
    \item $u_{\eps}^{\lambda}(s)
    \coloneqq u_{\eps}(z_{\eps}^{0}\,;z_{\eps}^{0,\lambda}(s))$
\end{itemize}
(here $z_{\eps}^{0,\lambda}(\cdot)$ stands for any  arclength parametrization of $z_{\eps}^{0,\lambda}$, which we then fix).
Note that 
\[
    \partial_{s}^{2}z_{\eps}^{0,\lambda}(s) = \partial_{s}\mathbf{T}^{\lambda}(s)
    = \kappa^{\lambda}(s)\mathbf{N}^{\lambda}(s)
    \qquad\textrm{and}\qquad
    \partial_{s}\mathbf{N}^{\lambda}(s)
    = -\kappa^{\lambda}(s)\mathbf{T}^{\lambda}(s).
\]

\begin{lemma}\label{L6.1}
    For any $\lambda\in\mathcal{L}$ we have
    \begin{equation}\label{6.1}
        \left.\partial_{t}\norm{z_{\eps}^{t,\lambda}}_{\dot{H}^{2}}^{2}
        \right|_{t=0} =
                2\int_{\ell^{\lambda}\bbT}
        \kappa^{\lambda}(s)
        \left(\partial_{s}^{2}u_{\eps}^{\lambda}(s)\cdot
        \mathbf{N}^{\lambda}(s)\right)ds
        -3\int_{\ell^{\lambda}\bbT}
        \kappa^{\lambda}(s)^{2}
        \left(\partial_{s}u_{\eps}^{\lambda}(s)\cdot
        \mathbf{T}^{\lambda}(s)\right)ds,
%        \\&\quad\quad\quad\quad
    \end{equation}
%    Since $\max_{\lambda\in\mathcal L} \norm{z_{\eps}^{t,\lambda}}_{\dot{H}^{2}} <\infty$ for each $t\in\bbR$, 
%    \eqref{6.1} 
    and this also holds for each $t_0\in\bbR$ in place of 0 when
     $z_{\eps}^{0}$ is replaced by $z_{\eps}^{t_0}$ in the  definitions
      of     $u_{\eps}^{\lambda}$, $\ell^{\lambda}$, $\mathbf{T}^{\lambda}$,    $\mathbf{N}^{\lambda}$, and $\kappa^{\lambda}$.
 %   In particular, $\norm{z_{\eps}^{t,\lambda}}_{\dot{H}^{2}}^{2}$ is differentiable in $t$.
\end{lemma}

\begin{proof}
    Fix any $\lambda\in\mathcal{L}$.  Since for any curve $\gamma\in\operatorname{RC}(\bbR^{2})$ we defined
    $\norm{\gamma}_{\dot{H}^{2}}$ to be the $\dot{H}^{2}$-norm of
    any arclength parametrization of $\gamma$, we also have
        \[
        \norm{\gamma}_{\dot{H}^{2}} = \left(\int_{\bbT}
        \frac{\abs{
            \partial_{\xi}\tilde{\gamma}(\xi)
            \cdot\partial_{\xi}^{2}\tilde{\gamma}(\xi)^{\perp}
        }^{2}}
        {\abs{\partial_{\xi}\tilde{\gamma}(\xi)}^{5}}
        d\xi\right)^{1/2}
    \]
    whenever
    $\tilde{\gamma}\in H^2(\bbT;\bbR^{2})$ is  any parametrization of $\gamma$ with non-vanishing $\partial_{\xi}\tilde{\gamma}$.
 %   Using this, we first derive the formula \eqref{6.1} for $t=0$. 
    Lemma~\ref{L5.5} shows that the above formula then holds for $\tilde{z}_{\eps}^{t,\lambda}$
 %   $\abs{\partial_{\xi}\tilde{z}_{\eps}^{t,\lambda}(\xi)}$ does not vanish 
    for each $t\in\bbR$.
 Hence, the $L^{2}$-convergence
    $\frac{1}{t}\left(\partial_{\xi}^{2}\tilde{z}_{\eps}^{t,\lambda}
    -\partial_{\xi}^{2}\tilde{z}_{\eps}^{0,\lambda}\right)
    \to \partial_{\xi}^{2}F_{\eps}(\tilde{z}_{\eps}^{0})^{\lambda}$ 
    and the uniform convergence
    $\frac{1}{t}\left(\partial_{\xi}\tilde{z}_{\eps}^{t,\lambda}
    -\partial_{\xi}\tilde{z}_{\eps}^{0,\lambda}\right)
    \to \partial_{\xi}F_{\eps}(\tilde{z}_{\eps}^{0})^{\lambda}$ as $t\to 0$ (see \eqref{5.1}) indeed show that
    \begin{align*}
        \left.\partial_{t}\norm{z_{\eps}^{t,\lambda}}_{\dot{H}^{2}}^{2}\right|_{t=0}
        &=
        \partial_{t} \left. \left(\int_{\bbT}
        \frac{\abs{
            \partial_{\xi}\tilde{z}_{\eps}^{t,\lambda}(\xi)
            \cdot\partial_{\xi}^{2}\tilde{z}_{\eps}^{t,\lambda}(\xi)^{\perp}
        }^{2}}
        {\abs{\partial_{\xi}\tilde{z}_{\eps}^{t,\lambda}(\xi)}^{5}}
        d\xi\right) \right|_{t=0} \\
        &=
        \int_{\bbT}
        \frac{2\left(
            \partial_{\xi}F_{\eps}(\tilde{z}_{\eps}^{0})^{\lambda}(\xi)
            \cdot\partial_{\xi}^{2}\tilde{z}_{\eps}^{0,\lambda}(\xi)^{\perp}
        \right)
        \left(
            \partial_{\xi}\tilde{z}_{\eps}^{0,\lambda}(\xi)
            \cdot\partial_{\xi}^{2}\tilde{z}_{\eps}^{0,\lambda}(\xi)^{\perp}
        \right)}
        {\abs{\partial_{\xi}\tilde{z}_{\eps}^{0,\lambda}(\xi)}^{5}}
        d\xi \\
        &\quad\quad\quad\quad
        + \int_{\bbT}
        \frac{2\left(
            \partial_{\xi}\tilde{z}_{\eps}^{0,\lambda}(\xi)
            \cdot\partial_{\xi}^{2}F(\tilde{z}_{\eps}^{0})^{\lambda}(\xi)^{\perp}
        \right)
        \left(
            \partial_{\xi}\tilde{z}_{\eps}^{0,\lambda}(\xi)
            \cdot\partial_{\xi}^{2}\tilde{z}_{\eps}^{0,\lambda}(\xi)^{\perp}
        \right)}
        {\abs{\partial_{\xi}\tilde{z}_{\eps}^{0,\lambda}(\xi)}^{5}}
        d\xi \\
        &\quad\quad\quad\quad
        - \int_{\bbT}
        \frac{5\abs{\partial_{\xi}\tilde{z}_{\eps}^{0,\lambda}(\xi)\cdot
        \partial_{\xi}^{2}\tilde{z}_{\eps}^{0,\lambda}(\xi)^{\perp}}^{2}
        \left(
            \partial_{\xi}\tilde{z}_{\eps}^{0,\lambda}(\xi)
            \cdot\partial_{\xi}F_{\eps}(\tilde{z}_{\eps}^{0})^{\lambda}(\xi)
        \right)}
        {\abs{\partial_{\xi}\tilde{z}_{\eps}^{0,\lambda}(\xi)}^{7}}
        d\xi \\
        &=
        -3\int_{\ell^{\lambda}\bbT}
        \kappa^{\lambda}(s)^{2}
        \left(\partial_{s}u_{\eps}^{\lambda}(s)\cdot
        \mathbf{T}^{\lambda}(s)\right)ds
        + 2\int_{\ell^{\lambda}\bbT}
        \kappa^{\lambda}(s)
        \left(\partial_{s}^{2}u_{\eps}^{\lambda}(s)\cdot
        \mathbf{N}^{\lambda}(s)\right)ds,
    \end{align*}
    where the last equality follows by reparametrizing all three integrals by
     arclength and then adding the first and the third integrals.

    For general $t_0\in\bbR$ in place of 0, note that Lemma~\ref{L5.5} shows
    \[
        \norm{z_{\eps}^{t_0,\lambda}}_{\dot{H}^{2}}^{2}
        \leq \frac{\norm{\tilde{z}_{\eps}^{t_0,\lambda}}_{\dot{H}^{2}(\bbT)}^{2}}
        {\min_{\xi\in\bbT}\abs{\partial_{\xi}\tilde{z}_{\eps}^{t_0,\lambda}(\xi)}^{3}} <\infty
%        \leq e^{tC_{\alpha,\theta,\eps,z_{\eps}^{0}}}
%        \norm{\tilde{z}_{\eps}^{t,\lambda}}_{\dot{H}^{2}(\bbT)}^{2},
    \]
%    so $\norm{z_{\eps}^{t,\lambda}}_{\dot{H}^{2}}^{2} < \infty$ holds for all $t$.
%    (Note that in the above, $\norm{\tilde{z}_{\eps}^{t,\lambda}}_{\dot{H}^{2}(\bbT)}$
%    is the $\dot{H}^{2}$-norm of the parametrization $\tilde{z}_{\eps}^{t,\lambda}$,
%    not the $\dot{H}^{2}$-norm of the curve $z_{\eps}^{t,\lambda}$.)
%    Then for a fixed $t_{0}$, 
   for each $\lambda\in\mathcal L$.  If we now solve \eqref{5.1} with initial data at $t=t_{0}$ being any constant-speed parametrizations of the $z_{\eps}^{t_{0},\lambda}$ ($\lambda\in\mathcal L$), which then all belong to $H^{2}(\bbT;\bbR^{2})$, Corollary~\ref{C5.3} shows that the projection of this solution
    onto $\operatorname{RC}(\bbR^{2})^{\mathcal{L}}$ via $\pi$ is again $z_{\eps}^{t}$.
    Since the right-hand side of \eqref{6.1}
    %formula we derived for
 %   $\left.\partial_{t}\norm{z_{\eps}^{t,\lambda}}_{\dot{H}^{2}}^{2}\right|_{t=0}$
    only depends on $z_{\eps}^{0}$ (and not on any specific parametrization), we see that the formula then also holds for  $t_0$ in place of 0.
%    but not depending on the choice of $\tilde{z}_{\eps}^{0}\in H^{2}(\bbT;\bbR^{2})^{\mathcal{L}}$,
%    the same argument applies so the formula still holds for all $t\in\bbR$
%    with $z_{\eps}^{0}$ replaced by $z_{\eps}^{t}$.
\end{proof}

We now need to estimate the right-hand side of  \eqref{6.1}.  The second term can be bounded easily, but this is not the case for the first term.  This is why we instead  bound a linear combination of those terms for all the $\lambda\in\mathcal L$, with the relevant coefficients being precisely $\abs{\theta^{\lambda}}$, as was explained at the end of the introduction. % (as in the definition of $Q$).  
The resulting cancellations will yield a good enough estimate on $\partial_{t}Q(z_{\eps}^{t})$, even though we will not control $\partial_{t}\norm{z_{\eps}^{t,\lambda}}_{\dot{H}^{2}}^{2}$ for  individual $\lambda$.  The specific term requiring these cancellations is $G_6$ in the proof of Lemma~\ref{L6.3} below, and for them to happen, it suffices for us to control $\Delta(z^{t,\lambda}_\eps,z^{t,\lambda'}_\eps)^{-1}$ only when  $\lambda'\notin\Sigma^\lambda(z^t)$.

%Now we multiply $\abs{\theta^{\lambda}}$ to \eqref{6.1}, sum over $\lambda\in\mathcal{L}$
%and then estimate the resulting formula for
%$(m(\theta)\partial_{t}Q(z_{\eps}^{t}))|_{t=0}$.

\begin{lemma}\label{L6.2}
    There is $C_{\alpha}$ such that for each  $\lambda\in\mathcal{L}$ we have
    \begin{align*}
        \abs{\int_{\ell^{\lambda}\bbT}
        \kappa^{\lambda}(s)^{2}
        \left(
            \partial_{s}u_{\eps}^{\lambda}(s)
            \cdot \mathbf{T}^{\lambda}(s)
        \right)
        ds}
        &\leq
        \norm{\partial_{s}u_{\eps}^{\lambda}}_{L^{\infty}}   \norm{z_{\eps}^{0,\lambda}}_{\dot{H}^{2}}^{2}
        \leq
        C_{\alpha}\abs{\theta}W_{0}\,
        L(z_{\eps}^{0})^{2+2\alpha}   \norm{z_{\eps}^{0,\lambda}}_{\dot{H}^{2}}^{2}.
    \end{align*}
\end{lemma}

\begin{proof}
    The first inequality is trivial, and the second  follows from
    \begin{align*}
        \norm{\partial_{s}u_{\eps}^{\lambda}}_{L^{\infty}}
        &\leq C_{\alpha}\sum_{\lambda'\in\mathcal{L}}\abs{\theta^{\lambda'}}
        \ell^{\lambda'}\max\left\{\norm{z_{\eps}^{0,\lambda}}_{\dot{C}^{1,1/2}},
        \norm{z_{\eps}^{0,\lambda'}}_{\dot{C}^{1,1/2}}\right\}^{2+4\alpha} \\
        &\leq C_{\alpha}\sum_{\lambda'\in\mathcal{L}}\abs{\theta^{\lambda'}}
        \frac{\abs{\Omega(z_{\eps}^{0,\lambda'})}}{\Delta^{\lambda'}}
        \max\left\{\norm{z_{\eps}^{0,\lambda}}_{\dot{H}^{2}},
        \norm{z_{\eps}^{0,\lambda'}}_{\dot{H}^{2}}\right\}^{2+4\alpha} \\
        &\leq C_{\alpha}\abs{\theta}W_{0}L(z_{\eps}^{0})^{2+2\alpha},
    \end{align*}
    where we used Lemmas~\ref{L4.6}, \ref{L3.9}, and \ref{L5.4}.
\end{proof}

\begin{lemma}\label{L6.3}
    There is $C_{\alpha}$ such that
    \beq \lb{111.6}
        \abs{\sum_{\lambda\in\mathcal{L}}\abs{\theta^{\lambda}}
        \int_{\ell^{\lambda}\bbT}
        \kappa^{\lambda}(s)
        \left(
            \partial_{s}^{2}u_{\eps}^{\lambda}(s)
            \cdot \mathbf{N}^{\lambda}(s)
        \right)
        ds} \leq C_{\alpha} m(\theta) \abs{\theta}W_{0} \,
        L(z_{\eps}^{0})^{3+2\alpha}.
    \eeq
\end{lemma}

\begin{proof}
    With $z_{\eps}^{0,\lambda}(\cdot)$ being the previously fixed  arclength parametrization of $z_{\eps}^{0,\lambda}$, we have
    \begin{align*}
        \partial_{s}^{2}u_{\eps}^{\lambda}(s)
        =& D^{2}(u_{\eps}(z_{\eps}^{0}))(z_{\eps}^{0,\lambda}(s))
        (\mathbf{T}^{\lambda}(s),\mathbf{T}^{\lambda}(s))
        + \kappa^{\lambda}(s)
        D(u_{\eps}(z_{\eps}^{0}))(z_{\eps}^{0,\lambda}(s))(\mathbf{N}^{\lambda}(s)) \\
        = & -\sum_{\lambda'\in\mathcal{L}}\theta^{\lambda'}\int_{\ell^{\lambda'}\bbT}
        D^{2}K_{\eps}(z_{\eps}^{0,\lambda}(s) - z_{\eps}^{0,\lambda'}(s'))
        (\mathbf{T}^{\lambda}(s),\mathbf{T}^{\lambda}(s))
       \, \mathbf{T}^{\lambda'}(s')\,ds'
        \\&
        - \kappa^{\lambda}(s)\sum_{\lambda'\in\mathcal{L}}\theta^{\lambda'}
        \int_{\ell^{\lambda'}\bbT}
        DK_{\eps}(z_{\eps}^{0,\lambda}(s) - z_{\eps}^{0,\lambda'}(s'))
        (\mathbf{N}^{\lambda}(s))
       \, \mathbf{T}^{\lambda'}(s')\,ds'.
    \end{align*}
    Hence, the sum in \eqref{111.6} is the sum of the terms
    \begin{align*}
        G_{1} &\coloneqq
        -\sum_{\lambda,\lambda'\in\mathcal{L}}\abs{\theta^{\lambda}}\theta^{\lambda'}
        \int_{\ell^{\lambda}\bbT}\kappa^{\lambda}(s)
        \int_{\ell^{\lambda'}\bbT}
        D^{2}K_{\eps}(z_{\eps}^{0,\lambda}(s) - z_{\eps}^{0,\lambda'}(s'))
        (\mathbf{T}^{\lambda}(s),\mathbf{T}^{\lambda}(s))
        \\&\qquad\qquad\quad\quad\quad\quad\quad\quad\quad\quad\quad\quad
        (\mathbf{T}^{\lambda'}(s')\cdot\mathbf{N}^{\lambda}(s))
        \,ds'\,ds, \\
        G_{2} &\coloneqq
        -\sum_{\lambda,\lambda'\in\mathcal{L}}\abs{\theta^{\lambda}}\theta^{\lambda'}
        \int_{\ell^{\lambda}\bbT}\kappa^{\lambda}(s)^{2}
        \int_{\ell^{\lambda'}\bbT}
        DK_{\eps}(z_{\eps}^{0,\lambda}(s) - z_{\eps}^{0,\lambda'}(s'))
        (\mathbf{N}^{\lambda}(s))
        \\&\qquad\qquad\quad\quad\quad\quad\quad\quad\quad\quad\quad\quad
        (\mathbf{T}^{\lambda'}(s')\cdot\mathbf{N}^{\lambda}(s))
        \,ds'\,ds.
    \end{align*}
    Similarly to \eqref{4.3}, % is derived in the proof of Lemma~\ref{L4.6},
    Lemma~\ref{L4.5} shows that
    \begin{align*}
        \abs{G_{2}} &\leq C_{\alpha}\sum_{\lambda\in\mathcal{L}}\abs{\theta^{\lambda}}
        \norm{z_{\eps}^{0,\lambda}}_{\dot{H}^{2}}^{2}
        \sum_{\lambda'\in\mathcal{L}}\abs{\theta^{\lambda'}}\ell^{\lambda'}\max\set{
            \norm{z_{\eps}^{0,\lambda}}_{\dot{C}^{1,1/2}},
            \norm{z_{\eps}^{0,\lambda'}}_{\dot{C}^{1,1/2}}
        }^{2+4\alpha},
    \end{align*}
    and as in the proof of Lemma~\ref{L6.2}, Lemmas~\ref{L3.9}
    and \ref{L5.4} yield
    \begin{align*}
        \sum_{\lambda'\in\mathcal{L}}\abs{\theta^{\lambda'}}\ell^{\lambda'}\max\set{
            \norm{z_{\eps}^{0,\lambda}}_{\dot{C}^{1,1/2}},
            \norm{z_{\eps}^{0,\lambda'}}_{\dot{C}^{1,1/2}}
        }^{2+4\alpha}
        &\leq \abs{\theta}W_{0}L(z_{\eps}^{0})^{2+2\alpha}.
    \end{align*}
    Hence  $\abs{G_{2}}
    \leq C_{\alpha}m(\theta)\abs{\theta}W_{0}\,L(z_{\eps}^{0})^{3+2\alpha}$.

Using twice both
\[
\mathbf{T}^{\lambda}(s)=(\mathbf{T}^{\lambda'}(s')\cdot \mathbf{T}^{\lambda}(s)) \, \mathbf{T}^{\lambda'}(s') + (\mathbf{N}^{\lambda'}(s')\cdot \mathbf{T}^{\lambda}(s)) \, \mathbf{N}^{\lambda'}(s')
\]
and  $\mathbf{N}^{\lambda'}(s')\cdot\mathbf{T}^{\lambda}(s)
    = -\mathbf{T}^{\lambda'}(s')\cdot\mathbf{N}^{\lambda}(s)$, we find that $G_{1}$ is the sum of the  terms
    \begin{align*}
        G_{3} &\coloneqq
        -\sum_{\lambda,\lambda'\in\mathcal{L}}\abs{\theta^{\lambda}}\theta^{\lambda'}
        \int_{\ell^{\lambda}\bbT}\kappa^{\lambda}(s)
        \int_{\ell^{\lambda'}\bbT}
        D^{2}K_{\eps}(z_{\eps}^{0,\lambda}(s) - z_{\eps}^{0,\lambda'}(s'))
        (\mathbf{T}^{\lambda}(s),\mathbf{T}^{\lambda'}(s'))
        \\&\qquad\qquad\quad\quad\quad\quad\quad\quad\quad\quad\quad\quad
        (\mathbf{T}^{\lambda'}(s')\cdot\mathbf{T}^{\lambda}(s))
        (\mathbf{T}^{\lambda'}(s')\cdot\mathbf{N}^{\lambda}(s))
        \,ds'\,ds, \\
        G_{4} &\coloneqq
        \sum_{\lambda,\lambda'\in\mathcal{L}}\abs{\theta^{\lambda}}\theta^{\lambda'}
        \int_{\ell^{\lambda}\bbT}\kappa^{\lambda}(s)
        \int_{\ell^{\lambda'}\bbT}
        D^{2}K_{\eps}(z_{\eps}^{0,\lambda}(s) - z_{\eps}^{0,\lambda'}(s'))
        (\mathbf{T}^{\lambda'}(s'),\mathbf{N}^{\lambda'}(s'))
        \\&\qquad\qquad\quad\quad\quad\quad\quad\quad\quad\quad\quad\quad
        (\mathbf{T}^{\lambda'}(s')\cdot\mathbf{T}^{\lambda}(s))
        (\mathbf{T}^{\lambda'}(s')\cdot\mathbf{N}^{\lambda}(s))^{2}
        \,ds'\,ds, \\
        G_{5} &\coloneqq
        -\sum_{\lambda,\lambda'\in\mathcal{L}}\abs{\theta^{\lambda}}\theta^{\lambda'}
        \int_{\ell^{\lambda}\bbT}\kappa^{\lambda}(s)
        \int_{\ell^{\lambda'}\bbT}
        D^{2}K_{\eps}(z_{\eps}^{0,\lambda}(s) - z_{\eps}^{0,\lambda'}(s'))
        (\mathbf{N}^{\lambda'}(s'),\mathbf{N}^{\lambda'}(s'))
        \\&\qquad\qquad\quad\quad\quad\quad\quad\quad\quad\quad\quad\quad
        (\mathbf{T}^{\lambda'}(s')\cdot\mathbf{N}^{\lambda}(s))^{3}
        \,ds'\,ds.
    \end{align*}
We now estimate these separately.
        
\smallskip
    \textbf{Estimate for $G_{3}$.} Since
    \begin{align*}
        \frac{\partial}{\partial s'} & \left(
            DK_{\eps}(z_{\eps}^{0,\lambda}(s) - z_{\eps}^{0,\lambda'}(s'))
            (\mathbf{T}^{\lambda}(s))
           \, (\mathbf{T}^{\lambda'}(s')\cdot\mathbf{N}^{\lambda}(s))
            (\mathbf{T}^{\lambda'}(s')\cdot\mathbf{T}^{\lambda}(s))
        \right) \\
        & =
        -D^{2}K_{\eps}(z_{\eps}^{0,\lambda}(s) - z_{\eps}^{0,\lambda'}(s'))
        (\mathbf{T}^{\lambda}(s),\mathbf{T}^{\lambda'}(s'))
        \, (\mathbf{T}^{\lambda'}(s')\cdot\mathbf{N}^{\lambda}(s))
        (\mathbf{T}^{\lambda'}(s')\cdot\mathbf{T}^{\lambda}(s))
        \\&\quad
        + \kappa^{\lambda'}(s')DK_{\eps}(z_{\eps}^{0,\lambda}(s)
        - z_{\eps}^{0,\lambda'}(s'))
        (\mathbf{T}^{\lambda}(s)) \,
        (\mathbf{T}^{\lambda'}(s')\cdot\mathbf{T}^{\lambda}(s))^{2}
        \\&\quad
        - \kappa^{\lambda'}(s')DK_{\eps}(z_{\eps}^{0,\lambda}(s)
        - z_{\eps}^{0,\lambda'}(s'))
        (\mathbf{T}^{\lambda}(s)) \,
        (\mathbf{T}^{\lambda'}(s')\cdot\mathbf{N}^{\lambda}(s))^{2},
    \end{align*}
    integration by parts shows that $G_{3}$ is the sum of the  terms
    \begin{align*}
        G_{6}&\coloneqq
        -\sum_{\lambda,\lambda'\in\mathcal{L}}\abs{\theta^{\lambda}}\theta^{\lambda'}
        \int_{\ell^{\lambda}\bbT\times \ell^{\lambda'}\bbT}
        \kappa^{\lambda}(s)\kappa^{\lambda'}(s')
        DK_{\eps}(z_{\eps}^{0,\lambda}(s) - z_{\eps}^{0,\lambda'}(s'))
        (\mathbf{T}^{\lambda}(s))
        \\&\quad\quad\quad\quad\quad\quad\quad\quad\quad\quad\quad\quad
        (\mathbf{T}^{\lambda'}(s')\cdot\mathbf{T}^{\lambda}(s))^{2}
        \,ds'\,ds, \\
        G_{7}&\coloneqq
        \sum_{\lambda,\lambda'\in\mathcal{L}}\abs{\theta^{\lambda}}\theta^{\lambda'}
        \int_{\ell^{\lambda}\bbT \times \ell^{\lambda'}\bbT}
        \kappa^{\lambda}(s)\kappa^{\lambda'}(s')
        DK_{\eps}(z_{\eps}^{0,\lambda}(s) - z_{\eps}^{0,\lambda'}(s'))
        (\mathbf{T}^{\lambda}(s))
        \\&\quad\quad\quad\quad\quad\quad\quad\quad\quad\quad\quad\quad
        (\mathbf{T}^{\lambda'}(s')\cdot\mathbf{N}^{\lambda}(s))^{2}
        \,ds'\,ds.
    \end{align*}
%    ({\color{red}Now here comes the first of the two places where we use the fact that
%    touches only happen in certain directions depending on the signs of the patches.})
    To estimate $G_{6}$, we symmetrize the integral to get
    \begin{align*}
        G_{6} &= -\frac{1}{2}\sum_{\lambda,\lambda'\in\mathcal{L}}
        \abs{\theta^{\lambda}\theta^{\lambda'}}
        \int_{\ell^{\lambda}\bbT\times\ell^{\lambda'}\bbT}
        \kappa^{\lambda}(s)\kappa^{\lambda'}(s')
        \\ & DK_{\eps}(z_{\eps}^{0,\lambda}(s) - z_{\eps}^{0,\lambda'}(s'))
        (\sgn(\theta^{\lambda'})\mathbf{T}^{\lambda}(s)- \sgn(\theta^{\lambda})\mathbf{T}^{\lambda'}(s'))
  %      \\&\quad\quad\quad\quad\quad\quad\quad\quad\quad\quad\quad\quad
        \,(\mathbf{T}^{\lambda'}(s')\cdot\mathbf{T}^{\lambda}(s))^{2}
        \,ds'\,ds.
    \end{align*}
%    where the sign for $\mathbf{T}^{\lambda}(s)$ is $+$ if and only if
%    $\theta^{\lambda'}>0$ and the sign for $\mathbf{T}^{\lambda'}(s')$ is $+$ if and only if
%    $\theta^{\lambda}<0$. 
    Hence, when $\lambda'\in\Sigma^{\lambda}(z_{\eps}^{0})$, Lemma~\ref{L4.5}(2,3) (with $\beta\coloneqq\frac 12$,
    $\gamma_{2}=\gamma_{3}$, $\phi\coloneqq{\rm Id}$, and $x\coloneqq\gamma_{1}(s')$)
    and
    %the inequality
%    $\abs{\kappa^{\lambda}(s)\kappa^{\lambda'}(s')}\leq\frac{1}{2}\left(
%        \abs{\kappa^{\lambda}(s)}^{2} + \abs{\kappa^{\lambda'}(s')}^{2}
%    \right)$, 
    Lemmas~\ref{L3.9} and \ref{L5.4} show that the absolute value of the integral above
    % in each summand 
    is bounded by
    \begin{align*}
        &\frac{C_{\alpha}}{\min\left\{
            \Delta^{\lambda},
            \Delta^{\lambda'}
        \right\}^{1+2\alpha}}
        \left(
            \ell^{\lambda'}\int_{\ell^{\lambda}\bbT}\abs{\kappa^{\lambda}(s)}^{2}\,ds
            + \ell^{\lambda}\int_{\ell^{\lambda'}\bbT}\abs{\kappa^{\lambda'}(s')}^{2}\,ds'
        \right)
        \\&\quad\quad\leq C_{\alpha}
        L(z_{\eps}^{0})^{1+2\alpha}\left(
            \frac{\abs{\Omega(z_{\eps}^{0,\lambda'})}}
            {\Delta^{\lambda'}}
            \norm{z_{\eps}^{0,\lambda}}_{\dot{H}^{2}}^{2}
            + \frac{\abs{\Omega(z_{\eps}^{0,\lambda})}}
            {\Delta^{\lambda}}
            \norm{z_{\eps}^{0,\lambda'}}_{\dot{H}^{2}}^{2}
        \right)
        \\&\quad\quad\leq C_{\alpha}W_{0}
        L(z_{\eps}^{0})^{2+2\alpha}\left(
            \norm{z_{\eps}^{0,\lambda}}_{\dot{H}^{2}}^{2}
            + \norm{z_{\eps}^{0,\lambda'}}_{\dot{H}^{2}}^{2}
        \right).
    \end{align*}
    When $\lambda'\notin \Sigma^{\lambda}(z_{\eps}^{0})$, 
    Lemmas~\ref{L4.2} (with $\beta\coloneqq\frac 12$), \ref{L3.9},
    and \ref{L5.4} show that that absolute value is instead  bounded by
    \begin{align*}
        &\frac{C_{\alpha}}{\Delta^{\lambda,\lambda'}}
        \int_{\ell^{\lambda}\bbT\times\ell^{\lambda'}\bbT}
        \frac{\abs{\kappa^{\lambda}(s)}^{2} + \abs{\kappa^{\lambda'}(s')}^{2}}
        {\abs{z^{\lambda}(s) - z^{\lambda'}(s')}^{2\alpha}}
        \,ds'\,ds
        \\&\quad\quad\leq
        C_{\alpha}L(z_{\eps}^{0})\bigg(
            \ell^{\lambda'}\norm{z_{\eps}^{0,\lambda'}}_{\dot{H}^{2}}^{4\alpha}
            \int_{\ell^{\lambda}\bbT}\abs{\kappa^{\lambda}(s)}^{2}\,ds
            + \ell^{\lambda}\norm{z_{\eps}^{0,\lambda}}_{\dot{H}^{2}}^{4\alpha}
            \int_{\ell^{\lambda'}\bbT}\abs{\kappa^{\lambda'}(s')}^{2}\,ds'
        \bigg)
        \\&\quad\quad\leq
        C_{\alpha}L(z_{\eps}^{0})^{1+2\alpha}\left(
            \frac{\abs{\Omega(z_{\eps}^{0,\lambda'})}}
            {\Delta^{\lambda'}}
            \norm{z_{\eps}^{0,\lambda}}_{\dot{H}^{2}}^{2}
            + \frac{\abs{\Omega(z_{\eps}^{0,\lambda})}}
            {\Delta^{\lambda}}
            \norm{z_{\eps}^{0,\lambda'}}_{\dot{H}^{2}}^{2}
        \right)
        \\&\quad\quad\leq C_{\alpha}W_{0}
        L(z_{\eps}^{0})^{2+2\alpha}\left(
            \norm{z_{\eps}^{0,\lambda}}_{\dot{H}^{2}}^{2}
            + \norm{z_{\eps}^{0,\lambda'}}_{\dot{H}^{2}}^{2}
        \right).
    \end{align*}
    This yields
       $ \abs{G_{6}} \leq C_{\alpha}m(\theta)\abs{\theta}W_{0}\,
        L(z_{\eps}^{0})^{3+2\alpha}$.

    Finally, we clearly have
    \begin{align*}
        \abs{G_{7}} \leq C_{\alpha}\sum_{\lambda,\lambda'\in\mathcal{L}}
        \abs{\theta^{\lambda}\theta^{\lambda'}}
        \int_{\ell^{\lambda}\bbT\times \ell^{\lambda'}\bbT}
        \frac{\left(\abs{\kappa^{\lambda}(s)}^{2} + \abs{\kappa^{\lambda'}(s')}^{2}\right)
        \abs{\mathbf{T}^{\lambda'}(s')\cdot\mathbf{N}^{\lambda}(s)}}
        {\abs{z^{\lambda}(s) - z^{\lambda'}(s')}^{1 + 2\alpha}}
        \,ds'\,ds.
    \end{align*}
    Thus the argument for $G_6$ in the case $\lambda'\in\Sigma^{\lambda}(z_{\eps}^{0})$, but with Lemma~\ref{L4.5}(1) in place of    Lemma~\ref{L4.5}(2,3), again yields
$        \abs{G_{7}} \leq C_{\alpha}m(\theta)\abs{\theta}W_{0}\,
        L(z_{\eps}^{0})^{3+2\alpha}$.
        
\smallskip
    \textbf{Estimate for $G_{4}$.} Similarly to $G_{3}$, we use integration by parts
    to reduce the number of derivatives applied to the kernel
    at the expense of producing the factor $\kappa^{\lambda'}(s')$. Since
    \begin{align*}
        \frac{\partial}{\partial s'} &\left(
            DK_{\eps}(z_{\eps}^{0,\lambda}(s) - z_{\eps}^{0,\lambda'}(s'))
            (\mathbf{N}^{\lambda'}(s'))
            \, (\mathbf{T}^{\lambda'}(s')\cdot\mathbf{N}^{\lambda}(s))^{2}
            (\mathbf{T}^{\lambda'}(s')\cdot\mathbf{T}^{\lambda}(s))
        \right) \\
        & =
        -D^{2}K_{\eps}(z_{\eps}^{0,\lambda}(s) - z_{\eps}^{0,\lambda'}(s'))
        (\mathbf{T}^{\lambda'}(s'),\mathbf{N}^{\lambda'}(s'))
        \, (\mathbf{T}^{\lambda'}(s')\cdot\mathbf{N}^{\lambda}(s))^{2}
        (\mathbf{T}^{\lambda'}(s')\cdot\mathbf{T}^{\lambda}(s))
        \\&\quad
        - \kappa^{\lambda'}(s')
        DK_{\eps}(z_{\eps}^{0,\lambda}(s) - z_{\eps}^{0,\lambda'}(s'))
        (\mathbf{T}^{\lambda'}(s'))
        \, (\mathbf{T}^{\lambda'}(s')\cdot\mathbf{N}^{\lambda}(s))^{2}
        (\mathbf{T}^{\lambda'}(s')\cdot\mathbf{T}^{\lambda}(s))
        \\&\quad
        + 2\kappa^{\lambda'}(s')
        DK_{\eps}(z_{\eps}^{0,\lambda}(s) - z_{\eps}^{0,\lambda'}(s'))
        (\mathbf{N}^{\lambda'}(s'))
        \, (\mathbf{T}^{\lambda'}(s')\cdot\mathbf{N}^{\lambda}(s))
        (\mathbf{T}^{\lambda'}(s')\cdot\mathbf{T}^{\lambda}(s))^{2}
        \\&\quad
        - \kappa^{\lambda'}(s')
        DK_{\eps}(z_{\eps}^{0,\lambda}(s) - z_{\eps}^{0,\lambda'}(s'))
        (\mathbf{N}^{\lambda'}(s'))\, 
        (\mathbf{T}^{\lambda'}(s')\cdot\mathbf{N}^{\lambda}(s))^{3},
    \end{align*}
     integration by parts shows that $G_{4}$ is the sum of the  terms
    \begin{align*}
        G_{8}&\coloneqq
        -\sum_{\lambda,\lambda'\in\mathcal{L}}\abs{\theta^{\lambda}}\theta^{\lambda'}
        \int_{\ell^{\lambda}\bbT\times\ell^{\lambda'}\bbT}
        \kappa^{\lambda}(s)\kappa^{\lambda'}(s')
        DK_{\eps}(z_{\eps}^{0,\lambda}(s) - z_{\eps}^{0,\lambda'}(s'))
        (\mathbf{T}^{\lambda'}(s'))
        \\&\quad\quad\quad\quad\quad\quad\quad\quad\quad\quad\quad\quad
        \, (\mathbf{T}^{\lambda'}(s')\cdot\mathbf{N}^{\lambda}(s))^{2}
        (\mathbf{T}^{\lambda'}(s')\cdot\mathbf{T}^{\lambda}(s))
        \,ds'\,ds, \\
        G_{9}&\coloneqq
        2\sum_{\lambda,\lambda'\in\mathcal{L}}\abs{\theta^{\lambda}}\theta^{\lambda'}
        \int_{\ell^{\lambda}\bbT\times\ell^{\lambda'}\bbT}
        \kappa^{\lambda}(s)\kappa^{\lambda'}(s')
        DK_{\eps}(z_{\eps}^{0,\lambda}(s) - z_{\eps}^{0,\lambda'}(s'))
        (\mathbf{N}^{\lambda'}(s'))
        \\&\quad\quad\quad\quad\quad\quad\quad\quad\quad\quad\quad\quad
        \, (\mathbf{T}^{\lambda'}(s')\cdot\mathbf{N}^{\lambda}(s))
        (\mathbf{T}^{\lambda'}(s')\cdot\mathbf{T}^{\lambda}(s))^{2}
        \,ds'\,ds, \\
        G_{10}&\coloneqq
        -\sum_{\lambda,\lambda'\in\mathcal{L}}\abs{\theta^{\lambda}}\theta^{\lambda'}
        \int_{\ell^{\lambda}\bbT\times\ell^{\lambda'}\bbT}
        \kappa^{\lambda}(s)\kappa^{\lambda'}(s')
        DK_{\eps}(z_{\eps}^{0,\lambda}(s) - z_{\eps}^{0,\lambda'}(s'))
        (\mathbf{N}^{\lambda'}(s'))
        \\&\quad\quad\quad\quad\quad\quad\quad\quad\quad\quad\quad\quad
        \, (\mathbf{T}^{\lambda'}(s')\cdot\mathbf{N}^{\lambda}(s))^{3}
        \,ds'\,ds.
    \end{align*}
    Since the $DK_{\eps}$ term is bounded by
    $\frac{C_{\alpha}}{\abs{z^{\lambda}(s) - z^{\lambda'}(s')}^{1+2\alpha}}$ and
    each of $G_{8}$, $G_{9}$ and $G_{10}$ has at least one factor 
    $\mathbf{T}^{\lambda'}(s')\cdot\mathbf{N}^{\lambda}(s)$,
    the argument from the estimate for $G_7$ shows that again
    \begin{align*}
       \max\{\abs{G_8}, \abs{G_9}, \abs{G_{10}}\} \le  C_{\alpha}m(\theta)\abs{\theta}W_{0}\,
        L(z_{\eps}^{0})^{3+2\alpha}.
    \end{align*}
        
\smallskip
    \textbf{Estimate for $G_{5}$.} Lemma~\ref{L3.4} yields
    \begin{align*}
        \abs{G_{5}} &\leq C_{\alpha}\sum_{\lambda,\lambda'\in\mathcal{L}}
        \abs{\theta^{\lambda}\theta^{\lambda'}}
        \int_{\ell^{\lambda}\bbT\times\ell^{\lambda'}\bbT}
        \frac{\abs{\kappa^{\lambda}(s)}\left(
            \mathcal{M}\kappa^{\lambda}(s)
            + \mathcal{M}\kappa^{\lambda'}(s')
        \right)\abs{\mathbf{T}^{\lambda'}(s')\cdot\mathbf{N}^{\lambda}(s)}}
        {\abs{z^{\lambda}(s) - z^{\lambda'}(s')}^{1+2\alpha}}
        \,ds'\,ds,
    \end{align*}
%    and the numerator of this integral is bounded by
%    \begin{align*}
%        \left(
%            \abs{\kappa^{\lambda}(s)}^{2}
%            + \mathcal{M}\kappa^{\lambda}(s)^{2}
%            + \mathcal{M}\kappa^{\lambda'}(s')^{2}
%        \right)\abs{\mathbf{T}^{\lambda'}(s')\cdot\mathbf{N}^{\lambda}(s)}.
%    \end{align*}
%    Therefore, 
so the argument from the estimate for $G_7$ and \eqref{3.1} give $\abs{G_{5}} \leq C_{\alpha}m(\theta)\abs{\theta}W_{0}\,        L(z_{\eps}^{0})^{3+2\alpha}$.
%    \begin{align*}
%        \abs{G_{5}} \leq C_{\alpha}m(\theta)\abs{\theta}W_{0}\,
%        L(z_{\eps}^{0})^{3+2\alpha}.
%    \end{align*}

 The obtained estimates  for $G_n$ with $n=2,5,6,7,8,9,10$ now yield \eqref{111.6}.
\end{proof}

%Since $\Sigma^{\lambda}(z_{\eps}^{t})$'s are invariant in $t$ and
%\eqref{6.1} holds not only at $t=0$ but at all $t\in\bbR$ with $z_{\eps}^{0}$ replaced by
%$z_{\eps}^{t}$ in the definitions of $u_{\eps}^{\lambda}(s)$, $\ell^{\lambda}$,
%    $\mathbf{T}^{\lambda}(s)$, $\mathbf{N}^{\lambda}(s)$ and $\kappa^{\lambda}(s)$,
%it follows that Lemma~\ref{L6.2} and Lemma~\ref{L6.3}
%also hold with the replacements of $z_{\eps}^{0}$ by $z_{\eps}^{t}$
%and also of $u_{\eps}^{\lambda}(s)$, $\ell^{\lambda}$,
%$\mathbf{T}^{\lambda}(s)$, $\mathbf{N}^{\lambda}(s)$ and $\kappa^{\lambda}(s)$
%by the corresponding quantities at $t$ (but keeping $W_{0}$). Therefore, we obtain:

The last three lemmas now together yield the following.
  
\begin{corollary}\label{C6.4}
    $Q(z_{\eps}^{t})$ is differentiable in $t$ and
    there is $C_{\alpha}$ such that for all $t\in\bbR$ we have
    \[
        \abs{\partial_{t}Q(z_{\eps}^{t})}
        \leq C_{\alpha}\abs{\theta}W_{0}L(z_{\eps}^{t})^{3+2\alpha}.
    \]
\end{corollary}

In the next two results we estimate $\partial_{t+}\Delta(z_{\eps}^{t,\lambda},z_{\eps}^{t,\lambda'})$
and $\partial_{t+}\Delta_{1/Q(z_{\eps}^{t})}(z_{\eps}^{t,\lambda})$.  

%({\color{red}For below, see JZ, equation (2.9).})

\begin{lemma}\label{L6.5}
    There is $C_{\alpha}$ such that
    $\Delta(z_{\eps}^{t,\lambda},z_{\eps}^{t,\lambda'})$
    is Lipschitz continuous in $t\in\bbR$ with Lipschitz constant 
    $C_{\alpha}\abs{\theta}W_{0}^{\frac{1}{2}-\alpha}$  for any distinct $\lambda,\lambda'\in\mathcal{L}$, and the following holds.  For any $t\in\bbR$
    and arbitrary arclength parametrizations
    of $z_{\eps}^{t,\lambda}$ and $z_{\eps}^{t,\lambda'}$,
    there are $s\in\ell(z_{\eps}^{t,\lambda})\bbT$ and
    $s'\in\ell(z_{\eps}^{t,\lambda'})\bbT$ such that
    $\Delta(z_{\eps}^{t,\lambda},
    z_{\eps}^{t,\lambda'})
    = \abs{z_{\eps}^{t,\lambda}(s) - z_{\eps}^{t,\lambda'}(s')}$,
    \begin{equation}\label{6.2}
    \begin{aligned}
        \partial_{t+}
        \Delta(z_{\eps}^{t,\lambda},z_{\eps}^{t,\lambda'})
        \geq \frac{z_{\eps}^{t,\lambda}(s) - z_{\eps}^{t,\lambda'}(s')}
        {\abs{z_{\eps}^{t,\lambda}(s) - z_{\eps}^{t,\lambda'}(s')}}
        \cdot \left[
            u_{\eps}(z_{\eps}^{t};z_{\eps}^{t,\lambda}(s))
            - u_{\eps}(z_{\eps}^{t};z_{\eps}^{t,\lambda'}(s'))
        \right],
    \end{aligned}
    \end{equation}
    and
    \begin{equation}\label{6.3}
        \partial_{t+}
        \Delta(z_{\eps}^{t,\lambda},z_{\eps}^{t,\lambda'}) \geq
        -C_{\alpha}\abs{\theta}W_{0}\, L(z_{\eps}^{t})^{2+2\alpha}
        \Delta(z_{\eps}^{t,\lambda},z_{\eps}^{t,\lambda'}).
    \end{equation}
\end{lemma}

{\it Remark.}  The bound \eqref{6.2} is a version of (2.9) in \cite{JeoZla}.
%\smallskip

\begin{proof}
%We clearly only need to consider the case $\lambda\neq\lambda'$.
    For each $t\in \bbR$ fix any
    $\xi_{t},\xi_{t}'\in\bbT$ such that
    $\Delta(z_{\eps}^{t,\lambda}, z_{\eps}^{t,\lambda'})
    = \abs{\tilde{z}_{\eps}^{t,\lambda}(\xi_{t})
    - \tilde{z}_{\eps}^{t,\lambda'}(\xi_{t}')}$.
    Then for any $h\in \bbR$, Lemmas~\ref{L4.1} and \ref{L5.4} show
    \begin{align*}
        \Delta(z_{\eps}^{t+h,\lambda},z_{\eps}^{t+h,\lambda'})
        - \Delta(z_{\eps}^{t,\lambda},z_{\eps}^{t,\lambda'})
        &\leq \abs{\tilde{z}_{\eps}^{t+h,\lambda}(\xi_{t})
        - \tilde{z}_{\eps}^{t+h,\lambda'}(\xi_{t}')}
        - \abs{\tilde{z}_{\eps}^{t,\lambda}(\xi_{t})
        - \tilde{z}_{\eps}^{t,\lambda'}(\xi_{t}')} \\
        &\leq \abs{\int_{t}^{t+h}
        u_{\eps}(z_{\eps}^{\tau};\tilde{z}_{\eps}^{\tau,\lambda}(\xi_{t}))\,d\tau}
        + \abs{\int_{t}^{t+h}
        u_{\eps}(z_{\eps}^{\tau};\tilde{z}_{\eps}^{\tau,\lambda'}(\xi_{t}'))\,d\tau} \\
        &\leq 2\abs{\int_{t}^{t+h}\norm{u_{\eps}(z_{\eps}^{\tau})}_{L^{\infty}}\,d\tau} \\
        &\leq C_{\alpha}\abs{\theta}W_{0}^{\frac{1}{2}-\alpha}\abs{h}.
    \end{align*}
    This proves the first claim.
%    shows that $\Delta(z_{\eps}^{t,\lambda},z_{\eps}^{t,\lambda'})$
%    is Lipschitz continuous in $t$ with Lipschitz constant at most
%    $C_{\alpha}\abs{\theta}W_{0}^{\frac{1}{2}-\alpha}$.

    Next fix $t\in \bbR$ and take any decreasing sequence
    $\seq{t_{n}}_{n=1}^{\infty}$ converging to $t$ such that
    \beq\lb{111.7}
        \partial_{t+}\Delta(z_{\eps}^{t,\lambda},
        z_{\eps}^{t,\lambda'})
        = \lim_{n\to\infty}
        \frac{
            \Delta(z_{\eps}^{t_{n},\lambda},
            z_{\eps}^{t_{n},\lambda'})
            - \Delta(z_{\eps}^{t,\lambda},
            z_{\eps}^{t,\lambda'})
        }{t_{n}-t}.
    \eeq
    By passing to a subsequence if needed, we can assume that
    $\seq{(\xi_{t_{n}},\xi_{t_{n}}')}_{n=1}^{\infty}$ converges to some
    $(\xi,\xi')\in\bbT^2$. Then we let 
    $(s,s')\in \ell(z_{\eps}^{t,\lambda})\bbT\times \ell(z_{\eps}^{t,\lambda'})\bbT$ be the corresponding arguments
    in some arclength parametrizations of
    $z_{\eps}^{t,\lambda}$ and $z_{\eps}^{t,\lambda'}$ (which we then fix).
    Since $\tau\mapsto \Delta(z_{\eps}^{\tau,\lambda},z_{\eps}^{\tau,\lambda'})$
    and $(\tau,\zeta,\zeta')\mapsto
    \tilde{z}_{\eps}^{\tau,\lambda}(\zeta) - \tilde{z}_{\eps}^{\tau,\lambda'}(\zeta')$
    are continuous, we obtain
    \beq\lb{111.8}
        \Delta(z_{\eps}^{t,\lambda}, z_{\eps}^{t,\lambda'})
        = \abs{\tilde{z}_{\eps}^{t,\lambda}(\xi)
        - \tilde{z}_{\eps}^{t,\lambda'}(\xi')}
        = \abs{z_{\eps}^{t,\lambda}(s) - z_{\eps}^{t,\lambda'}(s')}.
    \eeq
    Our choice of $\xi_t,\xi_t'$ now yields
    \begin{align*}
        &            \Delta(z_{\eps}^{t_{n},\lambda},
            z_{\eps}^{t_{n},\lambda'})
            - \Delta(z_{\eps}^{t,\lambda},
            z_{\eps}^{t,\lambda'}) \\
        &\quad\geq
        \frac{\tilde{z}_{\eps}^{t,\lambda}(\xi_{t_{n}})
        - \tilde{z}_{\eps}^{t,\lambda'}(\xi_{t_{n}}')}
        {\abs{\tilde{z}_{\eps}^{t,\lambda}(\xi_{t_{n}})
        - \tilde{z}_{\eps}^{t,\lambda'}(\xi_{t_{n}}')}}
        \cdot
        (\tilde{z}_{\eps}^{t_{n},\lambda}(\xi_{t_{n}})
        - \tilde{z}_{\eps}^{t_{n},\lambda'}(\xi_{t_{n}}'))
        - |\tilde{z}_{\eps}^{t,\lambda}(\xi_{t_{n}})
        - \tilde{z}_{\eps}^{t,\lambda'}(\xi_{t_{n}}')| \\
        &\quad=
        \frac{\tilde{z}_{\eps}^{t,\lambda}(\xi_{t_{n}})
        - \tilde{z}_{\eps}^{t,\lambda'}(\xi_{t_{n}}')}
        {\abs{\tilde{z}_{\eps}^{t,\lambda}(\xi_{t_{n}})
        - \tilde{z}_{\eps}^{t,\lambda'}(\xi_{t_{n}}')}}
        \cdot
        \int_{t}^{t_{n}} \left[
        u_{\eps}(z_{\eps}^{\tau};\tilde{z}_{\eps}^{\tau,\lambda}(\xi_{t_{n}}))
        - u_{\eps}(z_{\eps}^{\tau};\tilde{z}_{\eps}^{\tau,\lambda'}(\xi_{t_{n}}')) \right]
        \,d\tau.
    \end{align*}
    This, \eqref{111.7}, and \eqref{111.8} now yield \eqref{6.2}.
    Then since $\partial_{\xi}\tilde{z}_{\eps}^{t,\lambda}(\xi)$ and
    $\partial_{\xi}\tilde{z}_{\eps}^{t,\lambda'}(\xi')$ 
    are both orthogonal to $\tilde{z}_{\eps}^{t,\lambda}(\xi)
    - \tilde{z}_{\eps}^{t,\lambda'}(\xi')$,
    % (and hence parallel), 
    Lemmas~\ref{L4.7}, \ref{L3.9}, and \ref{L5.4} show that
    \begin{align*}
        &\partial_{t+}\Delta(z_{\eps}^{t,\lambda},z_{\eps}^{t,\lambda'})
        \\&\quad\quad
        \geq -C_{\alpha}\abs{z_{\eps}^{t,\lambda}(s) - z_{\eps}^{t,\lambda'}(s')}
        \sum_{\lambda''\in\mathcal{L}}\abs{\theta^{\lambda''}}\ell(z_{\eps}^{t,\lambda''})
        \max\left\{
            \norm{z_{\eps}^{t,\lambda}}_{\dot{H}^{2}},
            \norm{z_{\eps}^{t,\lambda''}}_{\dot{H}^{2}}
        \right\}^{2+4\alpha}
        \\&\quad\quad
        \geq -C_{\alpha}\Delta(z_{\eps}^{t,\lambda},z_{\eps}^{t,\lambda'})
        \sum_{\lambda''\in\mathcal{L}}\abs{\theta^{\lambda''}}
        \frac{\abs{\Omega(z_{\eps}^{t,\lambda''})}}
        {\Delta_{1/Q(z_{\eps}^{t})}(z_{\eps}^{t,\lambda''})}
        L(z_{\eps}^{t})^{1+2\alpha}
        \\&\quad\quad
        \geq -C_{\alpha}\abs{\theta}W_{0}\, L(z_{\eps}^{t})^{2+2\alpha}
        \Delta(z_{\eps}^{t,\lambda},z_{\eps}^{t,\lambda'}).
    \end{align*}
    This finishes the proof.
\end{proof}

\begin{lemma}\label{L6.6}
    For any $\lambda\in\mathcal{L}$, $\Delta_{1/Q(z_{\eps}^{t})}(z_{\eps}^{t,\lambda})$
    is continuous in $t$. There is $C_{\alpha}$ such that
    for any $\lambda\in\mathcal{L}$, any
    $t\in\bbR$ with $\Delta_{1/Q(z_{\eps}^{t})}(z_{\eps}^{t,\lambda})
    <\frac{3}{4Q(z_{\eps}^{t})}$, and
    any arclength parametrization of $z_{\eps}^{t,\lambda}$,
    there are $s,s'\in\ell(z_{\eps}^{t,\lambda})\bbT$ such that
    $s - s'\in [\frac{1}{Q(z_{\eps}^{t})}, \frac{\ell(z_{\eps}^{t,\lambda})}{2}]$,
    $\Delta_{1/Q(z_{\eps}^{t})}(z_{\eps}^{t,\lambda})
    = \abs{z_{\eps}^{t,\lambda}(s) - z_{\eps}^{t,\lambda}(s')}$,
    \begin{equation}\label{6.4}
        \begin{aligned}
            \partial_{t+}
            \Delta_{1/Q(z_{\eps}^{t})}(z_{\eps}^{t,\lambda})
            \geq \frac{z_{\eps}^{t,\lambda}(s) - z_{\eps}^{t,\lambda}(s')}
            {\abs{z_{\eps}^{t,\lambda}(s) - z_{\eps}^{t,\lambda}(s')}}
            \cdot \left[
                u_{\eps}(z_{\eps}^{t};z_{\eps}^{t,\lambda}(s))
                - u_{\eps}(z_{\eps}^{t};z_{\eps}^{t,\lambda}(s'))
            \right],
        \end{aligned}
    \end{equation}
    and
    \begin{equation}\label{6.5}
        \partial_{t+}
        \Delta_{1/Q(z_{\eps}^{t})}(z_{\eps}^{t,\lambda}) \geq
        -C_{\alpha}\abs{\theta}W_{0}\, L(z_{\eps}^{t})^{2+2\alpha}
        \Delta_{1/Q(z_{\eps}^{t})}(z_{\eps}^{t,\lambda}).
    \end{equation}
\end{lemma}

\begin{proof}
    Lemma~\ref{LA.6} shows that
    $\Delta_{1/Q(z_{\eps}^{t})}(z_{\eps}^{t,\lambda})$
    is upper semicontinuous in $t$, so we need to show that it is also lower semicontinuous.
    For each $t\in\bbR$, take any $(\xi_{t},\xi_{t}')\in\bbT^2$ such that
     $    \int_{\xi_{t}'}^{\xi_{t}}
    \abs{\partial_{\xi}\tilde{z}_{\eps}^{t,\lambda}(\zeta)}\,d\zeta 
    \in [\frac{1}{Q(z_{\eps}^{t})}, \frac{\ell(z_{\eps}^{t,\lambda})}{2}]$ and
    $\Delta_{1/Q(z_{\eps}^{t})}(z_{\eps}^{t,\lambda})
    = \abs{\tilde{z}_{\eps}^{t,\lambda}(\xi_{t})
    - \tilde{z}_{\eps}^{t,\lambda}(\xi_{t}')}$.
    Fix any $t\in\bbR$ and take any $t_n\to t$.
    % sequence    $\seq{t_{n}}_{n=1}^{\infty}$ convergent to $t$. 
%    We want to show that
%    \[
%        \liminf_{n\to\infty}\Delta_{1/Q(z_{\eps}^{t_{n}})}(z_{\eps}^{t_{n},\lambda})
%        \geq \Delta_{1/Q(z_{\eps}^{t})}(z_{\eps}^{t,\lambda}).
%    \]
    By passing to a subsequence, we can assume that
    $\lim_{n\to\infty}\Delta_{1/Q(z_{\eps}^{t_{n}})}(z_{\eps}^{t_{n},\lambda})$ exists,
    and by passing to a further subsequence, we can also assume that
    $\seq{(\xi_{t_{n}},\xi_{t_{n}}')}_{n=1}^{\infty}$ converges to some
    $(\xi,\xi') \in\bbT^2$. By continuity of
    $\tau\mapsto \tilde{z}_{\eps}^{\tau,\lambda}$ with respect to
    the $C^{1}$-norm, we see that $(\tau,\zeta,\zeta') \mapsto \int_{\zeta'}^{\zeta}
    \abs{\partial_{\xi}\tilde{z}^{\tau,\lambda}(\eta)}\,d\eta$ is continuous.
    This and continuity of $\tau\mapsto Q(z_{\eps}^{\tau})$ (which follows from Lemma~\ref{L6.1}) show that
    \begin{align*}
        \frac{1}{Q(z_{\eps}^{t})}
        \leq \int_{\xi'}^{\xi}
        \abs{\partial_{\xi}\tilde{z}_{\eps}^{t,\lambda}(\zeta)}\,d\zeta
        \leq \frac{\ell(z_{\eps}^{t,\lambda})}{2},
    \end{align*}
    and thus
    \begin{align*}
        \lim_{n\to\infty}\Delta_{1/Q(z_{\eps}^{t_{n}})}(z_{\eps}^{t_{n},\lambda})
        = \lim_{n\to\infty}\abs{\tilde{z}_{\eps}^{t_{n},\lambda}(\xi_{t_{n}})
        - \tilde{z}_{\eps}^{t_{n},\lambda}(\xi_{t_{n}}')}
        = \abs{\tilde{z}_{\eps}^{t,\lambda}(\xi)
        - \tilde{z}_{\eps}^{t,\lambda}(\xi')}
        \geq \Delta_{1/Q(z_{\eps}^{t})}(z_{\eps}^{t,\lambda}).
    \end{align*}
    Therefore $\Delta_{1/Q(z_{\eps}^{t})}(z_{\eps}^{t,\lambda})$
    is indeed lower semicontinuous in $t$.

    Now assume that $\seq{t_{n}}_{n=1}^{\infty}$ in the above argument also
    decreases to $t$ and satisfies
    \[
        \partial_{t+}\Delta_{1/Q(z_{\eps}^{t})}(z_{\eps}^{t,\lambda})
        = \lim_{n\to\infty}\frac{
            \Delta_{1/Q(z_{\eps}^{t_{n}})}(z_{\eps}^{t_{n},\lambda})
            - \Delta_{1/Q(z_{\eps}^{t})}(z_{\eps}^{t,\lambda})
        }{t_{n} - t},
    \]
    and we have $\Delta_{1/Q(z_{\eps}^{t})}(z_{\eps}^{t,\lambda})
    < \frac{3}{4Q(z_{\eps}^{t})}$. Lemmas~\ref{L3.8}
    and \ref{L3.1} (with $\beta = \frac 12$) then yield
    \beq \lb{111.9}
        \frac{1}{Q(z_{\eps}^{t})}
        < \int_{\xi'}^{\xi}\abs{\partial_{\xi}\tilde{z}_{\eps}^{t,\lambda}(\zeta)}\,d\zeta
        < \frac{3\ell(z_{\eps}^{t,\lambda})}{4}
        \leq \ell(z_{\eps}^{t,\lambda}) - \frac{1}{Q(z_{\eps}^{t})}.
    \eeq
    From $(\xi_{t_{n}},\xi_{t_{n}}')\to (\xi,\xi')$ we see that \eqref{111.9} continues to hold when the integration interval $(\xi,\xi')$ is replaced by   $(\xi_{t_{n}},\xi_{t_{n}}')$ for any large enough $n$.
    % (but without replacing $t$ by $t_{n}$)
     Therefore  we can again use  the argument proving \eqref{6.2} 
    to show \eqref{6.4}. And since
    $\partial_{\xi}\tilde{z}_{\eps}^{t,\lambda}(\xi)$ and
    $\partial_{\xi}\tilde{z}_{\eps}^{t,\lambda}(\xi')$ are again both orthogonal
    to $\tilde{z}_{\eps}^{t,\lambda}(\xi) - \tilde{z}_{\eps}^{t,\lambda}(\xi')$,
    the proof of \eqref{6.3} also yields \eqref{6.5}.
\end{proof}

The above results now allow us to obtain an $\eps$-independent estimate on $\partial_{t}^{+}L(z_{\eps}^{t})$ and $L(z_{\eps}^{t})$.

\begin{proposition}\label{P6.7}
    There is $C$ that only depends on ${\alpha,|\theta|,W_{0}}$
    such that with $p\coloneqq 3+2\alpha$ and $T_{\eps}\coloneqq \frac{1}{(p-1)CL(z_{\eps}^{0})^{p-1}}$ we have
    for all $t\in \left[0, T_{\eps}\right)$,
    \begin{equation}\label{6.6}
        \partial_{t}^{+}L(z_{\eps}^{t})
        \leq CL(z_{\eps}^{t})^{p}
    \end{equation}
    and hence
    \begin{equation}\label{6.7}
        L(z_{\eps}^{t})
        \leq \frac{L(z_{\eps}^{0})}
        {(1 - (p-1)CL(z_{\eps}^{0})^{p-1}t)^{\frac{1}{p-1}}}.
    \end{equation}
\end{proposition}

\begin{proof}
    Since $\Sigma^{\lambda}(z_{\eps}^{t})$ is independent of $t$ for each $\lambda\in\mathcal L$,
     Corollary~\ref{C6.4} and
    Lemmas~\ref{L6.5} and \ref{L6.6} show that
    $L(z_{\eps}^{t})$ is continuous in $t$. Since
    $Q(z_{\eps}^{t}) < \infty$ for all $t\in\bbR$ (see Lemma~\ref{L6.1}) and
    $\Delta_{h}(\gamma) > 0$  for any rectifiable simple closed curve $\gamma$ and $h\in\left(0,\frac{\ell(\gamma)}{2}\right]$, we see that $L(z_{\eps}^{t})<\infty$    for each $t\in\bbR$.
    
Since
    \[
        \partial_{t}^{+}\max\set{f,g}(t) \leq \begin{cases}
            \partial_{t}^{+}f(t) & \textrm{if $f(t) > g(t)$}, \\
            \max\set{\partial_{t}^{+}f(t),\partial_{t}^{+}g(t)}
            & \textrm{otherwise}
        \end{cases}
    \]
    holds for any continuous $f,g\colon\bbR\to\bbR$
    and all $t\in\bbR$,  if we have
    \[
        \max\set{
            2Q(z_{\eps}^{t}),
            \max_{\lambda\in\mathcal{L}}
            \max_{\lambda'\notin\Sigma^{\lambda}(z_{\eps}^{t})}
            \frac{1}{\Delta(z_{\eps}^{t,\lambda},z_{\eps}^{t,\lambda'})}
        }
        > \max_{\lambda\in\mathcal{L}}
        \frac{1}{\Delta_{1/Q(z_{\eps}^{t})}(z_{\eps}^{t,\lambda})}
    \]
     for some $t\in\bbR$,
    then Corollary~\ref{C6.4} and Lemma~\ref{L6.5} show that
    \begin{align*}
        \partial_{t}^{+}L(z_{\eps}^{t})
        \leq C_{\alpha}\abs{\theta}W_{0}\max\set{
            L(z_{\eps}^{t})^{3+2\alpha},
            \max_{\lambda\in\mathcal{L}}
            \max_{\lambda'\notin\Sigma^{\lambda}(z_{\eps}^{t})}
            \frac{L(z_{\eps}^{t})^{2+2\alpha}}
            {\Delta(z_{\eps}^{t,\lambda},z_{\eps}^{t,\lambda'})}
        }
        = C_{\alpha}\abs{\theta}W_{0}L(z_{\eps}^{t})^{3+2\alpha}.
    \end{align*}
    If instead
    \[
        \max\set{
            2Q(z_{\eps}^{t}),
            \max_{\lambda\in\mathcal{L}}
            \max_{\lambda'\notin\Sigma^{\lambda}(z_{\eps}^{t})}
            \frac{1}{\Delta(z_{\eps}^{t,\lambda},z_{\eps}^{t,\lambda'})}
        }
        \leq \max_{\lambda\in\mathcal{L}}
        \frac{1}{\Delta_{1/Q(z_{\eps}^{t})}(z_{\eps}^{t,\lambda})},
    \]
    pick any $\lambda\in\mathcal{L}$ that achieves the maximum of the right-hand side.
   Then
    \[
        \Delta_{1/Q(z_{\eps}^{t})}(z_{\eps}^{t,\lambda})
        \leq \frac{1}{2Q(z_{\eps}^{t})} < \frac{3}{4Q(z_{\eps}^{t})},
    \]
    so Lemma~\ref{L6.6} shows that
    \[
        \partial_{t}^{+}\left(\frac{1}
        {\Delta_{1/Q(z_{\eps}^{t})}(z_{\eps}^{t,\lambda})}\right)
        \leq \frac{C_{\alpha}\abs{\theta}W_{0}\, L(z_{\eps}^{t})^{2+2\alpha}}
        {\Delta_{1/Q(z_{\eps}^{t})}(z_{\eps}^{t,\lambda})}
        \leq C_{\alpha}\abs{\theta}W_{0}\,L(z_{\eps}^{t})^{3+2\alpha},
    \]
    and we still obtain
    \[
        \partial_{t}^{+}L(z_{\eps}^{t})
        \leq C_{\alpha}\abs{\theta}W_{0}\, L(z_{\eps}^{t})^{3+2\alpha}.
    \]
    This proves \eqref{6.6}, and \eqref{6.7} then follows     by a Gr\"{o}nwall-type argument.
\end{proof}

%%%%%%%%%%%%%%%%%%%%%%%%%%%%%%%%%%%%%%%%%%%%%%
\section{Existence of solutions to \eqref{2.4}}\label{S7}
%%%%%%%%%%%%%%%%%%%%%%%%%%%%%%%%%%%%%%%%%%%%%%

Consider again the setup from the start of Section \ref{S6}, so in particular \eqref{111.12} holds, and let
$C$, $p$, and $T_\eps$ be from Proposition~\ref{P6.7}.  Then we have
\beq\lb{111.13}
    T_{0} \coloneqq \frac{1}{(p-1)C\,(60L(z^{0}))^{p-1}}  \in \left(0,\inf_{\eps>0} T_{\eps} \right),
\eeq
and we will now find a solution to  \eqref{2.4} on $[0,T_{0}]$ as a  limit of
$z_{\eps_n}$   for some  $\eps_n\to 0^{+}$.
% (along a subsequence).

\begin{lemma}\label{P7.2}
    There are $\eps_n\to 0^+$
%     is a decreasing sequence
%    $\seq{\eps_{n}}_{n=1}^{\infty}$ of positive real numbers
%    converging to $0$ 
    such that $z_{\eps_{n}}$
    converges  in $C([0,T_0];\operatorname{PSC}(\bbR^{2})^{\mathcal{L}})$.
    % to some $z$
    %$z\colon [0,T_{0}] \to \operatorname{PSC}(\bbR^{2})^{\mathcal{L}}$
\end{lemma}

\begin{proof}
From \eqref{6.7} and \eqref{111.12} we see that with $C,p$ from \eqref{111.13} we have
%there is $M_0<\infty$ depending only on ${\alpha,|\theta|,W_{0},L(z^0)}$ (with $z^0$ the given initial condition) such that 
    \[
        \sup_{\eps>0}\sup_{t\in[0,T_{0}]}L(z_{\eps}^{t})\le M_0 \coloneqq \frac{56L(z^{0})}
        {(1 - (p-1)CT_{0}(56L(z^{0}))^{p-1})^{\frac{1}{p-1}}}.
    \]
    Then
    Lemmas~\ref{L3.9} and \ref{L5.4} show that Proposition~\ref{P7.1} with $M\coloneqq\max\{30M_0W_{0},M_0\}$ and $I\coloneqq[0,T_{0}]$ applies to $z_\eps$ for each $\eps>0$.
    Hence Proposition~\ref{P7.1} and Lemmas~\ref{L4.1} and \ref{L5.4} imply that
    for any  $0\le t_{1}\leq t_{2}\le T_0$ we have
    \begin{align*}
        d_{\mathrm{F}}(z_{\eps}^{t_{1}},z_{\eps}^{t_{2}})
        &\leq d_{\mathrm{F}}(z_{\eps}^{t_{1}},
        X_{u_{\eps}(z_{\eps}^{t_{1}})}^{t_{2}-t_{1}}[z_{\eps}^{t_{1}}])
        + d_{\mathrm{F}}(z_{\eps}^{t_{2}},
        X_{u_{\eps}(z_{\eps}^{t_{1}})}^{t_{2}-t_{1}}[z_{\eps}^{t_{1}}]) \\
        &\leq \norm{u_{\eps}(z_{\eps}^{t_{1}})}_{L^{\infty}}(t_{2}-t_{1})
        + C(t_{2}-t_{1})^{2-2\alpha} \\
        &\leq C(1 + T_0^{1-2\alpha})(t_{2}-t_{1}),
    \end{align*}
    with $C$ only depending on ${\alpha,|\theta|,W_{0},M_0}$ (and changing between inequalities).
    So $\set{z_{\eps}|_{[0,T_{0}]}}_{\eps>0}$ is an equicontinuous family of functions with values in
    $\set{w\in\operatorname{PSC}(\bbR^{2})^{\mathcal{L}}\colon L(w) \leq M_0}$,
%    \begin{align*}
%        X\coloneqq \set{w\in\operatorname{PSC}(\bbR^{2})^{\mathcal{L}}\colon L(w) \leq M_0}.
%    \end{align*}
   % Since also $z_\eps^0\to z^0$ in $X$ by 
   which together with \eqref{111.12} shows that all curves from $\{z_{\eps}^{t,\lambda}\,:\, \eps>0 \,\,\&\,\, (t,\lambda)\in[0,T_{0}]\times\mathcal L\}$ also lie in the same ball in $\bbR^2$. Hence Corollary~\ref{CA.8} and the
    Arzel\`{a}-Ascoli theorem finish the proof.
\end{proof}

\begin{proposition}\label{P7.3}
    Let $z\in C([0,T_0];\operatorname{PSC}(\bbR^{2})^{\mathcal{L}})$ be any limit given by Lemma~\ref{P7.2}     (with $z^{0}$ then being the given initial data due to \eqref{111.12}). Then $\sup_{t\in[0,T_{0}]}L(z^{t})<\infty$, and \eqref{111.27}
%    \begin{equation*}
%        \abs{\Omega(z^{t,\lambda})} = \abs{\Omega(z^{0,\lambda})}\qquad\text{and}\qquad
%        \Sigma^{\lambda}(z^{t}) = \Sigma^{\lambda}(z^{0})
%    \end{equation*}
    holds for any $(t,\lambda)\in[0,T_{0}]\times \mathcal{L}$.
    Moreover, there is $C$ that only depends on
    ${\alpha,|\theta|,W_{0},L(z^{0})}$ such that for any
     $t,t+h\in[0,T_{0}]$, $\lambda\in\mathcal{L}$, and any  arclength parametrizations
    of $z^{t,\lambda}$ and $z^{t+h,\lambda}$, there is a Lipschitz continuous
    %orientation-preserving 
    homeomorphism
    $\phi\colon\ell(z^{t,\lambda})\bbT\to\ell(z^{t+h,\lambda})\bbT$ such that
    $e^{-C\abs{h}} \leq \phi'(s) \leq e^{C\abs{h}}$ for almost all
    $s\in\ell(z^{t,\lambda})\bbT$ and
    \begin{equation}\lb{111.15}
        \norm{z^{t+h,\lambda}\circ\phi
        - z^{t,\lambda}
        - hu(z^{t})\circ z^{t,\lambda}}_{L^{\infty}(\ell(z^{t,\lambda})\bbT)}
        \leq C\abs{h}^{2-2\alpha}.
    \end{equation}
    In particular, for any $t,t+h\in [0,T_{0}]$ we have
    \[
        d_{\mathrm{F}}(z^{t+h},
        X_{u(z^{t})}^{h}[z^{t}])
        \leq C\abs{h}^{2-2\alpha},
    \]
    which means that $z$ also solves \eqref{2.4}.
    %is an $H^{2}$ solution to the g-SQG patch equation.
\end{proposition}

\begin{proof}
    Lemma~\ref{LA.7} shows that
    $\sup_{t\in[0,T_{0}]}L(z^{t}) \leq M_0$, with $M_0$ from Lemma \ref{P7.2}.
%    Take a sequence $\seq{z_{\eps_{n}}}_{n=1}^{\infty}$
%    convergent to $z$ given in Lemma~\ref{P7.2}.
    Since $\Delta(z_{\eps_{n}}^{t,\lambda},z_{\eps_{n}}^{t,\lambda'}) > 0$
    for any $(n,t,\lambda,\lambda')\in\bbN\times [0,T_{0}] \times \mathcal{L}^2$ with
    $\lambda\neq\lambda'$, Lemma~\ref{LA.5} shows that each such pair
    $(z_{\eps_{n}}^{t,\lambda},z_{\eps_{n}}^{t,\lambda'})$
    lies in the set $S_{1}\cup S_{2}\cup S_{3}$
    defined there. Since this set is closed in
    $\operatorname{PSC}(\bbR^{2})^{2}$, we see that $(z^{t,\lambda},z^{t,\lambda'})\in S_{1}\cup S_{2}\cup S_{3}$.
    And since $\set{z^{t}}_{t\in[0,T_{0}]}$ is a connected subset of
    $\operatorname{PSC}(\bbR^{2})^{\mathcal{L}}$, Lemma~\ref{LA.5} shows that
    $\Sigma^{\lambda}(z^{t}) = \Sigma^{\lambda}(z^{0})$ for all $(t,\lambda)\in[0,T_{0}]\times\mathcal L$.  Then the last claim in \eqref{111.27} follows from $\sup_{t\in[0,T_{0}]}L(z^{t}) <\infty$.

    Pick any $t,t+h\in[0,T_{0}]$ and $\lambda\in\mathcal{L}$, and
    fix some arclength parametrizations of $z^{t,\lambda}$ and $z^{t+h,\lambda}$, as well as of
    $z_{\eps_{n}}^{t,\lambda}$ and $z_{\eps_{n}}^{t+h,\lambda}$ for each $n\in\bbN$.
    Let ${\phi}_{n}\colon \ell(z_{\eps_{n}}^{t,\lambda})\bbT\to \ell (z_{\eps_{n}}^{t+h,\lambda})\bbT$ be
    the orientation-preserving homeomorphism    obtained by applying Proposition~\ref{P7.1} with $\eps=\eps_n$
    (and $M\coloneqq\max\{30M_0W_{0},M_0\}$ and $I\coloneqq[0,T_{0}]$
    as in Lemma~\ref{P7.2}).
    % to the chosen arclength parametrizations
    %of $z_{\eps_{n}}^{t,\lambda}$ and $z_{\eps_{n}}^{t+h,\lambda}$.
    Finally, let $\tilde{\gamma}_{1,n}(\xi)\coloneqq z_{\eps_{n}}^{t,\lambda}(\ell(z_{\eps_{n}}^{t,\lambda})\xi)$
    and $\tilde{\gamma}_{2,n}(\xi)\coloneqq z_{\eps_{n}}^{t+h,\lambda}(\ell(z_{\eps_{n}}^{t+h,\lambda})\xi)$ be the corresponding
    constant-speed parametrizations, and   
    $\tilde{\phi}_{n}(\xi) \coloneqq\ell(z_{\eps_{n}}^{t+h,\lambda})^{-1}\phi_n(\ell(z_{\eps_{n}}^{t,\lambda})\xi)$ the related endomorphism of $\bbT$.
    
    Then Proposition~\ref{P7.1} and the bounds on $\ell$ in Lemmas~\ref{L3.1} and
    \ref{L3.9} (with $\beta=\frac 12$) show that
    \begin{equation*}
    \sup_{(n,\xi) \in\bbN\times\bbT} \max \left\{ \tilde{\phi}_n' (\xi), \tilde{\phi}_n' (\xi)^{-1} \right\} \le C,
    \end{equation*}
%    \begin{align*}
%        \frac{\ell(z^{t,\lambda})}{\ell(z^{t+h,\lambda})}
%        e^{-C_{\alpha,\theta,W_{0},M}\abs{h}}
%        \leq\,\, & \tilde{\phi}' \leq
%        \frac{\ell(z^{t,\lambda})}{\ell(z^{t+h,\lambda})}
%        e^{C_{\alpha,\theta,W_{0},M}\abs{h}},
%        \\
%        \frac{\ell(z_{\eps_{n}}^{t+h,\lambda})}
%        {\ell(z_{\eps_{n}}^{t,\lambda})}e^{-C_{\alpha,|\theta|,W_{0},M}\abs{h}}
%        \leq (&\tilde{\phi}_{n}^{-1})' \leq
%        \frac{\ell(z_{\eps_{n}}^{t+h,\lambda})}
%        {\ell(z_{\eps_{n}}^{t,\lambda})}e^{C_{\alpha,|\theta|,W_{0},M}\abs{h}},
%    \end{align*}
with $C$ only depending on ${\alpha,|\theta|,W_{0},M_0}$.
    As we showed in the proof of Lemma \ref{P7.2}, there is $R<\infty$ such that
    $\norm{\tilde{\gamma}_{i,n}}_{L^{\infty}}\leq R$ for all $(n,i)\in\bbN\times\{1,2\}$.  This and \eqref{2.1} show that  
    \[
    \norm{\partial_{\xi}\tilde{\gamma}_{1,n}}_{\dot{C}^{0,1/2}(\bbT)}
    \leq \norm{\tilde{\gamma}_{1,n}}_{\dot{H}^{2}(\bbT)}
    = \ell(z_{\eps_{n}}^{t,\lambda})^{3/2}
    \norm{z_{\eps_{n}}^{t,\lambda}}_{\dot{H}^{2}}
    \leq C
    \]
    for all $n\in\bbN$, with $C$ only depending on $\alpha,|\tht|,z^0$, and the same bound holds for
    $\norm{\partial_{\xi}\tilde{\gamma}_{2,n}}_{\dot{C}^{0,1/2}(\bbT)}$. Therefore, by passing to a subsequence if needed, we can assume that
    $\seq{\tilde{\gamma}_{i,n}}_{n=1}^{\infty}$ converges in $C^{1}(\bbT;\bbR^{2})$
    to some $\tilde{\gamma}_{i}\in C^{1}(\bbT;\bbR^{2})$ ($i=1,2$), and that
    $\seq{\tilde{\phi}_{n}}_{n=1}^{\infty}$, $\seq{\tilde{\phi}_{n}^{-1}}_{n=1}^{\infty}$
    converge uniformly to some $\tilde{\phi},\tilde{\psi}\in C^{0,1}(\bbT;\bbT)$,
    respectively. Clearly $\tilde{\psi}=\tilde{\phi}^{-1}$,
    and $\tilde{\phi}_{n}'>0$ shows that $\tilde{\phi}$ is an orientation-preserving homeomorphism.

    Since  $\tilde{\gamma}_{i,n}\to \tilde{\gamma}_{i}$ in $C^1$,
    we see that $\tilde{\gamma}_{1},\tilde{\gamma}_{2}$ are constant-speed
    parametrizations of $z^{t,\lambda},z^{t+h,\lambda}$, respectively.  Moreover, we have
    $\ell(z^{t,\lambda}) = \lim_{n\to\infty}\ell(z_{\eps_{n}}^{t,\lambda})$,
    $\ell(z^{t+h,\lambda}) = \lim_{n\to\infty}\ell(z_{\eps_{n}}^{t+h,\lambda})$, and
    % \eqref{111.14} extends to $\tilde\phi$ in place of $\tilde \phi_n$, and
%    which then show
%    \begin{align*}
%        \frac{\ell(z^{t,\lambda})}{\ell(z^{t+h,\lambda})}
%        e^{-C_{\alpha,\theta,W_{0},M}\abs{h}}
%        \leq \tilde{\phi}' \leq
%        \frac{\ell(z^{t,\lambda})}{\ell(z^{t+h,\lambda})}
%        e^{C_{\alpha,\theta,W_{0},M}\abs{h}}.
%    \end{align*}
%    By Green's theorem, the $C^{1}$-convergence $\tilde{\gamma}_{1,n} \to \tilde{\gamma}_{1}$
%    also shows
    \[
        \abs{\Omega(z^{t,\lambda})} =
        \lim_{n\to\infty}\abs{\Omega(z_{\eps_{n}}^{t,\lambda})}
        = \lim_{n\to\infty}\abs{\Omega(z_{\eps_{n}}^{0,\lambda})}
        = \abs{\Omega(z^{0,\lambda})}.
    \]

For the initially chosen arclength parametrizations of $z^{t,\lambda}$ and $z^{t+h,\lambda}$,
%    Since $\tilde{\gamma}_{1},\tilde{\gamma}_{2}$ are constant-speed
%    parametrizations of $z^{t,\lambda},z^{t+h,\lambda}$, respectively,
    there are $s_{1}\in\ell(z^{t,\lambda})\bbT$ and $s_{2}\in\ell(z^{t+h,\lambda})\bbT$
    such that 
    \begin{equation*}
        z^{t,\lambda}(s) = \tilde{\gamma}_{1}\left(
            \frac{s-s_{1}}{\ell(z^{t,\lambda})}
        \right) \qquad\text{and}\qquad
        z^{t+h,\lambda}(s) = \tilde{\gamma}_{2}\left(
            \frac{s-s_{2}}{\ell(z^{t+h,\lambda})}
        \right),
    \end{equation*}
    and we let 
    %$\phi\colon\ell(z^{t,\lambda})\bbT\to\ell(z^{t+h,\lambda})\bbT$ as
    \[
        \phi(s) \coloneqq \ell(z^{t+h,\lambda})\tilde{\phi}\left(
            \frac{s-s_{1}}{\ell(z^{t,\lambda})}
        \right) + s_{2}.
    \]
    Clearly, $\phi$ then inherits the bound $e^{-C\abs{h}} \leq \phi' \leq e^{C\abs{h}}$ from Proposition~\ref{P7.1}.
%    then $e^{-C_{\alpha,|\theta|,W_{0},M}\abs{h}}\leq
%    \phi'\leq e^{C_{\alpha,|\theta|,W_{0},M}\abs{h}}$ holds,
    From Proposition~\ref{P7.1} and Lemmas~\ref{L4.4}, \ref{L4.3},
    and \ref{L3.9} we see that
    \begin{align*}
        &\norm{z^{t+h,\lambda}\circ\phi
        - z^{t,\lambda}
        - hu(z^{t})\circ z^{t,\lambda}}_{L^{\infty}}
        =\norm{\tilde{\gamma}_{2}\circ\tilde{\phi} - \tilde{\gamma}_{1}
        - hu(z^{t})\circ\tilde{\gamma}_{1}}_{L^{\infty}(\bbT)} \\
        &\quad\quad\leq
        \norm{\tilde{\gamma}_{2,n}\circ\tilde{\phi}_{n} - \tilde{\gamma}_{1,n}
        - hu_{\eps_{n}}(z_{\eps_{n}}^{t})\circ\tilde{\gamma}_{1,n}}_{L^{\infty}(\bbT)}
        \\&\quad\quad\quad\quad\quad
        + \norm{\partial_{\xi}\tilde{\gamma}_{2}}_{L^{\infty}(\bbT)}
        \norm{\tilde{\phi} - \tilde{\phi}_{n}}_{L^{\infty}(\bbT)}
        + \norm{\tilde{\gamma}_{2} - \tilde{\gamma}_{2,n}}_{L^{\infty}(\bbT)}
        + \norm{\tilde{\gamma}_{1} - \tilde{\gamma}_{1,n}}_{L^{\infty}(\bbT)}
        \\&\quad\quad\quad\quad\quad
        + \abs{h}\norm{u(z^{t})\circ\tilde{\gamma}_{1}
        - u(z_{\eps_{n}}^{t})\circ\tilde{\gamma}_{1,n}}_{L^{\infty}(\bbT)}
        + \abs{h}\norm{u(z_{\eps_{n}}^{t}) - u_{\eps_{n}}(z_{\eps_{n}}^{t})}
        _{L^{\infty}(\bbT)}\\
        &\quad\quad\leq
        C\abs{h}^{2-2\alpha}
        + \ell(z^{t+h,\lambda})\norm{\tilde{\phi} - \tilde{\phi}_{n}}_{L^{\infty}(\bbT)}
        + \norm{\tilde{\gamma}_{2} - \tilde{\gamma}_{2,n}}_{L^{\infty}(\bbT)}
        + \norm{\tilde{\gamma}_{1} - \tilde{\gamma}_{1,n}}_{L^{\infty}(\bbT)}
        \\&\quad\quad\quad\quad\quad
        + C\abs{h}\left(
            d_{\mathrm{F}}(z^{t},z_{\eps_{n}}^{t})^{1-2\alpha}
            + \norm{\tilde{\gamma}_{1}
            - \tilde{\gamma}_{1,n}}_{L^{\infty}(\bbT)}^{1-2\alpha}
        \right)
        + C\abs{h}\eps_{n}^{1-2\alpha}
    \end{align*}
     for all large $n$, with $C$ depending only on ${\alpha,|\theta|,W_{0},M_0}$, 
    where the last term  comes from 
    \begin{align*}
        u(z_{\eps_{n}}^{t};x) - u_{\eps_{n}}(z_{\eps_{n}}^{t};x)
        = \sum_{\lambda'\in\mathcal{L}}\theta^{\lambda'}
        \int_{\ell(z_{\eps_{n}}^{t,\lambda'})\bbT}
        \left[ K_{\eps_{n}}(x - z_{\eps_{n}}^{t,\lambda'}(s))
        - K(x - z_{\eps_{n}}^{t,\lambda'}(s)) \right]
        \partial_{s}z_{\eps_{n}}^{t,\lambda'}(s)\,ds
    \end{align*}
    for any $x\in\bbR^{2}$,
    with the integrand vanishing when
    $\abs{x - z_{\eps_{n}}^{t,\lambda'}(s)}>\eps_{n}$ (recall also that $0\le K_{\eps} \leq K$).
    Taking $n\to\infty$ now yields \eqref{111.15}.
\end{proof}

% \begin{proof}[Proof of existence.]
%     Let $z\colon[0,T_{0})\to \operatorname{PSC}(\bbR^{2})^{\mathcal{L}}$
%     be the local uniform limit of a Cauchy sequence provided by Proposition~\ref{P7.2}.
%     Then Proposition~\ref{P7.3} shows that $z$ is an $H^{2}$ patch solution
%     to the g-SQG equation with the initial data $z^{0}$.
% \end{proof}

\begin{proof}[Proof of existence in Theorem \ref{T2.5}]
Proposition \ref{P7.3} shows that the $H^2$ patch solution obtained there can be extended up to some time $T_{z^0}>0$ such that $\lim_{t\to T_{z^0}^-} L(z^t)=\infty$ when $T_{z^0}<\infty$.  It also shows that we have \eqref{111.27} 
%and $\min_{\lambda'\notin\Sigma^\lambda(z^0)} \Delta(z^{t,\lambda},z^{t,\lambda'})>0$  
 for any $(t,\lambda)\in[0,T_{z^0})\times \mathcal{L}$.  It remains to show \eqref{111.29}
 % that $\sup_{t\in[0, T_{z^0})} \norm{z^{t}}_{\dot H^2} = \infty$ 
 when $T_{z^0}<\infty$ and the last claim in Theorem \ref{T2.5}, which we do in Section \ref{S10}.
\end{proof}

%%%%%%%%%%%%%%%%%%%%%%%%%%%%%%%%%%%%%%%%%%%%%%
\section{Uniqueness of solutions to \eqref{2.4}}\label{S8}
%%%%%%%%%%%%%%%%%%%%%%%%%%%%%%%%%%%%%%%%%%%%%%

We will establish uniqueness of solutions to \eqref{2.4} by obtaining an appropriate estimate on the rate of divergence
of pairs of solutions.
Note that Lemma~\ref{L4.4} only yields a
H\"{o}lder-type estimate with exponent $1-2\alpha$, which is not sufficient to conclude uniqueness (it only provides an $O(t^{\frac{1}{2\alpha}})$ bound on the distance of solutions starting from
the same initial data).
We instead need the rate of divergence to be bounded by the distance of the solutions (or only a very slightly worse bound than that), but it seems unlikely that such an estimate can be obtained in terms of
the Fr\'{e}chet metric (with which we work here) or other $L^{\infty}$-type metrics.
We will therefore define a weaker, $L^{2}$-type notion of distance, for  which we will then derive the desired (Lipschitz-type) estimate, and thus at most exponential-in-time divergence of solutions.

For any rectifiable closed curves $\gamma_{1},\gamma_{2}$ let
\begin{align*}
    D(\gamma_{1},\gamma_{2}) &\coloneqq
    \int_{\ell(\gamma_{1})\bbT}d(\gamma_{1}(s),\operatorname{im}(\gamma_{2}))^{2}\,ds
\end{align*}
where in the integrand we use an arbitrary arclength parametrization of $\gamma_{1}$.
We clearly have 
\[
D(\gamma_{1},\gamma_{2})\leq
\ell(\gamma_{1})d_{\mathrm{F}}(\gamma_{1},\gamma_{2})^{2},
\]
but $D^{1/2}$ fails to be a metric.  Indeed,  $D(\gamma_{1},\gamma_{2})=0$ only implies
$\operatorname{im}(\gamma_{1})\subseteq\operatorname{im}(\gamma_{2})$,
 $D$ is not symmetric, and  the triangle inequality  fails for both
 $D^{1/2}$ and its symmetrization.
%({\color{red}though I do not have any counterexample}).
Nevertheless, Lemma~\ref{L8.1} below shows that
$D$ provides an upper bound on $d_{\mathrm{F}}$ on appropriate subsets of $\operatorname{PSC}(\bbR^{2})$, which will suffice for our purposes.

The key observation leading to a Lipschitz-type estimate on the growth of $D(z_{1}^{t,\lambda},z_{2}^{t,\lambda})$ for 
 two solutions $z_{1},z_{2}\in C([0,T];\operatorname{PSC}(\bbR^{2})^{\mathcal{L}})$ is that
if we can find a ($(t,\lambda)$-dependent) orientation-preserving  homeomorphism
$\phi\colon\ell(z_{1}^{t,\lambda})\bbT\to\ell(z_{2}^{t,\lambda})\bbT$ such that
\begin{align*}
    \abs{z_{1}^{t,\lambda} - z_{2}^{t,\lambda}\circ\phi} (s)
    = d(z_{1}^{t,\lambda}(s),\operatorname{im}(z_{2}^{t,\lambda}))
\end{align*}
holds for all $s\in \ell(z_{1}^{t,\lambda})\bbT$,
then $z_{1}^{t,\lambda}(s) - z_{2}^{t,\lambda}(\phi(s))$ must be parallel to the
normal vector to $z_{2}^{t,\lambda}$ at $\phi(s)$ (for each $s$). As a result,
the tangential component (with respect to $z_{2}^{t,\lambda}$) of the velocity difference at $z_{1}^{t,\lambda}(s)$ and $ z_{2}^{t,\lambda}(\phi(s))$ does not contribute
to the instantaneous departure velocity  of these points, while Lemma~\ref{L3.3} will provide us improved control
on the normal component.
It turns out, however, that the latter will only suffice to obtain the desired estimate if we ensure that the obtained $\phi$ is not too irregular, specifically, that it does not deviate too much from being constant-speed.

Existence of such $\phi$ is obtained via a short argument in the proof of \cite[Lemma 4.11]{KisYaoZla}, but only when the curves $z_{i}^{t,\lambda}$ are
% which is utilized in their uniqueness argument.
$C^{1,1}$.  However, our curves are only $H^{2}$ (and hence $C^{1,1/2}$), and it in fact turns out that $\phi$ as above {\it does not exist in general}
for $C^{1,\beta}$ curves  when $\beta<1$, even if they are assumed to be arbitrarily close to each other.
Indeed, if $d>0$ is arbitrary and $\phi(s)$ is the point on the curve $\{(x_1, \abs{x_1}^{1+\beta} - d)\,\,:\,\, x_1\in\bbR\}$  closest to $(s,0)$ for each $s\in \bbR$, then $\phi$ always has a discontinuity at $s=0$.

To solve this issue, we will replace $z_{2}^{t,\lambda}$ above by an approximation, which will be carefully chosen to be close enough to $z_{2}^{t,\lambda}$ but not too much so, in order for its $C^{1,1}$-norm
to not be too large.  We will then use that approximation to find a reparametrization that will yield a good enough ``alignment'' of the original curves.  This idea is at the core of the proof of the following lemma, which we postpone to Appendix~\ref{S9}.

\begin{lemma}\label{L8.1}
    For any $R_{1},R_{2}\in(0,\infty)$, there is    $\delta> 0$ such that
    for any $\gamma_{1},\gamma_{2}\in\operatorname{PSC}(\bbR^{2})$ satisfying
    $d_{\mathrm{F}}(\gamma_{1},\gamma_{2})\leq\delta$, $\ell(\gamma_{1}) \leq R_{1}$, and
    \[
        \max\set{\norm{\gamma_{1}}_{\dot{H}^{2}}^{2},
        \frac{1}{\Delta_{\norm{\gamma_{2}}_{\dot{H}^{2}}^{-2}}(\gamma_{2})}}
        \leq R_{2}^{2},
    \]
    and for any  arclength parametrizations of $\gamma_{1},\gamma_{2}$,
    there is a homeomorphism
    $\phi\colon\ell(\gamma_{1})\bbT\to\ell(\gamma_{2})\bbT$ with the following properties (all norms are with respect to these parametrizations):
    \begin{enumerate}
        \item[(a)] $\norm{\gamma_{1} - \gamma_{2}\circ\phi}_{L^{\infty}}
        \leq 2d_{\mathrm{F}}(\gamma_{1},\gamma_{2})
        \leq 12R_{1}^{3/10}R_{2}^{2/5}\norm{\gamma_{1} - \gamma_{2}\circ\phi}_{L^{2}}^{3/5}$,

        \item[(b)] $\norm{\gamma_{1} - \gamma_{2}\circ\phi}_{L^{2}}
        \leq 2D(\gamma_{1},\gamma_{2})^{1/2}$,

        \item[(c)] $\phi$ is Lipschitz continuous and
        $\frac{1}{3}\leq \phi' (s) \leq 3$ for almost all $s\in \ell(\gamma_{1})\bbT$,

        \item[(d)] $\norm{\partial_{s}\gamma_{1}
        - \partial_{s}(\gamma_{2}\circ\phi)}_{L^{2}}
        \leq 12R_{1}^{1/2}R_{2}^{2/3}d_{\mathrm{F}}(\gamma_{1},\gamma_{2})^{1/3}$,

        \item[(e)] $\norm{(\gamma_{1} - \gamma_{2}\circ\phi)\cdot
        (\partial_{s}\gamma_{2}\circ \phi)}_{L^{2}}
        \leq 10^{5}R_{1}^{9/10}R_{2}^{21/5}
        \norm{\gamma_{1} - \gamma_{2}\circ\phi}_{L^{2}}^{9/5}$, 

        \item[(f)] $\abs{s' - \phi(s)}\leq
        3d(\gamma_{1}(s),\operatorname{im}(\gamma_{2}))
        + 342R_{2}^{2}\,d_{\mathrm{F}}(\gamma_{1},\gamma_{2})^{2}$
         for any $(s,s')\in\ell(\gamma_{1})\bbT\times \ell(\gamma_{2})\bbT$ such that
        $\abs{\gamma_{1}(s) - \gamma_{2}(s')}
        = d(\gamma_{1}(s),\operatorname{im}(\gamma_{2}))$.
    \end{enumerate}
%    In the above, all norms are with respect to the given parametrization.
\end{lemma}

{\it Remark.}   Note that \eqref{2.2} yields $\norm{\gamma_{2}}_{\dot{H}^{2}} \leq R_{2}$.
\smallskip

The second ingredient in our uniqueness proof is a simple technical lemma that shows that
$\phi\colon\ell(\gamma_{1})\bbT\to\ell(\gamma_{2})\bbT$ such that
$\abs{\gamma_{1} - \gamma_{2}\circ\phi}(s)
= d(\gamma_{1}(s),\operatorname{im}(\gamma_{2}))$ for all $s\in \ell(\gamma_{1})\bbT$ does exist, although it need not be continuous.

\begin{lemma}\label{L8.2}
    Let $X$ be a measurable space, $Y$ a nonempty compact metric space equipped with
    the Borel $\sigma$-algebra, and $f\colon X\times Y\to[0,\infty)$ a
    measurable function such that $f(x,\,\cdot\,)$ is continuous for each $x\in X$.
    Then there is a measurable function $\psi\colon X\to Y$
    such that $f(x,\psi(x)) = \inf_{y\in Y}f(x,y)$ holds for each $x\in X$.
\end{lemma}

%({\color{red}Maybe a consequence of
%Kuratowski–Ryll-Nardzewski measurable selection theorem.})

\begin{proof}
    For each $k\in\bbN$ pick some finite set $\set{y_{k,n}}_{n=1}^{N_{k}}\subseteq Y$
    such that $Y = \bigcup_{n=1}^{N_{k}}B_{2^{-k}}(y_{k,n})$. Then for each $x\in X$ define $\nu_k(x)\in \set{1,\dots ,N_{k}}$ for $k=1,2,\dots$ (in that order) so that $\nu_k(x)$ is the smallest number such that
    \begin{align*}
        \inf_{y\in Y}f(x,y) =
        \inf_{y\in\bigcap_{j=1}^{k}B_{2^{-j}}(y_{j,\nu_{j}(x)})}f(x,y)
    \end{align*}
    (with $\inf\emptyset\coloneqq\infty$).
    Since $Y$ is separable and each $f(x,\,\cdot\,)$ is continuous,
    the function 
    \[
    g_{n_1,\dots,n_k}(x)\coloneqq \inf_{y\in\bigcap_{j=1}^{k}
    B_{2^{-j}}(y_{j,n_{j}})}f(x,y)
    \]
     is measurable for each
    $(n_1,\dots,n_k)\in\prod_{j=1}^{k}\set{1,\dots,N_{j}}$
    because it is the infimum of countably many measurable functions
    $\set{f(\,\cdot\,,y)}_{y\in S}$, with $S$ being any countable dense subset of
    $\bigcap_{j=1}^{k}B_{2^{-j}}(y_{j,n_{j}})$.
    Then $\nu_{k}\colon X\to\set{1,\dots,N_{k}}$
    is also measurable due to
    \[
        \nu_{k}^{-1}\left(n \right)
        = \set{x\in X\colon
        g_{\nu_{1}(x),\dots,\nu_{k-1}(x),n}(x) = \inf_{y\in Y}f(x,y)}
        \setminus \bigcup_{m=1}^{n-1}
        \nu_{k}^{-1}\left(m\right)
    \]
    for each $n\in\set{1,\dots,N_{k}}$, where
    $g_{\nu_{1}(\cdot),\dots,\nu_{k-1}(\cdot),n}(\cdot)$ is measurable
    since it is the composition of  functions $x\mapsto(\nu_{1}(x),\dots,\nu_{k-1}(x),x)
    \in\prod_{j=1}^{k-1}\set{1,\dots,N_{j}}\times X$ and
    $(n_{1},\dots,n_{k-1},x)\mapsto g_{n_{1},\dots,n_{k-1},n}(x)$.
    Hence, $\psi_{k}(x)\coloneqq y_{k,\nu_{k}(x)}$ is measurable as well.

    For each $(k,x)\in\bbN\times X$ we have
    \begin{align*}
        B_{2^{-k}}(\psi_{k}(x)) \cap B_{2^{-(k+1)}}(\psi_{k+1}(x)) \neq \emptyset
    \end{align*}
    because otherwise we would have $\inf_{y\in Y}f(x,y)
    = \inf_{y\in\bigcap_{j=1}^{k+1}B_{2^{-j}}(y_{j,\nu_{j}(x)})}f(x,y) = \infty$,
    which contradicts $Y\neq\emptyset$ and $f$ being finite.
    Hence $d(\psi_{k}(x),\psi_{k+1}(x)) \leq 2^{-(k-1)}$, so $\seq{\psi_{k}(x)}_{k=1}^{\infty}$ is a Cauchy sequence.
    Since $Y$ is complete, the (uniform) limit $\psi\coloneqq\lim_{k\to\infty} \psi_{k}$ is
     measurable.
     
    It remains to show $f(x,\psi(x)) = \inf_{y\in Y}f(x,y)$ for
    any  $x\in X$, so fix $x$ and any $\eps>0$. Continuity of
    $f(x,\,\cdot\,)$ yields $k\in\bbN$ such that
    $\abs{f(x,\psi(x)) - f(x,y)} \leq \eps$ whenever $d(\psi(x),y)\leq{2^{-(k-3)}}$.
    Take $y_{0}\in \bigcap_{j=1}^{k}B_{2^{-j}}(y_{j,\nu_{j}(x)})$ such that
    \begin{align*}
        \inf_{y\in Y}f(x,y) =
        \inf_{y\in\bigcap_{j=1}^{k}B_{2^{-j}}(y_{j,\nu_{j}(x)})}f(x,y)
        \geq f(x,y_{0}) - \eps.
    \end{align*}
    Then from $d(\psi(x),\psi_{k}(x)) \leq 2^{-(k-2)}$ and
    $y_{0} \in B_{2^{-k}}(y_{k,\nu_{k}(x)}) = B_{2^{-k}}(\psi_{k}(x))$
    we obtain $d(\psi(x),y_{0}) \leq 2^{-(k-3)}$, and so
    \begin{align*}
        f(x,\psi(x)) \leq f(x,y_{0}) + \eps
        \leq \inf_{y\in Y}f(x,y) + 2\eps.
    \end{align*}
    Since $\eps>0$ was arbitrary, the proof is finished.
\end{proof}

The last  ingredient in the uniqueness proof shows that if an $H^2$ patch solution to \eqref{2.4} has no patch boundary crossings initially, then 
%each pair of nested resp. disjoint patches must remain as such, therefore 
no such crossing can develop later.

\begin{proposition}\label{P8.4}
    Let $w\colon [0,T] \to \operatorname{PSC}(\bbR^{2})^{\mathcal{L}}$ be
    an $H^{2}$ patch solution to \eqref{2.4} with $d_{\mathrm{H}}$ in place of $d_{\mathrm{F}}$
    that satisfies
    \beq\lb{111.42}
        \sup_{t\in[0,T]}\max_{\lambda\in\mathcal{L}}
        \max\set{
            \ell(w^{t,\lambda}),
            {\Delta_{\norm{w^{t,\lambda}}_{\dot{H}^{2}}^{-2}}(w^{t,\lambda})}^{-1}
        } < \infty,
    \eeq
    and for each $(t,\lambda)\in[0,T]\times\mathcal{L}$ let
    \begin{align*}
        S_{1}^{t,\lambda}&\coloneqq \set{\lambda'\in\mathcal{L}\,\,:\,\,
        \Omega(w^{t,\lambda}) \subseteq \Omega(w^{t,\lambda'})}, \\
        S_{2}^{t,\lambda}&\coloneqq \set{\lambda'\in\mathcal{L}\,\,:\,\,
        \Omega(w^{t,\lambda}) \supseteq \Omega(w^{t,\lambda'})}, \\
        S_{3}^{t,\lambda}&\coloneqq \set{\lambda'\in\mathcal{L}\,\,:\,\,
        \Omega(w^{t,\lambda}) \cap \Omega(w^{t,\lambda'}) = \emptyset}.
    \end{align*}
    If $S_{1}^{0,\lambda}\cup S_{2}^{0,\lambda}\cup S_{3}^{0,\lambda} = \mathcal{L}$
    for each $\lambda\in\mathcal{L}$, then $S_{i}^{t,\lambda} = S_{i}^{0,\lambda}$ for each
    $(t,\lambda,i)\in[0,T]\times\mathcal{L}\times\set{1,2,3}$. 
%    In particular,
%    for any $(t,\lambda,\lambda')\in[0,T]\times\mathcal{L}^{2}$,
%    $w^{t,\lambda}$ and $w^{t,\lambda'}$ do not cross transversally.
\end{proposition}

\begin{proof}
    For each $(t,\lambda)\in[0,T]\times \mathcal{L}$, fix an arclength parametrization of $w^{t,\lambda}$
 and let $m(t)\coloneqq\max\{m_{1}(t),m_{2}(t),m_{3}(t)\}$, where   (with $\max\emptyset\coloneqq 0$)
    \begin{align*}
        m_{1}(t) &\coloneqq \max_{\lambda\in\mathcal{L}}\max_{\lambda'\in S_{1}^{0,\lambda}}
        \max_{s\in\ell(w^{t,\lambda})\bbT}
        d(w^{t,\lambda}(s), \Omega(w^{t,\lambda'})), \\
        m_{2}(t) &\coloneqq \max_{\lambda\in\mathcal{L}}\max_{\lambda'\in S_{2}^{0,\lambda}}
        \max_{s\in\ell(w^{t,\lambda})\bbT}
        d(w^{t,\lambda}(s), \bbR^{2}\setminus\Omega(w^{t,\lambda'})), \\
        m_{3}(t) &\coloneqq \max_{\lambda\in\mathcal{L}}\max_{\lambda'\in S_{3}^{0,\lambda}}
        \max_{s\in\ell(w^{t,\lambda})\bbT}
        d(w^{t,\lambda}(s), \bbR^{2}\setminus\Omega(w^{t,\lambda'})).
    \end{align*}
    Note that clearly $m(0)=0$.

        We first claim that for any fixed $\lambda,\lambda'\in\mathcal{L}$ and $A_t$ being either $\Omega(w^{t,\lambda'})$ or  $\bbR^{2}\setminus\Omega(w^{t,\lambda'})$ (same choice for all $t$), the function
    $f(t)\coloneqq \max_{s\in\ell(w^{t,\lambda})\bbT}d(w^{t,\lambda}(s),A_t)$
    is continuous on $[0,T]$. 
    Both cases are treated identically so we will only assume that $A_t=\Omega(w^{t,\lambda'})$.
    
    For any $t,\tau\in[0,T]$ we clearly have
    % and Lemma~\ref{L4.1} show that there is $\delta > 0$ such that
    % \[
    %     \delta \leq \min\left\{
    %         \frac{1}{4}\left(\frac{r^{2\alpha}}{C_{1}}\right)^{1/(1-2\alpha)},
    %         \frac{1}{\norm{u(w^{t})}_{L^{\infty}} + r}
    %         \left(\frac{r}{C_{1}}\right)^{1/(1-2\alpha)}
    %     \right\}
    % \]
    % and
    \beq\lb{111.41}
        f(\tau) \geq f(t)
        - \max_{s\in\ell(w^{t,\lambda})\bbT}d\left(
            w^{t,\lambda}(s) ,
            \operatorname{im}(w^{\tau,\lambda})
        \right)  
        - \sup_{x\in\Omega(w^{\tau,\lambda'})}d\left(
            x, \Omega(w^{t,\lambda'})
        \right).
    \eeq
    % If we fix $t$ and take $\tau\to t$, the continuity of
    % $w^{\lambda'}\colon[0,T]\to\operatorname{PSC}(\bbR^{2})$ with respect to $d_{\mathrm{H}}$
    % shows that the second and the third terms on the right-hand side converge to $0$,
    % which shows that $f$ is lower semicontinuous.
    If we fix $\tau$ and take $t\to \tau$, the second term on the right-hand side converges to $0$ by the hypothesis (note also that $\sup_{t\in[0,T]}\norm{u(w^{t})}_{L^\infty}<\infty$ by the hypothesis and Lemma \ref{L4.1}).
    Hence if we show the same about the third term, we will conclude that $f$ is upper semicontinuous.
    But this indeed holds because $w^{\lambda'}\colon[0,T]\to\operatorname{PSC}(\bbR^{2})$ is continuous (with respect to $d_{\mathrm{F}}$; note that continuity with respect to $d_{\mathrm{H}}$ does not suffice here).
    To see this, note that   Lemma~\ref{LA.2} shows that $\{w^{t,\lambda'}\,:\,t\in[0,T]\}$ is contained in some compact subset $X$ of  $(\operatorname{CC}(\bbR^{2}),d_{\mathrm{F}})$, where the hypotheses of the lemma hold by continuity of $w^{\lambda'}$  with respect to $d_{\mathrm{H}}$,  \eqref{111.42}, and \eqref{2.2}.  Moreover, \eqref{111.42}, the remark after Lemma~\ref{LA.6}, and   Lemmas \ref{L3.8} and \ref{LA.4} show that  $X$ can in fact be chosen to be a compact subset of    $(\operatorname{PSC}(\bbR^{2}),d_{\mathrm{F}})$.
    Since $\operatorname{Id}\colon(X,d_{\mathrm{F}})\to(X,d_{\mathrm{H}})$ is then a homeomorphism
    (although $d_{\mathrm{H}}$ is only a pseudometric on $\operatorname{CC}(\bbR^{2})$,
    it is a metric on $\operatorname{PSC}(\bbR^{2})$),
    continuity of $w^{\lambda'}$ with respect to $d_{\mathrm{H}}$ also yields
    its continuity with respect to $d_{\mathrm{F}}$.

    If we instead fix $t$ and take $\tau\to t$ in \eqref{111.41}, the second and third terms on the right again converge to zero as above (this time even continuity of $w^{\lambda'}$ with respect to $d_{\mathrm{H}}$ suffices), so $f$ is also lower semicontinuous. Therefore $f$, as well as $m_{i}$ ($i=1,2,3$) and $m$, is continuous.

    Next, Lemma~\ref{L4.4} (with $\beta\coloneqq\frac{1}{2}$)    and \eqref{2.2} show that there is $C_1$ such that
    \beq\lb{111.44}
        \abs{u(w^{t};x) - u(w^{t};y)} \leq C_{1}\abs{x - y}^{1-2\alpha}
    \eeq
     for all $(t,x,y)\in[0,T]\times\bbR^{4}$.
      We now fix any $t\in[0,T]$ and
will     derive a quantitative upper bound on $f(\tau)$ for $\tau$ near $t$, 
    which we will then use to   estimate $\partial_{t}^{+}m(t)$.
    Fix any $r>0$ and take
    $\delta\in\left(0,\frac{1}{4}(r^{2\alpha}/C_{1})^{1/(1-2\alpha)}\right]$ such that 
    for each $\tau\in[0,T]\cap [t-\delta,t+\delta]$ we have
    \beq\lb{111.43}
        d_{\mathrm{H}}\left(
            w^{\tau}, X_{u(w^{t})}^{\tau - t}[w^{t}]
        \right) \leq r\abs{\tau - t}.
    \eeq
%     (it exists because  $w$ satisfies \eqref{2.4}     with $d_{\mathrm{H}}$ in place of $d_{\mathrm{F}}$).
    Fix any such $\tau$ and note that
    \begin{align}
        f(\tau) &\leq \max_{s\in\ell(w^{\tau,\lambda})\bbT}
        d\left(
            w^{\tau,\lambda}(s), \operatorname{im}\left(
                X_{u(w^{t})}^{\tau - t}[w^{t,\lambda}]
            \right)
        \right) \notag
        \\&\quad\quad\quad
        + \max_{s\in\ell(w^{t,\lambda})\bbT}\inf_{x\in\Omega(w^{t,\lambda'})}
        \abs{
            \left(w^{t,\lambda}(s) + (\tau - t)u(w^{t};w^{t,\lambda}(s))\right)
            - \left(x + (\tau - t)u(w^{t};x)\right)
        }  \notag
        \\&\quad\quad\quad 
        + \sup_{x\in\Omega(w^{t,\lambda'})}
        d\left(
            x + (\tau - t)u(w^{t};x),
            \Omega(w^{\tau,\lambda'})
        \right). \lb{111.46}
    \end{align}
    The first term on the right-hand side is bounded above by $r\abs{\tau - t}$, while the second term is bounded by  $f(t) + C_{1}\abs{\tau - t}f(t)^{1-2\alpha}$.
We will now show that the third term is bounded by $6r\abs{\tau - t}$, which yields (for each $\tau\in[0,T]\cap [t-\delta,t+\delta]$)
      \begin{equation}\label{222.1}
      %\begin{aligned}
        f(\tau)
%        &\leq
%        \max_{s\in\ell(w^{t,\lambda})\bbT}\inf_{x\in\Omega(w^{t,\lambda'})}
%        \abs{
%            \left(w^{t,\lambda}(s) + (\tau - t)u(w^{t};w^{t,\lambda}(s))\right)
%            - \left(x + (\tau - t)u(w^{t};x)\right)
%        }
%        \\&\quad\quad\quad
%        + 7r\abs{\tau - t} \\
        \leq f(t)  +  (C_{1}f(t)^{1-2\alpha} + 7r)\abs{\tau - t}.
    %\end{aligned}
    \end{equation}

    Fix any $x\in\Omega(w^{t,\lambda'})$ and find
    $y\in\operatorname{im}(w^{t,\lambda'})$ such that
    $\abs{x - y} = d(x,\operatorname{im}(w^{t,\lambda'}))$.
    If $\abs{x - y} \leq 4r\abs{\tau - t}$, then by  \eqref{111.44} and \eqref{111.43} we have
    \begin{align*}
        d\left(x + (\tau - t)u(w^{t};x), \operatorname{im}(w^{\tau,\lambda'})\right)
        &\leq \abs{x - y} + C_{1}\abs{\tau - t}\abs{x - y}^{1-2\alpha}
        + r\abs{\tau - t} \\
        &\leq (5r + C_{1}    (4r\delta)^{1-2\alpha} ) \abs{\tau - t}\\
        &\leq 6r\abs{\tau - t}.
    \end{align*}
If instead $\abs{x - y} > 4r\abs{\tau - t}$, then
%    the definitions of $C_{1}$ and $\delta$ show that 
for any $\tau'$ between $t$ and $\tau$ we similarly have
    \begin{align*}
        d\left(x + (\tau' - t)u(w^{t};x), \operatorname{im}(w^{\tau',\lambda'})\right)
%        &\geq \abs{x - y}
%        - C_1 \abs{\tau' - t} \abs{x - y}^{1-2\alpha} %\abs{u(w^{t};x) - u(w^{t};y)}
%        \\&\quad\quad
%        - d\left(
%            y + (\tau' - t)u(w^{t};y), \operatorname{im}(w^{\tau',\lambda'})
%        \right) \\
        &\geq \abs{x - y} - C_{1}\abs{\tau' - t}\abs{x - y}^{1-2\alpha}
        - r\abs{\tau' - t} > 0. %2r\abs{\tau - t}.
%        &\geq \frac{\abs{x - y}}{2} > 0,
    \end{align*}
%   where the third inequality holds by
%    \begin{equation}\label{222.4}
%        C_{1}\abs{\tau - t}
%        \leq \frac{r^{2\alpha}\abs{\tau - t}}{(4\delta)^{1-2\alpha}}
%        \leq \left(\frac{\abs{x - y}}{4}\right)^{2\alpha}
%        \frac{\abs{\tau - t}^{1-2\alpha}}{(4\delta)^{1-2\alpha}}
%        \leq \frac{\abs{x - y}^{2\alpha}}{4}.
%    \end{equation}
    Hence the curve $\gamma^{\tau'}\coloneqq w^{\tau',\lambda'} - (\tau' - t)u(w^{t};x)
    \in \operatorname{CC}(\bbR^{2})$ does not contain $x$ for any $\tau'$
    between $t$ and $\tau$.
    % Note that the winding number with respect to $x$ is a continuous function of such curves
    % (with respect to $d_{\mathrm{F}}$). Note also that
    % $w^{\lambda'}\colon[0,T]\to\operatorname{PSC}(\bbR^{2})$
    % is continuous with respect to $d_{\mathrm{F}}$, thus
    % along the path $\left(\tau'\mapsto w^{\tau',\lambda'} - (\tau' - t)u(w^{t};x)\right)$
    % in $\operatorname{CC}(\bbR^{2})$, the winding number with respect to $x$ is constant,
    % which shows $x + (\tau - t)u(w^{t};x) \in \Omega(w^{\tau,\lambda'})$.
    %
    % To see why $w^{\lambda'}$ is continuous with respect to $d_{\mathrm{F}}$,
    % note that the continuity of $w^{\lambda'}$ with respect to $d_{\mathrm{H}}$,
    % Lemmas \ref{L3.9} and \ref{LA.2}, and the remark after Lemma~\ref{LA.6} show that
    % $\{w^{t,\lambda'}\,:\,t\in[0,T]\}$ is contained in some compact subset $X$ of
    % $(\operatorname{PSC}(\bbR^{2}),d_{\mathrm{F}})$.
    % Since $\operatorname{Id}\colon(X,d_{\mathrm{F}})\to(X,d_{\mathrm{H}})$ is a homeomorphism
    % (although $d_{\mathrm{H}}$ is only a pseudometric on $\operatorname{CC}(\bbR^{2})$,
    % it is a metric on $\operatorname{PSC}(\bbR^{2})$),
    % continuity of $w^{\lambda'}$ with respect to $d_{\mathrm{H}}$ also yields
    % its continuity with respect to $d_{\mathrm{F}}$.
    Since $x\in\Omega(w^{t,\lambda'})$ and $\gamma$ is continuous in $\tau'$ with respect to $d_F$ (because  $w^{\lambda'}$ is), 
    %the winding number with respect to $x$ is a continuous function of such curves
    %(with respect to $d_{\mathrm{F}}$), it must be constant along the path
    %$\left(\tau'\mapsto w^{\tau',\lambda'} - (\tau' - t)u(w^{t};x)\right)$
   % in $\operatorname{CC}(\bbR^{2})$.
    %This shows 
    we see that $x + (\tau - t)u(w^{t};x) \in \Omega(w^{\tau,\lambda'})$. Thus $d\left(  x + (\tau - t)u(w^{t};x), \Omega(w^{\tau,\lambda'})    \right)=0$ holds in this case, and \eqref{222.1} follows.

    Next we estimate $\partial_{t}^{+}m(t)$ in terms of $m(t)$.
    Fix any $t\in[0,T)$ and take any decreasing sequence $\set{t_{n}}_{n=1}^{\infty}$
    in $(t, T]$ converging to $t$ such that
    \[
        \partial_{t}^{+}m(t) = \lim_{n\to\infty}\frac{m(t_{n}) - m(t)}{t_{n} - t}.
    \]
    By passing to a subsequence if needed, we can assume that there is
    $(\lambda,\lambda',i)\in\mathcal{L}^{2}\times\set{1,2,3}$ such that
    $\lambda'\in S_{i}^{0,\lambda}$ and for each $n$ we have
    \beq\lb{111.45}
        m(t_{n}) = m_{i}(t_{n}) = \max_{s\in\ell(w^{t_{n},\lambda})\bbT}
        d(w^{t_{n},\lambda}(s), A_{t_n}),
    \eeq
    where
    \[
        A_{\tau} \coloneqq \begin{cases}
            \overline{\Omega(w^{\tau,\lambda'})}
            & \textrm{if $i=1$}, \\
            \bbR^{2} \setminus\Omega(w^{\tau,\lambda'})
            & \textrm{if $i\in\{2,3\}$}.
        \end{cases}
    \]
% Let also
%     \[
%        A \coloneqq \begin{cases}
%            \overline{\Omega(w^{t,\lambda'})}
%            & \textrm{if $i=1$}, \\
%            \bbR^{2} \setminus\Omega(w^{t,\lambda'})
%            & \textrm{if $i\in\{2,3\}$}.
%        \end{cases}
%    \]
% Define $A$ as $A_{n}$ above with $t$ in place of $t_{n}$.

    For each $n$, find $s_{n}\in\ell(w^{t,\lambda})\bbT$     such that 
    \begin{equation} \lb{111.47} \begin{split}
        &\min_{x\in A_t}\abs{
            w^{t,\lambda}(s_{n}) + (t_{n} - t)u(w^{t};w^{t,\lambda}(s_{n}))
            - \left(x + (t_{n} - t)u(w^{t};x)\right)
        }
        \\&\quad\quad=
        \max_{s\in\ell(w^{t,\lambda})\bbT}
        \min_{x\in A_t}\abs{
            w^{t,\lambda}(s) + (t_{n} - t)u(w^{t};w^{t,\lambda}(s))
            - \left(x + (t_{n} - t)u(w^{t};x)\right)
        }
    \end{split}\end{equation}
    and then $x_{n}\in  A_t$ such that $\abs{w^{t,\lambda}(s_{n}) - x_{n}}=d(w^{t,\lambda}(s_{n}),A_t) $.
    %Since $\set{x_{n}}_{n=1}^{\infty}$ is bounded, 
    By passing to a further subsequence, we can assume that $\exists\lim_{n\to\infty} (s_{n},x_{n}) = (s,x)\in \ell(w^{t,\lambda})\bbT\times A_t$.   
       Now take any $r>0$ and note that \eqref{222.1} holds with $(m,t_n)$ in place of $(f,\tau)$ for all large enough $n$ (due to \eqref{111.45}).       In fact, the definition of $s_n$ shows that \eqref{111.47} is the second term on the right-hand side of \eqref{111.46}, which is than no more than $|v_n|$, where
          \[
        v_{n} \coloneqq
        \left(w^{t,\lambda}(s_{n}) + (t_{n} - t)u(w^{t};w^{t,\lambda}(s_{n}))\right)
        - \left(x_{n} + (t_{n} - t)u(w^{t};x_{n})\right).
    \]
    Since the sum of the other two terms on the right side of \eqref{111.46} was shown to be bounded above by $7r|t_n-t|$,   we thus obtain
    %(and $A$ in place of $\Omega(w^{t,\lambda'}))$),  shows
    \begin{equation}\label{222.3}\begin{aligned}
        m(t_{n}) &\leq \abs{v_{n}} + 7r(t_{n} - t) \\
        &\leq \abs{w^{t,\lambda}(s_{n}) - x_{n}}
        + \left(C_{1}\abs{w^{t,\lambda}(s_{n}) - x_{n}}^{1-2\alpha}
        + 7r \right) (t_{n} - t) \\
        &\leq m(t) + \left(C_{1}m(t)^{1-2\alpha} + 7r \right) (t_{n} - t)
    \end{aligned}\end{equation}
    for all large $n$, where in the last inequality we  used  $\abs{w^{t,\lambda}(s_{n}) - x_{n}} \le m(t)$, which follows from the definition of $x_n$.  After taking $n\to\infty$, continuity of $m$ shows that
    \[
        m(t) = \abs{w^{t,\lambda}(s) - x} = d(w^{t,\lambda}(s), A_t).
    \]    
    
    If $m(t) = 0$, then \eqref{222.3} shows that $\partial_{t}^{+}m(t) \leq 7r$.
     If $m(t) > 0$, then $v_{n}\neq 0$ for large enough $n$ by \eqref{222.3} and continuity of $m$, 
%    (since $\norm{u(w^{t})}_{L^{\infty}} < \infty$ by Lemma~\ref{L4.1}),
    and then \eqref{222.3} and $\abs{w^{t,\lambda}(s_{n}) - x_{n}} \le m(t)$ also show that
    \begin{align*}
        m(t_{n}) - m(t) &\leq \abs{v_{n}} - \abs{w^{t,\lambda}(s_{n}) - x_{n}} + 7r(t_{n} - t) \\
        &\leq \frac{v_{n}}{\abs{v_{n}}}\cdot \left(
            v_{n} - \left(w^{t,\lambda}(s_{n}) - x_{n}\right)
        \right) + 7r(t_{n} - t).
    \end{align*}
    Dividing both sides by $t_{n} - t$ and then taking $n\to\infty$ yields
    \[
        \partial_{t}^{+}m(t) \leq
        \frac{w^{t,\lambda}(s) - x}{\abs{w^{t,\lambda}(s) - x}}
        \cdot \left[u(w^{t};w^{t,\lambda}(s)) - u(w^{t};x)\right] + 7r.
    \]
    Since $m(t) = d(w^{t,\lambda}(s),A_t) > 0$, we must have
    $x\in\partial A_t = \operatorname{im}(w^{t,\lambda'})$ and we let $s'\in\ell(w^{t,\lambda'})\bbT$ be such that $w^{t,\lambda'}(s') = x$.  Since
    $\abs{w^{t,\lambda}(s) - x} = d(w^{t,\lambda}(s),\operatorname{im}(w^{t,\lambda'}))$,
    $w^{t,\lambda}(s) - x$ must be parallel to $\partial_{s}w^{t,\lambda'}(s')^{\perp}$.
    Since $r>0$ was arbitrary, Lemma~\ref{L4.7} finally shows that
    \begin{align*}
        \partial_{t}^{+}m(t) &\leq C\left(m(t) + \Xi_t m(t)^{1-2\alpha}\right)
    \end{align*}
    holds for all $t\in[0,T)$, some $t$-independent $C$, and
    \begin{align*}
        \Xi_t \coloneqq \max_{\lambda,\lambda'\in\mathcal{L}}
        \max\big\{\abs{\partial_{s}w^{t,\lambda}(s)^{\perp}
        \cdot\partial_{s}w^{t,\lambda'}(s')}
        \colon (s,s')\in\ell(w^{t,\lambda})\bbT\times\ell(w^{t,\lambda'})\bbT
        \ \&\ 
        w^{t,\lambda}(s) = w^{t,\lambda'}(s')\big\}.
    \end{align*}

    We will now show that $\Xi_t\leq Cm(t)^{1/3}$ for some $t$-independent $C$. Once this is proved,
    a Gr\"{o}nwall-type argument yields $m\equiv 0$
    because $\alpha\leq\frac{1}{6}$, $m(0) = 0$, and $m$ is continuous.
    Then we can conclude that for any
    $(t,\lambda,\lambda')\in[0,T]\times\mathcal{L}^{2}$,
    $\operatorname{im}(w^{t,\lambda})\subseteq\overline{\Omega(w^{t,\lambda'})}$ when $\lambda'\in S_{1}^{0,\lambda}$, and  $\operatorname{im}(w^{t,\lambda})\subseteq\bbR^{2}\setminus\Omega(w^{t,\lambda'})$ when $\lambda'\in S_{2}^{0,\lambda}\cup  S_{3}^{0,\lambda}$.
    Since $w$ is continuous with respect to $d_{\mathrm{F}}$, Lemma~\ref{LA.5} shows that
    $S_{i}^{t,\lambda} = S_{i}^{0,\lambda}$ holds for $i=1,2,3$, completing the proof.
    
    Take any $\lambda,\lambda'\in\mathcal{L}$ and
    $(s,s')\in\ell(w^{t,\lambda})\bbT\times\ell(w^{t,\lambda'})\bbT$
    such that $w^{t,\lambda}(s) = w^{t,\lambda'}(s')$ and
    $\partial_{s}w^{t,\lambda}(s)^{\perp}\cdot
    \partial_{s}w^{t,\lambda'}(s')\neq 0$. Then
    % $s'$ must be the common end-point of two adjacent connected components
    % $I_{1},I_{2}$ of $\set{s_{2}'\in\ell(w^{t,\lambda'})\bbT\colon
    % w^{t,\lambda'}(s_{2}') \notin\operatorname{im}(w^{t,\lambda})}$ such that
    % $w^{t,\lambda'}(I_{1}) \subseteq \Omega(w^{t,\lambda})$ and
    % $w^{t,\lambda'}(I_{2}) \subseteq \bbR^{2}\setminus\Omega(w^{t,\lambda})$.
    % Let $i\in\{1,2,3\}$ be such that $\lambda' \in S_{i}^{0,\lambda}$,
    % and let $J$ be $I_{1}$ if $i\in\{1,3\}$, and $I_{2}$ if $i=2$.
    % Then Lemma~\ref{L3.10} (with $\beta\coloneqq\frac{1}{2}$)
    % shows that there is $s_{2}'\in J$ such that
    % \[
    %     \abs{\partial_{s}w^{t,\lambda}(s)^{\perp}\cdot
    %     \partial_{s}w^{t,\lambda'}(s')}
    %     \leq Cd(w^{t,\lambda'}(s_{2}'), \operatorname{im}(w^{t,\lambda}))^{1/3}
    % \]
    % for some $t$-independent $C$. Since $w^{t,\lambda'}(s_{2}')\in\Omega(w^{t,\lambda})$ if $i\in\{1,3\}$ and $w^{t,\lambda'}(s_{2}')\notin\Omega(w^{t,\lambda})$ if $i=2$,
    % we have
    % \[
    %     d(w^{t,\lambda'}(s_{2}'), \operatorname{im}(w^{t,\lambda}))
    %     \leq m_{i}(t) \leq m(t).
    % \]
    $s$ must be the common end-point of two adjacent connected components
    $I_{1},I_{2}$ of $\set{s''\in\ell(w^{t,\lambda})\bbT\colon
    w^{t,\lambda}(s'') \notin\operatorname{im}(w^{t,\lambda'})}$ such that
    $w^{t,\lambda}(I_{1}) \subseteq \Omega(w^{t,\lambda'})$ and
    $w^{t,\lambda}(I_{2}) \subseteq \bbR^{2}\setminus\Omega(w^{t,\lambda'})$.
    Let $i\in\{1,2,3\}$ be such that $\lambda' \in S_{i}^{0,\lambda}$,
    and let $J$ be $I_{1}$ if $i\in\{2,3\}$, and $I_{2}$ if $i=1$.
    Then Lemma~\ref{L3.10} (with $\beta\coloneqq\frac{1}{2}$)
    shows that there is $s''\in J$ such that
    \[
        \abs{\partial_{s}w^{t,\lambda}(s)^{\perp}\cdot
        \partial_{s}w^{t,\lambda'}(s')}
        \leq Cd(w^{t,\lambda}(s''), \operatorname{im}(w^{t,\lambda'}))^{1/3}
    \]
    for some $t$-independent $C$. Since $w^{t,\lambda}(s'')\in\Omega(w^{t,\lambda'})$ if $i\in\{2,3\}$ and $w^{t,\lambda}(s'')\notin\Omega(w^{t,\lambda'})$ if $i=1$,
    we have
    \[
        d(w^{t,\lambda}(s''), \operatorname{im}(w^{t,\lambda'}))
        \leq m_{i}(t) \leq m(t).
    \]
    Since $\lambda,\lambda',s,s'$ were arbitrary,
    it follows that $\Xi_{t} \leq Cm(t)^{1/3}$, as desired.
\end{proof}

We can now derive the desired estimate on the ``distance'' of two solutions to \eqref{2.4}.
%exponential-in-time divergence of $D$.
%({\color{red}Roughly corresponds to KYZ, Section 2.3 and Proposition 4.5.})
%In the following, recall $m(\theta) = \min_{\lambda\in\mathcal{L}}\abs{\theta^{\lambda}}$.    
We note that similarly to Lemma \ref{L6.3},  we will only control the rate of change of $\sum_{\lambda\in\mathcal{L}}\abs{\theta^{\lambda}} D(z^{t,\lambda},w^{t,\lambda})$ rather than of the individual terms.
The cancellations allowing this again require us to only control  $\Delta(z^{t,\lambda},z^{t,\lambda'})^{-1}$  when  $\lambda'\notin\Sigma^\lambda(z^t)$, and the term that benefits from them is $P_{8}^{\lambda,\lambda'}$ below.

\begin{proposition}\label{P8.3}
    For any $W_{0},M\in[0,\infty)$, there is $\delta > 0$  such that the following holds.  If
    $z\in C([0,T_{0}];\operatorname{PSC}(\bbR^{2})^{\mathcal{L}})$ is a solution to \eqref{2.4}
    from Proposition \ref{P7.3}  satisfying  $W(z^0) \le W_{0}$ and   $\sup_{t\in[0,T_0]}L(z^{t})\leq M$, 
    %and   $\sup_{t\in[0,T_{0}]}L(z_{\eps}^{t}) \leq M$, 
    and if
    $w\colon [0,T] \to \operatorname{PSC}(\bbR^{2})^{\mathcal{L}}$ is
    an $H^{2}$ patch solution to \eqref{2.4} with $d_{\mathrm{H}}$ in place of $d_{\mathrm{F}}$
    satisfying   $T\leq T_{0}$,
    $\sup_{t\in[0,T]}\max_{\lambda\in\mathcal{L}}
    \Delta_{\norm{w^{t,\lambda}}_{\dot{H}^{2}}^{-2}}(w^{t,\lambda})^{-1}\leq M$,
    and $\sup_{t\in[0,T]}d_{\mathrm{F}}(z^{t},w^{t}) \leq \delta$, then
    \begin{equation}\label{8.1}
        \sum_{\lambda\in\mathcal{L}}\abs{\theta^{\lambda}}
        D(z^{t,\lambda},w^{t,\lambda})
        \leq e^{C m(\tht)^{-1} t}\sum_{\lambda\in\mathcal{L}}\abs{\theta^{\lambda}}
        D(z^{0,\lambda},w^{0,\lambda})
    \end{equation}
    holds for each  $t\in[0,T]$, with $C$
    only depending on ${\alpha,|\theta|,W_{0},M}$.
    In particular,
    $z^{t} = w^{t}$  for all $t\in[0,T]$ when $z^{0} = w^{0}$.
\end{proposition}

% \textit{Remark}. In fact, the proof only requires $w$ to be a solution
% in terms of the Hausdorff metric $d_{\mathrm{H}}$ rather than the
% Fr\'{e}chet metric $d_{\mathrm{F}}$.  As explained in the remark after Theorem~\ref{T2.5}, this then shows equivalence
% of the two notions of solutions.

\begin{proof}
    Let $\delta>0$ be from Lemma~\ref{L8.1}
    with $R_{1}\coloneqq 30MW_{0}$ and $R_{2}\coloneqq M^{1/2}$ (without loss also assume $\delta\le 1$), and let $z,w,T$ be as above.
    Since Proposition~\ref{P7.3} and Lemma~\ref{L3.9} show that for all $t\in[0,T_0]$ we have
    $W(z^{t}) \leq W_{0}$ and
    \begin{equation}\label{8.2}
        \max_{\lambda\in\mathcal{L}}\ell(z^{t,\lambda}) \leq 30MW_{0},
    \end{equation}
  we can apply Lemma~\ref{L8.1} to
    $\gamma_{1} \coloneqq z^{t,\lambda}$ and $\gamma_{2} \coloneqq w^{t,\lambda}$ for any
    $(t,\lambda)\in[0,T]\times \mathcal{L}$. Then \eqref{8.2} and Lemma~\ref{L8.1}(c) show that
    \begin{equation}\label{8.3}
        \max_{\lambda\in\mathcal{L}}\ell(w^{t,\lambda}) \leq 90MW_{0},
    \end{equation}
    and \eqref{8.2} and $\delta\leq 1$ yield
    \begin{equation}\label{8.4}
        \sup_{t\in[0,T]}\max_{\lambda\in\mathcal{L}}
        D(z^{t,\lambda},w^{t,\lambda})
        \leq \sup_{t\in[0,T]}\max_{\lambda\in\mathcal{L}}\ell(z^{t,\lambda}) \leq 30MW_{0}.
    \end{equation}
    Also, \eqref{8.3} and Proposition~\ref{P8.4} show that 
    $w^{t,\lambda}$ and $w^{t,\lambda'}$ do not cross (transversally or otherwise)  for each
    $(t,\lambda,\lambda')\in[0,T]\times\mathcal{L}^{2}$.
    
    Since the second claim of Proposition~\ref{P8.3} immediately follows from \eqref{8.1}
    and Lemma~\ref{L8.1}(a,b), we only need to show \eqref{8.1}.
    This will in turn follow via  a Gr\"{o}nwall-type argument from $D(z^{t,\lambda},w^{t,\lambda})$ being lower semicontinuous in $t\in[0,T]$ and
    \begin{equation} \lb{111.16}
        \partial_{t}^{+}\left(
            \sum_{\lambda\in\mathcal{L}}\abs{\theta^{\lambda}}
            D(z^{t,\lambda},w^{t,\lambda})
        \right)
        \leq \frac{C}{m(\theta)}
        \sum_{\lambda\in\mathcal{L}}\abs{\theta^{\lambda}}
        D(z^{t,\lambda},w^{t,\lambda})
    \end{equation}
     for all $t\in[0,T)$, with $C$ that only depends on
    ${\alpha,|\theta|,W_{0},M}$.
    In the rest of the proof, all constants  $C$ will only depend on ${\alpha,|\theta|,W_{0},M}$  and can change from line to line.

    For any $(t,\lambda)\in[0,T]\times \mathcal{L}$, fix any arclength parametrizations
    of $z^{t,\lambda}$ and $w^{t,\lambda}$.  Now fix any $t\in[0,T]$, and  use Lemma~\ref{L8.2}    to obtain a measurable function    $\psi_{\lambda}\colon\ell(z^{t,\lambda})\bbT\to\ell(w^{t,\lambda})\bbT$ for each $\lambda\in \mathcal{L}$ such that  for each $s\in\ell(z^{t,\lambda})\bbT$ we have
    \beq \lb{111.17}
        \abs{z^{t,\lambda} - w^{t,\lambda}\circ\psi_{\lambda}}(s)
        = d(z^{t,\lambda}(s),\operatorname{im}(w^{t,\lambda})).
    \eeq
    To simplify notation, we let
    $u_{1}^{t,\lambda}(s)\coloneqq u(z^{t};z^{t,\lambda}(s))$ and
    $u_{2}^{t,\lambda}(s)\coloneqq u(w^{t};w^{t,\lambda}(s))$.
    
    Take an arbitrary $r\in(0,1]$ and then $h_{0}>0$ such that
    \begin{align*}
        d_{\mathrm{H}}\left(
            w^{t+h},
            X_{u(w^{t})}^{h}[w^{t}]
        \right)
        \leq r\abs{h}
    \end{align*}
    holds for any $h\in[-h_{0},h_{0}]$ with $t + h\in [0,T]$, as well as     $C_{1}h_{0}^{1-2\alpha} \leq r$, where $C_{1}$ is
    the constant $C$ from Proposition~\ref{P7.3} with $M$ in place of $L(z^0)$.
    Now fix any $\lambda\in\mathcal{L}$.
    For any $h\in[-h_{0},h_{0}]$ with $t + h\in [0,T]$, let
    $\eta_{h}\colon\ell(z^{t,\lambda})\bbT\to\ell(z^{t+h,\lambda})\bbT$
    be the Lipschitz continuous orientation-preserving homeomorphism $\phi$
    from Proposition~\ref{P7.3}.  Then our choice of $h_0$ and $\eta_h$ show that for each $s\in\ell(z^{t+h,\lambda})\bbT$ we have
%    (we reserve the letter $\phi$ in this proof for later use).    
%    Then by construction of $\eta_{h}$ and $h_{0}$,
    \begin{align*}
        &d\left(
            z^{t,\lambda}(\eta_{h}^{-1}(s)),
            \operatorname{im}(w^{t,\lambda})
        \right) \\
        &\quad\quad\leq 
            d\left(
                z^{t+h,\lambda}(s) - hu_{1}^{t,\lambda}(\eta_{h}^{-1}(s)),
                \operatorname{im}\left(X_{u(w^{t})}^{h}[w^{t,\lambda}]\right)
            \right)
            + r\abs{h} + \abs{h}\norm{u_{2}^{t,\lambda}}_{L^{\infty}}
      \\
        &\quad\quad\leq 
            d\left(
                z^{t+h,\lambda}(s),
                \operatorname{im}(w^{t+h,\lambda})
            \right)
            + \abs{h}\left(
                2r
                + \norm{u_{1}^{t,\lambda}}_{L^{\infty}}
                + \norm{u_{2}^{t,\lambda}}_{L^{\infty}}
            \right)   .
    \end{align*}
    We know from  \eqref{8.2}, \eqref{8.3}, Lemmas \ref{L4.1} and \ref{L4.4} (the latter to bound $u_2$ via $u_1$), 
    and Proposition \ref{P7.3}    that
    \begin{equation}\label{8.4a}
        \norm{u_{1}^{t,\lambda}}_{L^{\infty}}
        + \norm{u_{2}^{t,\lambda}}_{L^{\infty}} \leq C
        \qquad\textrm{and}\qquad
        e^{-C\abs{h}} \leq \eta_{h}' \leq e^{C\abs{h}},
    \end{equation}
    which yields
    \begin{align*}
        D(z^{t,\lambda}, w^{t,\lambda})
        &= \int_{\ell(z^{t+h,\lambda})\bbT}
        d\left(
            z^{t,\lambda}(\eta_{h}^{-1}(s)),
            \operatorname{im}(w^{t,\lambda})
        \right)^{2}
        (\eta_{h}^{-1})'(s)\,ds \\
        &\leq \int_{\ell(z^{t+h,\lambda})\bbT}
        e^{C\abs{h}}
        \left(
            d\left(
                z^{t+h,\lambda}(s),
                \operatorname{im}(w^{t+h,\lambda})
            \right)
            + C\abs{h}
        \right)^{2}\,ds.
    \end{align*}
This and \eqref{8.2} %, and $\delta\leq 1$ 
prove lower semicontinuity of    $D(z^{t,\lambda},w^{t,\lambda})$.

    We are left with proving \eqref{111.16}, so now also assume that  $h> 0$.  Then for each $s\in\ell(z^{t,\lambda})\bbT$,
    \begin{align*}
        d\left(
            z^{t+h,\lambda}(\eta_{h}(s)),
            \operatorname{im}(w^{t+h,\lambda})
        \right)^{2}
        &\leq \left(
            d\left(
                z^{t,\lambda}(s) + hu_{1}^{t,\lambda}(s),
                \operatorname{im}\left(X_{u(w^{t})}^{h}[w^{t,\lambda}]\right)
            \right)
            + 2rh
        \right)^{2} \\
        &\leq \left(
            \abs{z^{t,\lambda}(s) + hu_{1}^{t,\lambda}(s)
            - w^{t,\lambda}(\psi_{\lambda}(s)) - hu_{2}^{t,\lambda}(\psi_{\lambda}(s))}
            + 2rh
        \right)^{2} \\
        &\leq \abs{z^{t,\lambda}(s) - w^{t,\lambda}(\psi_{\lambda}(s))}^{2}
        + h^{2}\left(
            \abs{u_{1}^{t,\lambda}(s) - u_{2}^{t,\lambda}(\psi_{\lambda}(s))} + 2r
        \right)^{2}
        \\&\quad\quad\quad
        + 4rh\abs{z^{t,\lambda}(s) - w^{t,\lambda}(\psi_{\lambda}(s))}
        \\&\quad\quad\quad
        + 2h(z^{t,\lambda}(s) - w^{t,\lambda}(\psi_{\lambda}(s)))
        \cdot (u_{1}^{t,\lambda}(s) - u_{2}^{t,\lambda}(\psi_{\lambda}(s)))
    \end{align*}
    (using $(|u+v|+c)^2\le |u|^2+|v|^2+2u\cdot v+c^2 + 2c(|u|+|v|)$ at the end), so \eqref{111.17} yields
    \begin{align*}
        &D(z^{t+h,\lambda},w^{t+h,\lambda})
        - D(z^{t,\lambda}, w^{t,\lambda})\\
        &\quad\quad=
        \int_{\ell(z^{t,\lambda})\bbT}
        \left[ d\left(
            z^{t+h,\lambda}(\eta_{h}(s)),
            \operatorname{im}(w^{t+h,\lambda})
        \right)^{2}\eta_{h}'(s)
        - d\left(
            z^{t,\lambda}(s),
            \operatorname{im}(w^{t,\lambda})
        \right)^{2} \right]ds \\
        &\quad\quad\leq
        \int_{\ell(z^{t,\lambda})\bbT}
        d\left(
            z^{t,\lambda}(s),
            \operatorname{im}(w^{t,\lambda})
        \right)^{2}
        \abs{\eta_{h}'(s) - 1}ds
        + 4rh \int_{\ell(z^{t,\lambda})\bbT}
        d\left(
            z^{t,\lambda}(s),
            \operatorname{im}(w^{t,\lambda})
        \right)
        \eta_{h}'(s)\,ds
        \\&\quad\quad\quad\quad\quad
        + h^{2} \int_{\ell(z^{t,\lambda})\bbT}
        \left(
            \norm{u_{1}^{t,\lambda}}_{L^{\infty}}
            + \norm{u_{2}^{t,\lambda}}_{L^{\infty}}
            + 2r
        \right)^{2}
        \eta_{h}'(s)\,ds
        \\&\quad\quad\quad\quad\quad
        + 2h \int_{\ell(z^{t,\lambda})\bbT}
        \left( z^{t,\lambda}(s) - w^{t,\lambda}(\psi_{\lambda}(s)) \right)\cdot
        \left( u_{1}^{t,\lambda}(s) - u_{2}^{t,\lambda}(\psi_{\lambda}(s))\right)
        \,ds
        \\&\quad\quad\quad\quad\quad
        + 2h \int_{\ell(z^{t,\lambda})\bbT}
        d\left(
            z^{t,\lambda}(s),
            \operatorname{im}(w^{t,\lambda})
        \right)
        \left(
            \norm{u_{1}^{t,\lambda}}_{L^{\infty}}
            + \norm{u_{2}^{t,\lambda}}_{L^{\infty}}
        \right)
        \abs{\eta_{h}'(s) - 1} ds.
    \end{align*}
    Since this holds for any $r\in(0,1]$ and all small enough $h>0$, \eqref{8.4a} now shows that
    \begin{align*}
        \partial_{t}^{+}\left(
            \sum_{\lambda\in\mathcal{L}}\abs{\theta^{\lambda}}
            D(z^{t,\lambda},w^{t,\lambda})
        \right)
        \leq \sum_{\lambda\in\mathcal{L}}\abs{\theta^{\lambda}}\left(
            CD(z^{t,\lambda},w^{t,\lambda}) + 2P_{1}^{\lambda}
        \right),
    \end{align*}
     with
    \begin{align*}
        P_{1}^{\lambda}\coloneqq
        \int_{\ell(z^{t,\lambda})\bbT}
        \left( z^{t,\lambda}(s) - w^{t,\lambda}(\psi_{\lambda}(s)) \right)\cdot
        \left( u_{1}^{t,\lambda}(s) - u_{2}^{t,\lambda}(\psi_{\lambda}(s))\right)
        \,ds.
    \end{align*}
    Thus it suffices to prove  that
    \begin{equation}\label{8.5}
        \abs{ \sum_{\lambda\in\mathcal{L}}\abs{\theta^{\lambda}}P_{1}^{\lambda}}
        \leq \frac{C}{m(\theta)}
        \sum_{\lambda\in\mathcal{L}}\abs{\theta^{\lambda}}
        D(z^{t,\lambda},w^{t,\lambda}).
    \end{equation}

    Since all the following estimates involve a single $t$, we will now drop it from the notation for the sake of simplicity.
    For each $\lambda\in\mathcal{L}$, let
    $\phi_{\lambda}\colon \ell(z^{\lambda})\bbT\to\ell(w^{\lambda})\bbT$
    be the orientation-preserving homeomorphism from Lemma~\ref{L8.1}
    with $\gamma_{1}\coloneqq z^{\lambda}$ and $\gamma_{2}\coloneqq w^{\lambda}$.
    Denote $\mathbf{T}^{\lambda}(s)\coloneqq
    \partial_{s}w^{\lambda}(\phi_{\lambda}(s))$ and
    $\mathbf{N}^{\lambda}(s)\coloneqq \mathbf{T}^{\lambda}(s)^{\perp}$, and then
    decompose the integrand in $P_{1}^{\lambda}$ into
    \begin{align*}
%       & \left( z^{\lambda}(s) - w^{\lambda}(\psi_{\lambda}(s)) \right)\cdot
%        \left( u_{1}^{\lambda}(s) - u_{2}^{\lambda}(\psi_{\lambda}(s))\right)
%        \\&\quad\quad=
               & \left( z^{\lambda}(s) - w^{\lambda}(\psi_{\lambda}(s)) \right)\cdot
        \left( u_{2}^{\lambda}(\phi_{\lambda}(s)) - u_{2}^{\lambda}(\psi_{\lambda}(s))\right)
        \\& %\quad\quad\quad\quad\quad
        + \left[ (z^{\lambda}(s) - w^{\lambda}(\psi_{\lambda}(s)))
        \cdot \mathbf{T}^{\lambda}(s)\right]
        \left[ (u_{1}^{\lambda}(s) - u_{2}^{\lambda}(\phi_{\lambda}(s)))
        \cdot \mathbf{T}^{\lambda}(s)\right]
        \\& %\quad\quad\quad\quad\quad
        + \left[ (w^{\lambda}(\phi_{\lambda}(s)) - w^{\lambda}(\psi_{\lambda}(s)))
        \cdot \mathbf{N}^{\lambda}(s)\right]
        \left[ (u_{1}^{\lambda}(s) - u_{2}^{\lambda}(\phi_{\lambda}(s)))
        \cdot \mathbf{N}^{\lambda}(s)\right]
        \\& %\quad\quad\quad\quad\quad
        + \left[ (z^{\lambda}(s) - w^{\lambda}(\phi_{\lambda}(s)))
        \cdot \mathbf{N}^{\lambda}(s)\right]
        \left[ (u_{1}^{\lambda}(s) - u_{2}^{\lambda}(\phi_{\lambda}(s)))
        \cdot \mathbf{N}^{\lambda}(s)\right].
    \end{align*}
   We call the resulting integrals
    $P_{2}^{\lambda}$, $P_{3}^{\lambda}$, $P_{4}^{\lambda}$, and $P_{5}^{\lambda}$,
    respectively, and estimate them separately.
        
\smallskip
    \textbf{Estimate for $P_{2}^{\lambda}$.} By Lemma~\ref{L4.6} (with $\beta\coloneqq\frac{1}{2}$),
    \eqref{8.3}, \eqref{111.17}, and Lemma~\ref{L8.1}(f),
    \begin{align*}
        \abs{u_{2}^{\lambda}(\phi_{\lambda}(s))
        - u_{2}^{\lambda}(\psi_{\lambda}(s))}
        &\leq \norm{\partial_{s}u_{2}^{\lambda}}_{L^{\infty}}
        \abs{\phi_{\lambda}(s) - \psi_{\lambda}(s)}
        \leq C\left(d(z^{\lambda}(s),\operatorname{im}(w^{\lambda}))
        + d_{\mathrm{F}}(z^{\lambda},w^{\lambda})^{2}\right).
    \end{align*}
%    Multiplying $\abs{z^{\lambda}(s) - w^{\lambda}(\psi_{\lambda}(s))}$
%    to the right-hand side and then integrating over $s$ gives
    Therefore \eqref{111.17}, Lemma~\ref{L8.1}(a,b), \eqref{8.2}, and \eqref{8.4} yield
    \begin{align*}
       \abs{P_{2}^{\lambda}} & \le  C
            \int_{\ell(z^{\lambda})\bbT}
            d(z^{\lambda}(s),\operatorname{im}(w^{\lambda})) \left[ d(z^{\lambda}(s),\operatorname{im}(w^{\lambda}))
            + d_{\mathrm{F}}(z^{\lambda},w^{\lambda})^{2} \right]\,ds
        \\&
        \leq CD(z^{\lambda},w^{\lambda})
        + CD(z^{\lambda},w^{\lambda})^{1/2}\cdot
        D(z^{\lambda},w^{\lambda})^{3/5}
        \leq CD(z^{\lambda},w^{\lambda}).
    \end{align*}
%    where for the first inequality we used Cauchy-Schwarz inequality, \eqref{8.2}
%    and Lemma~\ref{L8.1}(a,b), and the second inequality follows by \eqref{8.4}.
        
\smallskip
    \textbf{Estimate for $P_{3}^{\lambda}$.} 
    We have $(z^{\lambda}(s) - w^{\lambda}(\psi_{\lambda}(s)))
    \cdot \partial_{s}w^{\lambda}(\psi_{\lambda}(s)) = 0$ because 
    $\psi_{\lambda}(s)$  minimizes
    $\abs{z^{\lambda}(s) - w^{\lambda}(\,\cdot\,)}$, hence    Lemma~\ref{L8.1}(f) shows that
    with $\kappa_{2}^{\lambda}\coloneqq \partial_{s}^{2}w^{\lambda}\cdot
    (\partial_{s}w^{\lambda})^{\perp}$ we have
    \begin{align*}
       |(z^{\lambda}(s) - w^{\lambda}(\psi_{\lambda}(s)))
        \,\cdot & \,\mathbf{T}^{\lambda}(s)|
        = \abs{(z^{\lambda}(s) - w^{\lambda}(\psi_{\lambda}(s)))
        \cdot (\partial_{s}w^{\lambda}(\phi_{\lambda}(s))
        - \partial_{s}w^{\lambda}(\psi_{\lambda}(s)))}
        \\&
        \leq Cd(z^{\lambda}(s), \operatorname{im}(w^{\lambda}))
        \,\mathcal{M}\kappa_{2}^{\lambda}(\phi_{\lambda}(s))
        \left(
            d(z^{\lambda}(s), \operatorname{im}(w^{\lambda}))
            + d_{\mathrm{F}}(z^{\lambda},w^{\lambda})^{2}
        \right),
    \end{align*}
  where $\mathcal{M}$ is the maximal operator defined in \eqref{111.18} (recall also that $\partial_{s}^{2}w^{\lambda}$ and   $(\partial_{s}w^{\lambda})^{\perp}$ are parallel because $w^\lambda(\cdot)$ is an arclength  parametrization).
%    appearing in Lemma~\ref{L3.4}.
    From Lemma~\ref{L4.4} (with $\beta\coloneqq\frac{1}{2}$), \eqref{8.2}, \eqref{8.3}, and
    Lemma~\ref{L8.1}(a,b) we know that
    \begin{align*}
        \norm{u_{1}^{\lambda} - u_{2}^{\lambda}\circ\phi_{\lambda}}_{L^{\infty}}
        &\leq Cd_{\mathrm{F}}(z,w)^{1-2\alpha}
        \leq C\max_{\lambda'\in\mathcal{L}}
        D(z^{\lambda'},w^{\lambda'})^{\frac{3-6\alpha}{10}},
    \end{align*}
    thus Lemma~\ref{L8.1}(a,b,c), 
    %Cauchy-Schwarz inequality,
    \eqref{8.2}, \eqref{8.4},  and \eqref{3.1} show that
    \begin{align*}
        \abs{P_{3}^{\lambda}}
        &\leq C 
            \max_{\lambda'\in\mathcal{L}}
            D(z^{\lambda'},w^{\lambda'})^{\frac{3-6\alpha}{10}}
        d_{\mathrm{F}}(z^{\lambda},w^{\lambda})
        \\&\quad\quad\quad\quad%\cdot
        \int_{\ell(z^{\lambda})\bbT}
          \mathcal{M}\kappa_{2}^{\lambda}(\phi_{\lambda}(s))
        \left(
            d(z^{\lambda}(s), \operatorname{im}(w^{\lambda}))
            + d_{\mathrm{F}}(z^{\lambda},w^{\lambda})^{2}
        \right)
      \,ds \\
        &\leq C
        \norm{\mathcal{M}\kappa_{2}^{\lambda}\circ\phi_{\lambda}}_{L^{2}}
            \max_{\lambda'\in\mathcal{L}}
            D(z^{\lambda'},w^{\lambda'})^{\frac{6-6\alpha}{10}}
        \left(
            D(z^{\lambda},w^{\lambda})^{1/2}
            + D(z^{\lambda},w^{\lambda})^{3/5}
        \right) \\
        &\leq C\norm{\mathcal{M}\kappa_{2}^{\lambda}}_{L^{2}}
            \max_{\lambda'\in\mathcal{L}}
            D(z^{\lambda'},w^{\lambda'})^{\frac{11-6\alpha}{10}} \\
        &\leq C\max_{\lambda'\in\mathcal{L}}
        D(z^{\lambda'},w^{\lambda'}).
    \end{align*}
        
\smallskip
    \textbf{Estimate for $P_{4}^{\lambda}$.} From \eqref{111.17} and Lemma~\ref{L8.1}(f) we see that
    \begin{align*}
        \abs{(w^{\lambda}(\phi_{\lambda}(s)) - w^{\lambda}(\psi_{\lambda}(s)))
        \cdot \mathbf{N}^{\lambda}(s)}
        &\le \abs{\int_{\psi_{\lambda}(s)}^{\phi_{\lambda}(s)}
        \left| \partial_{s}w^{\lambda}(s')
        - \partial_{s}w^{\lambda}(\phi_{\lambda}(s)) \right|
  %      \cdot \partial_{s}w^{\lambda}(\phi_{\lambda}(s))^{\perp}
        \,ds'} \\
        &\leq C\abs{\phi_{\lambda}(s) - \psi_{\lambda}(s)}^{2}
        \mathcal{M}\kappa_{2}^{\lambda}(\phi_{\lambda}(s)) \\
        &\leq C\left(
            d(z^{\lambda}(s), \operatorname{im}(w^{\lambda}))
            + d_{\mathrm{F}}(z^{\lambda},w^{\lambda})^{2}
        \right)^{2}
        \mathcal{M}\kappa_{2}^{\lambda}(\phi_{\lambda}(s)),
    \end{align*}
    and then a similar argument as for $P_{3}^{\lambda}$ yields
    \begin{align*}
        \abs{P_{4}^{\lambda}}
        &\leq C\max_{\lambda'\in\mathcal{L}}
        D(z^{\lambda'},w^{\lambda'}).
    \end{align*}
        
\smallskip
    \textbf{Estimate for $P_{5}^{\lambda}$.}
    We have $P_{5}^{\lambda}=\sum_{\lambda'\in\mathcal{L}}\theta^{\lambda'}
    \left(P_{6}^{\lambda,\lambda'} + P_{7}^{\lambda,\lambda'}
    + P_{8}^{\lambda,\lambda'}\right)$, where
    \begin{align*}
        &P_{6}^{\lambda,\lambda'} \coloneqq
        -\int_{\ell(z^{\lambda})\bbT\times\ell(z^{\lambda'})\bbT}
        \left[
            K(z^{\lambda}(s) - z^{\lambda'}(s'))
            - K(w^{\lambda}(\phi_{\lambda}(s)) - w^{\lambda'}(\phi_{\lambda'}(s')))
        \right]
        \\&\quad\quad\quad\quad\quad\quad\quad\quad\quad\quad\ 
        \left(
            \left(
                z^{\lambda}(s) - w^{\lambda}(\phi_{\lambda}(s))
            \right)
            \cdot\mathbf{N}^{\lambda}(s)
        \right)
        \left(\partial_{s}(w^{\lambda'}\circ\phi_{\lambda'})(s')
        \cdot \mathbf{N}^{\lambda}(s) \right)
        \,ds'\,ds, \\
        &P_{7}^{\lambda,\lambda'} \coloneqq
        \int_{\ell(z^{\lambda})\bbT\times\ell(z^{\lambda'})\bbT}
        K(z^{\lambda}(s) - z^{\lambda'}(s'))
        \left(
            (z^{\lambda}(s) - w^{\lambda}(\phi_{\lambda}(s)))
            \cdot\mathbf{T}^{\lambda}(s)
        \right)
        \\&\quad\quad\quad\quad\quad\quad\quad\quad\quad\ 
        \left[
            \left( \partial_{s}z^{\lambda'}(s')
            - \partial_{s}(w^{\lambda'}\circ\phi_{\lambda'})(s') \right)
            \cdot\mathbf{T}^{\lambda}(s)
        \right]\,ds'\,ds, \\
        &P_{8}^{\lambda,\lambda'}\coloneqq
        -\int_{\ell(z^{\lambda})\bbT\times\ell(z^{\lambda'})\bbT}
        K(z^{\lambda}(s) - z^{\lambda'}(s'))
        \\&\quad\quad\quad\quad\quad\quad\quad\quad\quad\quad
        \left(
            z^{\lambda}(s) - w^{\lambda}(\phi_{\lambda}(s))
        \right) \cdot \left(
            \partial_{s}z^{\lambda'}(s')
            - \partial_{s}(w^{\lambda'}\circ\phi_{\lambda'})(s')
        \right)ds'\,ds
    \end{align*}
and we also used $(u\cdot \mathbf{N}^\lambda)(v\cdot \mathbf{N}^\lambda) = u\cdot v - (u\cdot \mathbf{T}^\lambda)(v\cdot \mathbf{T}^\lambda)$.

    Since     % For $P_{6}^{\lambda,\lambda'}$, the inequalities
    $\abs{a^{2\alpha} - 1}\leq\abs{a - 1}$  and
    $a(b-c) \leq 2a^{2} + \max\set{b-c-a,0}^{2}$ when $a\geq 0$ and $b\geq c\ge 0$,
    \begin{align*}
        \abs{P_{6}^{\lambda,\lambda'}}&\leq
        \frac{c_{\alpha}}{2\alpha}
        \int_{\ell(z^{\lambda})\bbT\times\ell(z^{\lambda'})\bbT}
        \frac{\abs{z^{\lambda}(s) - w^{\lambda}(\phi_{\lambda}(s))}}
        {\abs{z^{\lambda}(s) - z^{\lambda'}(s')}^{2\alpha}}
        \abs{
            \frac{\abs{z^{\lambda}(s)
            - z^{\lambda'}(s')}}
            {\abs{w^{\lambda}(\phi_{\lambda}(s))
            - w^{\lambda'}(\phi_{\lambda'}(s'))}} - 1
        }
        \\&\quad\quad\quad\quad\quad\quad\quad\quad\ 
        \abs{\mathbf{T}^{\lambda'}(s')\cdot\mathbf{N}^{\lambda}(s)}
        \phi_{\lambda'}'(s')\,ds'\,ds \\
        &\leq
        \frac{c_{\alpha}}{2\alpha}
        \int_{\ell(z^{\lambda})\bbT\times\ell(z^{\lambda'})\bbT}
        \frac{2\abs{z^{\lambda}(s)
            - w^{\lambda}(\phi_{\lambda}(s))}^{2}
            + \abs{z^{\lambda'}(s')
            - w^{\lambda'}(\phi_{\lambda'}(s'))}^{2}
        }
        {\abs{z^{\lambda}(s) - z^{\lambda'}(s')}^{2\alpha}
        \abs{w^{\lambda}(\phi_{\lambda}(s))
        - w^{\lambda'}(\phi_{\lambda'}(s'))}}
        \\&\quad\quad\quad\quad\quad\quad\quad\quad\ 
        \abs{\mathbf{T}^{\lambda'}(s')\cdot\mathbf{N}^{\lambda}(s)}
        \phi_{\lambda'}'(s')\,ds'\,ds.
    \end{align*}
    Since $\phi_{\lambda'}'$ and $\phi_{\lambda}'$ are both between $\frac{1}{3}$
    and $3$ by Lemma~\ref{L8.1}(c), from Lemma~\ref{L4.5}(1) (once with $\beta\coloneqq\frac{1}{2}$,
    $\gamma_{1}\coloneqq w^{\lambda}$, $\gamma_{2}\coloneqq w^{\lambda'}$, $\gamma_{3}\coloneqq z^{\lambda'}$,
    $x\coloneqq z^{\lambda}(s)$, $s'\coloneqq\phi_{\lambda}(s)$ and $\phi\coloneqq \phi_{\lambda'}$,
    and a second time with $(\lambda,s)$ and $(\lambda',s')$ switched),
    \eqref{8.2}, \eqref{8.3}, and Lemma~\ref{L8.1}(b) we see that
    \begin{align*}
        \abs{P_{6}^{\lambda,\lambda'}}
        &\leq C\int_{\ell(z^{\lambda})\bbT}
        \abs{z^{\lambda}(s) - w^{\lambda}(\phi_{\lambda}(s))}^{2}\,ds
        + C\int_{\ell(z^{\lambda'})\bbT}
        \abs{z^{\lambda'}(s') - w^{\lambda'}(\phi_{\lambda'}(s'))}^{2}\,ds' \\
        &\leq C\max_{\lambda''\in\mathcal{L}}
        D(z^{\lambda''},w^{\lambda''}).
    \end{align*}

    Next, for any $c>0$ we have
    \begin{align*}
      \abs{P_{7}^{\lambda,\lambda'}} \le &   \frac{c_{\alpha}}{4\alpha c}
        \int_{\ell(z^{\lambda})\bbT\times\ell(z^{\lambda'})\bbT}
        \frac{\abs{\left(
            z^{\lambda}(s) - w^{\lambda}(\phi_{\lambda}(s))
        \right)
        \cdot\mathbf{T}^{\lambda}(s)}^{2}}
        {\abs{z^{\lambda}(s) - z^{\lambda'}(s')}^{2\alpha}}\,ds'\,ds \\
       & +  \frac{c_{\alpha}c}{4\alpha}\int_{\ell(z^{\lambda})\bbT\times\ell(z^{\lambda'})\bbT}
        \frac{\abs{
            \partial_{s}z^{\lambda'}(s')
            - \partial_{s}(w^{\lambda'}\circ\phi_{\lambda'})(s')
        }^{2}}
        {\abs{z^{\lambda}(s) - z^{\lambda'}(s')}^{2\alpha}}\,ds'\,ds.
    \end{align*}
By Lemma~\ref{L4.2} (with $\beta\coloneqq\frac{1}{2}$), \eqref{8.2}, and
    Lemma~\ref{L8.1}(b,e), the first term is bounded by
    \begin{align*}
        {Cc^{-1}D(z^{\lambda},w^{\lambda})^{9/5}}.
    \end{align*}
    By Lemma~\ref{L4.2} (with $\beta\coloneqq\frac{1}{2}$), \eqref{8.2}, and
    Lemma~\ref{L8.1}(a,b,d), the second term is bounded by
    \begin{align*}
        Cc D(z^{\lambda'},w^{\lambda'})^{1/5}.
    \end{align*}
    Hence, taking $c \coloneqq
    \max_{\lambda''\in\mathcal{L}}D(z^{\lambda''},w^{\lambda''})^{4/5}+r$
    and then letting $r\to 0^{+}$ shows that
    \begin{align*}
        \abs{P_{7}^{\lambda,\lambda'}} \leq
        C\max_{\lambda''\in\mathcal{L}}D(z^{\lambda''},w^{\lambda''}).
    \end{align*}

    Finally, we will bound $\sum_{\lambda,\lambda'\in\mathcal{L}}
    \abs{\theta^{\lambda}}\theta^{\lambda'}P_{8}^{\lambda,\lambda'}$.
    For each $\lambda\in\mathcal{L}$ and $s\in\ell(z^{\lambda})\bbT$ let
    \begin{align*}
        v^{\lambda}(s) &\coloneqq
        z^{\lambda}(s) - w^{\lambda}(\phi_{\lambda}(s)),
    \end{align*}
    so that
    \begin{align*}
        P_{8}^{\lambda,\lambda'}
        = -\int_{\ell(z^{\lambda})\bbT\times\ell(z^{\lambda'})\bbT}
        K(z^{\lambda}(s) - z^{\lambda'}(s'))
        \left(
            v^{\lambda}(s)\cdot\partial_{s}v^{\lambda'}(s')
        \right)
        \,ds'\,ds.
    \end{align*}
    If we let
    \begin{align*}
        P_{8,\eps}^{\lambda,\lambda'}
        \coloneqq  -\int_{\ell(z^{\lambda})\bbT\times\ell(z^{\lambda'})\bbT}
        K_{\eps}(z^{\lambda}(s) - z^{\lambda'}(s'))
        \left(
            v^{\lambda}(s)\cdot\partial_{s}v^{\lambda'}(s')
        \right)
        \,ds'\,ds
    \end{align*}
    for each $\eps>0$,
    then  $P_{8}^{\lambda,\lambda'}
    = \lim_{\eps\to 0^{+}}P_{8,\eps}^{\lambda,\lambda'}$ by Lemma~\ref{L4.3}.
    Fix any $\eps>0$.  Symmetrization shows that
    \begin{align*}
        \sum_{\lambda,\lambda'\in\mathcal{L}}
        \abs{\theta^{\lambda}}\theta^{\lambda'}P_{8,\eps}^{\lambda,\lambda'}
        &= -\frac{1}{2}\sum_{\lambda,\lambda'\in\mathcal{L}}
        \abs{\theta^{\lambda}\theta^{\lambda'}}
        \int_{\ell(z^{\lambda})\bbT\times\ell(z^{\lambda'})\bbT}
        K_{\eps}(z^{\lambda}(s) - z^{\lambda'}(s'))
        \\&\quad\quad\quad\quad\quad\quad\quad
        \left(
            \operatorname{sgn}(\theta^{\lambda'})
            v^{\lambda}(s)\cdot \partial_{s}v^{\lambda'}(s')
            + \operatorname{sgn}(\theta^{\lambda})
            \partial_{s}v^{\lambda}(s)\cdot v^{\lambda'}(s')
        \right)
        ds'\,ds.
    \end{align*}
    
    By splitting the two terms of the integrand into
    separate integrals and then performing the integration by parts on those integral
    in $s'$ and in $s$, respectively, we see that
    any summand with $\theta^{\lambda}\theta^{\lambda'}>0$ is equal to
    \begin{align*}
%        &\pm\int_{\ell(z^{\lambda})\bbT\times\ell(z^{\lambda'})\bbT}
%        K_{\eps}(z^{\lambda}(s) - z^{\lambda'}(s'))
%        \left(
%            v^{\lambda}(s)\cdot \partial_{s}v^{\lambda'}(s')
%            + \partial_{s}v^{\lambda}(s)\cdot v^{\lambda'}(s')
%        \right)
%        ds'\,ds \\
%        &\quad\quad=
        \pm\abs{\theta^{\lambda}\theta^{\lambda'}}
        \int_{\ell(z^{\lambda})\bbT\times\ell(z^{\lambda'})\bbT}
        DK_{\eps}(z^{\lambda}(s) - z^{\lambda'}(s'))
        \left(
            \partial_{s}z^{\lambda}(s) - \partial_{s}z^{\lambda'}(s')
        \right)
        (v^{\lambda}(s)\cdot v^{\lambda'}(s'))\,ds'\,ds,
    \end{align*}
    whose absolute value is bounded by
    \begin{align*}
        C\abs{\theta^{\lambda}\theta^{\lambda'}}
        \int_{\ell(z^{\lambda})\bbT\times\ell(z^{\lambda'})\bbT}
        \frac{\abs{\partial_{s}z^{\lambda}(s) - \partial_{s}z^{\lambda'}(s')}}
        {\abs{z^{\lambda}(s) - z^{\lambda'}(s')}^{1+2\alpha}}
        \left(
            \abs{v^{\lambda}(s)}^{2} + \abs{v^{\lambda'}(s')}^{2}
        \right)\,ds'\,ds.
    \end{align*}
    When $\lambda'\in\Sigma^{\lambda}(z)$,
    Lemma~\ref{L4.5}(2) (once with $\beta\coloneqq\frac{1}{2}$,
    $\gamma_{1}\coloneqq z^{\lambda'}$, $\gamma_{2}=\gamma_{3}\coloneqq z^{\lambda}$,
    $x\coloneqq z^{\lambda'}(s')$, and $\phi\coloneqq {\rm Id}$, and a second time
    with $(\lambda,s)$ and $(\lambda',s')$ switched), \eqref{8.2},
    and Lemma~\ref{L8.1}(b) show that the last integral is bounded by
    \begin{align*}
        C\int_{\ell(z^{\lambda})\bbT}\abs{v^{\lambda}(s)}^{2}\,ds
        + C\int_{\ell(z^{\lambda'})\bbT}\abs{v^{\lambda'}(s')}^{2}\,ds'
        \leq
        C\max_{\lambda''\in\mathcal{L}}D(z^{\lambda''},w^{\lambda''}).
    \end{align*}
    When $\lambda'\notin\Sigma^{\lambda}(z)$,
    Lemma~\ref{L4.2} (with $\beta\coloneqq\frac{1}{2}$), \eqref{8.2},
    and Lemma~\ref{L8.1}(b) show it is bounded by
    \begin{align*}
        &\frac{C}{\Delta(z^{\lambda},z^{\lambda'})}
        \int_{\ell(z^{\lambda})\bbT\times\ell(z^{\lambda'})\bbT}
        \frac{\abs{v^{\lambda}(s)}^{2} + \abs{v^{\lambda'}(s')}^{2}}
        {\abs{z^{\lambda}(s) - z^{\lambda'}(s')}^{2\alpha}}\,ds'\,ds
        \\&\quad\quad\quad\leq
        C\int_{\ell(z^{\lambda})\bbT}\abs{v^{\lambda}(s)}^{2}\,ds
        + C\int_{\ell(z^{\lambda'})\bbT}\abs{v^{\lambda'}(s')}^{2}\,ds'
        \\&\quad\quad\quad\leq
        C\max_{\lambda''\in\mathcal{L}}D(z^{\lambda''},w^{\lambda''}).
    \end{align*}
    
    Similarly, any summand with $\theta^{\lambda}\theta^{\lambda'}<0$ is equal to
    \begin{align*}
        &\pm\abs{\theta^{\lambda}\theta^{\lambda'}}
        \int_{\ell(z^{\lambda})\bbT \times\ell(z^{\lambda'})\bbT}
        DK_{\eps}(z^{\lambda}(s) - z^{\lambda'}(s'))
        \left(
            \partial_{s}z^{\lambda}(s) + \partial_{s}z^{\lambda'}(s')
        \right)
        (v^{\lambda}(s)\cdot v^{\lambda'}(s'))\,ds'\,ds,
    \end{align*}
    which again can be estimated similarly via  Lemma~\ref{L4.5}(3).
    Letting $\eps\to 0^{+}$ then shows that
    \begin{align*}
        \left| \sum_{\lambda,\lambda'\in\mathcal{L}}\abs{\theta^{\lambda}}\theta^{\lambda'}
        P_{8}^{\lambda,\lambda'} \right|
        \leq C \max_{\lambda\in\mathcal{L}}D(z^{\lambda},w^{\lambda}).
    \end{align*}
    This and the above estimates for
    $P_{j}^{\lambda}$ ($j=2,3,4$),  $P_{6}^{\lambda,\lambda'}$, $P_{7}^{\lambda,\lambda'}$ now yield
    \[
    \abs{ \sum_{\lambda\in\mathcal{L}}\abs{\theta^{\lambda}}P_{1}^{\lambda}}
     \le C   \max_{\lambda\in\mathcal{L}}D(z^{\lambda},w^{\lambda})
        \leq \frac{C }{m(\theta)}\sum_{\lambda\in\mathcal{L}}
        \abs{\theta^{\lambda}}D(z^{\lambda},w^{\lambda}),
    \]
    which is the desired estimate \eqref{8.5}.
\end{proof}

\begin{proof}[Proof of uniqueness in Theorem \ref{T2.5}]
    Let $z\in C_{\rm loc}([0,T_{z^0});\operatorname{PSC}(\bbR^{2})^{\mathcal{L}})$
    be an $H^{2}$ patch solution to \eqref{2.4} from the existence proof at the end of Section \ref{S7}.
    %Proposition \ref{P7.3}.
    % with the initial data
  %  $z^{0}\in\operatorname{PSC}(\bbR^{2})^{\mathcal{L}}$, and
    Let $w\colon [0,T]\to \operatorname{PSC}(\bbR^{2})^{\mathcal{L}}$
    be an $H^{2}$ patch solution  to \eqref{2.4}  with $d_{\mathrm{H}}$ in place of $d_{\mathrm{F}}$,
    satisfying $w^0=z^0$ and $T< T_{z^0}$. Let
%     \[
%         M \coloneqq \sup_{t\in [0,T]}\max\set{L(z^{t}),L(w^{t})} < \infty
% \qquad\text{and}\qquad
%         I\coloneqq \set{t\in[0,T]\colon z|_{[0,t]} = w|_{[0,t]}}.
% \]
    \[
        M \coloneqq \sup_{t\in [0,T]}\max\set{L(z^{t}),\norm{w^{t}}_{\dot{H}^{2}}^{2}} < \infty
        \qquad\text{and}\qquad
        I\coloneqq \set{t\in[0,T]\colon z|_{[0,t]} = w|_{[0,t]}},
    \]
    so that $z(T') = w(T')$ holds with $T'\coloneqq\sup I$.  We will now show that $\min_{\lambda\in\mathcal{L}}\Delta_{1/M}(w^{t,\lambda})$
    is bounded below for all $t\in[T', T]$ sufficiently close to $T'$.
    Fix $\lambda\in\mathcal{L}$, and for each $t\in[T',T]$,
    choose an arclength parametrization of $w^{t,\lambda}$ and
    $s_{1,t},s_{2,t}\in\ell(w^{t,\lambda})\bbT$ such that
    $\abs{s_{1,t} - s_{2,t}}\in\left[\frac{1}{M},\frac{\ell(w^{t,\lambda})}{2}\right]$ and
    $\Delta_{1/M}(w^{t,\lambda}) = \abs{w^{t,\lambda}(s_{1,t}) - w^{t,\lambda}(s_{2,t})}$.
    Also pick $s_{1,t}',s_{2,t}'\in\ell(w^{T',\lambda})\bbT$ so that
    $d(w^{t,\lambda}(s_{i,t}), \operatorname{im}(w^{T',\lambda}))
    = \abs{w^{t,\lambda}(s_{i,t}) - w^{T',\lambda}(s_{i,t}')}$ $(i=1,2)$, and then
    \[
        \Delta_{1/M}(w^{t,\lambda}) \geq
        \abs{w^{T',\lambda}(s_{1,t}') - w^{T',\lambda}(s_{2,t}')}
        - 2d_{\mathrm{H}}(w^{t,\lambda}, w^{T',\lambda}).
    \]
    If $\abs{s_{1,t}' - s_{2,t}'}\in\left[\frac{1}{M},\ell(w^{T',\lambda})-\frac{1}{M}\right]$, then
    \[
        \Delta_{1/M}(w^{t,\lambda}) \geq
        \Delta_{1/M}(w^{T',\lambda}) - 2d_{\mathrm{H}}(w^{t,\lambda}, w^{T',\lambda})
        = \Delta_{1/M}(z^{T',\lambda}) - 2d_{\mathrm{H}}(w^{t,\lambda}, w^{T',\lambda}).
      %  \ge \frac 1M-2d_{\mathrm{H}}(w^{t,\lambda}, w^{T',\lambda}).
    \]
    Otherwise Lemma~\ref{L3.1} (with $\beta\coloneqq\frac{1}{2}$)
    and the mean value theorem show that
    \begin{align*}
        \Delta_{1/M}(w^{t,\lambda}) &\geq
        \abs{(w^{T',\lambda}(s_{1,t}') - w^{T',\lambda}(s_{2,t}'))
        \cdot \partial_{s}w^{T',\lambda}(s_{1,t}')}
        - 2d_{\mathrm{H}}(w^{t,\lambda}, w^{T',\lambda}) \\
        &\geq \frac 12 {\min\set{\abs{s_{1,t}' - s_{2,t}'},
        \ell(w^{T',\lambda}) - \abs{s_{1,t}' - s_{2,t}'}}}
        - 2d_{\mathrm{H}}(w^{t,\lambda}, w^{T',\lambda}) \\
        &\geq \frac{1}{2M} - 2d_{\mathrm{H}}(w^{t,\lambda}, w^{T',\lambda}).
    \end{align*}
    Since $w^{\lambda}$ is continuous with respect to $d_{\mathrm{H}}$,
   we see that $\Delta_{1/M}(w^{t,\lambda}) \geq \frac 12\min\{ \frac 1{2M}, \Delta_{1/M}(z^{T',\lambda})\}$
   for all $t\in[T',T]$ sufficiently close to $T'$.
    % and continuity of $z,w$ with respect to $d_{\mathrm{H}}$ and Lemma~\ref{LA.9} show that
    %
    % since $z(T') = w(T')$ and $z,w$ are both continuous with respect to $d_{\mathrm{F}}$
    % by the second remark after Theorem~\ref{T2.5}, we see that
    % $\lim_{t\to T'} d_{\mathrm{F}}(z^{t},w^{t})=0$ holds.
    %
    % By the definition of $H^{2}$ patch solutions and Lemma~\ref{L4.1},
    % we know $z,w$ are both continuous. Hence, 
    % from $z(T') = w(T')$ we see that
    % $\lim_{t\to T'} d_{\mathrm{H}}(z^{t},w^{t})=0$.
    % Since all curves in a subset of $\operatorname{PSC}(\bbR^{2})$
    % bounded with respect to $d_{\mathrm{H}}$ lie in the same ball in $\bbR^{2}$,
    % Corollary~\ref{CA.8} shows that $w$ takes its values in some compact subset
    % $X$ of $\operatorname{PSC}(\bbR^{2})^{\mathcal{L}}$. Then
    % $\operatorname{Id}\colon (X,d_{\mathrm{F}})\to (X,d_{\mathrm{H}})$
    % is a homeomorphism, so $\lim_{t\to T'} d_{\mathrm{F}}(z^{t},w^{t})=0$ follows.
    % Then Lemma~\ref{LA.9} shows that
    % \[
    %   {\Delta_{\norm{w^{t,\lambda}}_{\dot{H}^{2}}^{-2}}(w^{t,\lambda})}^{-1} \leq 4M
    % \]
    % holds for each $\lambda\in\mathcal{L}$ and $t$ sufficiently close to $T'$,
    % so Proposition~\ref{P8.3} yields
    Hence Proposition~\ref{P8.3} yields
    $z^{t} = w^{t}$ for all $t\in[T',T]$ sufficiently close to $T'$,
    so we must have $T'=T$ and $w=z|_{[0,T]}$.
    \end{proof}

%%%%%%%%%%%%%%%%%%%%%%%%%%%%%%%%%%%%%%%%%%%%%%
\section{A blow-up criterion for \eqref{2.4} and the proof of Theorem \ref{T2.7}}\label{S10}
%%%%%%%%%%%%%%%%%%%%%%%%%%%%%%%%%%%%%%%%%%%%%%

We now assume that $T_{z^0}<\infty$ and let $z$ be the solution  from the existence proof at the end of Section~\ref{S7}.  We will show that $z$ satisfies  \eqref{111.29}.
% for all $t\in[0,T_{z^0})$.
%$\sup_{t\in[0, T_{z^0})} \norm{z^{t}}_{\dot H^2} = \infty$.  
Recall that we have $\lim_{t\to T_{z^0}^-} L(z^t)=\infty$, so it suffices to show that the second $\max$ in \eqref{111.23} cannot diverge to $\infty$ if \eqref{111.29} fails.  
%We have \eqref{111.27} and
% $\min_{\lambda'\notin\Sigma^\lambda(z^0)} \Delta(z^{t,\lambda},z^{t,\lambda'})>0$   for any $(t,\lambda)\in[0,T_{z^0})\times \mathcal{L}$.

%In this section, we prove Theorem~\ref{T2.6}.
%Let $z\colon [0,T)\to\operatorname{PSC}(\bbR^{2})^{\mathcal{L}}$ be an
%$H^{2}$ patch solution to the g-SQG equation with a given initial data
%$z^{0}\in\operatorname{PSC}(\bbR^{2})^{\mathcal{L}}$ with $L(z^{0}) < \infty$.
%Let $W_{0}\coloneqq W(z^{0})$. Also, for notational simplicity, let
%\[
%    \norm{z^{0}}_{L^\infty}\coloneqq
%    \max_{\lambda\in\mathcal{L}}\norm{z^{0,\lambda}}_{L^{\infty}}.
%\]
%
%By the existence and the uniqueness results we proved in Section~\ref{S7} and
%Section~\ref{S8}, we know that locally $z$ must be the solution constructed in
%Section~\ref{S7}. That is, each $t\in [0,T)$ admits a neighborhood $I^{t}\subseteq[0,T)$
%such that $I^{t}$ is a compact interval, $z|_{I^{t}}$ is the uniform limit of $z_{\eps}$'s
%as in Section~\ref{S7}, and
%$M^{t}\coloneqq\sup_{\tau\in I^{t}}L(z^{\tau})<\infty$.
%Hence, Proposition~\ref{P7.3} applies on $I^{t}$, which shows that
%\begin{equation}\label{10.1}
%    \abs{\Omega(z^{t,\lambda})} = \abs{\Omega(z^{0,\lambda})}
%    \quad\textrm{and}\quad
%    \Sigma^{\lambda}(z^{t}) = \Sigma^{\lambda}(z^{0})
%\end{equation}
%hold for all $t\in[0,T)$ and $\lambda\in\mathcal{L}$.

This is an immediate consequence of \eqref{111.27} and the last claims in Lemmas \ref{P10.2} and \ref{P10.3} below.  These are versions of Lemmas~\ref{L6.5} and \ref{L6.6} for $z$, with \eqref{111.15} playing the role of \eqref{5.1} and with the factor
$W_{0}L(z_{\eps}^{t})^{2+2\alpha}$ in \eqref{6.3} and \eqref{6.5}  replaced by one that includes  $Q(z^t)$ instead of $L(z_\eps^t)$ (here again $W_{0}\coloneqq W(z^{0})$).
% a quantity that is independent of
%$\Delta(z^{t,\lambda},z^{t,\lambda'})$ and $\Delta_{1/Q(z^{t})}(z^{t,\lambda})$.
This is possible because $L(z_\eps^{t})^{1+2\alpha}$ in the former factor appeared as an upper bound on $\max\left\{ \norm{z_{\eps}^{t,\lambda}}_{\dot{H}^{2}}, \norm{z_{\eps}^{t,\lambda''}}_{\dot{H}^{2}} \right\}^{2+4\alpha}$, and hence can clearly be replaced by $Q(z_\eps^{t})^{1+2\alpha}$.  The remaining factor $W_{0}L(z_\eps^{t})$ appeared when we
 bounded  $\ell(z_\eps^{t,\lambda})$
by ${\abs{\Omega(z_\eps^{t,\lambda})}}{\Delta_{1/Q(z_\eps^{t})}(z_\eps^{t,\lambda})}^{-1}$, but we will now use \eqref{2.1} for this purpose instead.  

Finally, the last claim in Theorem \ref{T2.5} follows from Lemma \ref{L6.5}.  Indeed, the lemma shows that patch boundaries that do not touch initially also do not touch prior to $T_{z^0}$, while time reversibility of \eqref{2.4} and the lemma also imply that patches that touch initially will continue touching until $T_{z^0}$. 

All constants $C_\alpha$ below will again only depend on $\alpha$ and may change from line to line.

\begin{lemma}\label{P10.1}
    There is $C_{\alpha}$ such that for each $(t,\lambda)\in[0,T_{z^0})\times \mathcal{L}$ we have
    \beq\lb{111.28}
    \partial_{t}^{+}\norm{z^{t,\lambda}}_{L^{\infty}}
        \leq C_{\alpha}\abs{\theta}W_{0}^{\frac{1}{2}-\alpha},
      \eeq
      and so
    \begin{equation}\label{10.2}
        \ell(z^{t,\lambda}) \leq
        \left(
            \norm{z^{0,\lambda}}_{L^{\infty}}
            + C_{\alpha}\abs{\theta}W_{0}^{\frac{1}{2} - \alpha}t
        \right)^{2}\norm{z^{t,\lambda}}_{\dot{H}^{2}}^{2}.
    \end{equation}
\end{lemma}

\begin{proof}
    Fix any $(t,\lambda)\in[0,T_{z^0})\times \mathcal{L}$ and let $T_0,C$ be  from Proposition~\ref{P7.3} with $L(z^t)$ in place of $L(z^0)$.  Then $t+T_0<  T_{z^0}$, and
    %initial condition $z^t$ at time $t$ (instead of $z^0$ at time 0) in place of $[0,T_{0}]$.
 for any $h\in(0,T_0]$  fix any arclenth parametrizations of $z^{t,\lambda}$ and $z^{t+h,\lambda}$.  Let $\phi$ be the corresponding orientation-preserving homeomorphism from Proposition~\ref{P7.3}.
%    $\phi\colon\ell(z^{t,\lambda})\bbT \to \ell(z^{t + h,\lambda})\bbT$ such that \eqref{111.15} holds.
%    \begin{equation*}
%        \norm{z^{t+h,\lambda}\circ\phi
%        - z^{t,\lambda}
%        - hu(z^{t})\circ z^{t,\lambda}}_{L^{\infty}}
%        \leq C_{1}\abs{h}^{2-2\alpha}
%    \end{equation*}
%    where $z^{t,\lambda}$, $z^{t+h,\lambda}$ are assumed to be of any of their
%    respective arclength parametrizations. 
    Then \eqref{111.15},
    Lemma~\ref{L4.1}, and \eqref{111.27} show that
    \begin{align*}
        \frac{\norm{z^{t + h,\lambda}}_{L^{\infty}}
        - \norm{z^{t,\lambda}}_{L^{\infty}}}{h}
        \leq \norm{u(z^{t})}_{L^{\infty}} + C\abs{h}^{1-2\alpha}
%        \leq C_{\alpha}\abs{\theta}W(z^{t})^{\frac{1}{2}-\alpha} + C\abs{h}^{1-2\alpha} \\
        \le C_{\alpha}\abs{\theta}W_{0}^{\frac{1}{2} - \alpha} + C\abs{h}^{1-2\alpha}.
    \end{align*}
    Taking $h\to 0^{+}$ yields \eqref{111.28},
%    shows that
%    \[
%        \partial_{t}^{+}\norm{z^{t,\lambda}}_{L^{\infty}}
%        \leq C_{\alpha}\abs{\theta}W_{0}^{\frac{1}{2}-\alpha}
%    \]
%for all $t\in[0,T)$ where $C_{\alpha}$ is independent of $t$,
and then  $\norm{z^{t,\lambda}}_{L^{\infty}}  \leq \norm{z^{0,\lambda}}_{L^{\infty}} + C_{\alpha}\abs{\theta}W_{0}^{\frac{1}{2} - \alpha}t$ follows for each $(t,\lambda)\in[0,T_{z^0})\times \mathcal{L}$, with \eqref{2.1} now implying  \eqref{10.2}.
\end{proof}

%({\color{red}$z$ does solve an analogue of \eqref{5.1}.})
%While we did not show that $z$  solves an ODE similar to \eqref{5.1}, and hence  cannot rely on integral identities we used in the proofs 
We  now obtain analogs of Lemmas~\ref{L6.5} and \ref{L6.6} via  similar proofs.
%Nevertheless, almost identical arguments continue to work.

\begin{lemma}\label{P10.2}
%    There is $C_{\alpha}$ such that
%    $\Delta(z^{t,\lambda},z^{t,\lambda'})$
%    is Lipschitz continuous in $t\in\bbR$ with Lipschitz constant 
%    $C_{\alpha}\abs{\theta}W_{0}^{\frac{1}{2}-\alpha}$  for any $\lambda,\lambda'\in\mathcal{L}$, and the following holds.  For any $t\in\bbR$
%    and arbitrary arclength parametrizations
%    of $z^{t,\lambda}$ and $z^{t,\lambda'}$,
%    there are $s\in\ell(z^{t,\lambda})\bbT$ and
%    $s'\in\ell(z^{t,\lambda'})\bbT$ such that
%    $\Delta(z^{t,\lambda},
%    z^{t,\lambda'})
%    = \abs{z^{t,\lambda}(s) - z^{t,\lambda'}(s')}$,
%    \begin{equation}\label{10.3}
%        \partial_{t+}\Delta(z^{t,\lambda},z^{t,\lambda'})
%        \geq \frac{z^{t,\lambda}(s) - z^{t,\lambda'}(s')}
%        {\abs{z^{t,\lambda}(s) - z^{t,\lambda'}(s')}} \cdot
%        \left(
%            u(z^{t};z^{t,\lambda}(s)) - u(z^{t};z^{t,\lambda'}(s'))
%        \right)
%    \end{equation}
%    and
Lemma \ref{L6.5} holds for all $t\in[0,T_{z^0})$, with $z_\eps,u_\eps$, and \eqref{6.3} replaced by $z,u$, and
    \begin{equation}\label{10.4}
        \partial_{t+}
        \Delta(z^{t,\lambda},z^{t,\lambda'}) \geq
        -C_{\alpha}m(\theta)\left(
            \norm{z^{0}}_{L^{\infty}}
            + C_{\alpha}\abs{\theta}W_{0}^{\frac{1}{2} - \alpha}t
        \right)^{2}
        Q(z^t)^{2 + 2\alpha}
        \Delta(z^{t,\lambda},z^{t,\lambda'}).
    \end{equation}
\end{lemma}

\begin{proof}
%We only need to consider $\lambda\neq\lambda'$.  
Consider again the setup from the proof of Lemma~\ref{P10.1}.
%    For given $t_{0}\in[0,T)$, let $C_{1}$ be the constant $C$
%    appearing in Proposition~\ref{P7.3} with $I^{t_{0}}$ in place of $[0,T_{0}]$.
    The construction of $z$ in Section \ref{S7} and Lemma~\ref{L6.5} show that
    $\Delta(z^{t,\lambda},z^{t,\lambda'})$ is Lipschitz continuous on $[0,T_{z^0})$
    with Lipschitz constant  $C_{\alpha}\abs{\theta}W_{0}^{\frac{1}{2}-\alpha}$.
%    and since $C_{\alpha}$ does not depend on $t_{0}$, the same is true on all of $[0,T)$.
  
   % For any $(t,\lambda)\in [0,T_{z^0})\times \mathcal{L}$ fix any arclength parametrization of $z^{t,\lambda}$, and then 
   Now fix any $t\in [0,T_{z^0})$ (then pick $T_0$ as in the proof of Lemma~\ref{P10.1}), and for each $\tau\in (t,t+T_{0}]$ let $\phi_{\tau,\lambda}\colon\ell(z^{t,\lambda})\bbT\to\ell(z^{\tau,\lambda})\bbT$ be an orientation-preserving homeomorphism
     such that
    \begin{equation}\label{10.5}
        \norm{z^{\tau,\lambda}\circ\phi_{\tau,\lambda}
        - z^{t,\lambda} - (\tau - t)u(z^{t})\circ z^{t,\lambda}}_{L^{\infty}(\ell(z^{t,\lambda})\bbT)}
        \leq C\abs{\tau - t}^{2-2\alpha}.
    \end{equation}
    Define $\phi_{\tau,\lambda'}$ in the same way, but with $\lambda'$ in place of $\lambda$, and then fix any
    $s_{\tau}\in\ell(z^{\tau,\lambda})\bbT$ and $s_{\tau}'\in\ell(z^{\tau,\lambda'})\bbT$
    such that $\Delta(z^{\tau,\lambda},z^{\tau,\lambda'})
    = \abs{z^{\tau,\lambda}(s_{\tau}) - z^{\tau,\lambda'}(s_{\tau}')}$.
    Then \eqref{10.5} and \eqref{111.27} show that 
    \begin{equation}\label{10.6}\begin{split}
        \Delta(z^{t,\lambda},z^{t,\lambda'})
        &\leq \abs{z^{t,\lambda}(\phi_{\tau,\lambda}^{-1}(s_{\tau}))
        - z^{t,\lambda'}(\phi_{\tau,\lambda'}^{-1}(s_{\tau}'))} \\
        &\leq \Delta(z^{\tau,\lambda},z^{\tau,\lambda'})
        + \norm{z^{\tau,\lambda}\circ\phi_{\tau,\lambda} - z^{t,\lambda}}_{L^{\infty}}
        + \norm{z^{\tau,\lambda'}\circ\phi_{\tau,\lambda'} - z^{t,\lambda'}}_{L^{\infty}} \\
        &\leq \Delta(z^{\tau,\lambda},z^{\tau,\lambda'})
        + 2\norm{u(z^{t})}_{L^{\infty}}(\tau - t) + 2C(\tau - t)^{2-2\alpha} .
    \end{split}\end{equation}

    To prove the analog of \eqref{6.2}, assume without loss that $\Delta(z^{t,\lambda},z^{t,\lambda'})> 0$
    and choose a decreasing sequence
    $\seq{t_{n}}_{n=1}^{\infty}$ in $(t,t+T_{0}]$ converging to $t$ such that
    \[
        \partial_{t+}\Delta(z^{t,\lambda},z^{t,\lambda'})
        = \lim_{n\to\infty}\frac{\Delta(z^{t_{n},\lambda},z^{t_{n},\lambda'})
        - \Delta(z^{t,\lambda},z^{t,\lambda'})}{t_{n} - t}.
    \]
    For each $n$ let $\tilde{s}_{n} \coloneqq
    \phi_{t_{n},\lambda}^{-1}(s_{t_{n}})$ and
    $\tilde{s}_{n}' \coloneqq \phi_{t_{n},\lambda'}^{-1}(s_{t_{n}}')$.
    By passing to a subsequence if necessary, we can assume that
    $\exists \lim_{n\to\infty} (\tilde{s}_{n}, \tilde{s}_{n}') = (s,s') \in \ell(z^{t,\lambda})\bbT\times \ell(z^{t,\lambda'})\bbT$. 
    Using \eqref{10.6} with $\tau\coloneqq t_{n}$ and then taking $n\to\infty$,
    we see from continuity of $\tau\mapsto\Delta(z^{\tau,\lambda},z^{\tau,\lambda'})$ that
    $\Delta(z^{t,\lambda},z^{t,\lambda'}) = \abs{z^{t,\lambda}(s) - z^{t,\lambda'}(s')} > 0$.  This and \eqref{10.5} show that for all large enough $n$ we have
    \begin{align*}
        &\frac{
            \Delta(z^{t_{n},\lambda},
            z^{t_{n},\lambda'})
            - \Delta(z^{t,\lambda},
            z^{t,\lambda'})
        }{t_{n}-t} \\
        &\quad\quad\geq
        \frac{z^{t,\lambda}(\tilde{s}_{n})
        - z^{t,\lambda'}(\tilde{s}_{n}')}
        {\abs{z^{t,\lambda}(\tilde{s}_{n})
        - z^{t,\lambda'}(\tilde{s}_{n}')}}
        \cdot
        \frac{(z^{t_{n},\lambda}(s_{t_{n}})
        - z^{t_{n},\lambda'}(s_{t_{n}}'))
        - (z^{t,\lambda}(\tilde{s}_{n})
        - z^{t,\lambda'}(\tilde{s}_{n}'))}{t_{n} - t} \\
        &\quad\quad\geq
        \frac{z^{t,\lambda}(\tilde{s}_{n})
        - z^{t,\lambda'}(\tilde{s}_{n}')}
        {\abs{z^{t,\lambda}(\tilde{s}_{n})
        - z^{t,\lambda'}(\tilde{s}_{n}')}}
        \cdot \left[
            u(z^{t};z^{t,\lambda}(\tilde{s}_{n}))
            - u(z^{t};z^{t,\lambda'}(\tilde{s}_{n}'))
        \right] - 2C(t_{n} - t)^{1-2\alpha}.
    \end{align*}
Since $u(z^{t})$ is continuous by Lemma~\ref{L4.4},   taking $n\to\infty$ yields the analog of \eqref{6.2}.

Next, from $\partial_{s}z^{t,\lambda}(s)$ and $\partial_{s}z^{t,\lambda'}(s')$
    being both orthogonal to $z^{t,\lambda}(s) - z^{t,\lambda'}(s')$,
    Lemma~\ref{L4.7}, and Lemma~\ref{P10.1} we see that
    \begin{align*}
        \partial_{t+}\Delta(z^{t,\lambda},z^{t,\lambda'})
        &\geq -C_{\alpha}\abs{z^{t,\lambda}(s) - z^{t,\lambda'}(s')}
        \sum_{\lambda''\in\mathcal{L}}\abs{\theta^{\lambda''}}\ell(z^{t,\lambda''})
        \max\left\{
            \norm{z^{t,\lambda}}_{\dot{H}^{2}},
            \norm{z^{t,\lambda''}}_{\dot{H}^{2}}
        \right\}^{2+4\alpha}
        \\&
        \geq -C_{\alpha}\Delta(z^{t,\lambda},z^{t,\lambda'})
        \left(
            \norm{z^{0}}_{L^{\infty}}
            + C_{\alpha}\abs{\theta}W_{0}^{\frac{1}{2} - \alpha}t
        \right)^{2}
        Q(z^{t})^{1+2\alpha}
        \sum_{\lambda''\in\mathcal{L}}\abs{\theta^{\lambda''}}
        \norm{z^{t,\lambda''}}_{\dot{H}^{2}}^{2}
        \\&
        = -C_{\alpha}m(\theta)
        \left(
            \norm{z^{0}}_{L^{\infty}}
            + C_{\alpha}\abs{\theta}W_{0}^{\frac{1}{2} - \alpha}t
        \right)^{2}
        Q(z^{t})^{2+2\alpha}
        \Delta(z^{t,\lambda},z^{t,\lambda'}),
    \end{align*}
    which is \eqref{10.4}.
\end{proof}

\begin{lemma}\label{P10.3}
Let $T\in(0,T_{z^0}]$ satisfy $\sup_{t\in[0,T)}Q(z^{t})<\infty$, and let
    $h\coloneqq[\sup_{t\in[0,T)}Q(z^{t})]^{-1}$.
Then Lemma \ref{L6.6} holds for all $t\in[0,T)$, with $z_\eps,u_\eps,\frac 1{Q(z_\eps^{t})}$, and \eqref{6.5} replaced by $z,u,h$, and
    \begin{equation}\label{10.8} %\begin{aligned}
        \partial_{t+}\Delta_{h}(z^{t,\lambda})
        \geq -C_{\alpha}m(\theta)\left(
            \norm{z^{0}}_{L^{\infty}}
            + C_{\alpha}\abs{\theta}W_{0}^{\frac{1}{2} - \alpha}t
        \right)^{2}
        Q(z^t)^{2 + 2\alpha}
        \Delta_{h}(z^{t,\lambda}).
    %\end{aligned}
    \end{equation}
\end{lemma}

\begin{proof}
Consider again the setup from the proof of Lemma~\ref{P10.1}, and for any $t\in [0,T)$ fix some $s_{t},s_{t}'\in\ell(z^{t,\lambda})\bbT$ such that
    $s_{t} - s_{t}' \in [h, \frac{\ell(z^{t,\lambda})}{2}]$ and
    $\Delta_{h}(z^{t,\lambda}) = \abs{z^{t,\lambda}(s_{t}) - z^{t,\lambda'}(s_{t}')}$.
Now fix any $t\in [0,T)$, and then for each $\tau\in [0,\min\{T,t+T_0\})$ let $\phi_{\tau}\colon\ell(z^{t,\lambda})\bbT\to\ell(z^{\tau,\lambda})\bbT$ be a homeomorphism    such that $e^{-C\abs{\tau - t}} \leq \phi_{\tau}' \leq e^{C\abs{\tau - t}}$ and  \eqref{10.5} hold. Then 
%
%   Fix $t\in I^{t_{0}}$, then for any $\tau\in I^{t_{0}}$, there exists
%    a Lipschitz continuous orientation-preserving homeomorphism
%    $\phi_{\tau}\colon\ell(z^{t,\lambda})\bbT\to\ell(z^{\tau,\lambda})\bbT$ such that
%    \begin{equation}\label{10.9}
%        \begin{gathered}
%            \norm{z^{\tau,\lambda}\circ\phi_{\tau}
%            - z^{t,\lambda}
%            - (\tau - t)u(z^{t})\circ z^{t,\lambda}}_{L^{\infty}}
%            \leq C_{1}\abs{\tau - t}^{2-2\alpha}
%            \quad\textrm{and} \\
%            e^{-C_{1}\abs{\tau - t}} \leq \phi_{\tau}' \leq e^{C_{1}\abs{\tau - t}}.
%        \end{gathered}
%    \end{equation}
%
%    Then for any $\tau\in I^{t_{0}}$, \eqref{10.9} shows
    \begin{equation}\label{10.10}\begin{split}
        \Delta_{h}(z^{t,\lambda}) &\leq
        \abs{z^{t,\lambda}(\phi_{\tau}^{-1}(s_{\tau}))
        - z^{t,\lambda}\left(
            \phi_{\tau}^{-1}(s_{\tau}) -
            \min\set{
                \max\set{h,\phi_{\tau}^{-1}(s_{\tau}) - \phi_{\tau}^{-1}(s_{\tau}')},
                \frac{\ell(z^{t,\lambda})}{2}
            }
        \right)} \\
        &\leq \abs{z^{t,\lambda}(\phi_{\tau}^{-1}(s_{\tau}))
        - z^{t,\lambda}(\phi_{\tau}^{-1}(s_{\tau}'))}
        + \max\set{h - \phi_{\tau}^{-1}(s_{\tau}) + \phi_{\tau}^{-1}(s_{\tau}'), 0}
        \\&\quad\quad\quad
        + \max\set{\phi_{\tau}^{-1}(s_{\tau}) - \phi_{\tau}^{-1}(s_{\tau}')
        - \frac{\ell(z^{t,\lambda})}{2}, 0} \\
        &\leq \Delta_{h}(z^{\tau,\lambda})
        + 2\norm{u(z^{t})}_{L^{\infty}} \abs{\tau - t} + 2C\abs{\tau - t}^{2-2\alpha}
        + (1 - e^{-C\abs{\tau - t}})h
        \\&\quad\quad\quad
        + (e^{C\abs{\tau - t}} - 1)\frac{\ell(z^{t,\lambda})}{2}.
    \end{split}\end{equation}
    %{\color{blue}Letting $\tau\to t$ then shows lower semicontinuity of $\Delta_{h}(z^{t,\lambda})$ in $t$,
    %and its upper semicontinuity follows by Lemma~\ref{LA.6}.}
    Since $t,\tau$ are arbitrary, the same inequality holds with $t$ and $\tau$
    swapped.  Since $\ell(z^{\tau,\lambda}) \leq e^{C\abs{\tau - t}}\ell(z^{t,\lambda})$ and from Lemma~\ref{L4.1} and \eqref{111.27}
    %, and \eqref{10.5},
    we know that 
    $\sup_{\tau\in[0,T_{z^0})} \norm{u(z^{\tau})}_{L^{\infty}}<\infty$,
    % are uniformly bounded for $\tau\in I^{t_{0}}$.
    we see that $\Delta_{h}(z^{t,\lambda})$ is continuous in $t$.

    Now assume that $\Delta_{h}(z^{t,\lambda}) < \frac{3h}{4}$ and pick a
    decreasing sequence $\seq{t_{n}}_{n=1}^{\infty}$ in $(t,\min\{T,t+T_0\})$ converging to $t$ such that
    \[
        \partial_{t+}\Delta_{h}(z^{t,\lambda})
        = \lim_{n\to\infty}\frac{\Delta_{h}(z^{t_{n},\lambda}) - \Delta_{h}(z^{t,\lambda})}
        {t_{n} - t}.
    \]
    For each $n$ let $\tilde{s}_{n}\coloneqq\phi_{t_{n}}^{-1}(s_{t_{n}})$
    and $\tilde{s}_{n}'\coloneqq\phi_{t_{n}}^{-1}(s_{t_{n}}')$. By passing to a subsequence
    if necessary, we can assume that $\exists\lim_{n\to\infty} (\tilde{s}_{n},\tilde{s}_{n}')=(s,s')\in(\ell(z^{t,\lambda})\bbT)^2$. Using \eqref{10.10} with $\tau\coloneqq t_{n}$ and then taking $n\to\infty$,
    we see from continuity of $\tau\mapsto\Delta_{h}(z^{\tau,\lambda})$ that
    $\Delta_{h}(z^{t,\lambda}) = \abs{z^{t,\lambda}(s) - z^{t,\lambda}(s')}$.
    %Also, \eqref{10.5} shows that
    For each $n$ we also clearly have
    \[
        e^{-C(t_{n} - t)}h \leq \tilde{s}_{n} - \tilde{s}_{n}'
        \leq e^{C(t_{n} - t)}\frac{\ell(z^{t_n,\lambda})}{2} \leq e^{2C(t_{n} - t)}\frac{\ell(z^{t,\lambda})}{2},
    \]
 so $s - s' \in[h, \frac{\ell(z^{t,\lambda})}{2}]$.
    Hence, Lemmas~\ref{L3.8} and \ref{L3.1} (with $\beta\coloneqq\frac{1}{2}$) show that
    \[
        h < s - s' \leq \frac{\ell(z^{t,\lambda})}{2}
        < \frac{3\ell(z^{t,\lambda})}{4} \leq \ell(z^{t,\lambda}) - h,
    \]
    and then also $\tilde{s}_{n} - \tilde{s}_{n}' \in(h, \ell(z^{t,\lambda}) - h)$ for all large enough $n$. Therefore, the arguments yielding the analog of \eqref{6.2} and \eqref{10.4} in the proof of  Lemma \ref{P10.2} can also be used to obtain the analog of \eqref{6.4} and \eqref{10.8}.
\end{proof}

\begin{proof}[Proof of Theorem \ref{T2.7}]
Consider the setup from Theorem \ref{T2.5} and assume that we also have $\bigcup_{\lambda\in\mathcal L} {\rm im}(z^{0,\lambda})\subseteq \bbR\times[0,\infty)$.  Let $\mathcal L'\coloneqq\{\lambda^+,\lambda^-\}_{\lambda\in\mathcal L}$, then $\psi^{\lambda^+}\coloneqq\tht^\lambda$ and $\psi^{\lambda^-}\coloneqq-\tht^\lambda$ for ${\lambda\in\mathcal L}$, as well as
$w^{0,\lambda^+}\coloneqq z^{0,\lambda}$ and $w^{0,\lambda^-}\coloneqq Rz^{0,\lambda}$, where  the operator $R$ on $\operatorname{PSC}(\bbR^2)$ is given by ${\rm im}(R\gamma) = R' ({\rm im}(\gamma))$ and $R'(x_1,x_2)\coloneqq(x_1,-x_2)$.  That is, $R$ represents the reflection across  $\bbR\times\{0\}$, with $\sum_{\lambda'\in\mathcal L'} \psi^{\lambda'} \chi_{\Omega(w^{t,\lambda'})}$ being odd in $x_2$.  Then $(\mathcal L',\psi,w^0)$ satisfies the hypotheses of Theorem \ref{T2.5} for $(\mathcal L,\tht,z^0)$, so the theorem holds for $(\mathcal L',\psi,w^0)$, with some $T_{w^0}>0$.  Moreover, by uniqueness and symmetry we have $w^{t,\lambda^-}=R w^{t,\lambda^+}$ for any $(t,\lambda)\in[0,T_{w^0})\times\mathcal L$.  
%(This clearly is an extension of Theorem \ref{T2.5} to the half-plane.)

For each $t\in[0,T_{w^0})$, the corresponding velocity $u(w^t)$, given by the Biot-Savart law \eqref{2.3} for $(\psi,w^t)$, is then also the g-SQG velocity obtained from the corresponding Biot-Savart law on $\bbR\times(0,\infty)$ for $\{(\tht^\lambda,w^{t,\lambda^+})\}_{\lambda\in \mathcal L}$ (which involves extending the latter ``oddly'' across $\bbR\times\{0\}$, i.e.~obtaining $(\psi,w^t)$, and then finding $u(w^t)$ via \eqref{2.3}).  Hence $\{w^{\lambda^+}\}_{\lambda\in \mathcal L}$ is an $H^2$ patch solution to \eqref{2.4} on $\bbR\times(0,\infty)$ (with either  $d_{\rm F}$ or $d_{\rm H}$).  Moreover, any other such solution clearly yields an $H^2$ patch solution to \eqref{2.4} on $\bbR^2$ after we extend it onto $\bbR^2$ at  each $t$ from its maximal interval of existence in the way we constructed $(\psi,w^0)$ from $(\tht,z^0)$.  

This shows that Theorem \ref{T2.5} also holds on $\bbR\times[0,\infty)$ instead of $\bbR^2$, and then  \cite[Theorem~1.3]{ZlaSQGsing} yields existence of  solutions as in this theorem
%Theorem \ref{T2.5} on $\bbR\times[0,\infty)$ 
satisfying $T_{z^0}<\infty$ and \eqref{111.29} (and the proof shows that the corresponding $z^{0,\lambda}$ can all be chosen to be smooth).  Extending them onto $\bbR^2$ at each $t\in[0,T_{z^0})$ as above now yields the desired solutions.
\end{proof}

\appendix

%%%%%%%%%%%%%%%%%%%%%%%%%%%%%%%%%%%%%%%%%%%%%%
\section{Various results on the geometry of curves} \lb{SAa}
%%%%%%%%%%%%%%%%%%%%%%%%%%%%%%%%%%%%%%%%%%%%%%

In this section we prove a number of  estimates concerning the geometry of curves with either $C^{1,\beta}$
or $H^{2}$ regularity. We let $\beta\in(0,1]$ throughout.

First we obtain a lower bound on how quickly the tangent vector of a $C^{1,\beta}$ curve
can turn, and thus also a lower bound on the length of the curve.

\begin{lemma}\label{L3.1}
    Let $\gamma\colon\ell\bbT\to\bbR^{2}$ be a $C^{1,\beta}$ closed curve
    parametrized by arclength and $\mathbf{T}\coloneqq\partial_{s}\gamma$.
    Then $\ell \geq \frac{2^{1+{1}/{2\beta}}}
    {\norm{\gamma}_{\dot{C}^{1,\beta}}^{1/\beta}}$, and    $\mathbf{T}(s)\cdot\mathbf{T}(s') \geq
    1 - \frac{1}{2}\norm{\gamma}_{\dot{C}^{1,\beta}}^{2}\abs{s - s'}^{2\beta}$  for any $s,s'\in\ell\bbT$.
\end{lemma}

\begin{proof}
    Since $\mathbf{T}$ is $C^{\beta}$, we have
    \begin{align*}
        \norm{\gamma}_{\dot{C}^{1,\beta}}^{2}\abs{s - s'}^{2\beta}
        &\geq \abs{\mathbf{T}(s) - \mathbf{T}(s')}^{2}
        = 2 - 2\mathbf{T}(s)\cdot \mathbf{T}(s'),
    \end{align*}
    and rearranging the above gives the second claim.
%    \[
%        \mathbf{T}(s)\cdot\mathbf{T}(s')
%        \geq 1 - \frac{1}{2}\norm{\gamma}_{\dot{C}^{1,\beta}}^{2}\abs{s - s'}^{2\beta}.
%    \]
    Next note that 
    %$\gamma$ must make at least two
    %$90^{\circ}$-turns because 
    there must be $s\in(0,\ell)$ such that 
%    $\mathbf{T}(s)\cdot \mathbf{T}(s')$ must change 
%    sign at least twice as $s$ increases from $s'$ to $s'+\ell$.
%     while traveling $\ell\bbT$ starting at any
%    $s\in\ell\bbT$, because if there is no such turn, then its $\mathbf{T}(s)$-coordinate
%    increases forever contradicting to that $\gamma$ is a closed curve,
%    and if there is at least one $90^{\circ}$-turn, then $\gamma$ must have another
%    $90^{\circ}$-turn in order to turn back to its initial direction. By the inequality above,
%    The first claim shows that we must have $\abs{s' - s} \geq \frac{2^{\frac{1}{2\beta}}}
 %   {\norm{\gamma}_{\dot{C}^{1,\beta}}^{1/\beta}}$ in order to have
    \begin{align*}
        0 = \mathbf{T}(s)\cdot\mathbf{T}(0) \geq
        1 - \frac{1}{2}\norm{\gamma}_{\dot{C}^{1,\beta}}^{2}\min\{s,\ell-s\}^{2\beta},
    \end{align*}
    with the inequality due to the second claim. This yields the first claim. 
%    thus $\ell\geq\frac{2^{1+\frac{1}{2\beta}}}{\norm{\gamma}_{\dot{C}^{1,\beta}}^{1/\beta}}$ follows.
\end{proof}

Next, for any $C^{1,\beta}$ curve $\gamma$ and any $x\in\bbR^2$, we identify the structure of the part of $\gamma$ that lies $O(\norm{\gamma}_{\dot{C}^{1,\beta}}^{-1/\beta})$-close to $x$.  We use this result extensively in Section~\ref{S4}.
% the intersectioninto a bounded number of
%segments close to a given point
%The next lemma is about decomposing a  and the rest. We use this decomposition in proving
%many estimates we state in Section~\ref{S4}.
%({\color{red}Roughly corresponds to JZ, Lemma 2.3.})

\begin{lemma}\label{L3.2}
    Let $\gamma\colon\ell\bbT\to\bbR^{2}$ be a $C^{1,\beta}$ closed curve
    parametrized by arclength and
    $\mathbf{T}\coloneqq\partial_{s}\gamma$. Then for any $x\in\bbR^{2}$,
    there is $N\leq \ell\norm{\gamma}_{\dot{C}^{1,\beta}}^{1/\beta}$ and $s_{1},\dots, s_N \in \ell\bbT$  such that
    \begin{enumerate}
        \item  $\left\{I_{i}\coloneqq
        \left[s_{i}-\frac{1}{2}\norm{\gamma}_{\dot{C}^{1,\beta}}^{-1/\beta},
        s_{i}+\frac{1}{2}\norm{\gamma}_{\dot{C}^{1,\beta}}^{-1/\beta}\right] \right\}_{i=1}^N$
        are pairwise disjoint subintervals of $\ell\bbT$;

        \item $\set{s\in\ell\bbT \colon \abs{x - \gamma(s)} \leq
        \frac{1}{4}\norm{\gamma}_{\dot{C}^{1,\beta}}^{-1/\beta}}
        \subseteq\bigcup_{i=1}^{N}I_{i}$;
    \end{enumerate}
    and with
    $\tilde{I}_{i}\coloneqq \left[s_{i}-\norm{\gamma}_{\dot{C}^{1,\beta}}^{-1/\beta},
    s_{i}+\norm{\gamma}_{\dot{C}^{1,\beta}}^{-1/\beta}\right]$ and
    $g_{i}(s)\coloneqq (x - \gamma(s))\cdot\mathbf{T}(s_{i})$ we have for each $i=1,\dots ,N$,
    \begin{enumerate}
        \item[(3)] $\abs{x - \gamma(s_{i})}
        = \min_{s\in \tilde{I}_{i}}\abs{x - \gamma(s)}\leq
        \frac{1}{4}\norm{\gamma}_{\dot{C}^{1,\beta}}^{-1/\beta}$ and
        $(x-\gamma(s_{i}))\cdot\mathbf{T}(s_{i}) = 0$;
        
        \item[(4)]
        $g_{i}' \leq -\frac{1}{2}$  on $\tilde{I}_{i}$, so $g_{i}|_{\tilde{I}_{i}}$ is a homeomorphism and $\abs{g_{i}(s)} \geq \frac{1}{2}\abs{s - s_{i}}$ holds
        for all $s\in \tilde{I}_i$.
    \end{enumerate}
\end{lemma}

\begin{proof}
    Let $d_{0}\coloneqq\frac{1}{4}\norm{\gamma}_{\dot{C}^{1,\beta}}^{-1/\beta}$
    and consider the sets 
    $A\coloneqq\set{s\in\ell\bbT\colon \abs{x - \gamma(s)} \leq d_{0}}$ and
    $B\coloneqq\set{s\in\ell\bbT\colon \abs{x - \gamma(s)} < 2d_{0}}$.
    Since $B$ is open, it must be either $\ell\bbT$ or a countable union
    of proper open subintervals of $\ell\bbT$.
    Let $\set{J_{i}}_{i=1}^{N}$ be  all the connected components
    of $B$ that have non-empty intersections with $A$ (each has length at least $2d_0$ so there are finitely many).
    
    Fix any $i$ and pick $s_{i}\in \overline{J}_{i}$ that achieves the minimum of
    $\abs{x - \gamma(s)}$ over $\overline{J}_{i}$.  Then $s_{i}\in A$, so $s_{i}\in J_{i}$ and thus $(x - \gamma(s_{i}))\cdot\mathbf{T}(s_{i}) = 0$.
    Also, for any $s\in\ell\bbT$ with $\abs{s_{i} - s} < 4d_{0}$ we have
    \[
        \partial_{s}\left((x - \gamma(s))\cdot\mathbf{T}(s_{i})\right)
        = -\mathbf{T}(s_{i})\cdot\mathbf{T}(s)
        \leq  \frac{1}{2}\norm{\gamma}_{\dot{C}^{1,\beta}}^{2}
        \abs{s_{i} - s}^{2\beta} -1 < -\frac{1}{2}.
    \]
    Therefore $|x-\gamma(s)|\ge \frac{|s-s_i|}2$ when $|s-s_i|\le 4d_0$, so we have
    %Since $s_{i}\in A\cap J_{i}$, we obtain
    $\tilde{I}_{i}\coloneqq [s_{i} - 4d_{0}, s_{i} + 4d_{0}]\supseteq J_{i}$ and then also 
    $I_{i} \coloneqq [s_{i} - 2d_{0}, s_{i} + 2d_{0}] \supseteq A\cap J_{i}$.
    Since clearly $(s_{i} - d_{0}, s_{i} + d_{0}) \subseteq J_{i}$,
    any $s\in I_{i}\setminus \overline{J}_{i}$ satisfies $d(s, \partial J_{i}) \leq d_{0}$
    and so $\abs{x - \gamma(s)} \geq d_{0}$.
    This and  $\abs{x - \gamma(s)} > \frac{2d_0}2=d_{0}$ for $s\in \tilde{I}_{i}\setminus I_{i}$ show that $s_{i}$ also minimizes $\abs{x - \gamma(s)}$ on $\tilde{I}_{i}$.
    This shows (2)--(4), so it remains to prove (1) because $N\leq \ell\norm{\gamma}_{\dot{C}^{1,\beta}}^{1/\beta}$ then follows.

Pick any  $i\neq j$, so that $J_i\cap J_j=\emptyset$.
    Since $(s_{i} - d_{0}, s_{i} + d_{0})\subseteq J_{i}$ and the same holds for $j$,
    we have $|s_j-s_{i}|\ge 2d_{0}$. But  $\abs{x - \gamma(s)} > d_{0}$
    for all  $s\in\tilde{I}_{i}\setminus I_{i}$ then shows that in fact $|s_{j}-s_i|> 4d_{0}$. Therefore indeed   $I_{i}\cap I_{j} = \emptyset$.
\end{proof}

Next we obtain an estimate on the angle between tangent vectors of two curves
at points that are far enough from any large-angle crossings of those curves.
Note that if there are no transversal crossings (as in Lemma \ref{L3.4}), the hypothesis involving $(s_1',s_2')$ is always satisfied.
%This is a slight extension of a similar result appeared in \cite{JeoZla}
%which did not allow any crossings between the curves.
%Dependence on the $C^{1,\beta}$ norms of the curves is also more explicitly included
%in the following lemma.
%Next we obtain an estimate on the angle between tangent vectors of two non-crossing curves.
%While such an estimate already appeared \cite{JeoZla},
%dependence on the $C^{1,\beta}$ norms of the curves is included more explicitly
%in the following lemma.
%({\color{red}JZ, Lemma 2.2. A bit more generous since we do not have
%restrictions on $\abs{\gamma_{1}(s_{1}) - \gamma_{2}(s_{2})}$, and
%the dependence on $\norm{\gamma}_{\dot{C}^{1,\beta}}$ is directly stated.
%I think it got a bit longer than in JZ since I tried to be more explicit about the constant.})

% \begin{lemma}\label{L3.3}
%     Let $\gamma_{i}\colon\ell_{i}\bbT\to\bbR^{2}$ ($i=1,2$) be
%     $C^{1,\beta}$ closed curves parametrized by arclength that do not cross transversally,
%     and let $\mathbf{T}_{i}\coloneqq\partial_{s}\gamma_{i}$ and
%     $\mathbf{N}_{i}\coloneqq\mathbf{T}_{i}^{\perp}$.
%     Then for any $s_{1}\in\ell_{1}\bbT$ and $s_{2}\in\ell_{2}\bbT$ we have
%     \begin{align*}
%         \abs{\mathbf{T}_{1}(s_{1})\cdot\mathbf{N}_{2}(s_{2})}
%         \leq 12\max\left\{\norm{\gamma_{1}}_{\dot{C}^{1,\beta}},
%         \norm{\gamma_{2}}_{\dot{C}^{1,\beta}}\right\}^{\frac{1}{1+\beta}}
%         \abs{\gamma_{1}(s_{1}) - \gamma_{2}(s_{2})}^{\frac{\beta}{1+\beta}}.
%     \end{align*}
% \end{lemma}

\begin{lemma}\label{L3.3}
    Let $\gamma_{i}\colon\ell_{i}\bbT\to\bbR^{2}$ ($i=1,2$) be
    $C^{1,\beta}$ closed curves parametrized by arclength,
    and let $\mathbf{T}_{i}\coloneqq\partial_{s}\gamma_{i}$ and
    $\mathbf{N}_{i}\coloneqq\mathbf{T}_{i}^{\perp}$.
    Let $(s_{1},s_2)\in\ell_{1}\bbT\times\ell_{2}\bbT$ be such that 
    for all $(s_{1}',s_{2}')\in\ell_{1}\bbT\times\ell_{2}\bbT$
    satisfying $\gamma_{1}(s_{1}') = \gamma_{2}(s_{2}')$ and
    \[
        \max\set{\abs{s_{1}' - s_{1}},\abs{s_{2}' - s_{2}}}
        \leq \left(
            \frac{\abs{\gamma_{1}(s_{1}) - \gamma_{2}(s_{2})}}
            {\max\left\{\norm{\gamma_{1}}_{\dot{C}^{1,\beta}},
            \norm{\gamma_{2}}_{\dot{C}^{1,\beta}}\right\}}
        \right)^{\frac{1}{1+\beta}},
    \]
we have
    \[
        \abs{\mathbf{T}_{1}(s_{1}')\cdot\mathbf{N}_{2}(s_{2}')}
        \leq 10\max\left\{\norm{\gamma_{1}}_{\dot{C}^{1,\beta}},
        \norm{\gamma_{2}}_{\dot{C}^{1,\beta}}\right\}^{\frac{1}{1+\beta}}
        \abs{\gamma_{1}(s_{1}) - \gamma_{2}(s_{2})}^{\frac{\beta}{1+\beta}}.
    \]
    Then
    \begin{align*}
        \abs{\mathbf{T}_{1}(s_{1})\cdot\mathbf{N}_{2}(s_{2})}
        \leq 12\max\left\{\norm{\gamma_{1}}_{\dot{C}^{1,\beta}},
        \norm{\gamma_{2}}_{\dot{C}^{1,\beta}}\right\}^{\frac{1}{1+\beta}}
        \abs{\gamma_{1}(s_{1}) - \gamma_{2}(s_{2})}^{\frac{\beta}{1+\beta}}.
    \end{align*}
\end{lemma}

\begin{proof}
    Let $M\coloneqq \max\left\{\norm{\gamma_{1}}_{\dot{C}^{1,\beta}},
    \norm{\gamma_{2}}_{\dot{C}^{1,\beta}}\right\}$ and
    $R\coloneqq \frac{1}{4^{1+1/\beta}M^{1/\beta}}$. Note that the result holds trivially when $r\coloneqq\abs{\gamma_{1}(s_{1}) - \gamma_{2}(s_{2})}>R$,
    so we assume that $r\leq R$.   Then    $r^{\frac{\beta}{1+\beta}} \leq \frac{1}{4}M^{-\frac{1}{1+\beta}}$,
    so  $r\leq h \coloneqq \frac{1}{4}(r/M)^{\frac{1}{1+\beta}}$ and for any $s_{1}' \in [s_{1} - h, s_{1} + h]$ we have
    \[
    \abs{(\gamma_{1}(s_{1}') - \gamma_{2}(s_{2}))\cdot \mathbf{T}_{2}(s_{2})}\leq r + h\le 2h.
    \]
Since $4h \leq (R/M)^{\frac{1}{1+\beta}} \leq M^{-1/\beta}$,
    Lemma~\ref{L3.1} %and the mean value theorem 
    yields
    \begin{align*}
        (\gamma_{2}(s_{2}+4h) - \gamma_{2}(s_{2}))\cdot \mathbf{T}_{2}(s_{2})
        &\geq 4h\left(
            1 - \frac{1}{2}M^{2}
            (4h)^{2\beta}
        \right)
        \geq 2h
    \end{align*}
    and similarly
    \[
        (\gamma_{2}(s_{2}-4h) - \gamma_{2}(s_{2}))\cdot \mathbf{T}_{2}(s_{2}) \leq -2h.
    \]
    This all shows that the segment $\gamma_{1}|_{[s_{1}-h,s_{1}+h]}$ 
    is  contained in the strip
    \[
        S_{1}\coloneqq \set{y\in\bbR^{2} \colon y\cdot\mathbf{T}_{2}(s_{2})\in
        [\gamma_{2}(s_{2}-4h)\cdot \mathbf{T}_{2}(s_{2}), \gamma_{2}(s_{2}+4h)\cdot \mathbf{T}_{2}(s_{2})]}.
    \]

    Next, for any $s_{2}'\in[s_{2}-4h,s_{2}+4h]$,
    the mean value theorem gives some $\tau$ between $s_{2}$ and $s_{2}'$ such that
    \begin{align*}
        \abs{(\gamma_{2}(s_{2}') - \gamma_{2}(s_{2}))\cdot\mathbf{N}_{2}(s_{2})}
        &=\abs{s_{2}' - s_{2}}
        \abs{(\mathbf{T}_{2}(\tau) - \mathbf{T}_{2}(s_{2}))\cdot\mathbf{N}_{2}(s_{2})}
        \leq M(4h)^{1+\beta} = r.
    \end{align*}
    Therefore, the segment $\gamma_{2}|_{[s_{2}-4h,s_{2}+4h]}$
    is contained in the strip
    \[
        S_{2}\coloneqq \set{y\in\bbR^{2} \colon
        \abs{(y - \gamma_{2}(s_{2}))\cdot\mathbf{N}_{2}(s_{2})} \leq r}.
    \]
    On the other hand, the mean value theorem gives some $\tau_{\pm}$ between
    $s_{1}$ and $s_{1}\pm h$ such that
    \begin{align*}
        &(\gamma_{1}(s_{1}\pm h) - \gamma_{2}(s_{2}))\cdot\mathbf{N}_{2}(s_{2}) \\
        &\quad\quad
        = (\gamma_{1}(s_{1}) - \gamma_{2}(s_{2}))\cdot\mathbf{N}_{2}(s_{2}) \pm h\mathbf{T}_{1}(s_{1})\cdot\mathbf{N}_{2}(s_{2})
        \pm h(\mathbf{T}_{1}(\tau_{\pm}) - \mathbf{T}_{1}(s_{1}))\cdot\mathbf{N}_{2}(s_{2}).
    \end{align*}
    Since
    \begin{align*}
        \abs{(\gamma_{1}(s_{1}) - \gamma_{2}(s_{2}))\cdot\mathbf{N}_{2}(s_{2})
        \pm h(\mathbf{T}_{1}(\tau_{\pm}) - \mathbf{T}_{1}(s_{1}))\cdot\mathbf{N}_{2}(s_{2})}
        &\leq r + Mh^{1+\beta} \leq 2r,
    \end{align*}
assuming
    \begin{equation} \lb{111.4}
        \mathbf{T}_{1}(s_{1})\cdot\mathbf{N}_{2}(s_{2})
        > 12M^{\frac{1}{1+\beta}}
        r^{\frac{\beta}{1+\beta}} \quad (= 3rh^{-1})
    \end{equation}
now yields
    \begin{align*}
        (\gamma_{1}(s_{1}+h) - \gamma_{2}(s_{2}))\cdot\mathbf{N}_{2}(s_{2})
        > 3r - 2r = r
    \end{align*}
    and
    \begin{align*}
        (\gamma_{1}(s_{1}-h) - \gamma_{2}(s_{2}))\cdot\mathbf{N}_{2}(s_{2})
        < -3r + 2r = -r.
    \end{align*}
    Hence the segment $\gamma_{1}|_{[s_{1}-h,s_{1}+h]}$ passes through both sides of
    the rectangle $S_1\cap S_2$ that are parallel to $\mathbf{T}_2(s_2)$,
    while being contained in $S_1$.  Similarly, $\gamma_{2}|_{[s_{2}-4h,s_{2}+4h]}$
    connects the other two sides of $S_1\cap S_2$, while being contained in $S_2$.
    Therefore the segments must intersect at some $x \in S_{1}\cap S_{2}$,
    and so there are $s_{1}'\in[s_{1}-h,s_{1}+h]$ and
    $s_{2}'\in[s_{2}-4h,s_{2}+4h]$ such that
    $x = \gamma_{1}(s_{1}') = \gamma_{2}(s_{2}')$.
    By our assumption on $(s_{1},s_{2})$ we then  have
    \begin{align*}
        \abs{\mathbf{T}_{1}(s_{1})\cdot\mathbf{N}_{2}(s_{2})}
        \leq \abs{\mathbf{T}_{1}(s_{1}')\cdot\mathbf{N}_{2}(s_{2}')}
        + Mh^{\beta} + M(4h)^{\beta}
        \leq 12M^{\frac{1}{1+\beta}}r^{\frac{\beta}{1+\beta}},
    \end{align*}
    which contradicts \eqref{111.4}.
    % Since $\gamma_{1}$ and $\gamma_{2}$ do not cross transversally,
    % we have $\mathbf{T}_{1}(s_{1}') = \pm\mathbf{T}_{2}(s_{2}')$, which shows that
    % \[
    %     \abs{\mathbf{T}_{1}(s_{1})\cdot\mathbf{N}_{2}(s_{2}')}
    %     = \abs{(\mathbf{T}_{1}(s_{1}) - \mathbf{T}_{1}(s_{1}'))\cdot\mathbf{N}_{1}(s_{1}')}
    %     \leq Mh^{\beta} \leq M^{\frac{1}{1+\beta}}r^{\frac{\beta}{1+\beta}}.
    % \]
    % However, since also
    % \[
    %     \mathbf{T}_{1}(s_{1})\cdot\mathbf{N}_{2}(s_{2}')
    %     \geq \mathbf{T}_{1}(s_{1})\cdot\mathbf{N}_{2}(s_{2})
    %     - M(4h)^{\beta} > 11 M^{\frac{1}{1+\beta}}r^{\frac{\beta}{1+\beta}},
    % \]
    % boundary lines of $S_{2}$ which contains the segment
    % $\gamma_{2}|_{[s_{2}-4h,s_{2}+4h]}$. Since $\gamma_{2}|_{[s_{2}-4h,s_{2}+4h]}$
    % connects the two boundary lines of $S_{1}$ which contains
    % $\gamma_{1}|_{[s_{1}-h,s_{1}+h]}$, this means that these two segments
    % $\gamma_{1}|_{[s_{1}-h,s_{1}+h]}$ and $\gamma_{2}|_{[s_{2}-4h,s_{2}+4h]}$
    % must cross at some $x\in S_{1}\cap S_{2}$, contradicting
    % to the hypothesis.
    %this yields a contradiction. 
    We similarly obtain a contradiction when we assume
    \[
        \mathbf{T}_{1}(s_{1})\cdot\mathbf{N}_{2}(s_{2}) <
        -12M^{\frac{1}{1+\beta}}
        r^{\frac{\beta}{1+\beta}}
    \]
    instead of \eqref{111.4}, so the proof is finished.
\end{proof}

If the $\gamma_{i}$'s are $H^{2}$ instead of $C^{1,\beta}$ (and no transversal crossings are assumed), then an identical
proof works with  $\norm{\gamma_{i}}_{\dot{C}^{1,\beta}}$ and $\beta$ replaced by $\mathcal{M}\kappa_{i}(s_{i})$ and 1, where
$\kappa_{i}\coloneqq \partial_{s}\mathbf{T}_{i}\cdot\mathbf{N}_{i}$
is the (signed) curvature of $\gamma_i$ and $\mathcal{M}$ is the maximal operator 
\beq\lb{111.18}
    \mathcal{M}f(s)\coloneqq \max\set{
        \sup_{h\in\left(0,\frac{\ell}{2}\right]}
        \frac{1}{h}\int_{s}^{s+h}\abs{f(s')}\,ds',
        \sup_{h\in\left(0,\frac{\ell}{2}\right]}
        \frac{1}{h}\int_{s-h}^{s}\abs{f(s')}\,ds'
    }
\eeq
for an integrable $f\colon \ell\bbT\to\bbR$. We thus obtain
the following result.

\begin{lemma}\label{L3.4}
    Let $\gamma_{i}\colon\ell_{i}\bbT\to\bbR^{2}$ ($i=1,2$) be
    $H^{2}$ closed curves parametrized by arclength that do not cross transversally,
    and let $\mathbf{T}_{i}\coloneqq\partial_{s}\gamma_{i}$,
    $\mathbf{N}_{i}\coloneqq\mathbf{T}_{i}^{\perp}$, and
    $\kappa_{i}\coloneqq \partial_{s}\mathbf{T}_{i}\cdot\mathbf{N}_{i}$.
    Then for any $s_{1}\in\ell_{1}\bbT$ and $s_{2}\in\ell_{2}\bbT$ we have
    \begin{align*}
        \abs{\mathbf{T}_{1}(s_{1})\cdot\mathbf{N}_{2}(s_{2})}
        \leq 12\max\left\{\mathcal{M}\kappa_{1}(s_{1}),\mathcal{M}\kappa_{2}(s_{2})\right\}^{1/2}
        \abs{\gamma_{1}(s_{1}) - \gamma_{2}(s_{2})}^{1/2}.
    \end{align*}
\end{lemma}

We also recall that there is a universal constant $C<\infty$ such that
\begin{equation}\label{3.1}
    \norm{\mathcal{M}f}_{L^{2}} \leq C\norm{f}_{L^{2}}
\end{equation}
holds for any $f\in L^{2}(\ell\bbT)$ and $\ell\in(0,\infty)$.

Even though $\mathbf{T}_{1}(s_{1})\cdot\mathbf{N}_{2}(s_{2})$
can be bounded by some power of $\abs{\gamma_{1}(s_{1}) - \gamma_{2}(s_{2})}$,
this does not mean the same for $\mathbf{T}_{1}(s_{1}) - \mathbf{T}_{2}(s_{2})$
because $\mathbf{T}_{1}(s_{1})$ and $\mathbf{T}_{2}(s_{2})$ may be pointing
in roughly opposite directions. Nevertheless, the latter will not happen
if we assume that $\gamma_{1},\gamma_{2}$ are both positive simple closed curves
and one lies inside the region enclosed by the other.

\begin{lemma}\label{L3.5}
    Let $\gamma_{i}\colon\ell_{i}\bbT\to\bbR^{2}$ ($i=1,2$) be
    $C^{1,\beta}$ positive simple closed curves parametrized by arclength,
    and let $\mathbf{T}_{i}\coloneqq\partial_{s}\gamma_{i}$.
    Then for any $s_{1}\in\ell_{1}\bbT$ and $s_{2}\in\ell_{2}\bbT$ we have:
    \begin{enumerate}
        \item If $\Omega(\gamma_{1})\subseteq\Omega(\gamma_{2})$
        or $\Omega(\gamma_{2})\subseteq\Omega(\gamma_{1})$, then
        \begin{align*}
            &\abs{\mathbf{T}_{1}(s_{1}) - \mathbf{T}_{2}(s_{2})}
            \leq 60
            \min\set{
                \Delta_{\norm{\gamma_{1}}_{\dot{C}^{1,\beta}}^{-1/\beta}}(\gamma_{1}),
                \Delta_{\norm{\gamma_{2}}_{\dot{C}^{1,\beta}}^{-1/\beta}}(\gamma_{2}),
            }^{-\frac{\beta}{1+\beta}}
            \abs{\gamma_{1}(s_{1}) - \gamma_{2}(s_{2})}^{\frac{\beta}{1+\beta}}.
        \end{align*}

        \item If $\Omega(\gamma_{1})\cap\Omega(\gamma_{2}) = \emptyset$, then
        \begin{align*}
            &\abs{\mathbf{T}_{1}(s_{1}) + \mathbf{T}_{2}(s_{2})}
            \leq 60
            \min\set{
                \Delta_{\norm{\gamma_{1}}_{\dot{C}^{1,\beta}}^{-1/\beta}}(\gamma_{1}),
                \Delta_{\norm{\gamma_{2}}_{\dot{C}^{1,\beta}}^{-1/\beta}}(\gamma_{2}),
            }^{-\frac{\beta}{1+\beta}}
            \abs{\gamma_{1}(s_{1}) - \gamma_{2}(s_{2})}^{\frac{\beta}{1+\beta}}.
        \end{align*}
    \end{enumerate}
\end{lemma}

\begin{proof}
    We only prove the case $\Omega(\gamma_{1})\subseteq\Omega(\gamma_{2})$ as
   the others are analogous. Let
    $d_{i}\coloneqq\norm{\gamma_{i}}_{\dot{C}^{1,\beta}}^{-1/\beta}$ for $i=1,2$, and then
    $d\coloneqq \min\set{d_{1},d_{2}}$ and
    \[
        k\coloneqq \min\left\{
            \frac{1}{5\cdot 12^{1+1/\beta}},
            \frac{\Delta_{d_{1}}(\gamma_{1})}{2d},
            \frac{\Delta_{d_{2}}(\gamma_{2})}{2d}
        \right\}.
    \]
    Since $\abs{\mathbf{T}_{1}(s_{1}) - \mathbf{T}_{2}(s_{2})}\leq 2$,
    we may assume that $\abs{\gamma_{1}(s_{1}) - \gamma_{2}(s_{2})} < kd$ because
    otherwise the result follows from (recall also \eqref{2.2})
    \begin{align*}
        (kd)^{\frac{\beta}{1+\beta}}
        \geq \frac{1}{5^{\frac{\beta}{1+\beta}}\cdot 12} \min\set{
            d, \Delta_{d_{1}}(\gamma_{1}), \Delta_{d_{2}}(\gamma_{2})
        }^{\frac{\beta}{1+\beta}}
        \geq \frac{\min\set{\Delta_{d_{1}}(\gamma_{1}),
        \Delta_{d_{2}}(\gamma_{2})}^{\frac{\beta}{1+\beta}}}{30}.
    \end{align*}
%    where the last inequality is because of \eqref{2.2}.
    
    Take $s_{2,1},\dots,s_{2,N_{2}} \in \ell_{2}\bbT$ obtained by applying
    Lemma~\ref{L3.2} with $x\coloneqq\gamma_{1}(s_{1})$ and $\gamma\coloneqq\gamma_{2}$ so that
    $(\gamma_{1}(s_{1}) - \gamma_{2}(s_{2,i}))\cdot
    \mathbf{T}_{2}(s_{2,i}) = 0$ holds for $i=1,\dots,N_2$ and 
    \[
    \left\{ s_{2}'\in \ell_{2}\bbT \,:\, \abs{\gamma_{1}(s_{1}) - \gamma_{2}(s_{2}')} \leq \frac{d_{2}}{4} \right\}
    \subseteq  \bigcup_{i=1}^{N_2} \left[s_{2,i}-\frac{d_{2}}{2},s_{2,i}+\frac{d_{2}}{2}\right].
    \]
    Without loss of generality, suppose that
    $\abs{\gamma_{1}(s_{1}) - \gamma_{2}(s_{2,1})}
    = d(\gamma_{1}(s_{1}), \operatorname{im}(\gamma_{2})) < kd$.
    We now claim that any $s_{2}'\in\ell_{2}\bbT$ with
    $\abs{\gamma_{1}(s_{1}) - \gamma_{2}(s_{2}')}\leq kd$
    must be in $[s_{2,1}-4kd,s_{2,1}+4kd]$. This is because from
    $ \abs{s_{2}' - s_{2,1}} \in[d_2, \frac{\ell_{2}}{2}]$ and
    the definition of $\Delta_{d_{2}}(\gamma_{2})$ we see that
    \[
        \abs{\gamma_{1}(s_{1}) - \gamma_{2}(s_{2}')}
        \geq \abs{\gamma_{2}(s_{2,1}) - \gamma_{2}(s_{2}')}
        - \abs{\gamma_{1}(s_{1}) - \gamma_{2}(s_{2,1})}
        > \Delta_{d_{2}}(\gamma_{2}) - kd
        \geq kd,
    \]
    while from     $ \abs{s_{2}' - s_{2,1}} \in(4kd, d_{2})$ and Lemma~\ref{L3.2} we obtain
    \[
        \abs{\gamma_{1}(s_{1}) - \gamma_{2}(s_{2}')}
        \geq \abs{\gamma_{2}(s_{2,1}) - \gamma_{2}(s_{2}')}
        - \abs{\gamma_{1}(s_{1}) - \gamma_{2}(s_{2,1})}
        > 2kd - kd = kd.
    \]
 
    Similarly, take $s_{1,1},\dots,s_{1,N_{1}}\in \ell_{1}\bbT$ obtained by applying
    Lemma~\ref{L3.2} with $x\coloneqq\gamma_{2}(s_{2,1})$ and $\gamma\coloneqq\gamma_{1}$, and without loss assume that
    $\abs{\gamma_{1}(s_{1,1}) - \gamma_{2}(s_{2,1})}
    = d(\gamma_{2}(s_{2,1}), \operatorname{im}(\gamma_{1})) < kd$.
    Then again any $s_{1}'\in\ell_{1}\bbT$ with $\abs{\gamma_{1}(s_{1}') - \gamma_{2}(s_{2,1})}
    \leq kd$ must be in $[s_{1,1}-4kd,s_{1,1}+4kd]$.
    
    Next, we claim that the interior of the line segment joining $\gamma_{1}(s_{1})$
    and $\gamma_{2}(s_{2,1})$ is disjoint with
    $\operatorname{im}(\gamma_{1})\cup\,\operatorname{im}(\gamma_{2})$. It is clearly disjoint with $\operatorname{im}(\gamma_{2})$, by the definition of $s_{2,1}$.
%     is a point    on $\operatorname{im}(\gamma_{2})$ achieving the minimum distance to $\gamma_{1}(s_{1})$.
    Assume, towards contradiction that  $t\gamma_{1}(s_{1}) + (1-t)\gamma_{2}(s_{2,1}) = \gamma_{1}(s_{1}')$
    holds for some $t\in(0,1)$ and $s_{1}'\in\ell_{1}\bbT$ (as well as that $\gamma_{1}(s_{1})\neq\gamma_{2}(s_{2,1})$).
    Then  $|\gamma_{1}(s_{1}') - \gamma_{2}(s_{2,1})|
    = t|\gamma_{1}(s_{1}) - \gamma_{2}(s_{2,1})|<tkd$, so we must have
    $s_{1},s_{1}'\in [s_{1,1}-4kd,s_{1,1}+4kd]$. Then
    \[
        f(s_{1}'')\coloneqq
        (\gamma_{1}(s_{1}'') - \gamma_{2}(s_{2,1}))
        \cdot (\gamma_{1}(s_{1}) - \gamma_{2}(s_{2,1}))^{\perp}
    \]
    satisfies $f(s_{1}) = f(s_{1}') = 0$, so there is $s_{1}''$ strictly in between $s_{1}$ and $s_{1}'$ such that
    with $\mathbf{N}_{i}\coloneqq\mathbf{T}_{i}^{\perp}$ we have
    $(\gamma_{1}(s_{1}) - \gamma_{2}(s_{2,1}))\cdot\mathbf{N}_{1}(s_{1}'') = 0$.
    Since $\gamma_{1}(s_{1}) - \gamma_{2}(s_{2,1})$ is a nonzero vector
    parallel to $\mathbf{N}_{2}(s_{2,1})$, this implies
    $\abs{\mathbf{T}_{1}(s_{1}'')\cdot\mathbf{N}_{2}(s_{2,1})} = 1$.  Then
    Lemma~\ref{L3.3} shows that
    \begin{align*}
        \abs{\gamma_{1}(s_{1}'') - \gamma_{2}(s_{2,1})}
        \geq \frac{d}{12^{1+1/\beta}}
    \end{align*}
    and we obtain a contradiction via
    \begin{align*}
        \abs{\gamma_{1}(s_{1,1}) - \gamma_{2}(s_{2,1})}
        &\geq \abs{\gamma_{1}(s_{1}'') - \gamma_{2}(s_{2,1})}
        - \abs{\gamma_{1}(s_{1,1}) - \gamma_{1}(s_{1}'')}
        \geq \frac{d}{12^{1+1/\beta}} - 4kd \geq kd.
    \end{align*}
  
    So the line segment joining $\gamma_{1}(s_{1})$ and $\gamma_{2}(s_{2,1})$
    also joins  $\partial\Omega(\gamma_{1})$ and  $\partial\Omega(\gamma_{2})$, and its interior does not cross either boundary.  Hence  this interior must  lie fully inside
    $\Omega(\gamma_{2})\setminus\overline{\Omega(\gamma_{1})}$, and then the inward normal vectors
    $\mathbf{N}_{1}(s_{1})$ %of $\partial\Omega(\gamma_{1})$ 
    (at $\gamma_{1}(s_{1})$) and $\mathbf{N}_{2}(s_{2,1})$ %of $\partial\Omega(\gamma_{2})$
    (at $\gamma_{2}(s_{2,1})$) both have non-negative dot products with 
    $\gamma_{1}(s_{1}) - \gamma_{2}(s_{2,1})$. Since $\gamma(s_{1}) - \gamma_{2}(s_{2,1})$
    is parallel to $\mathbf{N}_{2}(s_{2,1})$, we obtain
    \begin{align*}
        0&\leq (\gamma(s_{1}) - \gamma_{2}(s_{2,1})) \cdot \mathbf{N}_{1}(s_{1})
        = \abs{\gamma_{1}(s_{1}) - \gamma_{2}(s_{2,1})}
        \mathbf{N}_{1}(s_{1})\cdot\mathbf{N}_{2}(s_{2,1}).
    \end{align*}
    This yields $\mathbf{N}_{1}(s_{1})\cdot\mathbf{N}_{2}(s_{2,1}) \geq 0$ if also $\gamma_{2}(s_{2,1}) \neq \gamma_{1}(s_{1})$, while  $\gamma_{2}(s_{2,1}) = \gamma_{1}(s_{1})$  implies
    $\mathbf{N}_{1}(s_{1})\cdot\mathbf{N}_{2}(s_{2,1}) \geq 0$ as well because $\Omega(\gamma_{1})\subseteq\Omega(\gamma_{2})$.
    We now obtain
    \begin{align*}
        \mathbf{T}_{1}(s_{1})\cdot\mathbf{T}_{2}(s_{2})
        \geq \mathbf{T}_{1}(s_{1})\cdot\mathbf{T}_{2}(s_{2,1})
        - \norm{\gamma_{2}}_{\dot{C}^{1,\beta}}(4kd)^{\beta} \geq
        -\frac{1}{12\cdot 15^{\beta}} \geq -\frac{1}{12},
    \end{align*}
    which together with Lemma~\ref{L3.3} and \eqref{2.2} shows that
    \begin{align*}
        \abs{\mathbf{T}_{1}(s_{1}) - \mathbf{T}_{2}(s_{2})}
        &= \left(\frac{2\abs{\mathbf{T}_{1}(s_{1})\cdot\mathbf{N}_{2}(s_{2})}^{2}}
        {1 + \mathbf{T}_{1}(s_{1})\cdot\mathbf{T}_{2}(s_{2})}\right)^{1/2} \\
        &\leq 18\max\set{\norm{\gamma_{1}}_{\dot{C}^{1,\beta}},
        \norm{\gamma_{2}}_{\dot{C}^{1,\beta}}}^{\frac{1}{1+\beta}}
        \abs{\gamma_{1}(s_{1}) - \gamma_{2}(s_{2})}^{\frac{\beta}{1+\beta}} \\
        &\leq 18\min\set{\Delta_{d_{1}}(\gamma_{1}),
        \Delta_{d_{2}}(\gamma_{2})}^{-\frac{\beta}{1+\beta}}
        \abs{\gamma_{1}(s_{1}) - \gamma_{2}(s_{2})}^{\frac{\beta}{1+\beta}},
    \end{align*}
and the proof is finished.
\end{proof}

Next, we obtain an angle estimate similar to Lemma~\ref{L3.3} which, instead of
a condition on angles of nearby crossings, assumes that the curves are close in the Fr\'{e}chet metric.
This estimate is used in our uniqueness proof, specifically in Appendix~\ref{S9}.

\begin{lemma}\label{L3.6}
    Let $\gamma_{i}\colon\ell_{i}\bbT\to\bbR^{2}$ ($i=1,2$) be
    $C^{1,\beta}$ closed curves parametrized by arclength,
    and let $\mathbf{T}_{i}\coloneqq\partial_{s}\gamma_{i}$ and
    $\mathbf{N}_{i}\coloneqq\mathbf{T}_{i}^{\perp}$.
    Then for any orientation-preserving homeomorphisms
    $\phi_{i}\colon\bbT\to\ell_{i}\bbT$ ($i=1,2$) and any $\xi\in\bbT$ we have
    \[
        \abs{\mathbf{T}_{1}(\phi_{1}(\xi))\cdot\mathbf{N}_{2}(\phi_{2}(\xi))}
        \leq 4\max\left\{\norm{\gamma_{1}}_{\dot{C}^{1,\beta}},
        \norm{\gamma_{2}}_{\dot{C}^{1,\beta}}\right\}^{\frac{1}{1+\beta}}
        \norm{\gamma_{1}\circ\phi_{1} - \gamma_{2}\circ\phi_{2}}
        _{L^{\infty}}^{\frac{\beta}{1+\beta}}
    \]
    and if
    $\norm{\gamma_{1}\circ\phi_{1} - \gamma_{2}\circ\phi_{2}}_{L^{\infty}}
    \leq {6^{-1-1/\beta}}\max\left\{\norm{\gamma_{1}}_{\dot{C}^{1,\beta}},
    \norm{\gamma_{2}}_{\dot{C}^{1,\beta}}\right\}^{-1/\beta}$, then
%      \[
%        \mathbf{T}_{1}(\phi_{1}(\xi))\cdot \mathbf{T}_{2}(\phi_{2}(\xi))
%        \geq 1 - \frac{8\max\left\{\norm{\gamma_{1}}_{\dot{C}^{1,\beta}},
%        \norm{\gamma_{2}}_{\dot{C}^{1,\beta}}\right\}^{\frac{2}{1+\beta}}
%        \norm{\gamma_{1}\circ\phi_{1} - \gamma_{2}\circ\phi_{2}}
%        _{L^{\infty}}^{\frac{2\beta}{1+\beta}}}
%        {1 - 3\max\left\{\norm{\gamma_{1}}_{\dot{C}^{1,\beta}},
%        \norm{\gamma_{2}}_{\dot{C}^{1,\beta}}\right\}^{\frac{1}{1+\beta}}
%        \norm{\gamma_{1}\circ\phi_{1} - \gamma_{2}\circ\phi_{2}}
%        _{L^{\infty}}^{\frac{\beta}{1+\beta}}} .
%    \]
    \[
        \mathbf{T}_{1}(\phi_{1}(\xi))\cdot \mathbf{T}_{2}(\phi_{2}(\xi))
        \geq 1 - 10\max\left\{\norm{\gamma_{1}}_{\dot{C}^{1,\beta}},
        \norm{\gamma_{2}}_{\dot{C}^{1,\beta}}\right\}^{\frac{2}{1+\beta}}
        \norm{\gamma_{1}\circ\phi_{1} - \gamma_{2}\circ\phi_{2}}
        _{L^{\infty}}^{\frac{2\beta}{1+\beta}}.
    \]
\end{lemma}

\begin{proof}
    Let us identify $\phi_{i}$ with its lifting $\bbR\to\bbR$
    (so that $\phi_{i}$ is strictly increasing and $\phi_{i}(\xi+n) = \phi_{i}(\xi) + n\ell_{i}$
    holds for all $\xi\in\bbR$ and $n\in\bbZ$)
    and consider $\gamma_{i}$, $\mathbf{T}_{i}$, $\mathbf{N}_{i}$
    as functions  on $\bbR$. 
    
    Then for any $\xi,\xi'\in\bbR$ we have
    \begin{align*}
         (\gamma_{1}(\phi_{1}(\xi')) & - \gamma_{2}(\phi_{2}(\xi')))
         \cdot \mathbf{N}_{1}(\phi_{1}(\xi))
        - (\gamma_{1}(\phi_{1}(\xi)) - \gamma_{2}(\phi_{2}(\xi)))
        \cdot \mathbf{N}_{1}(\phi_{1}(\xi))
        \\& %\quad\quad\quad
        = (\gamma_{1}(\phi_{1}(\xi')) - \gamma_{1}(\phi_{1}(\xi)))
        \cdot \mathbf{N}_{1}(\phi_{1}(\xi))
 %       \\&\quad\quad\quad
        - (\gamma_{2}(\phi_{2}(\xi')) - \gamma_{2}(\phi_{2}(\xi)))
        \cdot \mathbf{N}_{1}(\phi_{1}(\xi)).
    \end{align*}
    The mean value theorem and
    $\abs{\mathbf{T}_{1}(s)\cdot\mathbf{N}_{1}(\phi_{1}(\xi))}
    \leq \norm{\gamma_{1}}_{\dot{C}^{1,\beta}}\abs{s - \phi_{1}(\xi)}^{\beta}$
    for $s\in\bbR$ yield
    \begin{align*}
        \abs{(\gamma_{1}(\phi_{1}(\xi')) - \gamma_{1}(\phi_{1}(\xi)))
        \cdot\mathbf{N}_{1}(\phi_{1}(\xi))}
        &\leq \norm{\gamma_{1}}_{\dot{C}^{1,\beta}}
        \abs{\phi_{1}(\xi') - \phi_{1}(\xi)}^{1+\beta}
    \end{align*}
    and 
    \begin{align*}
        &\abs{(\gamma_{2}(\phi_{2}(\xi')) - \gamma_{2}(\phi_{2}(\xi)))
        \cdot\mathbf{N}_{1}(\phi_{1}(\xi))}
        \\&\quad\quad\quad
        \geq \abs{\mathbf{T}_{2}(\phi_{2}(\xi))\cdot\mathbf{N}_{1}(\phi_{1}(\xi))}
        \abs{\phi_{2}(\xi') - \phi_{2}(\xi)}
        - \norm{\gamma_{2}}_{\dot{C}^{1,\beta}}
        \abs{\phi_{2}(\xi') - \phi_{2}(\xi)}^{1+\beta}.
    \end{align*}
    Therefore
    \begin{align*}
        &\abs{\mathbf{T}_{2}(\phi_{2}(\xi))\cdot\mathbf{N}_{1}(\phi_{1}(\xi))}
        \abs{\phi_{2}(\xi') - \phi_{2}(\xi)}
        \\&\quad\quad
        \leq 2\norm{\gamma_{1}\circ\phi_{1} - \gamma_{2}\circ\phi_{2}}_{L^{\infty}}
        + \norm{\gamma_{1}}_{\dot{C}^{1,\beta}}
        \abs{\phi_{1}(\xi') - \phi_{1}(\xi)}^{1+\beta}
        + \norm{\gamma_{2}}_{\dot{C}^{1,\beta}}
        \abs{\phi_{2}(\xi') - \phi_{2}(\xi)}^{1+\beta},
    \end{align*}
    and similarly
    \begin{align*}
        &\abs{\mathbf{T}_{1}(\phi_{1}(\xi))\cdot\mathbf{N}_{2}(\phi_{2}(\xi))}
        \abs{\phi_{1}(\xi') - \phi_{1}(\xi)}
        \\&\quad\quad
        \leq 2\norm{\gamma_{1}\circ\phi_{1} - \gamma_{2}\circ\phi_{2}}_{L^{\infty}}
        + \norm{\gamma_{1}}_{\dot{C}^{1,\beta}}
        \abs{\phi_{1}(\xi') - \phi_{1}(\xi)}^{1+\beta}
        + \norm{\gamma_{2}}_{\dot{C}^{1,\beta}}
        \abs{\phi_{2}(\xi') - \phi_{2}(\xi)}^{1+\beta}.
    \end{align*}
    Since $\xi'\mapsto\max\left\{
        \abs{\phi_{1}(\xi')-\phi_{1}(\xi)},\abs{\phi_{2}(\xi')-\phi_{2}(\xi)}
    \right\}$ is a continuous surjection from $\bbR$ onto $[0,\infty)$,
    we conclude that with  $M\coloneqq \max\left\{\norm{\gamma_{1}}_{\dot{C}^{1,\beta}},
    \norm{\gamma_{2}}_{\dot{C}^{1,\beta}}\right\}$ we have 
    \begin{align*}
        \abs{\mathbf{T}_{1}(\phi_{1}(\xi))\cdot\mathbf{N}_{2}(\phi_{2}(\xi))}t
        \leq 2\norm{\gamma_{1}\circ\phi_{1} - \gamma_{2}\circ\phi_{2}}_{L^{\infty}}
        + 2Mt^{1+\beta}
    \end{align*}
    holds for all $t\geq 0$. Taking
    $t: = \left(\frac{\norm{\gamma_{1}\circ\phi_{1} - \gamma_{2}\circ\phi_{2}}_{L^{\infty}}}
    {M}\right)^{1/({1+\beta})}$ now shows the first claim of the lemma.
%    \begin{align*}
%        \abs{\mathbf{T}_{1}(\phi_{1}(\xi))\cdot\mathbf{N}_{2}(\phi_{2}(\xi))}
%        \leq 4M^{\frac{1}{1+\beta}}
%        \norm{\gamma_{1}\circ\phi_{1} - \gamma_{2}\circ\phi_{2}}_{L^{\infty}}^{\frac{\beta}{1+\beta}}.
%    \end{align*}
%The second claim follows immediately from the first.
%
    
    A similar argument with $\mathbf{T}_{i}(\phi_{i}(\xi))$
    in place of $\mathbf{N}_{i}(\phi_{i}(\xi))$ shows
    \begin{align*}
       (\phi_{1}(\xi') & -  \phi_{1}(\xi))  - (\phi_{2}(\xi') - \phi_{2}(\xi))
        \, \mathbf{T}_{1}(\phi_{1}(\xi))\cdot \mathbf{T}_{2}(\phi_{2}(\xi))
        \\&
        \leq  2\norm{\gamma_{1}\circ\phi_{1} - \gamma_{2}\circ\phi_{2}}_{L^{\infty}}
       +  \norm{\gamma_{1}}_{\dot{C}^{1,\beta}}\abs{\phi_{1}(\xi') - \phi_{1}(\xi)}^{1+\beta}
        + \norm{\gamma_{2}}_{\dot{C}^{1,\beta}}\abs{\phi_{2}(\xi') - \phi_{2}(\xi)}^{1+\beta}
    \end{align*}
    and
    \begin{align*}
       (\phi_{2}(\xi') & -  \phi_{2}(\xi))  - (\phi_{1}(\xi') - \phi_{1}(\xi))
        \, \mathbf{T}_{1}(\phi_{1}(\xi))\cdot \mathbf{T}_{2}(\phi_{2}(\xi))
        \\&
        \leq  2\norm{\gamma_{1}\circ\phi_{1} - \gamma_{2}\circ\phi_{2}}_{L^{\infty}}
       +  \norm{\gamma_{1}}_{\dot{C}^{1,\beta}}\abs{\phi_{1}(\xi') - \phi_{1}(\xi)}^{1+\beta}
        + \norm{\gamma_{2}}_{\dot{C}^{1,\beta}}\abs{\phi_{2}(\xi') - \phi_{2}(\xi)}^{1+\beta}
    \end{align*}
    for any $\xi,\xi'\in\bbR$.
    Since $\xi'\mapsto (\phi_{1}(\xi')-\phi_{1}(\xi)) + (\phi_{2}(\xi')-\phi_{2}(\xi))$
    is a continuous surjection from $\bbR$ to $\bbR$, adding these two inequalities yields
    \begin{align*}
        \mathbf{T}_{1}(\phi_{1}(\xi))\cdot \mathbf{T}_{2}(\phi_{2}(\xi)) \, t \geq t
        - 2M|t|^{1+\beta}
        - 4\norm{\gamma_{1}\circ\phi_{1} - \gamma_{2}\circ\phi_{2}}_{L^{\infty}}
    \end{align*}
    for all $t\in\bbR$.  Taking
    $t \coloneqq \left(\frac{2\norm{\gamma_{1}\circ\phi_{1} - \gamma_{2}\circ\phi_{2}}_{L^{\infty}}}
    {M}\right)^{1/(1+\beta)}$ now shows
    \begin{align*}
        \mathbf{T}_{1}(\phi_{1}(\xi))\cdot \mathbf{T}_{2}(\phi_{2}(\xi))
        \geq 1 - 6M^{\frac{1}{1+\beta}}
        \norm{\gamma_{1}\circ\phi_{1} - \gamma_{2}\circ\phi_{2}}_{L^{\infty}}^{\frac{\beta}{1+\beta}}.
    \end{align*}
%    (using $4\sort 2\le 6$).
    If we now assume that $\norm{\gamma_{1}\circ\phi_{1} - \gamma_{2}\circ\phi_{2}}_{L^{\infty}}
    \leq {6^{-1-1/\beta}}M^{-1/\beta}$, then
    $\mathbf{T}_{1}(\phi_{1}(\xi))\cdot \mathbf{T}_{2}(\phi_{2}(\xi))\geq 0$
follows for all $\xi\in\bbR$.  Since the first claim of the lemma yields
    $\abs{\mathbf{T}_{1}(\phi_{1}(\xi))\cdot \mathbf{N}_{2}(\phi_{2}(\xi))}\leq \frac{2}{3}$,
    the second claim now follows via
    $\sqrt{1 - x} \geq 1 - \frac{5}{8}x$ for $x\in\left[0,\frac{4}{9}\right]$.
    % which then implies
    % \begin{align*}
    %     \mathbf{T}_{1}(\phi_{1}(\xi))\cdot \mathbf{T}_{2}(\phi_{2}(\xi))
    %     &= 1 - \frac{\abs{\mathbf{T}_{1}(\phi_{1}(\xi))\cdot \mathbf{N}_{2}(\phi_{2}(\xi))}^{2}}
    %     {1 + \mathbf{T}_{1}(\phi_{1}(\xi))\cdot \mathbf{T}_{2}(\phi_{2}(\xi))} \\
    %     &\geq 1 - \frac{8M^{\frac{2}{1+\beta}}
    %     \norm{\gamma_{1}\circ\phi_{1} - \gamma_{2}\circ\phi_{2}}_{L^{\infty}}^{\frac{2\beta}{1+\beta}}}
    %     {1 - 3M^{\frac{1}{1+\beta}}
    %     \norm{\gamma_{1}\circ\phi_{1} - \gamma_{2}\circ\phi_{2}}_{L^{\infty}}^{\frac{\beta}{1+\beta}}}.
    % \end{align*}
\end{proof}

 Next we show that the angle of a crossing
of two curves yields a certain lower bound on how far one of the crossing curve segments departs from the other curve.
We will use this in Section~\ref{S8} when proving that crossings of $H^2$ patch boundaries
cannot develop from initial data with no crossings.

\begin{lemma}\label{L3.10}
    Let $\gamma_{i}\colon\ell_{i}\bbT\to\bbR^{2}$ ($i=1,2$) be $C^{1,\beta}$ closed curves
    parametrized by arclength such that
    $\operatorname{im}(\gamma_{1})\cap\operatorname{im}(\gamma_{2})\neq\emptyset$,
    and let $\mathbf{T}_{i}\coloneqq\partial_{s}\gamma_{i}$
    and $\mathbf{N}_{i}\coloneqq\mathbf{T}_{i}^{\perp}$.
    Let $J$ be a connected component of
    $\set{s\in\ell_{1}\bbT\colon \gamma_1(s)\notin\operatorname{im}(\gamma_{2})}$, let
    $s_{1}\in\partial J$, and let $s_{2}\in\ell_{2}\bbT$ be such that
    $\gamma_{1}(s_{1}) = \gamma_{2}(s_{2})$. Then there is $s_{1}'\in J$ such that
    \[
        \abs{\mathbf{T}_{1}(s_{1})\cdot\mathbf{N}_{2}(s_{2})}
        \leq 6\max\set{
            \norm{\gamma_{1}}_{\dot{C}^{1,\beta}}^{\frac{1}{1+\beta}},
            \Delta_{\norm{\gamma_{2}}_{\dot{C}^{1,\beta}}^{-1/\beta}}(\gamma_{2})
            ^{-\frac{\beta}{1+\beta}}
        }
        d(\gamma_{1}(s_{1}'), \operatorname{im}(\gamma_{2}))^{\frac{\beta}{1+\beta}}.
    \]
\end{lemma}

\begin{proof}
    Assume without loss that $\abs{\mathbf{T}_{1}(s_{1})\cdot\mathbf{N}_{2}(s_{2})}>0$,
 and let
    \[
        d\coloneqq \frac{1}{2\cdot 4^{1/\beta}}
        \min\set{
            \norm{\gamma_{1}}_{\dot{C}^{1,\beta}}^{-1/\beta},
            \Delta_{\norm{\gamma_{2}}_{\dot{C}^{1,\beta}}^{-1/\beta}}(\gamma_{2})
        }
        \qquad \text{and} \qquad
        h\coloneqq d\abs{\mathbf{T}_{1}(s_{1})\cdot\mathbf{N}_{2}(s_{2})}^{1/\beta}.
    \]
    Without loss assume that $s_{1}$ is the left end-point of $J$, and pick any $s_{1}'\in (s_{1},s_{1} + h]$. Then
    \[
    \abs{\gamma_{1}(s_{1}') - \gamma_{2}(s_{2})} \leq s_{1}' - s_{1} \leq h \le 
    \frac{1}{4}\norm{\gamma_{2}}_{\dot{C}^{1,\beta}}^{-1/\beta}
    \]
     holds by \eqref{2.2}, so
     Lemma~\ref{L3.2} with $\gamma\coloneqq\gamma_{2}$ and $x\coloneqq \gamma_{1}(s_{1}')$
    yields at least one $\tilde{I}_{i}$ (as defined in the lemma) such that $s_{2}\in\tilde{I}_{i}$.  Since
    $\abs{\gamma_{1}(s_{1}') - \gamma_{2}(s_{2})}
    \leq \frac{1}{4}\Delta_{\norm{\gamma_{2}}_{\dot{C}^{1,\beta}}^{-1/\beta}}(\gamma_{2})$,     for all $s\in\ell_{2}\bbT$ with    $\abs{s - s_{2}}\in\left[\norm{\gamma_{2}}_{\dot{C}^{1,\beta}}^{-1/\beta},
    \frac{\ell_{2}}{2}\right]$   we have
    \[
        \abs{\gamma_{1}(s_{1}') - \gamma_{2}(s)}
        \geq \Delta_{\norm{\gamma_{2}}_{\dot{C}^{1,\beta}}^{-1/\beta}}(\gamma_{2})
        - \abs{\gamma_{1}(s_{1}') - \gamma_{2}(s_{2})}
        > \abs{\gamma_{1}(s_{1}') - \gamma_{2}(s_{2})}.
    \]
 So there is $\tilde{I}_{i}$ containing
    $s_{2}$ whose center $s_{2}'$
    satisfies $\abs{\gamma_{1}(s_{1}') - \gamma_{2}(s_{2}')}
    = d(\gamma_{1}(s_{1}'), \operatorname{im}(\gamma_{2}))$.
    Then Lemma~\ref{L3.2}(4) shows that
    \beq\lb{111.40}
        \frac{\abs{s_{2}' - s_{2}}}{2}
        \leq \abs{\gamma_{1}(s_{1}') - \gamma_{2}(s_{2})}
        \leq s_{1}' - s_{1} \le h.
    \eeq
    Since the mean value theorem and
    $\abs{\mathbf{T}_{2}(s)\cdot\mathbf{N}_{2}(s_{2})}
    \leq \norm{\gamma_{2}}_{\dot{C}^{1,\beta}}\abs{s - s_{2}}^{\beta}$
    for $s\in\ell_{2}\bbT$ yield
    \begin{align*}
        \abs{(\gamma_{2}(s_{2}') - \gamma_{2}(s_{2}))\cdot\mathbf{N}_{2}(s_{2})}
        &\leq \norm{\gamma_{2}}_{\dot{C}^{1,\beta}}
        \abs{s_{2}' - s_{2}}^{1+\beta}
    \end{align*}
    and 
    \begin{align*}
        \abs{(\gamma_{1}(s_{1}') - \gamma_{1}(s_{1}))\cdot\mathbf{N}_{2}(s_{2})}
        &\geq \abs{\mathbf{T}_{1}(s_{1})\cdot\mathbf{N}_{2}(s_{2})}(s_{1}' - s_{1})
        - \norm{\gamma_{1}}_{\dot{C}^{1,\beta}}(s_{1}' - s_{1})^{1+\beta},
    \end{align*}
    we obtain
    \begin{equation*}\begin{aligned}
        \abs{\gamma_{1}(s_{1}') - \gamma_{2}(s_{2}')}
        &\geq \abs{
            (\gamma_{1}(s_{1}') - \gamma_{1}(s_{1}))\cdot\mathbf{N}_{2}(s_{2})
            - (\gamma_{2}(s_{2}') - \gamma_{2}(s_{2}))\cdot\mathbf{N}_{2}(s_{2})
        } \\
        &\geq (s_{1}' - s_{1})\abs{\mathbf{T}_{1}(s_{1})\cdot\mathbf{N}_{2}(s_{2})}
        - \norm{\gamma_{1}}_{\dot{C}^{1,\beta}}(s_{1}' - s_{1})^{1+\beta}
        - \norm{\gamma_{2}}_{\dot{C}^{1,\beta}}\abs{s_{2}' - s_{2}}^{1+\beta} \\
        &\geq (s_{1}' - s_{1})\abs{\mathbf{T}_{1}(s_{1})\cdot\mathbf{N}_{2}(s_{2})}
        \left(
            1 - \norm{\gamma_{1}}_{\dot{C}^{1,\beta}}d^{\beta}
            - 2\norm{\gamma_{2}}_{\dot{C}^{1,\beta}}(2d)^{\beta}
        \right) \\
        &\geq \frac{s_{1}' - s_{1}}{4}
        \abs{\mathbf{T}_{1}(s_{1})\cdot\mathbf{N}_{2}(s_{2})} >0,
    \end{aligned}\end{equation*}
    where in the third and fourth inequalities we used \eqref{111.40}
    and \eqref{2.2}, respectively.
    Since $s_{1}'\in(s_{1},s_{1}+h]$ was arbitrary, it follows that $(s_{1},s_{1}+h]\subseteq J$.
    Then letting $s_{1}' \coloneqq s_{1} + h$ and substituting $s_1'-s_1=h$ in  the last estimate finishes the proof.
\end{proof}

The next lemma shows that we can deform any $H^{2}$ positive simple closed curve $\gamma$
slightly so that the new curve lies strictly inside $\Omega(\gamma)$ and is still
a positive simple closed curve, while its $\dot{H}^{2}$-norm and $\Delta_{h}$ are controlled by the corresponding quantities for $\gamma$.
This  is used in the existence proof in Section~\ref{S6},
where we employ arbitrarily small perturbations of the initial data such that
all the perturbed initial curves are separated from each other.

%({\color{red}I couldn't find any Sobolev regularity results for level sets of
%solutions to the Poisson equation $\Delta\psi = 1$. If I recall correctly, the issue is that,
%something like this seems to be closely related to, and apparently usually studied
%in the context of, the boundary regularity of elliptic equations, but most of
%those boundary regularity results seem to concern how to pass the Sobolev regularity of
%the boundary data into the corresponding Sobolev regularity of the solution $\psi$.
%But unlike H\"{o}lder-type regularity results, it does not seem so clear how to pass this
%Sobolev regularity of the solution into the corresponding Sobolev regularity of level sets,
%or if it is possible at all.})

\begin{lemma}\label{L3.7}
    There is  $M>0$ such that
    for any $\ell>0$, any $H^{2}$ positive simple closed curve $\gamma\colon\ell\bbT\to\bbR^{2}$
    parametrized by arclength, and any $h\in\left(0,\norm{\gamma}_{\dot{C}^{1,1/2}}^{-2}\right]$
    and $c\in(0,1]$, the following holds with
    \begin{align*}
        \eps_{0} \coloneqq \min\left\{
            \frac{c \Delta_{h}(\gamma)}{10}, \,
            \frac{1}{17M\ell\norm{\gamma}_{\dot{C}^{1,1/2}}^{4}}
        \right\}.
    \end{align*}
    For any $\eps\in(0,\eps_{0}]$, there is
    an $H^{2}$ positive simple closed curve $\gamma_{\eps}$ such that
    \begin{enumerate}
        \item $d_{\mathrm{F}}(\gamma_{\eps},\gamma) \leq \eps$,

        \item $\operatorname{im}(\gamma_{\eps}) \subseteq \Omega(\gamma)$
        with $\Delta(\gamma_{\eps},\gamma) \geq \frac{\eps}{4}$,

        \item $\norm{\gamma_{\eps}}_{\dot{H}^{2}}
        \leq 4\norm{\gamma}_{\dot{H}^{2}}$,

        \item $\ell(\gamma_{\eps}) \geq 3h$ and
        $\Delta_{ch}(\gamma_{\eps}) \geq \frac{2c}{7}\Delta_{h}(\gamma)$.
    \end{enumerate}
\end{lemma}

\begin{proof}
    Let $\mathbf{T}\coloneqq\partial_{s}\gamma$,
    $\mathbf{N}\coloneqq\mathbf{T}^{\perp}$, and let $\eps_0$ be as above, with $M$ to be determined later.  
    
    First we claim that for any $(\eps,s,s_0)\in(0,\eps_{0}]\times(\ell\bbT)^2$ we have    $\abs{\gamma(s) - (\gamma(s_{0}) + \eps\mathbf{N}(s_{0}))} \geq \frac{\eps}{2}$.
To show this, first assume that $ \abs{s - s_{0}} \in( {\norm{\gamma}_{\dot{C}^{1,1/2}}^{-2}} , \frac{\ell}{2}]$. Then $ \abs{s - s_{0}}\ge h $ indeed yields
    \[
        \abs{\gamma(s) - (\gamma(s_{0}) + \eps\mathbf{N}(s_{0}))}
        \geq \abs{\gamma(s) - \gamma(s_{0})} - \eps
        \geq \Delta_{h}(\gamma) - \eps \geq 9\eps . % \geq \frac{1}{2}\eps.
    \]
If instead $\abs{s - s_{0}} \in (\eps, {\norm{\gamma}_{\dot{C}^{1,1/2}}^{-2}}]$, then
     for any $s'$ between $s$ and $s_0$ Lemma~\ref{L3.1} shows 
    \[
        \mathbf{T}(s')\cdot\mathbf{T}(s_{0})
        \geq 1 - \frac{1}{2}\norm{\gamma}_{\dot{C}^{1,1/2}}^{2}\abs{s' - s_{0}}
        \geq \frac{1}{2},
    \]
and thus
%    so taking the $\mathbf{T}(s_{0})$-component of   $\gamma(s) - (\gamma(s_{0}) + \eps\mathbf{N}(s_{0}))$ shows
    \begin{equation}\label{3.2}%\begin{aligned}
        \abs{\gamma(s) - (\gamma(s_{0}) + \eps\mathbf{N}(s_{0}))}
        \geq \abs{(\gamma(s) - \gamma(s_{0}))\cdot\mathbf{T}(s_{0})} 
       = \abs{\int_{s_{0}}^{s}\mathbf{T}(s')\cdot\mathbf{T}(s_{0})\,ds'}
        \geq \frac{\abs{s - s_{0}}}{2} 
     \geq \frac{\eps}{2}.
   % \end{aligned}
   \end{equation}
    Finally, if $\abs{s - s_{0}} \leq \eps$, then   
    % taking the $\mathbf{N}(s_{0})$-component of    $\gamma(s) - (\gamma(s_{0}) + \eps\mathbf{N}(s_{0}))$ shows
    \begin{align*}
        \abs{\gamma(s) - (\gamma(s_{0}) + \eps\mathbf{N}(s_{0}))}
        &\geq \eps - \abs{(\gamma(s) - \gamma(s_{0}))\cdot\mathbf{N}(s_{0})} \\
        &= \eps - \abs{\int_{s_{0}}^{s}(\mathbf{T}(s') - \mathbf{T}(s_{0}))
        \cdot \mathbf{N}(s_{0})\,ds'} \\
        &\geq \eps - \frac{2\norm{\gamma}_{\dot{C}^{1,1/2}}}{3}\abs{s - s_{0}}^{3/2}.
    \end{align*}
    The definition of $\eps_{0}$ and \eqref{2.2} show that
    \begin{align*}
        \frac{2\norm{\gamma}_{\dot{C}^{1,1/2}}}{3}\eps^{1/2} \leq
        \frac{2}{3}\left(\frac{\eps}{h}\right)^{1/2}
        \leq \frac{2}{3}\left(\frac{\eps}{\Delta_{h}(\gamma)}\right)^{1/2}
        \leq \frac{2}{3\sqrt{10}} \leq \frac{1}{2},
    \end{align*}
    so  we also obtain $\abs{\gamma(s) - (\gamma(s_{0}) + \eps\mathbf{N}(s_{0}))}
    \geq \frac{\eps}{2}$ in this case. 
    
    So    $d(\gamma(s_{0}) + \eps\mathbf{N}(s_{0}), \operatorname{im}(\gamma))
    \geq \frac{\eps}{2}$  for all $(\eps,s_0)\in(0,\eps_{0}]\times\ell\bbT$, while clearly
    $\gamma(s_{0}) + \eps\mathbf{N}(s_{0})\in \Omega(\gamma)$
    for all small enough $\eps>0$.  Hence $\gamma(s_{0}) + \eps\mathbf{N}(s_{0}) \in \Omega(\gamma)$ for all  $(\eps,s_0)\in(0,\eps_{0}]\times\ell\bbT$.

    Next, fix a smooth  $\psi\colon\bbR\to[0,\infty)$ such that
    $\operatorname{supp}(\psi)=[-1,1]$ and $\psi>0$ on $(-1,1)$.
    Let $N$ be the smallest integer such that
    $N\geq 16\ell\norm{\gamma}_{\dot{C}^{1,1/2}}^{2}$, then let $d\coloneqq\frac{\ell}{N}$ and
    $s_{1}, \dots ,s_{N}$ be points that  divide $\ell\bbT$ into
    $N$ subintervals of  length $d$. Since $\ell\norm{\gamma}_{\dot{C}^{1,1/2}}^{2} \geq 4$
    by Lemma~\ref{L3.1}, we know that $N\leq 17\ell \norm{\gamma}_{\dot{C}^{1,1/2}}^{2}$ and $d\in [\frac 1{17}\norm{\gamma}_{\dot{C}^{1,1/2}}^{-2}, \frac 1{16}\norm{\gamma}_{\dot{C}^{1,1/2}}^{-2}]$.
%    \begin{align*}
%        \frac{1}{17\norm{\gamma}_{\dot{C}^{1,1/2}}^{2}}
%        \leq d \leq \frac{1}{16\norm{\gamma}_{\dot{C}^{1,1/2}}^{2}}.
%    \end{align*}

    For each $k=1, \dots ,N$, let
    $\psi_{k}(s)\coloneqq\psi\left(\frac{s - s_{k}}{d}\right)$ for $s\in\ell\bbT$
    (with $s - s_{k}$ being considered as a number from $(-\frac\ell 2, \frac \ell2]$)
    and $\varphi_{k} \coloneqq \frac{\psi_{k}}{\sum_{l=1}^{N}\psi_{l}}$.  Then 
    $\set{\varphi_{k}}_{k=1}^{N}$ is a smooth partition of unity on $\ell\bbT$ such that at most two $\varphi_k$ are positive at any $s\in\ell\bbT$, and from $d\ge \frac1{17}\norm{\gamma}_{\dot{C}^{1,1/2}}^{-2}$ we see that there is  $M>0$ (depending only on  $\psi$) such that for each $k$ we have
    \begin{align*}
        \norm{\varphi_{k}'}_{L^{\infty}(\ell\bbT)} \leq M\norm{\gamma}_{\dot{C}^{1,1/2}}^{2}
        \qquad\textrm{and}\qquad
        \norm{\varphi_{k}''}_{L^{2}(\ell\bbT)} \leq M\norm{\gamma}_{\dot{C}^{1,1/2}}^{3}.
    \end{align*}
    
    This now defines $\eps_0$ as above. For any $\eps\in(0,\eps_0]$, let us  define $\tilde{\gamma}_{\eps}\colon\ell\bbT\to\bbR$ via
    \begin{align*}
        \tilde{\gamma}_{\eps}(s) \coloneqq
        \gamma(s) + \eps\sum_{k=1}^{N}\varphi_{k}(s)\mathbf{N}(s_{k})
    \end{align*}
    and let $\gamma_{\eps}$ be the curve defined by $\tilde{\gamma}_\eps$ (note that $\tilde{\gamma}_{\eps}$ is not necessarily    an arclength parametrization of $\gamma_{\eps}$).
    It is clear from this construction that (1) holds, and (2) also follows  because for each $s\in\ell\bbT$ we have
%    $d_{\mathrm{F}}(\gamma_{\eps},\gamma)\leq \eps$,
%    and we claim that
%    \[
%        \operatorname{im}(\gamma_{\eps})\subseteq\Omega(\gamma)
%        \qquad\textrm{and}\qquad
%        \Delta(\gamma_{\eps},\gamma)\geq\frac{\eps}{4}.
%    \]
%    Indeed, since 
$\gamma(s) + \eps\mathbf{N}(s) \in \Omega(\gamma)$,
    $d(\gamma(s) + \eps\mathbf{N}(s),\operatorname{im}(\gamma))
    \geq \frac{\eps}{2}$ , 
%    it suffices to show that $|\tilde{\gamma}_{\eps}(s)-(\gamma(s) + \eps\mathbf{N}(s))| \le \frac{\eps}{4}$ for
%    each $s\in\ell\bbT$. This holds because
%this holds because for each $s\in\ell\bbT$ we have
and
    \begin{align*}
        \abs{\tilde{\gamma}_{\eps}(s) - (\gamma(s) + \eps\mathbf{N}(s)) }
        &\leq \eps\sum_{k=1}^{N}\varphi_{k}(s)
        \abs{\mathbf{N}(s_{k}) - \mathbf{N}(s)} \\
        &\leq\eps\norm{\gamma}_{\dot{C}^{1,1/2}}
        \sum_{k=1}^{N}\varphi_{k}(s)\abs{s_{k} - s}^{1/2} 
        \leq \eps\norm{\gamma}_{\dot{C}^{1,1/2}}d^{1/2}
        \leq \frac{\eps}{4}.
    \end{align*}

    Next, %we note the following estimate on
  %  $\norm{\partial_{s}\tilde{\gamma}_{\eps} - \partial_{s}\gamma}_{L^{\infty}}$:
    for any $s\in\ell\bbT$ we have
    \begin{equation} \lb{111.1}
        \abs{\partial_{s}\tilde{\gamma}_{\eps}(s) - \partial_{s}\gamma(s)}
        \leq \eps\sum_{k=1}^{N}\abs{\varphi_{k}'(s)}
        \leq 2\eps M\norm{\gamma}_{\dot{C}^{1,1/2}}^{2}
        \leq \frac{2}{17\ell\norm{\gamma}_{\dot{C}^{1,1/2}}^{2}}
        \leq \frac{1}{34},
    \end{equation}
    where we again used $\ell\norm{\gamma}_{\dot{C}^{1,1/2}}^{2} \geq 4$.
%    the second inequality is because there are at most two $\varphi_{k}'(s)$'s
%    that are not zero, and the last inequality is by Lemma~\ref{L3.1}.
    This also yields the first claim in (4) via
    %holds because Lemma~\ref{L3.1} yields
    \[
       \ell(\gamma_{\eps})
        \geq \frac{33}{34}\ell \geq \frac{66}{17\norm{\gamma}_{\dot{C}^{1,1/2}}^{2}}
        \geq 3 h.
    \]
    
        To show the second claim in (4), consider arbitrary $s',s\in\ell\bbT$ such that $s'\le s$ and
    \[
        ch \leq \int_{s'}^{s}\abs{\partial_{s}\tilde{\gamma}_{\eps}(s'')}ds''
        \leq \frac{\ell(\gamma_{\eps})}{2}.
    \]
    From \eqref{111.1} we see that
      \[
      s - s' \le  \frac{34}{33}\int_{s'}^{s}\abs{\partial_{s}\tilde{\gamma}_{\eps}(s'')}ds''
      \le  \frac{17\ell(\gamma_{\eps})}{33}
       \leq \frac{35\ell}{66}
        \leq \ell - \frac{31}{66}\cdot \frac{4}{\norm{\gamma}_{\dot{C}^{1,1/2}}^{2}}
        \leq \ell - \frac{1}{\norm{\gamma}_{\dot{C}^{1,1/2}}^2}.
    \]  
%    We first show that
%    $ s - s' \leq \ell - \frac{1}{\norm{\gamma}_{\dot{C}^{1,1/2}}^{2}}$.
%    Indeed, since $\frac{33}{34}\leq \abs{\partial_{s}\tilde{\gamma}_{\eps}}
%    \leq \frac{35}{34}$, we have
%    \[
%        0 \leq \frac{33}{34}(s - s')
%        \leq \int_{s'}^{s}\abs{\partial_{s}\tilde{\gamma}_{\eps}(s'')}ds''
%        \leq \frac{\ell(\gamma_{\eps})}{2}
%        \leq \frac{35\ell}{68},
%    \]
%    which together with Lemma~\ref{L3.1} show
%    \[
%        0 \leq s - s' \leq \frac{35\ell}{66}
%        \leq \ell - \frac{31}{66}\cdot \frac{4}{\norm{\gamma}_{\dot{C}^{1,1/2}}^{2}}
%        \leq \ell - \frac{1}{\norm{\gamma}_{\dot{C}^{1,1/2}}}.
%    \]
 If also $s-s'\ge  {\norm{\gamma}_{\dot{C}^{1,1/2}}^{-2}}$, then either $ s - s' \in[h, \frac{\ell}{2}]$ or 
    $(s' + \ell) - s \in [h, \frac{\ell}{2}]$, and so
    \[
        \abs{\tilde{\gamma}_{\eps}(s) - \tilde{\gamma}_{\eps}(s')}
        \geq \abs{\gamma(s) - \gamma(s')} - 2\eps
        \geq \Delta_{h}(\gamma) - 2\eps
        \geq \frac{4}{5}\Delta_{h}(\gamma)
        \geq \frac{2c}{7}\Delta_{h}(\gamma).
    \]
    If instead $ s - s' <{\norm{\gamma}_{\dot{C}^{1,1/2}}^{-2}}$, then $\abs{\gamma(s) - \gamma(s')} \geq \frac{s - s'}{2}$ as in \eqref{3.2}.  Therefore
%    , in this case
%    \[
%        \abs{\gamma(s) - \gamma(s')} \geq \frac{s - s'}{2}
%    \]
%    holds.
    \[
        \abs{\tilde{\gamma}_{\eps}(s) - \tilde{\gamma}_{\eps}(s')}
        \geq \frac{s - s'}{2} - 2\eps
        \ge \frac{17}{35}  \int_{s'}^{s}\abs{\partial_{s}\tilde{\gamma}_{\eps}(s'')}ds'' -2\eps
        \ge  \frac{17ch}{35} -\frac{c\Delta_{h}(\gamma)}{5}
        \geq \frac{2c}{7}\Delta_{h}(\gamma),
    \]
    where we also used \eqref{2.2} at the end.  This proves the claim.

    It remains to prove (3). Recall that $\norm{\gamma_{\eps}}_{\dot{H}^{2}}$ is defined via any arclength parametrization of the curve (so 
%    it may differ from $\norm{\tilde\gamma_{\eps}}_{\dot{H}^{2}}$).
%More precisely,
%    $\norm{\gamma_{\eps}}_{\dot{H}^{2}}$ 
it is the $L^{2}$-norm of its curvature), and therefore
    \begin{align*}
        \norm{\gamma_{\eps}}_{\dot{H}^{2}}
        = \norm{\frac{\partial_{s}^{2}\tilde{\gamma}_{\eps}
        \cdot\partial_{s}\tilde{\gamma}_{\eps}^{\perp}}
        {\abs{\partial_{s}\tilde{\gamma}_{\eps}}^{3}}}_{L^{2}(\ell\bbT)}.
    \end{align*}
    Let $\kappa\coloneqq\partial_{s}\mathbf{T}\cdot\mathbf{N}$
    so that $\partial_{s}\mathbf{T} = \kappa \mathbf{N}$.  Then
    \begin{align*}
        \partial_{s}^{2}\tilde{\gamma}_{\eps}(s) & \cdot
        \partial_{s}\tilde{\gamma}_{\eps}(s)^{\perp}
        =
        \left(
            \kappa(s)\mathbf{N}(s)
            + \eps\sum_{k=1}^{N}\varphi_{k}''(s)\mathbf{N}(s_{k})
        \right)
        \cdot
        \left(
            \mathbf{N}(s)
            - \eps\sum_{k=1}^{N}\varphi_{k}'(s)\mathbf{T}(s_{k})
        \right) \\
        &= \kappa(s)\, \mathbf{T}(s)\cdot \partial_{s}\tilde{\gamma}_{\eps}(s)
        + \eps\sum_{k=1}^{N}\varphi_{k}''(s)(\mathbf{N}(s_{k})\cdot \mathbf{N}(s))
        - \eps^{2}\sum_{k\neq l}\varphi_{k}''(s)\varphi_{l}'(s)
        (\mathbf{N}(s_{k})\cdot\mathbf{T}(s_{l})).
    \end{align*}
    Since $\abs{\partial_{s}\tilde{\gamma}_{\eps}}\geq\frac{33}{34}$, we get
    \begin{align*}
        \norm{\frac{\partial_{s}^{2}\tilde{\gamma}_{\eps}
        \cdot\partial_{s}\tilde{\gamma}_{\eps}^{\perp}}
        {\abs{\partial_{s}\tilde{\gamma}_{\eps}}^{3}}}_{L^{2}}
        &\leq \left(\frac{34}{33}\right)^{2}\norm{\kappa}_{L^{2}}
        + \left(\frac{34}{33}\right)^{3}
        \eps N\norm{\varphi_{1}''}_{L^{2}}
        + \left(\frac{34}{33}\right)^{3}
        \eps^{2}N^{2}\norm{\varphi_{1}''}_{L^{2}}
        \norm{\varphi_{1}'}_{L^{\infty}} \\
        &\leq \left(\frac{34}{33}\right)^{2}\norm{\gamma}_{\dot{H}^{2}}
        + \left(\frac{34}{33}\right)^{3}\left(
            17M\eps\ell\norm{\gamma}_{\dot{C}^{1,1/2}}^{5}
            + 17^{2}M^{2}\eps^{2}\ell^{2}\norm{\gamma}_{\dot{C}^{1,1/2}}^{9}
        \right). 
%        \\
%        &\leq 4\norm{\gamma}_{\dot{H}^{2}}.
    \end{align*}
    The definition of $\eps_0$ and $\norm{\gamma}_{\dot{C}^{1,1/2}}\le \norm{\gamma}_{\dot{H}^{2}}$ now show that this is indeed less than $4\norm{\gamma}_{\dot{H}^{2}}$.
\end{proof}

The next lemma yields a sufficient condition for $\Delta_{h}(\gamma)$ to be achieved at
two points with arclength parameters differing by more than $h$.  Note that when that happens, the line joining such points is orthogonal to both tangent lines to $\gamma$ at these points.

\begin{lemma}\label{L3.8}
    Let $\gamma\colon\ell\bbT\to\bbR^{2}$ be a $C^{1,\beta}$ closed curve
    parametrized by arclength and
    $h\in\left(0,\norm{\gamma}_{\dot{C}^{1,\beta}}^{-1/\beta}\right]$.
    Then $\Delta_{h}(\gamma) > 0$ if and only if $\gamma$ is simple.
    And if $\Delta_{h}(\gamma) \in \left(0, \frac{1 + 4\beta}{2(1 + 2\beta)}h\right)$, then for any $s,s'\in\ell\bbT$ with
    $\abs{s - s'} \in[h, \frac{\ell}{2}]$ and
    $\abs{\gamma(s) - \gamma(s')} = \Delta_{h}(\gamma)$ we have
    $\abs{s - s'}>h$ and
    $\partial_{s}\gamma(s)\perp \gamma(s) - \gamma(s')\perp \partial_{s}\gamma(s')$.
\end{lemma}

\begin{proof}
    Let $\mathbf{T}\coloneqq \partial_{s}\gamma$. By Lemma~\ref{L3.1},
    for any $s,s'\in\ell\bbT$ with $\abs{s - s'} \leq h$ we have
    \begin{align*}
        \abs{\gamma(s) - \gamma(s')} &\geq
        \abs{(\gamma(s) - \gamma(s'))\cdot \mathbf{T}(s)}
        = \abs{\int_{s'}^{s}\mathbf{T}(s'')\cdot \mathbf{T}(s)\,ds''} \\
        &\geq \abs{s - s'} - \frac{1}{2(1+2\beta)}\norm{\gamma}_{\dot{C}^{1,\beta}}^{2}
        \abs{s - s'}^{1 + 2\beta} 
%        \\&\geq \left(
%            1 - \frac{1}{2(1+2\beta)}
%            \left(\frac{h}{\norm{\gamma}_{\dot{C}^{1,\beta}}^{-1/\beta}}\right)^{2\beta}
%        \right)\abs{s - s'}
        \geq \frac{1 + 4\beta}{2(1 + 2\beta)}\abs{s - s'}.
    \end{align*}
    This proves the first claim, as well as that $\abs{s - s'}>h$ whenever the hypotheses of the second claim hold.
%     Hence, if $\gamma$ is not simple, then any distinct $s,s'\in\ell\bbT$
%    such that $\gamma(s) = \gamma(s')$ must satisfy
%    $h<\abs{s - s'}\leq\frac{\ell}{2}$, so $\Delta_{h}(\gamma) = 0$ holds in that case.
%    Clearly, if $\Delta_{h}(\gamma) = 0$, then $\gamma$ is not simple.
    Since any such $(s,s')$ minimize $\abs{\gamma(s) - \gamma(s')}^{2}$
    on an open subset of $(\ell\bbT)^2$, we see that
    $(\gamma(s) - \gamma(s'))\cdot\partial_{s}\gamma(s)
    =0 = (\gamma(s) - \gamma(s'))\cdot\partial_{s}\gamma(s') $, and the proof is finished.
%    
%    Also, if $\Delta_{h}(\gamma) < \frac{1 + 4\beta}{2(1 + 2\beta)}h$, then
%    for any $s,s'\in\ell\bbT$ such that
%    $h \leq \abs{s - s'} \leq \frac{\ell}{2}$ and
%    $\Delta_{h}(\gamma) = \abs{\gamma(s) - \gamma(s')}$,
%    we must have $\abs{s - s'} > h$ because otherwise
%    \[
%        \Delta_{h}(\gamma) = \abs{\gamma(s) - \gamma(s')}
%        \geq \frac{1 + 4\beta}{2(1 + 2\beta)}\abs{s - s'}
%        = \frac{1 + 4\beta}{2(1 + 2\beta)}h
%    \]
%    holds. Hence, any such $(s,s')$ is a minimizer of $\abs{\gamma(s) - \gamma(s')}^{2}$
%    on an open subset of $(\ell\bbT)^2$, which shows
%    $(\gamma(s) - \gamma(s'))\cdot\partial_{s}\gamma(s)
%    = (\gamma(s) - \gamma(s'))\cdot\partial_{s}\gamma(s') = 0$.
\end{proof}

The last result in this appendix, which is extensively used in Section~\ref{S4}, provides a bound on the length of a curve $\gamma$ in terms of $\abs{\Omega(\gamma)} $ and $\Delta_{\norm{\gamma}_{\dot{C}^{1,\beta}}^{-1/\beta}}(\gamma)$.

\begin{lemma}\label{L3.9}
    Let $\gamma\colon\ell\bbT\to\bbR^{2}$ be a $C^{1,\beta}$
    positive simple closed curve parametrized by arclength. Then
    for any $h\in\left(0,\norm{\gamma}_{\dot{C}^{1,\beta}}^{-1/\beta}\right]$ we have
    $ \ell \le  \frac{ 30 \abs{\Omega(\gamma)} }{\Delta_{h}(\gamma)}$.
\end{lemma}

\begin{proof}
    Let $\mathbf{T}\coloneqq\partial_{s}\gamma$,
    $\mathbf{N}\coloneqq\mathbf{T}^{\perp}$, and
    $d\coloneqq\frac{1}{3}\norm{\gamma}_{\dot{C}^{1,\beta}}^{-1/\beta}$.
    Let $N$ be the largest integer such that $5dN < \ell$.
    Since $\ell \geq 2^{1+\frac{1}{2\beta}}(3d)
    \geq 6\sqrt{2}d$ by Lemma~\ref{L3.1}, we have
    $ N \geq \frac{\ell}{5d} - 1
        %\geq \left(\frac{1}{5} - \frac{1}{12\sqrt{2}}\right)\frac{\ell}{d}
        > \frac{\ell}{15d}$.
        
    Let $\set{s_{i}}_{i=1}^{N}$ be $N$ equally spaced points in $\ell\bbT$, so that
    $\abs{s_{i} - s_{j}} \geq \frac{\ell}{N} > 5d$ holds whenever $i\neq j$.
    For $i=1,\dots ,N$ let
    \[
        A_{i} \coloneqq \set{\gamma(s) + t\mathbf{N}(s_{i})
        \colon (s,t)\in[s_{i} - d, s_{i} + d] \times (0,\tfrac 12\Delta_{h}(\gamma))}.
    \]
    For each $s,s'\in\ell\bbT$ with $\abs{s - s'}\leq d$, Lemma~\ref{L3.1} shows that $\mathbf{T}(s)\cdot\mathbf{T}(s') \ge \frac 12$.
%    \begin{equation}\label{3.3}
%        \mathbf{T}(s)\cdot\mathbf{T}(s')
%        \geq 1 - \frac{1}{2}\norm{\gamma}_{\dot{C}^{1,\beta}}^{2}\abs{s - s'}^{2\beta}
%        \geq 1 - \frac{1}{2}\left(\frac{1}{3}\right)^{2\beta} \geq \frac{1}{2},
%    \end{equation}
    Hence $s\mapsto \gamma(s)\cdot \mathbf{T}(s_{i})$
    is strictly increasing on $[s_{i} - d, s_{i} + d]$ and
    $(\gamma(s_{i} + d) - \gamma(s_{i} - d))\cdot\mathbf{T}(s_{i}) \geq d$, for each $i$.
    It follows  that %$A_{i}$ is a path-connected subset of $\bbR^{2}$ with
    $\abs{A_{i}} \geq \frac{1}{2}\Delta_{h}(\gamma)d$, which will now yield the result once we also prove that $A_{1},\dots,A_N$ are contained in $\Omega(\gamma)$ and
    are pairwise disjoint (since $ N > \frac{\ell}{15d}$). 
        
       So assume that
    \[
        \gamma(s) + t\mathbf{N}(s_{i}) = \gamma(s') + t'\mathbf{N}(s_{j})
    \]
    holds for some $i\neq j$,
$s\in [s_{i} - d, s_{i} + d]$, %$s'\in [s_{j} - d, s_{j} + d]$
    and $t,t'\in[0,\frac 12\Delta_{h}(\gamma))$ with $t>0$ (the case $t'=0$ will be used to show that $A_i\subseteq\Omega(\gamma)$). Then
    \[
        \abs{\gamma(s) - \gamma(s')} \leq t + t' < \Delta_{h}(\gamma),
    \]
    so we must have $\abs{s - s'} < h \leq 3d$ and hence
    $\max\set{\abs{s - s_{i}}, \abs{s' - s_{i}}} \leq 4d$.
    If $t'=0$, Lemma~\ref{L3.1} now shows that
    $\mathbf{T}(s'')\cdot\mathbf{T}(s_{i})\geq \frac{1}{9}$ holds for all
    $s''$  between $s$ and $s'$.  Therefore
    \[
        0 = \abs{(\gamma(s) - \gamma(s'))\cdot\mathbf{T}(s_{i})}
        = \abs{\int_{s'}^{s}\mathbf{T}(s'')\cdot\mathbf{T}(s_{i})\,ds''}
        \geq \frac{\abs{s - s'}}{9},
    \]
    and we see that $s = s'$ and $t = 0$. So
    $A_{i}\cap\operatorname{im}(\gamma)=\emptyset$, and since
    $\gamma(s_{i}) + \eps\mathbf{N}(s_{i}) \in \Omega(\gamma)$
     for all small enough $\eps>0$, we obtain
    $A_{i} \subseteq \Omega(\gamma)$.

    Similarly, if we now assume that $s'\in [s_{j}-d, s_{j}+d]$ and $t' > 0$ instead of $t'=0$, then
    $\abs{s_{i} - s_{j}} \leq \abs{s - s'} + 2d \leq 5d$, which contradicts $i \neq j$.
    Hence $A_{1},\dots,A_N$ are also pairwise disjoint, and the proof is finished. 
%    Therefore, we conclude
%    \[
%        \abs{\Omega(\gamma)} \geq \sum_{i=1}^{N}\abs{A_{i}}
%        \geq \left(\frac{\ell}{10d}\right)
%        \left(\frac{1}{2}\Delta_{h}(\gamma)d\right)
%        = \frac{\ell\Delta_{h}(\gamma)}{20}.
%    \]
\end{proof}

%%%%%%%%%%%%%%%%%%%%%%%%%%%%%%%%%%%%%%%%%%%%%%
\section{Results involving the space of closed  curves}\label{SA}
%%%%%%%%%%%%%%%%%%%%%%%%%%%%%%%%%%%%%%%%%%%%%%

\begin{lemma}\label{LA.1}
    For any $\tilde{\gamma}_{1},\tilde{\gamma}_{2}\in C(\bbT;\bbR^{2})$ we have
    \[
        d_{\mathrm{F}}(\tilde{\gamma}_{1},\tilde{\gamma}_{2}) =
        \inf_{\phi}
        \norm{\tilde{\gamma}_{1} - \tilde{\gamma}_{2}\circ\phi}_{L^{\infty}},
    \]
    where the infimum is taken over all orientation-preserving
    diffeomorphisms $\phi \colon\bbT\to\bbT$.
\end{lemma}

\begin{proof}
    It suffices to show that any orientation-preserving homeomorphism
    $\phi\colon\bbT\to\bbT$ can be approximated arbitrarily well in $L^\infty(\bbT)$ by orientation-preserving diffeomorphisms. %uniformly approximating $\phi$ arbitrarily well.
    This can be done via convolution with a smooth delta function (see \eqref{111.22}).
%    Let $\tilde{\phi}\colon\bbR\to\bbR$ be the lifting of $\phi$, then we can set
%    $\psi$ to be the projection of $\tilde{\phi}$ convolved with an
%    approximate identity. Details are omitted.
    % $\psi$ to be the projection of the function $\tilde{\psi}\colon\bbR\to\bbR$ given as
    % \[
    %     \tilde{\psi}\colon x\mapsto
    %     \tilde{\phi}*\eta(0) +
    %     \frac{\tilde{\phi}*\eta(x) - \tilde{\phi}*\eta(0) + \delta x}{1+\delta},
    % \]
    % where $\delta>0$ is sufficiently small and $\eta\colon\bbR\to[0,\infty)$
    % is a smooth bump function with $\int\eta = 1$ sufficiently close to
    % the Dirac measure at $0$. Details are omitted.
\end{proof}

\begin{lemma}\label{LA.2}
    For any $R_{0},R_{2}\in[0,\infty)$, the set
    \begin{align*}
        X\coloneqq\set{\gamma\in\operatorname{CC}(\bbR^{2})
        \colon \norm{\gamma}_{L^{\infty}}\leq R_{0} \,\,\&\,\, 
        \norm{\gamma}_{\dot{H}^{2}}\leq R_{2}}
    \end{align*}
    is compact with respect to $d_{\mathrm{F}}$.
\end{lemma}

\begin{proof}
    For each $\gamma\in X$, choose any constant-speed parametrization
    $\tilde{\gamma}\in C(\bbT;\bbR^{2})$, and let $\tilde{X}$ be the resulting
    subset of $C(\bbT;\bbR^{2})$. We have $\partial_{\xi}^{2}\tilde{\gamma}(\xi)
    = \ell(\gamma)^{2}\partial_{s}^{2}\gamma(\ell(\gamma) \xi)$ when 
    $\gamma(\cdot)$ is an appropriate arclength parametrization, and 
    \eqref{2.1} shows that 
    $\norm{\partial_{\xi}^{2}\tilde{\gamma}}_{L^{2}(\bbT)}
    = \ell(\gamma)^{3/2}\norm{\gamma}_{\dot{H}^{2}} \leq
    \norm{\gamma}_{L^{\infty}}^{3}\norm{\gamma}_{\dot{H}^{2}}^{4}
    \leq R_{0}^{3}R_{2}^{4}$.
%satisfies $\norm{\tilde{\gamma}}_{L^{\infty}}\leq R_{0}$

    Hence $\tilde{X}$ is a bounded subset of $C^{1,1/2}(\bbT;\bbR^{2})$, which means that $\tilde{X}$ is precompact in
    $C^{1}(\bbT;\bbR^{2})$. If now $\gamma_{n}\in X$ ($n=1,2,\dots$) are arbitrary, there is a subsequence $\seq{\gamma_{n_{k}}}_{k=1}^{\infty}$
    such that $\tilde{\gamma}_{n_{k}} \to \tilde{\gamma}$ in $C^{1}(\bbT;\bbR^{2})$.
    Then $\tilde{\gamma}$ is a constant-speed parametrization of the curve
    $\gamma\in\operatorname{CC}(\bbR^{2})$ defined by $\tilde{\gamma}$, and 
    $\ell(\gamma) = \lim_{k\to\infty}\ell(\gamma_{n_{k}})
    \geq \frac{4}{R_{2}^{2}}$  by Lemma~\ref{L3.1}.
    From the uniform bound on
    $\norm{\partial_{\xi}^{2}\tilde{\gamma}_{n_{k}}}_{L^{2}(\bbT)}$ we see that after passing to a further subsequence, we can assume that    $\seq{\partial_{\xi}^{2}\tilde{\gamma}_{n_{k}}}_{k=1}^{\infty}$
  weakly converges to some $f\in L^{2}(\bbT;\bbR^{2})$.  This means that $f$ is the weak derivative of
    $\partial_{\xi}\tilde{\gamma}$, and since the $C^{1}$-convergence yields
    $\lim_{k\to\infty}\ell(\gamma_{n_{k}}) = \ell(\gamma)$, we also have
    \begin{align*}
        \norm{f}_{L^{2}(\bbT)} &\leq
        \liminf_{k\to\infty}\norm{\partial_{\xi}^{2}\tilde{\gamma}_{n_{k}}}_{L^{2}(\bbT)}
       % = \liminf_{k\to\infty}\ell(\gamma_{n_{k}})^{3/2}       \norm{\gamma_{n_{k}}}_{\dot{H}^{2}}
        = \ell(\gamma)^{3/2}\liminf_{k\to\infty}\norm{\gamma_{n_{k}}}_{\dot{H}^{2}}\le \ell(\gamma)^{3/2} R_2.
    \end{align*}
    It follows that $\norm{\gamma}_{\dot{H}^{2}} \leq R_2$,
%    \[
%        \norm{\gamma}_{\dot{H}^{2}} \leq
%        \liminf_{k\to\infty}\norm{\gamma_{n_{k}}}_{\dot{H}^{2}} \leq R_{2}.
%    \]
    so $\gamma\in X$ and hence $X$ is compact.
\end{proof}

\begin{corollary}\label{CA.3}
    The functional $\norm{\,\cdot\,}_{\dot{H}^{2}}
    \colon\operatorname{CC}(\bbR^{2})\to[0,\infty]$
    is lower semicontinuous.
\end{corollary}

\begin{proof}
    It suffices to show that the set
    $S\coloneqq \set{\gamma\in\operatorname{CC}(\bbR^{2})\colon
    \norm{\gamma}_{\dot{H}^{2}}\leq R_{2}}$
    is closed in $\operatorname{CC}(\bbR^{2})$ for any given $R_{2}\in[0,\infty)$, so assume that $\gamma_{n} \to \gamma$ in $\operatorname{CC}(\bbR^{2})$
    for some $\set{\gamma_{n}}_{n=1}^{\infty}$ in $S$.
    Then clearly   $\sup_{n}\norm{\gamma_{n}}_{L^{\infty}}<\infty$, so $\set{\gamma_n}_{n=1}^\infty$ is contained in a
    compact subset of $S$ by Lemma~\ref{LA.2}. Therefore $\gamma\in S$, and hence $S$ is closed.
\end{proof}

The next two results
follow easily from the winding number with respect to any fixed $x\in\bbR^{2}$
being constant near any $\gamma\in\operatorname{SC}(\bbR^{2})$ with $x\notin{\rm im}(\gamma)$, and we omit their proofs.

\begin{lemma}\label{LA.4}
    $\operatorname{PSC}(\bbR^{2})$ is a clopen subset of $\operatorname{SC}(\bbR^{2})$.
\end{lemma}

\begin{lemma}\label{LA.5}
    The sets
    \begin{align*}
        S_{1}&\coloneqq \set{(\gamma_{1},\gamma_{2})\in\operatorname{PSC}(\bbR^{2})^{2}\colon
        \Omega(\gamma_{1}) \subseteq \Omega(\gamma_{2})}, \\
        S_{2}&\coloneqq \set{(\gamma_{1},\gamma_{2})\in\operatorname{PSC}(\bbR^{2})^{2}\colon
        \Omega(\gamma_{1}) \supseteq \Omega(\gamma_{2})}, \\
        S_{3}&\coloneqq \set{(\gamma_{1},\gamma_{2})\in\operatorname{PSC}(\bbR^{2})^{2}\colon
        \Omega(\gamma_{1}) \cap \Omega(\gamma_{2}) = \emptyset}
    \end{align*}
    are all closed in $\operatorname{PSC}(\bbR^{2})^{2}$.
    In particular, $S_{1}\cup S_{2}$ and $S_{3}$ are disjoint clopen subsets
    of $S_{1}\cup S_{2}\cup S_{3}$.
    Moreover,   $\{ (\gamma_{1},\gamma_{2})    \in\operatorname{PSC}(\bbR^{2})^{2}\,:\, \Delta(\gamma_{1},\gamma_{2}) > 0\} = \bigcup_{i=1}^3 {\rm int}(S_i)$.
%     is the union of the interiors of
%    $S_{1}$, $S_{2},$ and $S_{3}$ as subsets
%    $\operatorname{PSC}(\bbR^{2})^{2}$.
\end{lemma}

Recall now that $\mathcal{L}$ is a fixed finite set of patch indices and
$Q(z)=\frac{1}{m(\theta)}\sum_{\lambda\in\mathcal{L}}
\abs{\theta^{\lambda}}\norm{z^{\lambda}}_{\dot{H}^{2}}^{2}$, with
$\theta^\lambda\in \bbR\setminus\set{0}$ the patch strength 
and $m(\theta)= \min_{\lambda\in\mathcal{L}}\abs{\theta^{\lambda}}$, is lower semi-continuous by Corollary~\ref{CA.3}.

\begin{lemma}\label{LA.6}
    For any $\tht\in(\bbR\setminus\set{0})^{\mathcal{L}}$ and
    $\lambda\in\mathcal{L}$, the functional
%    $M^{\lambda}\colon \operatorname{CC}(\bbR^{2})^{\mathcal{L}}\to [0,\infty)$ given by
    \[
        M^{\lambda}(z) \coloneqq \Delta_{1/Q(z)}(z^{\lambda})
    \]
    is upper semicontinuous on $\operatorname{CC}(\bbR^{2})^{\mathcal{L}}$.
\end{lemma}

\textit{Remark}. The proof below equally applies when $Q(z)$ is replaced by
$\norm{z^{\lambda}}_{\dot{H}^{2}}^{2}$ (or any lower semicontinuous functional
bounded below by $\norm{z^{\lambda}}_{\dot{H}^{2}}^{2}$),
so the lemma also holds in that setting.

\begin{proof}
    It is enough to show that for any $z\in\operatorname{CC}(\bbR^{2})^{\mathcal{L}}$
    and a sequence $\seq{z_{n}}_{n=1}^{\infty}$ converging to $z$  in
    $\operatorname{CC}(\bbR^{2})^{\mathcal{L}}$ and composed of non-trivial curves,
    \[
        \lim_{n\to\infty}\Delta_{1/Q(z_{n})}(z_{n}^{\lambda})
        \leq M^{\lambda}(z)
    \]
    holds whenever the limit exists.  We may also assume that  $\limsup_{n\to\infty}Q(z_{n}) < \infty$
    because otherwise the conclusion is trivial due to \eqref{2.2}, and after passing to a subsequence 
    we can even assume that $\exists \lim_{n\to\infty}Q(z_{n}) < \infty$.
    Letting
    $\tilde{\gamma}_{n}\colon\bbT \to \bbR^{2}$ be any constant-speed parametrization of $z_{n}^{\lambda}$,
    from \eqref{2.1} we see that
    \[
        \sup_{n\in\bbN}\norm{\partial_{\xi}^{2}\tilde{\gamma}_{n}}_{L^{2}(\bbT)}
        = \sup_{n\in\bbN}\ell(z_{n}^{\lambda})^{3/2}\norm{z_{n}^{\lambda}}_{\dot{H}^{2}}
        \leq \sup_{n\in\bbN}\norm{z_{n}^{\lambda}}_{L^{\infty}}^{3}
        \norm{z_{n}^{\lambda}}_{\dot{H}^{2}}^{4} < \infty.
    \]
    By passing to a further subsequence, we can therefore assume that
    $\seq{\tilde{\gamma}_{n}}_{n=1}^{\infty}$ converges to some constant-speed parametrization 
    $\tilde{\gamma}$ of $z^{\lambda}$ in $C^{1}(\bbT;\bbR^{2})$, as in the proof of Lemma~\ref{LA.2}. 

    If $z^{\lambda}$ is trivial, then this $C^{1}$ convergence shows that
    \[
        \lim_{n\to\infty}\Delta_{1/Q(z_{n})}(z_{n}^{\lambda})
        \leq \liminf_{n\to\infty}\frac{\ell(z_{n}^{\lambda})}{2}
        = \frac{\ell(z^{\lambda})}{2} = 0.
    \]
    So assume otherwise and take any $\xi,\xi'\in\bbT$ such that
    $\frac{1}{\ell(z^{\lambda})Q(z)} \leq \abs{\xi - \xi'} \leq \frac{1}{2}$ and
    \[
        \Delta_{1/Q(z)}(z^{\lambda})
        = \abs{\tilde{\gamma}(\xi) - \tilde{\gamma}(\xi')}.
    \]
    Without loss of generality
    assume that $\xi' = 0\le \xi \le\frac{1}{2}$.
    Since $Q(z) \leq \lim_{n\to\infty}Q(z_{n})$ holds by
    Corollary~\ref{CA.3}, for any given $\eps\in\left(0,\frac{1}{4}\right]$ and all large enough $n$ we have
    \[
        \frac{1}{\ell(z_{n}^{\lambda})Q(z_{n})}
        \leq \xi + \eps \leq \frac{1}{2} + \eps.
    \]
    If $\xi + \eps \leq \frac{1}{2}$, then 
    \[
        \Delta_{1/Q(z_{n})}(z_{n}^{\lambda})
        \leq \abs{\tilde{\gamma}_{n}(\xi + \eps) - \tilde{\gamma}_{n}(0)}.
    \]
    Otherwise
    \begin{align*}
        \frac{1}{2} > 1 - (\xi + \eps)
        \geq \frac{1}{2} - \eps \geq \frac{1}{\ell(z_{n}^{\lambda})Q(z_{n})}
    \end{align*}
   because Lemma~\ref{L3.1} shows that $\frac{1}{Q(z_{n})} \leq \frac{\ell(z_{n}^{\lambda})}{4}$, so again we have
    \[
        \Delta_{1/Q(z_{n})}(z_{n}^{\lambda})
        \leq \abs{\tilde{\gamma}_{n}(1) - \tilde{\gamma}_{n}(\xi + \eps)}
        = \abs{\tilde{\gamma}_{n}(\xi + \eps) - \tilde{\gamma}_{n}(0)}.
    \]
    Sending $n\to\infty$ and then $\eps\to 0^{+}$ now gives
    \[
        \lim_{n\to\infty}\Delta_{1/Q(z_{n})}(z_{n}^{\lambda}) \le \abs{\tilde{\gamma}(\xi) - \tilde{\gamma}(0)}
        = \Delta_{1/Q(z)}(z^{\lambda}),
    \]
      as desired.
\end{proof}

\begin{lemma}\label{LA.7}
    For any $\tht\in(\bbR\setminus\set{0})^{\mathcal{L}}$,
    the functional $L\colon\operatorname{PSC}(\bbR^{2})^{\mathcal{L}}\to [0,\infty]$    from \eqref{111.23} is lower semicontinuous.
\end{lemma}

\begin{proof}
    By Corollary~\ref{CA.3} and Lemma~\ref{LA.6}, it suffices to prove that
    for any given $\lambda\in\mathcal{L}$, the functional
    $z\mapsto \min_{\lambda'\notin\Sigma^{\lambda}(z)}\Delta(z^{\lambda},z^{\lambda'})$
    is upper semicontinuous. For each $\lambda'\in\mathcal{L}$, let
    \[
        A^{\lambda,\lambda'}\coloneqq 
        \begin{cases}
        \set{z\in\operatorname{PSC}(\bbR^{2})^{\mathcal{L}}\colon
        \Omega(z^{\lambda}) \subseteq \Omega(z^{\lambda'}) 
        \ \textrm{or}\         \Omega(z^{\lambda}) \supseteq \Omega(z^{\lambda'})}
        & \text{if }\theta^{\lambda}\theta^{\lambda'} > 0, \\
        \set{z\in\operatorname{PSC}(\bbR^{2})^{\mathcal{L}}\colon
        \Omega(z^{\lambda}) \cap \Omega(z^{\lambda'}) = \emptyset} 
        & \text{if }\theta^{\lambda}\theta^{\lambda'} < 0.
        \end{cases}
    \]
    Also let 
    %$I^{\lambda,\lambda'}\colon \operatorname{PSC}(\bbR^{2})^{\mathcal{L}}    \to\set{1,\infty}$ by 
    $I^{\lambda,\lambda'}(z) \coloneqq  \infty$ when
    $z\in A^{\lambda,\lambda'}$ and $I^{\lambda,\lambda'}(z) \coloneqq  1$ otherwise.
    Since 
        \[
        \min_{\lambda'\notin\Sigma^{\lambda}(z)}\Delta(z^{\lambda},z^{\lambda'})
        = \min_{\lambda'\in\mathcal{L}}
        I^{\lambda,\lambda'}(z)\Delta(z^{\lambda},z^{\lambda'})
    \]
    and $(\gamma_{1},\gamma_{2})\mapsto \Delta(\gamma_{1},\gamma_{2})$
    is clearly continuous on $\operatorname{PSC}(\bbR^{2})^2$,
    it suffices to show that each $I^{\lambda,\lambda'}$ is upper semicontinuous on
    $\operatorname{PSC}(\bbR^{2})^{\mathcal{L}}$. This follows from
    Lemma~\ref{LA.5}, which shows that $A^{\lambda,\lambda'}$ is closed there.
\end{proof}

\begin{corollary}\label{CA.8}
    For any $\tht\in(\bbR\setminus\set{0})^{\mathcal{L}}$ and $R_{0},M\in[0,\infty)$, the set
    \[
        X \coloneqq \set{z\in\operatorname{PSC}(\bbR^{2})^{\mathcal{L}} \colon
        \norm{\gamma}_{L^{\infty}} \leq R_{0}\,\,\&\,\, L(z) \leq M}
    \]
    is compact with respect to $d_{\mathrm{F}}$.
\end{corollary}

\begin{proof}
    By Lemmas \ref{LA.2} and \ref{LA.7}, it suffices to prove that
    if a sequence $\seq{z_{n}}_{n=1}^{\infty}$ in $X$ converges to some
    $z\in\operatorname{CC}(\bbR^{2})^{\mathcal{L}}$, then
    $z\in\operatorname{PSC}(\bbR^{2})^{\mathcal{L}}$.
    Note that Lemmas~\ref{LA.2} and \ref{LA.6} show that
    %$z^{\lambda}$ is non-trivial for each $\lambda\in\mathcal L$, as well as
    \[
        Q(z) \leq \frac{M}{2}
        \qquad\textrm{and}\qquad
        \min_{\lambda\in\mathcal{L}}\Delta_{1/Q(z)}(z^{\lambda})
        \geq \frac{1}{M}.
    \]
    Lemma~\ref{L3.8} then shows that each $z^{\lambda}$ is a simple closed curve,
    so $z\in\operatorname{PSC}(\bbR^{2})^{\mathcal{L}}$ follows by Lemma~\ref{LA.4}.
\end{proof}

\section{Proof of Lemma~\ref{L8.1}}\label{S9}
%%%%%%%%%%%%%%%%%%%%%%%%%%%%%%%%%%%%%%%%%%%%%%

We start by collecting some basic estimates
on the convolution of a function with a bump function.
For $f\in L^1(\ell\bbT)$
and $\sigma\in C_c(\bbR)$, we extend $f$ periodically on $\bbR$ and for $x\in \ell\bbT$ let 
%$\sigma * f \colon\ell\bbT\to\bbR$ as
\beq\lb{111.22}
    (\sigma * f)( x) \coloneqq  \int_{\bbR}f(x - y)\sigma(y)\,dy.
\eeq

\begin{lemma}\label{L9.1}
    Let $\ell_{1},\ell_{2}\in(0,\infty)$ and $\sigma\colon\bbR\to[0,\infty)$ be
    a smooth even function with $\operatorname{supp}(\sigma) \subseteq [-\ell_{1},\ell_{1}]$
    and $\int_{\bbR}\sigma(x)\,dx = 1$. Let $p\in[1,\infty]$ and 
 $\sigma_{r}(x)\coloneqq \frac{1}{r}\sigma(\frac xr)$ for $r\in(0,\infty)$.  
    \begin{enumerate}
        \item[(a)] $\norm{\sigma_{r} * f}_{L^{p}(\ell_{2}\bbT)} \leq
        \norm{f}_{L^{p}(\ell_{2}\bbT)}$
        for any $f\in L^{p}(\ell_{2}\bbT)$.

        \item[(b)] $\norm{\sigma_{r} * f}_{L^{\infty}(\ell_{2}\bbT)}
        \leq r^{-1/2} \left\lceil\frac{2r\ell_{1}}{\ell_{2}}\right\rceil^{1/2}
        \norm{\sigma}_{L^{2}(\bbR)}\norm{f}_{L^{2}(\ell_{2}\bbT)}$
        for any $f\in L^{2}(\ell_{2}\bbT)$.

        \item[(c)] $\norm{\sigma_{r}*f - f}_{L^{\infty}(\ell_{2}\bbT)} \leq
        r^{\beta}\ell_{1}^{\beta}\norm{f}_{\dot{C}^{0,\beta}(\ell_{2}\bbT)}$
        for any $\beta\in(0,1]$ and $f\in C^{0,\beta}(\ell_{2}\bbT)$.

        \item[(d)] $\norm{\sigma_{r}*f - f}_{L^{p}(\ell_{2}\bbT)} \leq
        r\ell_{1}\norm{\partial_{s}f}_{L^{p}(\ell_{2}\bbT)}$
        for any $f\in W^{1,p}(\ell_{2}\bbT)$.

        \item[(e)] $\norm{\sigma_{r}*f - f}_{L^{p}(\ell_{2}\bbT)} \leq
        \frac{r^{2}\ell_{1}^{2}}{2}\norm{\partial_{s}^{2}f}_{L^{p}(\ell_{2}\bbT)}$
        for any $f\in W^{2,p}(\ell_{2}\bbT)$.
    \end{enumerate}
\end{lemma}

\begin{proof}
    These follow easily from Jensen's inequality, Cauchy-Schwarz inequality,
    Taylor expansion, and evenness of $\sigma$.  The details are left to the reader.
\end{proof}

We will also use the following interpolation-type estimate.

\begin{lemma}\label{L9.2}
    For any $\ell\in(0,\infty)$ and $f\in H^{1}(\ell\bbT)$,
    \begin{align*}
        \norm{f}_{L^{\infty}(\ell\bbT)}
        &\leq \frac{1}{\ell}\abs{\int_{\ell\bbT}f(s)\,ds}
        + \norm{\partial_{s}f}_{L^{2}(\ell\bbT)}^{1/2}
        \norm{f - \frac{1}{\ell}\int_{\ell\bbT}f(s)\,ds}_{L^{2}(\ell\bbT)}^{1/2} \\
        &\leq \frac{1}{\sqrt{\ell}}\norm{f}_{L^{2}(\ell\bbT)}^{1/2}\left(
            \norm{f}_{L^{2}(\ell\bbT)}^{1/2}
            + \sqrt{\ell}\norm{\partial_{s}f}_{L^{2}(\ell\bbT)}^{1/2}
        \right).
    \end{align*}
\end{lemma}

\begin{proof}    
    Consider the unitary Fourier transform
    \begin{align*}
        \hat{f}(k) \coloneqq  \frac{1}{\sqrt{\ell}}\int_{\ell\bbT}
        e^{-\frac{2\pi iks}{\ell}}f(s)\,ds
    \end{align*}
    for all $k\in\bbZ$. For any $a>0$, Cauchy-Schwarz inequality shows that
    \begin{align*}
        \norm{f}_{L^{\infty}(\ell\bbT)}
        &\leq \frac{1}{\sqrt{\ell}}\norm{\hat{f}}_{L^{1}(\bbZ)}
        = \frac{1}{\ell}\abs{\int_{\ell\bbT}f(s)\,ds}
        + \frac{1}{\sqrt{\ell}}\sum_{k\neq 0}\abs{\hat{f}(k)} \\
        &\leq \frac{1}{\ell}\abs{\int_{\ell\bbT}f(s)\,ds} +
        \frac{1}{\sqrt{\ell}}
        \left(
            \sum_{k\neq 0}\frac{a}{1 + \abs{2\pi a k/\ell}^{2}}
        \right)^{1/2}
        \left(
            \sum_{k\neq 0}\left(\frac{1}{a} + a\abs{\frac{2\pi k}{\ell}}^{2}\right)
            \abs{\hat{f}(k)}^{2}
        \right)^{1/2} \\
        &\leq \frac{1}{\ell}\abs{\int_{\ell\bbT}f(s)\,ds}
        + \frac{1}{\sqrt{2}}
        \left(
            \frac{1}{a}\norm{f - \frac{1}{\ell}\int_{\ell\bbT}f(s)\,ds}_{L^{2}(\ell\bbT)}^{2}
            + a\norm{\partial_{s}f}_{L^{2}(\ell\bbT)}^{2}
        \right)^{1/2},
    \end{align*}
    where in the last inequality we used
    $\sum_{k\neq 0}\frac{a}{1+\abs{2\pi ak/\ell}^{2}}
    \leq \int_{-\infty}^{\infty}\frac{a}{1 + \abs{2\pi ax/\ell}^{2}}\,dx = \frac{\ell}{2}$.
    Optimizing the right-hand side for $a>0$ now gives the desired estimate.
\end{proof}

%\begin{proof}[Proof of Lemma~\ref{L8.1}]

Let  $\gamma_{1},\gamma_{2}\in\operatorname{PSC}(\bbR^{2})$ be as in the statement of Lemma \ref{L8.1} (we also assume  $d_{\mathrm{F}}(\gamma_{1},\gamma_{2})>0$ because otherwise we are done), fix their (arbitrary) arclength parametrizations, and let
    \begin{align*}
        \delta \coloneqq\min\left\{
            \frac{1}{512R_{2}^{2}},
            \frac{1}{(2^{56}\cdot 3^{10})
            R_{1}^{6}R_{2}^{14}}
        \right\}.
    \end{align*}
%    not the intrinsic norms of curves. Since there is no risk of confusion,
%    we will not write the domain of the functions
%    when we write norms for notational simplicity. Take
%    \begin{align*}
%        \delta \coloneqq\min\left\{
%            \frac{1}{512R_{2}^{2}},
%            \frac{1}{(2^{56}\cdot 3^{10})
%            R_{1}^{6}R_{2}^{14}}
%        \right\}.
%    \end{align*}
    We note that the second term in the definition of $\delta$ will only be used in the proof of (b), while first term will be used in all the other parts.  

%    Let curves $\gamma_{1},\gamma_{2}\in\operatorname{PSC}(\bbR^{2})$
%    with $d_{\mathrm{F}}(\gamma_{1},\gamma_{2})\leq \delta$,
%    \[
%        \ell(\gamma_{1}) \leq R_{1}
%        \quad\textrm{and}\quad
%        \max\set{\norm{\gamma_{1}}_{\dot{H}^{2}}^{2},
%        \frac{1}{\Delta_{1/\norm{\gamma_{2}}_{\dot{H}^{2}}^{2}}(\gamma_{2})}}
%        \leq R_{2}^{2}
%    \]
%    be given, and fix any arclength parametrizations of $\gamma_{1},\gamma_{2}$.
%    We may assume that $d_{\mathrm{F}}(\gamma_{1},\gamma_{2})>0$ since otherwise the conclusion
%    is trivial. 

    Fix a smooth even  $\sigma\colon\bbR\to[0,\infty)$
    with $\operatorname{supp}(\sigma) \subseteq [-\ell(\gamma_{1}),\ell(\gamma_{1})]$,
    $\int_{\bbR}\sigma(x)\,dx = 1$, as well as
    %. For calculational simplicity, we further assume that
    $\norm{\sigma}_{L^{2}(\bbR)} = \frac{1}{\ell(\gamma_{1})^{1/2}}$.
    % (we can, for instance,
    %consider the function that is constantly equal to $h$ on $[-w,w]$ and equal
    %to an appropriately scaled smooth transition function on
   % $[-\ell(\gamma_{1}),\ell(\gamma_{1})]\setminus[-w,w]$,
   % and then solve for $w$ and $h$). 
   Fix also the mollification scale
    $r\coloneqq \frac{64\norm{\gamma_{2}}_{\dot{H}^{2}}^{2}}{\ell(\gamma_{1})}
    d_{\mathrm{F}}(\gamma_{1},\gamma_{2})^{2}$, let $\sigma_{r}(x)\coloneqq \frac{1}{r}\sigma(\frac xr)$, and
    define $\hat{\gamma}_{1}\coloneqq \sigma_{r} * \gamma_{1}$ and
    $\hat{\gamma}_{2}\coloneqq \sigma_{r} * \gamma_{2}$. Note that $\hat{\gamma}_{i}$
    is a \emph{path} rather than \emph{curve}, but  for the sake of simplicity, we will use the same notation
    to denote the curve it defines.
    In general, $\hat{\gamma}_{i}$ is not an arclength parametrization of that curve, so the two can have different $\dot{C}^{k,\beta}$-norms and $\dot{H}^{2}$-norms. 
All norms below will be defined with respect to the relevant parametrizations (i.e., for paths), and we will suppress the corresponding domains in the notation.
% for the sake of simplicity.
Also, whenever needed,    we will consider $\gamma_{i}$ and $\hat \gamma_{i}$ to be  $\ell(\gamma_{i})$-periodic functions on     $\bbR$ rather than  functions on $\ell(\gamma_{i})\bbT$.

    By Lemma~\ref{L3.1} (with $\beta\coloneqq\frac{1}{2}$) and the definition of $\delta$, we have
    \[
        \frac{2r\ell(\gamma_{1})}{\ell(\gamma_{i})}
        \leq 32\norm{\gamma_{2}}_{\dot{H}^{2}}^{2}\norm{\gamma_{i}}_{\dot{H}^{2}}^{2}
        d_{\mathrm{F}}(\gamma_{1},\gamma_{2})^{2} \leq 1
    \]
   for $i=1,2$.  Hence Lemma~\ref{L9.1}(a,b) shows that
    $\norm{\partial_{s}^{2}\hat{\gamma}_{i}}_{L^{2}}
    \leq \norm{\gamma_{i}}_{\dot{H}^{2}}$ and
    \begin{equation}\label{9.1}
        \norm{\partial_{s}^{2}\hat{\gamma}_{i}}_{L^{\infty}}
        \leq \frac{1}{r^{1/2}}\left\lceil
            \frac{2r\ell(\gamma_{1})}{\ell(\gamma_{2})}
        \right\rceil^{1/2}
        \norm{\sigma}_{L^{2}}
        \norm{\gamma_{i}}_{\dot{H}^{2}}
        \leq \frac{\norm{\gamma_{i}}_{\dot{H}^{2}}}{r^{1/2}\ell(\gamma_{1})^{1/2}},
    \end{equation}
%    (which forces us to take    $r$ not too small), 
while Lemma~\ref{L9.1}(c) (with $\beta=1$) shows
    \begin{equation}\label{9.2}
        \norm{\hat{\gamma}_{i} - \gamma_{i}}_{L^{\infty}}
        \leq r\ell(\gamma_{1}).    
    \end{equation}
%     Note that \eqref{9.1} forces us to take
%    $r$ not too small.
    Definition of $\delta$ yields
    \begin{equation}\label{9.3}
        r\ell(\gamma_{1})
        \leq 64R_{2}^{2}\,d_{\mathrm{F}}(\gamma_{1},\gamma_{2})^{2}
        \leq \frac{d_{\mathrm{F}}(\gamma_{1},\gamma_{2})}8,
    \end{equation}
and so Lemma~\ref{L3.1} (with $\beta\coloneqq\frac{1}{2}$) shows
    \begin{equation}\label{9.4}
        \begin{aligned}
            \abs{1 - \abs{\partial_{s}\hat{\gamma}_{i}(s)}^{2}}
            &= \abs{\int_{\bbR^2}\sigma(x_{1})\sigma(x_{2})
            \left(
                1 - \partial_{s}\gamma_{i}(s - rx_{1})
                \cdot \partial_{s}\gamma_{i}(s - rx_{2})
            \right)
            \,dx_{1}\,dx_{2}} \\
            &\leq
            \frac{r}{2}\norm{\gamma_{i}}_{\dot{C}^{1,1/2}}^{2}
            \int_{\bbR^2}\sigma(x_{1})\sigma(x_{2})
            \abs{x_{1} - x_{2}}\,dx_{1}\,dx_{2} \\
            &\leq r\ell(\gamma_{1})\norm{\gamma_{i}}_{\dot{H}^{2}}^{2}
            \leq r\ell(\gamma_{1})R_{2}^{2}
            \leq \frac{d_{\mathrm{F}}(\gamma_{1},\gamma_{2})
            R_{2}^{2}}8 \leq \frac{1}{4}
        \end{aligned}
    \end{equation}
    for all $s\in\ell(\gamma_{i})\bbT$ ($i=1,2$),
    where the last inequality again uses the definition of $\delta$.

    By Lemma~\ref{LA.1}, there is an orientation-preserving
    diffeomorphism $\psi\colon\ell(\gamma_{1})\bbT\to\ell(\gamma_{2})\bbT$ such that
    $\norm{\gamma_{1} - \gamma_{2}\circ\psi}_{L^{\infty}}
    \leq \frac{3}{2}d_{\mathrm{F}}(\gamma_{1},\gamma_{2})$.
    By reparametrizing $\gamma_{2}$ if necessary, we can assume without loss of generality
    that $\psi(0) = 0$. Whenever needed, $\psi$ will also represent the corresponding $\ell(\gamma_{1})$-periodic function on $\bbR$.
    Consider now the ODE
    \[
            \partial_{t}\eta^{t}(s)
            = \left(\hat{\gamma}_{1}(s) - \hat{\gamma}_{2}(\eta^{t}(s)) \right)\cdot
            \partial_{s}\hat{\gamma}_{2}(\eta^{t}(s))
    \]
    on $C^{k}([a,b];\bbR)$ with $ \eta^{0}(s) = \psi(s)$ for each $s\in[a,b]$, where $k\geq 0$ and $[a,b]$ is an arbitrary compact interval   (with $\hat{\gamma}_{i},\psi$ being    functions on $\bbR$). Since $\hat{\gamma}_{1},\hat{\gamma}_{2}$ are smooth, Picard-Lindel\"{o}f theorem shows that there is a unique global solution $\eta\in C^1(\bbR; C^{k}([a,b];\bbR))$. Moreover, uniqueness  for the ODE with any fixed $s\in\bbR$ shows that
    $\eta^{t}(s + n\ell(\gamma_{1})) = \eta^{t}(s) + n\ell(\gamma_{2})$ holds for all $(s,t)\in[a,b]\times \bbR$ and  $n\in\bbZ$ such that $s + n\ell(\gamma_{1})\in[a,b]$, so we can extend $\eta^t$ to $\bbR$ for each $t\in\bbR$ with
    $\eta^{t}\in C^{\infty}(\bbR;\bbR)$ as well as
    $\eta^{t}\in C^{\infty}(\ell(\gamma_{1})\bbT;\ell(\gamma_{2})\bbT)$.

    We will obtain the desired 
    homeomorphism $\phi\colon\ell(\gamma_{1})\bbT\to\ell(\gamma_{2})\bbT$
    as $\phi\coloneqq \lim_{t\to\infty} \eta^{t}$, and show that
    $\abs{\partial_{t}\eta^{t}}\to 0$ uniformly  as $t\to\infty$, which will  yield
    $(\hat{\gamma}_{1} - \hat{\gamma}_{2}\circ\phi)
    \cdot (\partial_{s}\hat{\gamma}_{2}\circ\phi)\equiv 0$.
    Our proof of this requires that
    $\norm{\partial_{s}^{2}\hat{\gamma}_{2}}_{L^{\infty}}$ is not too large,
    thus $r$ cannot be too small (specifically, we will need
    $r\gtrsim d_{\mathrm{F}}(\gamma_{1},\gamma_{2})^{2}$ to conclude \eqref{9.6} below from \eqref{9.1}). 
    At the same time, $r$ must be small enough
    in order to translate estimates for $\hat{\gamma}_{1},\hat{\gamma_{2}}$
    into good enough estimates for $\gamma_{1},\gamma_{2}$ and vice versa, which is why we need precisely
    $r \sim d_{\mathrm{F}}(\gamma_{1},\gamma_{2})^{2}$ here.
%    the conclusions of Lemma~\ref{L8.1}.

From
    \begin{align*}
        \partial_{t}\abs{\hat{\gamma}_{1} - \hat{\gamma}_{2}\circ\eta^{t}}^{2}
        = -2\abs{\partial_{t}\eta^{t}}^{2}
    \end{align*}
    we see that
    $\abs{\hat{\gamma}_{1} - \hat{\gamma}_{2}\circ\eta^{t}}(s)$ is decreasing in $t\in\bbR$ for each $s\in\bbR$.
    Thus,  \eqref{9.2} and \eqref{9.3} show that
    \begin{equation}\label{9.5}
            \norm{\hat{\gamma}_{1} - \hat{\gamma}_{2}\circ\eta^{t}}_{L^{\infty}}
            \leq \norm{\hat{\gamma}_{1} - \hat{\gamma}_{2}\circ\psi}_{L^{\infty}} 
%            &\leq \norm{\gamma_{1} - \gamma_{2}\circ\psi}_{L^{\infty}}
%            + \norm{\hat{\gamma}_{1} - \gamma_{1}}_{L^{\infty}}
%            + \norm{\hat{\gamma}_{2} - \gamma_{2}}_{L^{\infty}} \\
            \leq \norm{\gamma_{1} - \gamma_{2}\circ\psi}_{L^{\infty}}
            + 2r\ell(\gamma_{1})
            \leq \frac{7}{4}d_{\mathrm{F}}(\gamma_{1},\gamma_{2}),    
    \end{equation}
    for each $t\ge 0$.
    Then \eqref{9.1} yields
    \begin{equation}\label{9.6}
        \norm{\hat{\gamma}_{1} - \hat{\gamma}_{2}\circ\eta^{t}}_{L^{\infty}}
        \norm{\partial_{s}^{2}\hat{\gamma}_{2}\circ\eta^{t}}_{L^{\infty}}
        \leq
        \frac{7}{4}d_{\mathrm{F}}(\gamma_{1},\gamma_{2})
        \cdot \frac{\norm{\gamma_{2}}_{\dot{H}^{2}}}{8\norm{\gamma_{2}}_{\dot{H}^{2}}
        d_{\mathrm{F}}(\gamma_{1},\gamma_{2})}
        \leq \frac{1}{4},
    \end{equation}
    and using \eqref{9.4} we conclude that for all $t\ge 0$ we have
    \begin{equation}\label{9.7}
        \frac{1}{2} \leq
        \abs{\partial_{s}\hat{\gamma}_{2}\circ\eta^{t}}^{2}
        - (\hat{\gamma}_{1} - \hat{\gamma}_{2}\circ\eta^{t})
        \cdot (\partial_{s}^{2}\hat{\gamma}_{2}\circ\eta^{t})
        \leq \frac{3}{2}.
    \end{equation}
    This implies
    \begin{align*}
        \partial_{t}\abs{\partial_{t}\eta^{t}}^{2}
        &= -2\abs{\partial_{t}\eta^{t}}^{2}\left(
            \abs{\partial_{s}\hat{\gamma}_{2}\circ\eta^{t}}^{2}
            - (\hat{\gamma}_{1} - \hat{\gamma}_{2}\circ\eta^{t})
            \cdot (\partial_{s}^{2}\hat{\gamma}_{2}\circ\eta^{t})
        \right)
        \leq -\abs{\partial_{t}\eta^{t}}^{2}
    \end{align*}
     for all $t\ge 0$, thus for all $t\ge 0$ and $s\in\bbR$ we have
    \begin{align*}
        \abs{\partial_{t}\eta^{t}(s)}^{2}
        \leq e^{-t} \abs{\partial_{t}\eta^{0}(s)}^{2}.
    \end{align*}
    
    Hence  $\int_{0}^{\infty}\norm{\partial_{t}\eta^{t}}_{L^{\infty}}\,dt<\infty$, so 
    %so we obtain a continuous function
    %$\phi\colon\ell(\gamma_{1})\bbT\to\ell(\gamma_{2})\bbT$ given as
    \begin{align*}
        \phi(s) \coloneqq  \lim_{t\to\infty} \eta^{t}(s) = \psi(s) + \int_{0}^{\infty}\partial_{t}\eta^{t}(s)\,dt \,\, \in C(\ell(\gamma_{1})\bbT;\ell(\gamma_{2})\bbT)
    \end{align*}
    is well defined and satisfies
    \begin{equation}\label{9.8}
        \abs{(\hat{\gamma}_{1} - \hat{\gamma}_{2}\circ\phi)\cdot
        (\partial_{s}\hat{\gamma}_{2}\circ\phi)} (s)
        = \lim_{t\to\infty}\abs{\partial_{t}\eta^{t}(s)} = 0
    \end{equation}
    for all $s\in \ell(\gamma_{1})\bbT$. From \eqref{9.2}, \eqref{9.3}, and \eqref{9.5} we see that
    \begin{equation}\label{9.9}
        \norm{\gamma_{1} - \gamma_{2}\circ\eta^{t}}_{L^{\infty}}
        \leq \norm{\hat{\gamma}_{1} - \hat{\gamma}_{2}\circ\eta^{t}}_{L^{\infty}}
        + 2r\ell(\gamma_{1})
        \leq 2d_{\mathrm{F}}(\gamma_{1},\gamma_{2})
    \end{equation}
    for all $t\ge 0$, so taking $t\to\infty$ yields the first inequality in (a).

    Next we obtain  lower and upper bounds on $\partial_{s}\eta^{t}$ for all large $t$.
    For any $t\in\bbR$ we have
    \begin{equation}\label{9.10}
        \begin{aligned}
            \partial_{t}(\partial_{s}\eta^{t})
            &= \partial_{s}\hat{\gamma}_{1}
            \cdot (\partial_{s}\hat{\gamma}_{2}\circ\eta^{t})
            - \left(
                \abs{\partial_{s}\hat{\gamma}_{2}\circ\eta^{t}}^{2}
                - (\hat{\gamma}_{1} - \hat{\gamma}_{2}\circ\eta^{t})
                \cdot (\partial_{s}^{2}\hat{\gamma}_{2}\circ\eta^{t})
            \right) \partial_{s}\eta^{t},
        \end{aligned}
    \end{equation}
    and \eqref{9.7} shows that the coefficient
    in front of $\partial_{s}\eta^{t}$ is from
    $[\frac{1}{2},\frac{3}{2}]$ for $t\ge 0$, while \eqref{9.4} yields
    \begin{equation}\label{9.11}
        \partial_{s}\hat{\gamma}_{1}
        \cdot(\partial_{s}\hat{\gamma}_{2}\circ\eta^{t})
        \leq (1+r\ell(\gamma_{1})R_{2}^{2})^{1/2}
        (1+r\ell(\gamma_{1})R_{2}^{2})^{1/2}
        \leq \frac{5}{4}
    \end{equation}
    (these will give us the upper bound in  \eqref{111.19} below).
    Lemma~\ref{L3.6}
    with $\phi_{1}(\xi)\coloneqq  \ell(\gamma_{1})\xi$ and
    $\phi_{2}(\xi)\coloneqq  \eta^{t}(\ell(\gamma_{1})\xi)$, together with \eqref{9.9}, shows
    \begin{equation}\label{8.11a}
        \begin{aligned}
            \partial_{s}\gamma_{1}\cdot
            (\partial_{s}\gamma_{2}\circ\eta^{t})
            &\geq 1 - 10\max\{\norm{\gamma_{1}}_{\dot{H}^{2}},
            \norm{\gamma_{2}}_{\dot{H}^{2}}\}^{4/3}
            \norm{\gamma_{1} - \gamma_{2}\circ\eta^{t}}_{L^{\infty}}^{2/3} \\
            &\geq  1 - 16R_{2}^{4/3}d_{\mathrm{F}}(\gamma_{1},\gamma_{2})^{2/3}
        \end{aligned}
    \end{equation}
    for $t\ge 0$ because $\norm{\gamma_{1} - \gamma_{2}\circ\eta^{t}}_{L^{\infty}}
    \leq 2d_{\mathrm{F}}(\gamma_{1},\gamma_{2})\leq \frac{1}{6^{3}R_{2}^{2}}$ by \eqref{9.9}.
    Also, Lemma~\ref{L9.1}(c) (with $\beta\coloneqq\frac{1}{2}$), the definition of $r$, and 
    \begin{equation}\label{9.12}
        d_{\mathrm{F}}(\gamma_{1},\gamma_{2})^{1/3}\leq
        \delta^{1/3}\leq \frac{1}{8R_{2}^{2/3}}
    \end{equation}
    show that
    \begin{align*}
        \norm{\partial_{s}\hat{\gamma}_{i} - \partial_{s}\gamma_{i}}_{L^{\infty}}
        &\leq r^{1/2}\ell(\gamma_{1})^{1/2}\norm{\gamma_{i}}_{\dot{C}^{1,1/2}}
        \leq 8\norm{\gamma_{2}}_{\dot{H}^{2}}\norm{\gamma_{i}}_{\dot{H}^{2}}
        d_{\mathrm{F}}(\gamma_{1},\gamma_{2}) \\
        &\leq 8R_{2}^{2}\,d_{\mathrm{F}}(\gamma_{1},\gamma_{2})
        \leq R_{2}^{4/3}d_{\mathrm{F}}(\gamma_{1},\gamma_{2})^{2/3}
        \leq \frac{1}{64}
    \end{align*}
    for $i=1,2$, so
    \begin{equation}\label{9.13}
        \begin{aligned}
            &\partial_{s}\hat{\gamma}_{1}
            \cdot(\partial_{s}\hat{\gamma}_{2}\circ\eta^{t})
            \\&\quad\quad\quad\geq
            \partial_{s}\gamma_{1}\cdot(\partial_{s}\gamma_{2}\circ\eta^{t})
            - \abs{\partial_{s}\gamma_{1}}
            \abs{\partial_{s}\gamma_{2}\circ\eta^{t}
            - \partial_{s}\hat{\gamma}_{2}\circ\eta^{t}}
            - \abs{\partial_{s}\hat{\gamma}_{2}\circ\eta^{t}}
            \abs{\partial_{s}\gamma_{1} - \partial_{s}\hat{\gamma}_{1}}
            \\&\quad\quad\quad\geq
            1 - 16R_{2}^{4/3}d_{\mathrm{F}}(\gamma_{1},\gamma_{2})^{2/3}
            - R_{2}^{4/3}d_{\mathrm{F}}(\gamma_{1},\gamma_{2})^{2/3}
            - \frac{65}{64}R_{2}^{4/3}d_{\mathrm{F}}(\gamma_{1},\gamma_{2})^{2/3}
            \\&\quad\quad\quad\geq
            1 - 19R_{2}^{4/3}d_{\mathrm{F}}(\gamma_{1},\gamma_{2})^{2/3}
            \geq \frac{1}{2}
        \end{aligned}
    \end{equation}
    for all $t\ge 0$. Therefore, for any $t\ge 0$ and $s\in\ell(\gamma_{1})\bbT$ we have
    \beq \lb{111.19}
        \frac{1}{2} - \frac{3}{2}\partial_{s}\eta^{t}(s)
        \leq \partial_{t} (\partial_{s}\eta^{t}(s)) \leq
        \frac{5}{4} - \frac{1}{2}\partial_{s}\eta^{t}(s)
    \eeq
     whenever $\partial_{s}\eta^{t}(s) > 0$.
    Since $\partial_{s}\eta^{0}(s)=\psi'(s) > 0$, \eqref{111.19} shows that $\partial_{s}\eta^{t}(s)>0$ for all $t\ge 0$ and $s\in\ell(\gamma_{1})\bbT$,
%    a Gr\"{o}nwall-type argument shows that
%    \begin{align*}
%        \frac{1}{3} + \left(\psi'(s) - \frac{1}{3}\right)e^{-\frac{3}{2}t}
%        \leq \partial_{s}\eta^{t}(s) \leq
%        \frac{5}{2} + \left(\psi'(s) - \frac{5}{2}\right)e^{-\frac{1}{2}t}
%    \end{align*}
%    holds at least until $\partial_{s}\eta^{t}(s)$ becomes zero.
%    Clearly, the left-hand side of the above is strictly positive
%    for all $t\in [0,\infty)$, so in fact the above holds
%    for all $t\in[0,\infty)$ and $s\in\ell(\gamma_{1})\bbT$.
    and then for any $\eps>0$ and all large enough $t\ge 0$ we have
    \begin{equation}\label{9.14}
        \frac{1}{3} - \eps \leq \partial_{s}\eta^{t}(s) \leq \frac{5}{2} + \eps
    \end{equation}
    for all  $s\in\ell(\gamma_{1})\bbT$.  This also proves (c).
%    In particular, $\eta^{t}\colon\ell(\gamma_{1})\bbT\to\ell(\gamma_{2})\bbT$
 %   is an orientation-preserving homeomorphism for all large $t$.
 %   This  shows that $\phi$ is Lipschitz continuous and $\phi'$ satisfies
%    the same bound almost everywhere as above for all $\eps>0$, showing (c).

    Next we estimate $\norm{\partial_{s}\eta^{t} - 1}_{L^{2}}$.
    The ODE \eqref{9.10} for $\partial_{s}\eta^{t}$
    and \eqref{9.7} show that
    \begin{align*}
        &\frac{1}{2}\partial_{t}\abs{\partial_{s}\eta^{t} - 1}^{2}
        = (\partial_{s}\eta^{t} - 1)
        \left(
            (\partial_{s}\hat{\gamma}_{1}\cdot
            (\partial_{s}\hat{\gamma}_{2}\circ\eta^{t}))
            - \abs{\partial_{s}\hat{\gamma}_{2}\circ\eta^{t}}^{2}
            + (\hat{\gamma}_{1} - \hat{\gamma}_{2}\circ\eta^{t})
            \cdot (\partial_{s}^{2}\hat{\gamma}_{2}\circ\eta^{t})
        \right) \\
        &\qquad\qquad\quad\quad\quad\quad
        - \left(
            \abs{\partial_{s}\hat{\gamma}_{2}\circ\eta^{t}}^{2}
            - (\hat{\gamma}_{1} - \hat{\gamma}_{2}\circ\eta^{t})
            \cdot (\partial_{s}^{2}\hat{\gamma}_{2}\circ\eta^{t})
        \right)
        \abs{\partial_{s}\eta^{t} - 1}^{2}
        \\&\quad
        \leq \abs{\partial_{s}\eta^{t} - 1}
        \left(
            \abs{(\partial_{s}\hat{\gamma}_{1}\cdot
            (\partial_{s}\hat{\gamma}_{2}\circ\eta^{t}))
            - \abs{\partial_{s}\hat{\gamma}_{2}\circ\eta^{t}}^{2}}
            + \abs{\hat{\gamma}_{1} - \hat{\gamma}_{2}\circ\eta^{t}}
            \abs{\partial_{s}^{2}\hat{\gamma}_{2}\circ\eta^{t}}
        \right)
        - \frac{\abs{\partial_{s}\eta^{t} - 1}^{2}}2 .
    \end{align*}
   From \eqref{9.11} and \eqref{9.13} we see that
    \begin{align*}
        \abs{\partial_{s}\hat{\gamma}_{1}
        \cdot(\partial_{s}\hat{\gamma}_{2}\circ\eta^{t}) - 1}
        &\leq \max\left\{
            r\ell(\gamma_{1})R_{2}^{2},
            19R_{2}^{4/3}d_{\mathrm{F}}(\gamma_{1},\gamma_{2})^{2/3}
        \right\} \\
        &\leq \max\left\{
            64R_{2}^{4}d_{\mathrm{F}}(\gamma_{1},\gamma_{2})^{2},
            19R_{2}^{4/3}d_{\mathrm{F}}(\gamma_{1},\gamma_{2})^{2/3}
        \right\} \\
        &= 19R_{2}^{4/3}d_{\mathrm{F}}(\gamma_{1},\gamma_{2})^{2/3}
    \end{align*}
    because
    \begin{align*}
        r\ell(\gamma_{1})R_{2}^{2}
        \leq 64R_{2}^{4}d_{\mathrm{F}}(\gamma_{1},\gamma_{2})^{2}
        \leq \frac{64R_{2}^{4}\cdot d_{\mathrm{F}}(\gamma_{1},\gamma_{2})^{2/3}}
        {2^{12}R_{2}^{8/3}}
        \leq \frac{R_{2}^{4/3}d_{\mathrm{F}}(\gamma_{1},\gamma_{2})^{2/3}}{64}
    \end{align*}
    holds by the definition of $r$ and \eqref{9.12}. Hence \eqref{9.4} shows that
    \begin{equation}\label{9.15}
        \begin{aligned}
            \abs{(\partial_{s}\hat{\gamma}_{1}\cdot
            (\partial_{s}\hat{\gamma}_{2}\circ\eta^{t}))
            - \abs{\partial_{s}\hat{\gamma}_{2}\circ\eta^{t}}^{2}}
            &\leq r\ell(\gamma_{1})R_{2}^{2}
            + 19R_{2}^{4/3}d_{\mathrm{F}}(\gamma_{1},\gamma_{2})^{2/3} \\
            &\leq 20R_{2}^{4/3}d_{\mathrm{F}}(\gamma_{1},\gamma_{2})^{2/3}.
        \end{aligned}
    \end{equation}
    On the other hand, Cauchy-Schwarz inequality and \eqref{9.14} show that
    \begin{align*}
        \int_{\ell(\gamma_{1})\bbT}\abs{\partial_{s}\eta^{t}(s) - 1}
        \abs{\partial_{s}^{2}\hat{\gamma}_{2}(\eta^{t}(s))}\,ds
        &\leq \norm{\partial_{s}\eta^{t} - 1}_{L^{2}}
        \left(4\int_{\ell(\gamma_{1})\bbT}
        \abs{\partial_{s}^{2}\hat{\gamma}_{2}(\eta^{t}(s))}^{2}
        \partial_{s}\eta^{t}(s)\,ds\right)^{1/2} \\
        &= 2\norm{\partial_{s}^{2}\hat{\gamma}_{2}}_{L^{2}}
        \norm{\partial_{s}\eta^{t} - 1}_{L^{2}}
    \end{align*}
    holds for all sufficiently large $t\ge 0$.
    This, \eqref{9.15}, Cauchy-Schwarz inequality, and \eqref{9.5} yield
    \begin{equation}\label{9.16}
        \begin{aligned}
            \frac{1}{2}\partial_{t}\norm{\partial_{s}\eta^{t} - 1}_{L^{2}}^{2}
            &\leq \left(
                20\ell(\gamma_{1})^{1/2}R_{2}^{4/3}
                d_{\mathrm{F}}(\gamma_{1},\gamma_{2})^{2/3}
                + \frac{7}{2}\norm{\gamma_{2}}_{\dot{{H}^{2}}}
                d_{\mathrm{F}}(\gamma_{1},\gamma_{2})
            \right)
            \norm{\partial_{s}\eta^{t} - 1}_{L^{2}}
            \\&\quad\quad\quad
            - \frac{1}{2}\norm{\partial_{s}\eta^{t} - 1}_{L^{2}}^{2} \\
            &\leq 21\ell(\gamma_{1})^{1/2}R_{2}^{4/3}
            d_{\mathrm{F}}(\gamma_{1},\gamma_{2})^{2/3}
            \norm{\partial_{s}\eta^{t} - 1}_{L^{2}}
            - \frac{1}{2}\norm{\partial_{s}\eta^{t} - 1}_{L^{2}}^{2}
        \end{aligned}
    \end{equation}
    for all large $t\ge 0$, where the last inequality uses
    \begin{align*}
        \frac{7}{2}\norm{\gamma_{2}}_{\dot{{H}^{2}}}
        d_{\mathrm{F}}(\gamma_{1},\gamma_{2})^{1/3}
        &\leq \frac{7R_{2}}{2}\, \frac{1}{8R_{2}^{2/3}}
        \leq \frac{7}{32}R_{2}^{4/3}\, \frac{2}{R_{2}},
    \end{align*}
    which holds by \eqref{9.12} and $\frac{2}{R_{2}} \leq \frac{2}{\norm{\gamma_{1}}_{\dot{C}^{1,1/2}}}
    \leq \ell(\gamma_{1})^{1/2}$ (the latter is due to Lemma~\ref{L3.1}).
    So a Gr\"{o}nwall-type argument shows that
    \begin{align*}
        \norm{\partial_{s}\eta^{t} - 1}_{L^{2}}
        &\leq 42\ell(\gamma_{1})^{1/2}R_{2}^{4/3}
        d_{\mathrm{F}}(\gamma_{1},\gamma_{2})^{2/3}
        \\&\quad\quad\quad
        + \left(
            \norm{\partial_{s}\eta^{t_{0}} - 1}_{L^{2}}
            - 42\ell(\gamma_{1})R_{2}^{4/3}
            d_{\mathrm{F}}(\gamma_{1},\gamma_{2})^{2/3}
        \right)e^{-\frac{1}{2}(t-t_{0})}
    \end{align*}
    holds for all $t\ge t_{0}$ when $t_0\ge 0$ is  large enough.
    In particular, for any $\eps>0$, we have
    \[
        \norm{\partial_{s}\eta^{t} - 1}_{L^{2}} \leq
        42\ell(\gamma_{1})^{1/2}R_{2}^{4/3}
        d_{\mathrm{F}}(\gamma_{1},\gamma_{2})^{2/3} + \eps
    \]
    for all large enough $t\ge 0$. Then a simple application of
    Banach-Alaoglu theorem shows that $\partial_{s}\eta^{t}$ converges
    weakly to $\phi'$ in $L^{2}(\ell(\gamma_{1})\bbT)$ and
    \beq\lb{111.20}
        \norm{\phi' - 1}_{L^{2}}
        \leq 42\ell(\gamma_{1})^{1/2}R_{2}^{4/3}
        d_{\mathrm{F}}(\gamma_{1},\gamma_{2})^{2/3}.
    \eeq
    
    We can now turn to bounding
    $d_{\mathrm{F}}(\gamma_{1},\gamma_{2})$
    in terms of $\norm{\gamma_{1} - \gamma_{2}\circ\phi}_{L^{2}}$. By Lemma~\ref{L9.2},
    \begin{align*}
        \norm{\gamma_{1} - \gamma_{2}\circ\phi}_{L^{\infty}}
        \leq \frac{\norm{\gamma_{1} - \gamma_{2}\circ\phi}_{L^{2}}^{1/2}}
        {\ell(\gamma_{1})^{1/2}}
        \left(
            \norm{\gamma_{1} - \gamma_{2}\circ\phi}_{L^{2}}^{1/2}
            + \ell(\gamma_{1})^{1/2}\norm{\partial_{s}\gamma_{1}
            - \partial_{s}(\gamma_{2}\circ\phi)}_{L^{2}}^{1/2}
        \right).
    \end{align*}
    We have
    \begin{align*}
        \norm{\partial_{s}\gamma_{1}
        - \partial_{s}(\gamma_{2}\circ\phi)}_{L^{2}}
        &\leq \norm{\partial_{s}\gamma_{1}
        - \partial_{s}\gamma_{2}\circ\phi}_{L^{2}}
        + \norm{\partial_{s}\gamma_{2}}_{L^{\infty}}
        \norm{\phi' - 1}_{L^{2}},
    \end{align*}
    and \eqref{8.11a} (which clearly also holds with $\phi$ in place of $\eta^{t}$) shows that
    \begin{align*}
        \abs{\partial_{s}\gamma_{1} - \partial_{s}\gamma_{2}\circ\phi}^{2}
        &= \abs{\partial_{s}\gamma_{1}}^{2}
        + \abs{\partial_{s}\gamma_{2}\circ\phi}^{2}
        - 2(\partial_{s}\gamma_{1}\cdot
        (\partial_{s}\gamma_{2}\circ\phi))
        \leq 32R_{2}^{4/3}d_{\mathrm{F}}(\gamma_{1},\gamma_{2})^{2/3}.
    \end{align*}
    From this, \eqref{9.12}, and \eqref{111.20} we now conclude (d) because
    \begin{align*}
        \norm{\partial_{s}\gamma_{1}
        - \partial_{s}(\gamma_{2}\circ\phi)}_{L^{2}}
        &\leq 6\ell(\gamma_{1})^{1/2}R_{2}^{2/3}
        d_{\mathrm{F}}(\gamma_{1},\gamma_{2})^{1/3}
        + 42\ell(\gamma_{1})^{1/2}R_{2}^{4/3}
        d_{\mathrm{F}}(\gamma_{1},\gamma_{2})^{2/3} \\
        &\leq 12\ell(\gamma_{1})^{1/2}R_{2}^{2/3}
        d_{\mathrm{F}}(\gamma_{1},\gamma_{2})^{1/3}.
    \end{align*}
    Using again $\frac{4}{R_{2}^{2}}
    \leq \frac{4}{\norm{\gamma_{1}}_{\dot{C}^{1,1/2}}^{2}}
    \leq \ell(\gamma_{1})$, as well as
    % we know from Lemma~\ref{L3.1},
    \eqref{9.9} with $\phi$ in place of $\eta^{t}$ 
    and \eqref{9.12}, yields
    \begin{align*}
        \norm{\gamma_{1} - \gamma_{2}\circ\phi}_{L^{2}}
        &\leq \ell(\gamma_{1})^{1/2}
        \norm{\gamma_{1} - \gamma_{2}\circ\phi}_{L^{\infty}}
        \leq 2\ell(\gamma_{1})^{1/2}d_{\mathrm{F}}(\gamma_{1},\gamma_{2}) \\
        &\leq 2\ell(\gamma_{1})^{1/2} \,
        \frac{1}{64R_{2}^{4/3}}\, d_{\mathrm{F}}(\gamma_{1},\gamma_{2})^{1/3}
        \leq \frac{\ell(\gamma_{1})^{3/2}}{128}R_{2}^{2/3}
        d_{\mathrm{F}}(\gamma_{1},\gamma_{2})^{1/3},
    \end{align*}
    which  shows that
    \begin{align*}
        d_{\mathrm{F}}(\gamma_{1},\gamma_{2})
        &\leq \norm{\gamma_{1} - \gamma_{2}\circ\phi}_{L^{\infty}} \\
        &\leq \frac{\norm{\gamma_{1} - \gamma_{2}\circ\phi}_{L^{2}}^{1/2}}
        {\ell(\gamma_{1})^{1/2}}
        \left(
            \frac{\ell(\gamma_{1})^{3/4}}{\sqrt{128}}R_{2}^{1/3}
            d_{\mathrm{F}}(\gamma_{1},\gamma_{2})^{1/6}
            + \sqrt{12}\ell(\gamma_{1})^{3/4}R_{2}^{1/3}
            d_{\mathrm{F}}(\gamma_{1},\gamma_{2})^{1/6}
        \right) \\
        &\leq 4\ell(\gamma_{1})^{1/4}R_{2}^{1/3}
        d_{\mathrm{F}}(\gamma_{1},\gamma_{2})^{1/6}
        \norm{\gamma_{1} - \gamma_{2}\circ\phi}_{L^{2}}^{1/2}.
    \end{align*}
    This now proves the second inequality in (a):
    \begin{equation}\label{9.17}
        d_{\mathrm{F}}(\gamma_{1},\gamma_{2})
        \leq 6\ell(\gamma_{1})^{3/10}R_{2}^{2/5}
        \norm{\gamma_{1} - \gamma_{2}\circ\phi}_{L^{2}}^{3/5}.
    \end{equation}

    Next we estimate $\norm{(\gamma_{1} - \gamma_{2}\circ\phi)\cdot
    (\partial_{s}\gamma_{2}\circ\phi)}_{L^{2}}$. By \eqref{9.8} we have
    \begin{align*}
        \norm{(\gamma_{1} - \gamma_{2}\circ\phi)\cdot
        (\partial_{s}\gamma_{2}\circ\phi)}_{L^{2}}
        &\leq \norm{\hat{\gamma}_{1} - \gamma_{1}}_{L^{2}}
        + \norm{\hat{\gamma}_{2}\circ\phi - \gamma_{2}\circ\phi}_{L^{2}}
        \\&\quad\quad
        + \norm{\hat{\gamma}_{1} - \hat{\gamma}_{2}\circ\phi}_{L^{\infty}}
        \norm{\partial_{s}\hat{\gamma}_{2}\circ\phi - \partial_{s}\gamma_{2}\circ\phi}_{L^{2}}.
    \end{align*}
    Note that (c), Lemma~\ref{L9.1}(d) (with $p=2$), and the definition of $r$  show that
    \begin{align*}
        \norm{\partial_{s}\hat{\gamma}_{2}\circ\phi
        - \partial_{s}\gamma_{2}\circ\phi}_{L^{2}}
        &\leq 2\norm{\partial_{s}\hat{\gamma}_{2} - \partial_{s}\gamma_{2}}_{L^{2}}
        \leq 2r\ell(\gamma_{1})\norm{\gamma_{2}}_{\dot{H}^{2}}
        \leq 128R_{2}^{3}\,d_{\mathrm{F}}(\gamma_{1},\gamma_{2})^{2},
    \end{align*}
    so using \eqref{9.5} with  $\phi$ in place of $\eta^{t}$ yields
    \begin{align*}
        \norm{\hat{\gamma}_{1} - \hat{\gamma}_{2}\circ\phi}_{L^{\infty}}
        \norm{\partial_{s}\hat{\gamma}_{2}\circ\phi - \partial_{s}\gamma_{2}\circ\phi}_{L^{2}}
        \leq 224R_{2}^{3}\,d_{\mathrm{F}}(\gamma_{1},\gamma_{2})^{3}.
    \end{align*}
    On the other hand, Lemma~\ref{L9.1}(e) (with $p=2$) and \eqref{9.3} show
    \begin{align*}
        \norm{\hat{\gamma}_{1} - \gamma_{1}}_{L^{2}}
        &\leq \frac{1}{2}\, \frac{d_{\mathrm{F}}(\gamma_{1},\gamma_{2})}8
        r\ell(\gamma_{1})R_{2}
        \leq 4R_{2}^{3}\,d_{\mathrm{F}}(\gamma_{1},\gamma_{2})^{3},
    \end{align*}
    and similarly
    \begin{align*}
        \norm{\hat{\gamma}_{2}\circ\phi - \gamma_{2}\circ\phi}_{L^{2}}
        &\leq 2\norm{\hat{\gamma}_{2} - \gamma_{2}}_{L^{2}}
        \leq 8R_{2}^{3}\,d_{\mathrm{F}}(\gamma_{1},\gamma_{2})^{3}.
    \end{align*}
    We therefore obtain
    \begin{align*}
        \norm{(\gamma_{1} - \gamma_{2}\circ\phi)\cdot
        (\partial_{s}\gamma_{2}\circ\phi)}_{L^{2}}
        \leq 236R_{2}^{3}\,d_{\mathrm{F}}(\gamma_{1},\gamma_{2})^{3},
    \end{align*}
    and combining this with \eqref{9.17} proves (e).

    Now we claim that for any $s\in\ell(\gamma_{1})\bbT$,
    $\phi(s)$ is a unique element in $\ell(\gamma_{2})\bbT$ such that
    \begin{equation}\label{9.18}
        \abs{\hat{\gamma}_{1}(s) - \hat{\gamma}_{2}(\phi(s))}
        = d(\hat{\gamma}_{1}(s), \operatorname{im}(\hat{\gamma}_{2})).
    \end{equation}
    (This would be false in general if $r$ were chosen too small
    because we will need sufficient control of $\norm{\partial_{s}^{2}\hat{\gamma}_{2}}_{L^{\infty}}$.)
 %   see the discussion before the statement of Lemma~\ref{L8.1}.)
    So fix $s\in\ell(\gamma_{1})\bbT$ and take any $s'\in\ell(\gamma_{2})\bbT$ such that
    $\abs{\hat{\gamma}_{1}(s) - \hat{\gamma}_{2}(s')}
    = d(\hat{\gamma}_{1}(s), \operatorname{im}(\hat{\gamma}_{2}))$.
    We can now regard $s',\phi(s)$ as elements of $\bbR$ with
    $\abs{s' - \phi(s)}\leq \frac{\ell(\gamma_{2})}{2}$.
    Since $(\hat{\gamma}_{1}(s) - \hat{\gamma}_{2}(s'))
    \cdot\partial_{s}\hat{\gamma}_{2}(s') = 0$, we obtain
    \begin{align*}
        0 &= (\hat{\gamma}_{1}(s) - \hat{\gamma}_{2}(\phi(s)))
        \cdot \partial_{s}\hat{\gamma}_{2}(\phi(s)) \\
        &= (\hat{\gamma}_{1}(s) - \hat{\gamma}_{2}(\phi(s)))
        \cdot \left(            \partial_{s}\hat{\gamma}_{2}(\phi(s))
            - \partial_{s}\hat{\gamma}_{2}(s')         \right)
        + (\hat{\gamma}_{2}(s') - \hat{\gamma}_{2}(\phi(s)))
        \cdot \partial_{s}\hat{\gamma}_{2}(s'),
    \end{align*}
    which yields the equality in
    \begin{align*}
        &        \abs{\hat{\gamma}_{1}(s) - \hat{\gamma}_{2}(\phi(s))}
     \norm{\partial_{s}^{2}\hat{\gamma}_{2}}_{L^{\infty}}   \abs{s' - \phi(s)} \\
        &\quad\quad\quad\quad\geq
        \abs{(\hat{\gamma}_{1}(s) - \hat{\gamma}_{2}(\phi(s)))
        \cdot (\partial_{s}\hat{\gamma}_{2}(\phi(s)) - \partial_{s}\hat{\gamma}_{2}(s'))} \\
        &\quad\quad\quad\quad=
        \abs{\int_{\phi(s)}^{s'}\partial_{s}\hat{\gamma}_{2}(s'')
        \cdot \partial_{s}\hat{\gamma}_{2}(s')\,ds''} \\
        &\quad\quad\quad\quad\geq
        \abs{\partial_{s}\hat{\gamma}_{2}(s')}^{2}\abs{s' - \phi(s)}
        - \abs{\int_{\phi(s)}^{s'}
            (\partial_{s}\hat{\gamma}_{2}(s'')
            - \partial_{s}\hat{\gamma}_{2}(s'))
            \cdot \partial_{s}\hat{\gamma}_{2}(s')\,ds''
        } \\
        &\quad\quad\quad\quad\geq
        \abs{\partial_{s}\hat{\gamma}_{2}(s')}^{2}\abs{s' - \phi(s)}
        - \frac{1}{2}\norm{\partial_{s}^{2}\hat{\gamma}_{2}}_{L^{\infty}}
        \abs{\partial_{s}\hat{\gamma}_{2}(s')}\abs{s' - \phi(s)}^{2}.
    \end{align*}
    If $s'\neq\phi(s)$, then
    this, \eqref{9.4}, and
    \eqref{9.6} with $\phi$ in place of $\eta^{t}$ show that
    \begin{align*}
        \frac{1}{2}\norm{\partial_{s}^{2}\hat{\gamma}_{2}}_{L^{\infty}}
        \abs{\partial_{s}\hat{\gamma}_{2}(s')}\abs{s' - \phi(s)}        
        \geq \abs{\partial_{s}\hat{\gamma}_{2}(s')}^{2}
        - \norm{\partial_{s}^{2}\hat{\gamma}_{2}}_{L^{\infty}}
        \norm{\hat{\gamma}_{1} - \hat{\gamma}_{2}\circ\phi}_{L^{\infty}}
        \geq \frac{1}{2},
    \end{align*}
    so \eqref{9.4}, \eqref{9.1}, and the definition of $r$ yield
    \beq\lb{111.21}
        \abs{s' - \phi(s)}
        \geq \frac{1}{\norm{\partial_{s}^{2}\hat{\gamma}_{2}}_{L^{\infty}}
        \abs{\partial_{s}\hat{\gamma}_{2}(s')}}
        \geq \frac{2r^{1/2}\ell(\gamma_{1})^{1/2}}
        {\sqrt{5}\norm{\gamma_{2}}_{\dot{H}^{2}}}
        = \frac{16}{\sqrt{5}}d_{\mathrm{F}}(\gamma_{1},\gamma_{2}).
    \eeq
    On the other hand, our choice of $s'$, \eqref{9.2}, \eqref{9.3},
    \eqref{9.5} with $\phi$ in place of $\eta^{t}$,
    and the definition of $\delta$ show
    \begin{align*}
        \abs{\gamma_{2}(s') - \gamma_{2}(\phi(s))}
        &\leq \abs{\gamma_{1}(s) - \gamma_{2}(s')}
        + \abs{\gamma_{1}(s) - \gamma_{2}(\phi(s))} \\
        &\leq \abs{\hat{\gamma}_{1}(s) - \hat{\gamma}_{2}(s')}
        + \abs{\hat{\gamma}_{1}(s) - \hat{\gamma}_{2}(\phi(s))}
        + 2\norm{\hat{\gamma}_{1} - \gamma_{1}}_{L^{\infty}}
        + 2\norm{\hat{\gamma}_{2} - \gamma_{2}}_{L^{\infty}} \\
        &\leq 2\norm{\hat{\gamma}_{1} - \hat{\gamma}_{2}\circ\phi}_{L^{\infty}}
        + 2\norm{\hat{\gamma}_{1} - \gamma_{1}}_{L^{\infty}}
        + 2\norm{\hat{\gamma}_{2} - \gamma_{2}}_{L^{\infty}} \\
        &\leq 2
            \norm{\hat{\gamma}_{1} - \hat{\gamma}_{2}\circ\phi}_{L^{\infty}}
            + 4r\ell(\gamma_{1}) \\
        &\leq 4d_{\mathrm{F}}(\gamma_{1},\gamma_{2})
        \leq 4\delta
        \leq \frac{1}{128R_{2}^{2}}
        < \Delta_{1/\norm{\gamma_{2}}_{\dot{H}^{2}}^{2}}(\gamma_{2}).
    \end{align*}
    It follows that $\abs{s' - \phi(s)} < \frac{1}{\norm{\gamma_{2}}_{\dot{H}^{2}}^{2}}$,
    and Lemma~\ref{L3.1} (with $\beta\coloneqq\frac{1}{2}$) then shows
    \begin{align*}
        \abs{\gamma_{2}(s') - \gamma_{2}(\phi(s))}
        &\geq \abs{\int_{\phi(s)}^{s'}
        \partial_{s}\gamma_{2}(s'')\cdot \partial_{s}\gamma_{2}(s')\,ds''} \\
        &\geq \abs{s' - \phi(s)} - \frac{\norm{\gamma_{2}}_{\dot{H}^{2}}^{2}}{4}
        \abs{s' - \phi(s)}^{2} > \frac{3}{4}\abs{s' - \phi(s)}.
    \end{align*}
    But this implies
    \[
        \abs{s' - \phi(s)} < \frac{4}{3}\abs{\gamma_{2}(s') - \gamma_{2}(\phi(s))}
        \leq \frac{16}{3}d_{\mathrm{F}}(\gamma_{1},\gamma_{2}),
    \]
    contradicting \eqref{111.21}.
    Hence we must have $s' =\phi(s)$.

    Next fix any $s\in\ell(\gamma_{1})\bbT$ and then any
    $s'\in\ell(\gamma_{2})\bbT$ such that $\abs{\gamma_{1}(s) - \gamma_{2}(s')}
    = d(\gamma_{1}(s), \operatorname{im}(\gamma_{2}))$.
    (Note that \eqref{9.18} concerns $\hat{\gamma}_{1},\hat{\gamma}_{2}$,
    not $\gamma_{1},\gamma_{2}$.) We can again regard $s',\phi(s)$ as elements of $\bbR$ with
    $\abs{s' - \phi(s)}\leq \frac{\ell(\gamma_{2})}{2}$. Then by \eqref{9.2},
    \begin{equation}\label{9.19}
        \begin{aligned}
            \abs{\gamma_{1}(s) - \gamma_{2}(\phi(s))}
            &\leq \abs{\hat{\gamma}_{1}(s) - \hat{\gamma}_{2}(\phi(s))}
            + \norm{\hat{\gamma}_{1} - \gamma_{1}}_{L^{\infty}}
            + \norm{\hat{\gamma}_{2} - \gamma_{2}}_{L^{\infty}} \\
            &\leq \abs{\hat{\gamma}_{1}(s) - \hat{\gamma}_{2}(s')}
            + \norm{\hat{\gamma}_{1} - \gamma_{1}}_{L^{\infty}}
            + \norm{\hat{\gamma}_{2} - \gamma_{2}}_{L^{\infty}} \\
            &\leq \abs{\gamma_{1}(s) - \gamma_{2}(s')}
            + 2\norm{\hat{\gamma}_{1} - \gamma_{1}}_{L^{\infty}}
            + 2\norm{\hat{\gamma}_{2} - \gamma_{2}}_{L^{\infty}} \\
            &\leq \abs{\gamma_{1}(s) - \gamma_{2}(s')}
            + 4r\ell(\gamma_{1}) \\
            &\leq \abs{\gamma_{1}(s) - \gamma_{2}(s')}
            + 256R_{2}^{2}\,d_{\mathrm{F}}(\gamma_{1},\gamma_{2})^{2},
        \end{aligned}
    \end{equation}
    so the definition of $\delta$ shows that
    \begin{align*}
        \abs{\gamma_{2}(s') - \gamma_{2}(\phi(s))}
        &\leq \abs{\gamma_{1}(s) - \gamma_{2}(s')}
        + \abs{\gamma_{1}(s) - \gamma_{2}(\phi(s))} \\
        &\leq 2\abs{\gamma_{1}(s) - \gamma_{2}(s')}
        + 256R_{2}^{2}\,d_{\mathrm{F}}(\gamma_{1},\gamma_{2})^{2} \\
        &\leq 2d_{\mathrm{F}}(\gamma_{1},\gamma_{2})
        + \frac{1}{2}d_{\mathrm{F}}(\gamma_{1},\gamma_{2})
        = \frac{5}{2}d_{\mathrm{F}}(\gamma_{1},\gamma_{2})
        < \Delta_{1/\norm{\gamma_{2}}_{\dot{H}^{2}}^{2}}(\gamma_{2}).
    \end{align*}
    So again $\abs{s' - \phi(s)} < \frac{1}{\norm{\gamma_{2}}_{\dot{H}^{2}}^{2}}$,
    and similarly to the proof of \eqref{9.18}, Lemma~\ref{L3.1} gives
    \begin{align*}
        \abs{s' - \phi(s)}
        &< \frac{4}{3}\abs{\gamma_{2}(s') - \gamma_{2}(\phi(s))}
        \leq \frac{8}{3}\abs{\gamma_{1}(s) - \gamma_{2}(s')}
        + \frac{1024}{3}R_{2}^{2}\,d_{\mathrm{F}}(\gamma_{1},\gamma_{2})^{2}.
    \end{align*}
    Hence our choice of $s'$ now yields (f):
    \[
        \abs{s' - \phi(s)}
        \leq 3d(\gamma_{1}(s),\operatorname{im}(\gamma_{2}))
        + 342R_{2}^{2}\,d_{\mathrm{F}}(\gamma_{1},\gamma_{2})^2.
    \]

    Finally, \eqref{9.19} and the choice of $s'$ imply
    \begin{align*}
        \abs{\gamma_{1}(s) - \gamma_{2}(\phi(s))}
        &\leq d(\gamma_{1}(s),\operatorname{im}(\gamma_{2}))
        + 256R_{2}^{2}\,d_{\mathrm{F}}(\gamma_{1},\gamma_{2})^{2}
    \end{align*}
     for all $s\in\ell(\gamma_{1})\bbT$, and we get
    \begin{equation*}
        \norm{\gamma_{1} - \gamma_{2}\circ\phi}_{L^{2}}
        \leq D(\gamma_{1},\gamma_{2})^{1/2}
        + 256R_{1}^{1/2}R_{2}^{2}\,d_{\mathrm{F}}(\gamma_{1},\gamma_{2})^{2}.
    \end{equation*}
    Then \eqref{9.17} shows that
    \begin{equation}\label{9.20}
        \norm{\gamma_{1} - \gamma_{2}\circ\phi}_{L^{2}}
        \leq D(\gamma_{1},\gamma_{2})^{1/2}
        + (2^{10}\cdot 3^{2})R_{1}^{11/10}R_{2}^{14/5}
        \norm{\gamma_{1} - \gamma_{2}\circ\phi}_{L^{2}}^{6/5},
    \end{equation}
    and since \eqref{9.9} with   $\phi$ in place of $\eta^{t}$ yields
    \begin{align*}
        \norm{\gamma_{1} - \gamma_{2}\circ\phi}_{L^{2}}
        \leq 2\ell(\gamma_{1})^{1/2}d_{\mathrm{F}}(\gamma_{1},\gamma_{2}),
    \end{align*}
    using the definition of $\delta$ we also have
    \begin{align*}
        (2^{10}\cdot 3^{2})R_{1}^{11/10}
        R_{2}^{14/5}
        \norm{\gamma_{1} - \gamma_{2}\circ\phi}_{L^{2}}^{1/5}
        \leq (2^{10}\cdot 3^{2})\cdot 2^{1/5}
        R_{1}^{6/5}R_{2}^{14/5}
        d_{\mathrm{F}}(\gamma_{1},\gamma_{2})^{1/5}
        \leq \frac{1}{2}.
    \end{align*}
    This and \eqref{9.20} now conclude the proof of (b):
    \begin{align*}
        \norm{\gamma_{1} - \gamma_{2}\circ\phi}_{L^{2}}
        \leq 2D(\gamma_{1},\gamma_{2})^{1/2}.
    \end{align*}
%\end{proof}

%%%%%%%%%%%%%%%%%%%%%%%%%%%%%%%%%%%%%%%%%%%%%%%%%%%%%%%%%%%%%%%%%%%

\end{document}